\documentclass[letterpaper]{amsart}

\usepackage[letterpaper]{geometry}
\usepackage{amsmath, amsthm, amsfonts, amsbsy, thmtools, amssymb,bm,xfrac,stmaryrd,mathrsfs,bbm}
\usepackage[sans]{dsfont}

\usepackage{thm-restate}
\usepackage{hyperref}
\usepackage[mathscr]{euscript}
\usepackage{tikz}
\usetikzlibrary{arrows}
\usepackage{cleveref}
\SetSymbolFont{stmry}{bold}{U}{stmry}{m}{n}

\numberwithin{equation}{section} 

\DeclareMathOperator{\fG}{fG}
\DeclareMathOperator{\SV}{SV}

\DeclareMathOperator{\Comp}{Comp}

\DeclareMathOperator{\scL}{scL}
\DeclareMathOperator{\FV}{FV}
\DeclareMathOperator{\hs}{hs}
\DeclareMathOperator{\dqb}{dqb}
\DeclareMathOperator{\we}{we}
\DeclareMathOperator{\cL}{cL}
\DeclareMathOperator{\qdp}{qp}
\DeclareMathOperator{\bq}{b}

\newtheorem{thm}{Theorem}[section]
\newtheorem{prop}[thm]{Proposition}
\newtheorem{lem}[thm]{Lemma}
\newtheorem{cor}[thm]{Corollary}

\theoremstyle{remark}
\newtheorem{rem}[thm]{Remark}
\theoremstyle{definition}
\newtheorem{definition}[thm]{Definition}

\title{Colored Line Ensembles for Stochastic Vertex Models}

\author{Amol Aggarwal and Alexei Borodin}

\date{}

\begin{document}

\begin{abstract} 

In this paper we assign a family of $n$ coupled line ensembles to any $U_q (\widehat{\mathfrak{sl}}_{n+1})$ colored stochastic fused vertex model, which satisfies two properties. First, the joint law of their top curves coincides with that of the colored height functions for the vertex model. Second, the $n$ line ensembles satisfy an explicit Gibbs property prescribing their laws if all but a few of their curves are conditioned upon. We further describe several examples of such famlies of line ensembles, including the ones for the colored stochastic six-vertex and $q$-boson models. The appendices (which may be of independent interest) include an explanation of how the $U_q (\widehat{\mathfrak{sl}}_{n+1})$ colored stochastic fused vertex model degenerates to the log-gamma polymer, and an effective rate of convergence of the colored stochastic six-vertex model to the colored ASEP.

\end{abstract}	
	
	\maketitle
	
	\tableofcontents

	\section{Introduction} 
	
	\label{Introduction}	
	
	\subsection{Preface}

	Over the past twenty-five years, a striking interplay has materialized between equilibrium random surface models and out-of-equilibrium stochastic growth systems. One of the first such correspondences is due to Jockusch--Propp--Shor \cite{RTT} in 1998, who showed an equality in law between the facet edge for a uniformly random domino tiling of the Aztec diamond, and the height function of a certain discrete-time totally asymmetric simple exclusion process (TASEP). Using ideas of Rost \cite{NBPP}, they proved a hydrodynamical limit theorem for the latter TASEP, which together with their matching result implied that the limiting trajectory for the domino tiling facet boundary is a circle. 
	
	Such correspondences have also been fruitful in reverse (namely, to study stochastic growth models through random surfaces), starting with the work of Pr\"{a}hofer--Spohn \cite{SIDP}. They analyzed the polynuclear growth (PNG) model by introducing an associated \emph{line ensemble}, which is a sequence of random curves (that may be viewed as level lines of a surface model), whose top curve coincides in law with the PNG height function. Using the solvability of this line ensemble through the framework of determinantal point processes, they (and also subsequently Johansson \cite{DPGDP}) showed that its fluctuations converge to a scaling limit called the Airy line ensemble (which they introduced in \cite{SIDP} as a determinantal point process with the extended Airy correlation kernel), now known to be a universal object in the Kardar--Parisi--Zhang (KPZ) universality class \cite{DSGI}. From this, they deduced that the PNG height fluctuations converge to its top curve, the Airy$_2$ process. 
	
	The combinatorial underpinnings behind the two matchings described above were originally quite different. The first was based on the shuffling algorithm introduced by Elkies--Kuperberg--Larsen--Propp \cite{AMT}, to sample random domino tilings of the Aztec diamond. The second was based on the Robinson--Schensted--Knuth (RSK) correspondence, which was also used by Baryshnikov \cite{GQ}, O'Connell--Yor \cite{AT,PTRW}, and Warren \cite{II} to produce line ensembles associated with various models of last passage percolation. Borodin--Petrov \cite{NNDP} later explained that both can be viewed as special cases of a natural family $(2 + 1)$-dimensional Markov chains (whose first forms date back to Borodin--Ferrari \cite{AGRS}) on the Schur processes of Okounkov--Reshetikhin \cite{CFPALG}. 
	
	Line ensembles have also been introduced for certain random polymers at positive temperature, based on a geometric lift of the RSK correspondence due to Kirillov \cite{TC} and Noumi--Yamada \cite{TCBA} (which, due to work of Matveev--Petrov \cite{RCRP}, can also be thought of as a special case of certain $(2 + 1)$-dimensional Markov chains on the $q$-Whittaker procesess of Borodin--Corwin \cite{P}). These include for the O'Connell--Yor polymer and KPZ equation through works of O'Connell--Warren \cite{DPQL,MSE}, Corwin--Hammond \cite{LE}, and Nica \cite{IDL}, as well as for the log-gamma polymer through works of Corwin--O'Connell--Sepp\"{a}l\"{a}inen--Zygouras \cite{TCF}, Johnston--O'Connell \cite{SLNP}, and Wu \cite{TDLEE}. For the (single-species) asymmetric simple exclusion process (ASEP) and stochastic six-vertex model, line ensembles were produced in a different way (which will be closer to the direction of this paper) by Borodin--Bufetov--Wheeler \cite{SSVMP}, Corwin--Dimitrov \cite{TFSVM}, and Bufetov--Petrov \cite{FSSF}. They first used the Yang--Baxter equation to match the height functions of these systems with specific marginals of the Hall--Littlewood measure, and then interpreted the latter measure as a line ensemble.
	
	All of the above line ensembles admit explicit \emph{Gibbs properties} describing their laws if all but a few of their curves are conditioned upon. Starting with the paper \cite{PLE} of Corwin--Hammond, such \emph{Gibbsian line ensembles} have emerged as fundamental instruments for probabilistically analyzing the associated stochastic growth models. For example, work of Hammond \cite{RLE,MCPWP,PRPW,ERDP} used them to study the on-scale polymer geometry of Brownian last passage percolation in detail. Later, Matetski--Quastel--Remenik \cite{TFP} and Dauvergne--Ortmann--Vir\'{a}g \cite{TL} provided the full space-time scaling limit for TASEPs and last passage percolation models under arbitrary initial data, and the latter \cite{TL} showed that this limit can be described entirely through the Airy line ensemble. 
	
	Gibbsian line ensembles have also been used to understand the fine probabilistic structure of non-determinantal models in the KPZ universality class, such as the KPZ equation, (single-species) ASEP and stochastic six-vertex model, and log-gamma polymer. Results in this direction include proofs of tightness and correlation bounds by Corwin--Hammond--Ghosal \cite{LE,ECT}, Wu \cite{TDLEE}, and Barraquand--Corwin--Dimitrov \cite{TFSVM,STEL}, as well as polymer path properties (possibly under large deviation events) by Das--Zhu \cite{LCDP}, Wu \cite{BRLE,EL}, and Ganguly--Hegde--Zhang \cite{UTEL,BLPM}. More recent work of Aggarwal--Huang \cite{SCLE} established that the Airy line ensemble is the unique line ensemble satisfying the Gibbs property for non-intersecting Brownian bridges, whose top curve is approximately parabolic. This (together with the above-mentioned tightness frameworks) could potentially lead to a systematic way of proving that discrete stochastic growth models converge to their scaling limit, whenever such models can be associated with a Gibbsian line ensemble. 
	
	This activity leads to the (closely related) questions of, (a) in what generality can Gibbsian line ensembles be associated with a stochastic growth model, and (b) what is the mechanism that enables their appearance? The purpose of this paper is to work towards these questions. 
	
	We consider the colored stochastic fused vertex models, associated with the affine quantum group $U_q (\widehat{\mathfrak{sl}}_{n+1})$, introduced by Kuniba--Mangazeev--Maruyama--Okado \cite{SRM} and studied in detail by Borodin--Wheeler \cite{CSVMST}. Our main result is that any such model can be associated with a family of line ensembles satisfying two properties. The first is that their top curves coincide in law with the colored height functions of the stochastic vertex model (\Cref{lmu2} and \Cref{lmu2fused}); the second is that they satisfy an explicit Gibbs property (\Cref{conditionl} and \Cref{conditionlfused}). Our arguments generalize those in \cite{SSVMP,TFSVM,FSSF}, by using the Yang--Baxter equation underlying these vertex models to match their colored height functions to marginals of certain measures on compositions (\Cref{fgsv} and \Cref{fgsvfused}); by their definitions, the latter can be interpreted as families of line ensembles with explicit Gibbs properties (\Cref{lmu} and \Cref{lmufused}). In a sense, this pinpoints the Yang--Baxter equation as the algebraic source for Gibbsian line ensembles associated with the integrable stochastic vertex models studied here.\footnote{There also exist stochastic systems satisfying the Yang--Baxter equation, which are not special cases of our $U_q (\widehat{\mathfrak{sl}}_{n+1})$ stochastic fused vertex model. These include ones considered by Cantini \cite{SDHR} and Chen--de Gier--Hiki--Sasamoto--Usui \cite{LCDP}, as well as ones with boundary conditions, studied for example by Barraquand--Borodin--Corwin--Wheeler \cite{SVMHH}, He \cite{BCFHS,SIHSIM}, and Yang \cite{SMHSVMS}. It would be interesting to investigate whether Gibbsian line ensembles can be associated with these models, too. For the half-space log-gamma polymer, this has been done by Barraquand--Corwin--Das \cite{EHSP} (using the geometric RSK correspondence).}
	
	We have several reasons for operating at the level of the $U_q (\widehat{\mathfrak{sl}}_{n+1})$ stochastic fused vertex models. The first is their scope; they are fairly general objects that degenerate to most systems proven to be in the KPZ universality class (though not all of them, such as the non-nearest neighbor exclusion processes considered by Quastel--Sarkar \cite{CEPEFP}). See \Cref{qmodel} for a (not entirely complete) list of degenerations to known models;\footnote{Many of these degenerations were previously discussed by Borodin--Gorin--Wheeler \cite[Figure 2]{SVMP}, but that work does not explain how to degenerate the colored vertex model to the log-gamma polymer. We address this point in \Cref{VertexPolymer} below.} all of them should be associated with Gibbsian line ensembles, obtained by taking the appropriate specializations or limits of our most general ones for the stochastic fused vertex model. While we will not describe these line ensembles in detail for all of the models depicted in \Cref{qmodel}, we will do so for a few examples (such as the colored stochastic six-vertex and discrete-time $q$-boson models) in \Cref{ExampleL} and \Cref{L2} below.

	\begin{figure}
		\begin{center} 	
			\begin{tikzpicture}[
				>=stealth,
				scale = .35
				]
				
				\node (cfsvm) at (0, 10) [scale = 1]{\framebox{Colored fused stochastic vertex model}}; 
				\node (cssvm) at (12.5, 6.5) [scale = .8]{\framebox{\parbox{2.8cm}{Colored stochastic six-vertex model}}};
				\node (cep) at (10.5, 2.5) [scale = .8]{\framebox{\parbox{2cm}{Colored ASEP/SSEP}}};
				\node (cste) at (17, 2.5) [scale = .8]{\framebox{\parbox{2.85cm}{Colored stochastic telegraph equation}}}; 
				\node (ct1) at (7, -2) [scale = .8]{\framebox{\parbox{1.2cm}{Colored TASEP}}};
				\node (ahe) at (16, -2.5) [scale = .8]{\framebox{Additive stochastic heat equation}};
				\node (ct2) at (-.5, 5.5) [scale  = .8]{\framebox{\parbox{1.2cm}{Colored $q$-Hahn TASEP}}};
				\node (ct3) at (5.5, 4) [scale = .8]{\framebox{\parbox{1.25cm}{Colored $q$-boson}}};
				\node (qp) at (-19, 6.5) [scale = .8]{\framebox{$q$-PNG}};
				\node (p) at (-19, 4) [scale = .8]{\framebox{PNG}};
				\node (qcp) at (-14, 5) [scale = .8]{\framebox{\parbox{1.9cm}{Colored $q$-PushTASEP}}};
				\node (cp) at (-14, 1) [scale = .8]{\framebox{\parbox{2cm}{Colored PushTASEP}}};
				\node (bp) at (-2, 2) [scale = .8]{\framebox{Beta polymer}};
				\node (swp) at (-2.75, -.5) [scale = .8]{\framebox{Strict-weak polymer}};
				\node (p2) at (0, -3) [scale=.8]{\framebox{O'Connell--Yor polymer}};
				\node (lp) at (3, -6.5) [scale=.8]{\framebox{\parbox{1.8cm}{Brownian last passage percolation}}};
				\node (lg) at (-9, -1) [scale = .8]{\framebox{\parbox{1.3cm}{Log-gamma polymer}}};
				\node (e) at (-3, -6.5) [scale = .8]{\framebox{KPZ equation}};
				\node (f) at (0, -10) [scale = .8]{\framebox{KPZ fixed point / directed landscape}};
				\draw[->] (cfsvm) -- (cssvm);
				\draw[->] (cssvm) -- (cep);
				\draw[->] (cssvm) -- (cste);
				\draw[->] (cfsvm) -- (qcp);
				\draw[->] (qcp) -- (cp);
				\draw[->] (cfsvm) -- (qp);
				\draw[->] (cfsvm) -- (ct2);
				\draw[->] (cfsvm) -- (ct3);
				\draw[->] (qp) -- (p);
				\draw[->] (cep) -- (ct1);
				\draw[->] (cep) -- (ahe);
				\draw[->] (cste) -- (ahe);
				\draw[->] (ct2) -- (ct3);
				\draw[->] (ct3) -- (ct1);
				\draw[->] (ct2) -- (bp);
				\draw[->] (bp) -- (swp);
				\draw[->] (swp) -- (p2);
				\draw[->] (ct3) -- (p2);
				\draw[->] (p2) -- (e);
				\draw[->] (p2) -- (lp);
				\draw[->] (e) -- (f);
				\draw[->] (lp) -- (f);
				\draw[->] (cfsvm) -- (lg);
				\draw[->] (lg) -- (p2);
				
			\end{tikzpicture}
			\end{center} 

			\caption{ \label{qmodel} Depicted above are various degenerations of the $U_q (\widehat{\mathfrak{sl}}_{n+1})$ colored fused stochastic vertex model.} 
			
		\end{figure}
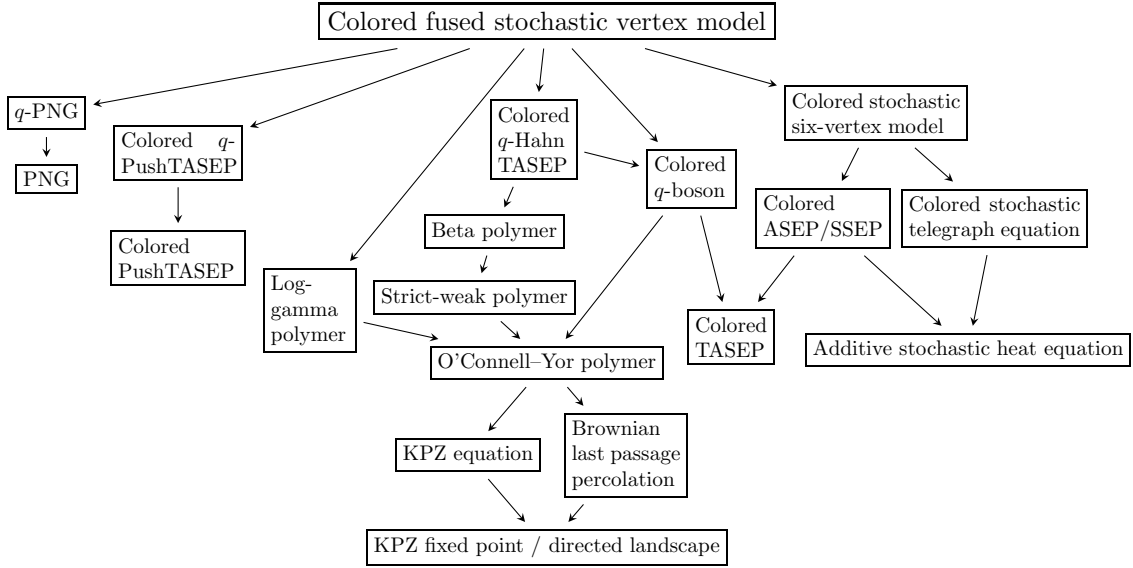 
	
	The second is that for $n > 1$ these models enable us to access \emph{colored}, also called \emph{multi-species}, systems (in which some particles may have a higher priority than others). Prior to this work, we were unaware of Gibbsian line ensembles associated with any example of a multispecies model. A new effect arises here; when the model comprises $n > 1$ species, we associate not one but a family of $n$ coupled line ensembles, called a \emph{colored line ensemble}, with the multi-species system. The top curve in the $c$-th ensemble, jointly over all $c \in [1, n]$, of the family coincides in law with the height function tracking particles in the model of color at least $c$. The full colored line ensemble further satisfies an explicit Gibbs property that prescribes the joint law of all $n$ ensembles in the family, upon conditioning on all but a few of their curves. This provides a potential way of asymptotically analyzing colored systems. We refrain from pursuing such probabilistic studies in this paper and instead point to the forthcoming work of Aggarwal--Corwin--Hegde \cite{PP} that will use the colored line ensembles introduced here to analyze the scaling limit of the multi-species ASEP and stochastic six-vertex model. 
	
	Before continuing, let us briefly comment on $(2+1)$-dimensional Markov chains. As mentioned previously, they have been prevalent in many prior studies on line ensembles but at first may not seem to make a direct presence here. However, such dynamics do implicitly underly the proofs behind our matching statements (\Cref{fgsv} and \Cref{fgsvfused}), which involve several sequences of applications of the Yang--Baxter equation to move every vertex of an $M \times N$ rectangle across a lattice. Starting from a frozen (or ``empty'') configuration, this lattice gets randomly transformed every time a vertex is moved through it. The process of moving one vertex at time induces a $(2+1)$-dimensional Markovian evolution on the lattice, which in the uncolored ($n=1$) case was studied in works of Bufetov--Mucciconi--Petrov \cite{FSSF,RFSVM,SPDQ} under the name \emph{Yang--Baxter bijectivization}. The resulting dynamics are quite general and encapsulate the RSK correspondence as a special case \cite[Section 5]{SPDQ}. An interesting question is to introduce colors ($n > 1$) in these bijectivization dynamics, and to investigate whether different (possibly nonsymmetric\footnote{One nonsymmetric RSK algorithm was introduced by Mason \cite{DFA} and studied by Haglund--Mason--Remmel \cite{PNA}.}) RSK-type correspondences arise.
	
	We now proceed to give a more detailed sense of our results. To keep the notation as light as possible in this introductory section, we will not state them in fullest generality here. Instead, we only describe a fairly special case of our results that is still new, namely, for the $q=0$ colored stochastic six-vertex model. For the most general versions of our results, we refer to \Cref{lmu2} and \Cref{conditionl} (for the colored stochastic six-vertex model), and to \Cref{lmu2fused} and \Cref{conditionlfused} (for the colored stochastic fused vertex model). For further examples and degenerations, we refer to \Cref{ExampleL} (for other special cases of the stochastic six-vertex model) and \Cref{L2} (for the colored discrete time $q$-boson model). 
	
	Throughout this work, for any real numbers $a, b \in \mathbb{R}$ with $a \le b$, we write $\llbracket a, b \rrbracket = [a, b] \cap \mathbb{Z}$.

	\subsection{Colored Stochastic Six-Vertex Model}
	
	\label{ModelVertex}
	
	The colored stochastic six-vertex model is a certain probability measure on colored six-vertex ensembles on $\mathbb{Z}_{> 0}^2$; we begin by defining the latter. To that end, a \emph{colored six-vertex arrow configuration} is a quadruple $(a, b; c, d) \in \mathbb{Z}_{\ge 0}^4$ of nonnegative integers, which we view as an assignment of directed up-right arrows to a vertex $v \in \mathbb{Z}_{> 0}^2$, as follows. We assume that each of the four edges incident to $v$ accomodates one arrow, and that each arrow is labeled by a nonnegative integer, called a \emph{color}; edges occupied by an arrow of color $0$ are typically viewed as unoccupied (so arrows of color $0$ are ignored). We then interpret the integers $a$, $b$, $c$, and $d$ of the arrow configuration as the colors of the arrows vertically entering $v$, horizontally entering $v$, vertically exiting $v$, and horizontally exiting $v$, respectively; see the left side of \Cref{arrows} for an example. We will typically impose that $\{ a, b \} = \{ c, d \}$ as multi-sets, a restriction known as \emph{arrow conservation}; it indicates that an arrow of any color entering $v$ must also exit $v$.	

		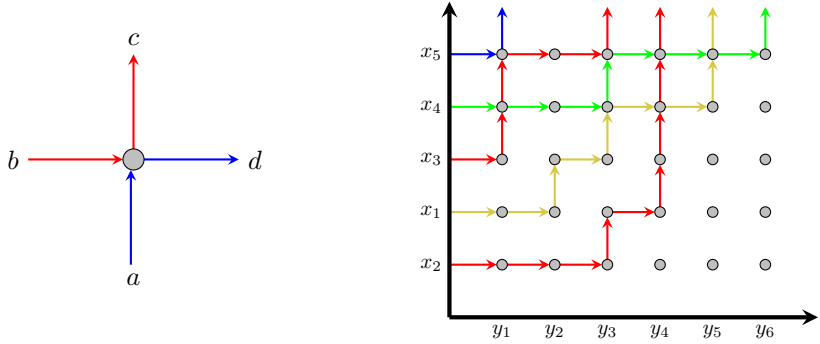
\begin{figure}
		\begin{center} 	
			\begin{tikzpicture}[
				>=stealth,
				scale = .7
				]
				\draw[->,red, thick] (0, .2) -- (0, 2); 
				\draw[->,blue, thick] (-.05,-2) -- (-.05, -.2);
				\draw[->,red, thick] (-2, 0) -- (-.2, 0);
				\draw[->,blue, thick] (.2, 0) -- (2, 0); 
				
				\draw[] (0, 2) circle [radius = 0] node[above]{$c$};
				\draw[] (0, -2) circle [radius = 0] node[below]{$a$};
				\draw[] (-2, 0) circle [radius = 0] node[left]{$b$};
				\draw[] (2, 0) circle [radius = 0] node[right]{$d$};
				
				\filldraw[fill=gray!50!white, draw=black] (0, 0) circle [radius=.2];
				
				\draw[->, thick, red] (6, -2) -- (6.9, -2);
				
				\draw[->, thick, yellow!80!black] (6, -1) -- (6.9, -1);
				
				\draw[->, thick, red] (6, 0) -- (6.9, 0);
				
				\draw[->, thick, green] (6, 1) -- (6.9, 1);

				\draw[->, thick, blue] (6, 2) -- (6.9, 2);
				
				\draw[->, thick, red] (7.1, -2) -- (7.9, -2);
				\draw[->, thick, red] (8.1, -2) -- (8.9, -2);
				\draw[->, thick, red] (9, -1.9) -- (9, -1.1);
				
				\draw[->, thick, yellow!80!black] (7.1, -1) -- (7.9, -1);
				\draw[->, thick, yellow!80!black] (8, -.9) -- (8, -.1);
				\draw[->, thick, red] (9.1, -1) -- (9.9, -1);
				\draw[->, thick, red] (10, -.9) -- (10, -.1);
				
				\draw[->, thick, red] (7, .1) -- (7, .9);
				\draw[->, thick, yellow!80!black] (8.1, 0) -- (8.9, 0);
				\draw[->, thick, yellow!80!black] (9, .1) -- (9, .9);
				\draw[->, thick, red] (10, .1) -- (10, .9);
				
				\draw[->, thick,red] (7, 1.1) -- (7, 1.9);
				\draw[->, thick, green] (7.1, 1) -- (7.9, 1);
				\draw[->, thick, green] (8.1, 1) -- (8.9, 1);
				\draw[->, thick, green] (9, 1.1) -- (9, 1.9);
				\draw[->, thick, yellow!80!black] (9.1, 1) -- (9.9, 1);
				\draw[->, thick, red] (10, 1.1) -- (10, 1.9);
				\draw[->, thick, yellow!80!black] (10.1, 1) -- (10.9, 1);
				\draw[->, thick, yellow!80!black] (11, 1.1) -- (11, 1.9);
				
				\draw[->, thick, blue] (7, 2.1) -- (7, 2.9);
				\draw[->, thick, red] (7.1, 2) -- (7.9, 2);
				\draw[->, thick, red] (8.1, 2) -- (8.9, 2);
				\draw[->, thick, red] (9, 2.1) -- (9, 2.9);
				\draw[->, thick, green] (9.1, 2) -- (9.9, 2);
				\draw[->, thick, red] (10, 2.1) -- (10, 2.9);
				\draw[->, thick, green] (10.1, 2) -- (10.9, 2);
				\draw[->, thick, yellow!80!black] (11, 2.1) -- (11, 2.9);
				\draw[->, thick, green] (11.1, 2) -- (11.9, 2);
				\draw[->, thick, green] (12, 2.1) -- (12, 2.9);
				
				\draw[->, ultra thick] (6, -3) -- (13, -3);
				\draw[->, ultra thick] (6, -3) -- (6, 3);
				
				\filldraw[fill=gray!50!white, draw=black] (7, -2) circle [radius=.1];
				\filldraw[fill=gray!50!white, draw=black] (8, -2) circle [radius=.1];
				\filldraw[fill=gray!50!white, draw=black] (9, -2) circle [radius=.1];
				\filldraw[fill=gray!50!white, draw=black] (10, -2) circle [radius=.1];
				\filldraw[fill=gray!50!white, draw=black] (11, -2) circle [radius=.1];
				\filldraw[fill=gray!50!white, draw=black] (12, -2) circle [radius=.1];
				
				\filldraw[fill=gray!50!white, draw=black] (7, -1) circle [radius=.1];
				\filldraw[fill=gray!50!white, draw=black] (8, -1) circle [radius=.1];
				\filldraw[fill=gray!50!white, draw=black] (9, -1) circle [radius=.1];
				\filldraw[fill=gray!50!white, draw=black] (10, -1) circle [radius=.1];
				\filldraw[fill=gray!50!white, draw=black] (11, -1) circle [radius=.1];
				\filldraw[fill=gray!50!white, draw=black] (12, -1) circle [radius=.1];
				
				\filldraw[fill=gray!50!white, draw=black] (7, 0) circle [radius=.1];
				\filldraw[fill=gray!50!white, draw=black] (8, 0) circle [radius=.1];
				\filldraw[fill=gray!50!white, draw=black] (9, 0) circle [radius=.1];
				\filldraw[fill=gray!50!white, draw=black] (10, 0) circle [radius=.1];
				\filldraw[fill=gray!50!white, draw=black] (11, 0) circle [radius=.1];
				\filldraw[fill=gray!50!white, draw=black] (12, 0) circle [radius=.1];
				
				\filldraw[fill=gray!50!white, draw=black] (7, 1) circle [radius=.1];
				\filldraw[fill=gray!50!white, draw=black] (8, 1) circle [radius=.1];
				\filldraw[fill=gray!50!white, draw=black] (9, 1) circle [radius=.1];
				\filldraw[fill=gray!50!white, draw=black] (10, 1) circle [radius=.1];
				\filldraw[fill=gray!50!white, draw=black] (11, 1) circle [radius=.1];
				\filldraw[fill=gray!50!white, draw=black] (12, 1) circle [radius=.1];
				
				\filldraw[fill=gray!50!white, draw=black] (7, 2) circle [radius=.1];
				\filldraw[fill=gray!50!white, draw=black] (8, 2) circle [radius=.1];
				\filldraw[fill=gray!50!white, draw=black] (9, 2) circle [radius=.1];
				\filldraw[fill=gray!50!white, draw=black] (10, 2) circle [radius=.1];
				\filldraw[fill=gray!50!white, draw=black] (11, 2) circle [radius=.1];
				\filldraw[fill=gray!50!white, draw=black] (12, 2) circle [radius=.1];
				
				\draw[] (6, -1) circle [radius = 0] node[left, scale = .8]{$x_1$};
				\draw[] (6, -2) circle [radius = 0] node[left, scale = .8]{$x_2$};
				\draw[] (6, 0) circle [radius = 0] node[left, scale = .8]{$x_3$};
				\draw[] (6, 1) circle [radius = 0] node[left, scale = .8]{$x_4$};
				\draw[] (6, 2) circle [radius = 0] node[left, scale = .8]{$x_5$};
				
				\draw[] (7, -3) circle [radius = 0] node[below, scale = .8]{$y_1$};
				\draw[] (8, -3) circle [radius = 0] node[below, scale = .8]{$y_2$};
				\draw[] (9, -3) circle [radius = 0] node[below, scale = .8]{$y_3$};
				\draw[] (10, -3) circle [radius = 0] node[below, scale = .8]{$y_4$};
				\draw[] (11, -3) circle [radius = 0] node[below, scale = .8]{$y_5$};
				\draw[] (12, -3) circle [radius = 0] node[below, scale = .8]{$y_6$};
			\end{tikzpicture}
		\end{center}
		\caption{\label{arrows} Shown to the left is a vertex with arrow configuration $(a, b; c, d) = (2, 1; 1, 2)$, where red and blue are colors $1$ and $2$, respectively. Shown to the right is a colored model on the quadrant.} 
	\end{figure}

	A \emph{domain} is a subset $\mathcal{D} \subseteq \mathbb{Z}^2$, and a \emph{colored six-vertex ensemble} on a domain $\mathcal{D} \subset \mathbb{Z}^2$ is an assignment of an arrow configuration to each vertex of $\mathcal{D}$ in such a way that neighboring arrow configurations are \emph{consistent}; this means that, if $v_1, v_2 \in \mathcal{D}$ are two adjacent vertices, then there is an arrow of color $c \in \mathbb{Z}_{\ge 0}$ to $v_2$ in the configuration at $v_1$ if and only if there is an arrow of color $c$ from $v_1$ in the configuration at $v_2$. Observe in particular that the arrows in a colored six-vertex ensemble form colored up-right directed paths connecting vertices of $\mathcal{D}$. 
	
	\emph{Boundary data} for a colored six-vertex ensemble is prescribed by dictating which points on the boundary of a domain are entrance (or exit) sites for a path of a given color. If the domain $\mathcal{D}$ is a rectangle or quadrant, we will typically restrict to the case when paths only enter horizontally through the west boundary of $\mathcal{D}$; see the right side of \Cref{arrows} for a depiction. Given a function $\sigma : \llbracket 1, N \rrbracket \rightarrow \mathbb{Z}_{\ge 0}$, we say that a colored six-vertex ensemble on the rectangle domain $\mathcal{D}_{M;N} = \llbracket 1, M \rrbracket \times \llbracket 1, N \rrbracket$ has \emph{$\sigma$-entrance data} if the following holds. For each $j \in \llbracket 1, N \rrbracket$, one path of color $\sigma (j)$ horizontally enters $\mathcal{D}_{M;N}$ from the site $(0, j)$ on the $y$-axis, and no path horizontally enters $\mathcal{D}_{M;N}$ from any site on the $x$-axis.
	
	Associated with any six-vertex ensemble $\mathcal{E}$ on a domain $\mathcal{D} \subseteq \mathbb{Z}^2$ are \emph{height functions} $\mathfrak{h}_{\ge c}^{\leftarrow} : \mathbb{Z}^2 \rightarrow \mathbb{Z}$, which for any integer $c \ge 1$ are defined as follows. For any vertex $v = (i, j) \in \mathbb{Z}^2$, let $\mathfrak{h}_{\ge c}^{\leftarrow} (v)$ denote the number of paths of color at least $c$ in $\mathcal{E}$ that pass to the left of (equivalently, above) $v$, namely, that do not intersect the vertical ray (pointing south) connecting $( i+1 / 2, j + 1 / 2)$ to $(i + 1 / 2, -\infty )$.

	The \emph{colored stochastic six-vertex model} is a probability measure on colored six-vertex ensembles on $\mathbb{Z}_{> 0}^2$ that depends on two infinite sequences of real parameters $\bm{x} = (x_1, x_2, \ldots)$ and $\bm{y} = (y_1, y_2, \ldots)$. We view $x_j$ as associated with the $j$-th row and $y_i$ as associated with the $i$-th column, so $\bm{x}$ and $\bm{y}$ are called \emph{row rapidities} and \emph{column rapidities}, respectively. The specific forms of these probability measures are expressed through weights $R_{y_i / x_j} (a, b, c, d)$ associated with each vertex $v = (i, j) \in \mathbb{Z}_{> 0}^2$. In addition to depending on the arrow configuration $(a, b, c, d)$ at $v$, this vertex weight will also be governed by several parameters. The first is the \emph{quantization parameter} $q$, which is fixed throughout the model. The second is the \emph{spectral variable} $z = z_{i,j} = x_j^{-1} y_i$, which is given by the ratio of the column and row rapidities at the vertex $v = (i, j)$. Given this notation, we define the following vertex weights originally introduced in \cite{TSTECA,QMGS}.
	
	\begin{definition}
		
		\label{rzabcd} 
		
		For any complex number $z \in \mathbb{C}$ and integers $a, b, c, d \ge 0$, define the vertex weight $R_z (a, b, c, d)$ as follows. For $i < j$, set (see \Cref{rz})
		\begin{flalign}
			\label{rzij} 
			\begin{aligned}
				& R_z (i, i; i, i) = 1; \qquad R_z (j, i; j,i) = \displaystyle\frac{q(1-z)}{1-qz}; \qquad R_z (i, j; i, j) = \displaystyle\frac{1-z}{1-qz}; \\
				& \qquad \qquad  R_z (j, i; i, j) = \displaystyle\frac{1-q}{1-qz}; \qquad R_z (i, j; j, i) = \displaystyle\frac{z(1-q)}{1-qz}.
			\end{aligned} 
		\end{flalign}
		
		\noindent If $(a, b, c, d)$ is not of the above form for some $0 \le i < j$, then set $R_z (a, b; c, d) = 0$. 
	\end{definition}
	
	\begin{rem}
		\label{1rsum}
		
		These $R_z$ weights are \emph{stochastic} in that the sum of all weights with a fixed pair of incoming arrows is equal to $1$, namely, $\sum_{c, d \ge 0} R_z (a, b; c, d) = 1$ for each $z \in \mathbb{C}$ and $a, b \ge 0$.   
		
	\end{rem}

	\begin{figure} 
		\begin{center}
			\begin{tikzpicture}[
				>=stealth,
				scale = .75
				]
				
				\draw[-, black] (-5, 3.1) -- (7.5, 3.1);
				\draw[-, black] (-5, -2.5) -- (7.5, -2.5);
				\draw[-, black] (-5, -1.1) -- (7.5, -1.1);
				\draw[-, black] (-5, -.4) -- (7.5, -.4);
				\draw[-, black] (-5, 2.4) -- (7.5, 2.4);
				\draw[-, black] (7.5, -2.5) -- (7.5, 3.1);
				\draw[-, black] (-5, -2.5) -- (-5, 3.1);
				\draw[-, black] (5, -2.5) -- (5, 2.4);
				\draw[-, black] (-2.5, -2.5) -- (-2.5, 3.1);
				\draw[-, black] (2.5, -2.5) -- (2.5, 2.4);
				\draw[-, black] (0, -2.5) -- (0, 2.4);			
				
				\draw[->, thick, blue] (3.75, .1) -- (3.75, 1) -- (4.65, 1);
				\draw[->, thick, red] (2.85, 1) -- (3.75, 1) -- (3.75, 1.9); 
				
				\draw[->, thick, blue] (-1.25, .1) -- (-1.25, 1.9);
				\draw[->, thick,  red] (-2.15, 1) -- (-.35, 1);
				
				\draw[->, thick, blue] (.35, 1) -- (2.15, 1);
				\draw[->, thick, red] (1.25, .1) -- (1.25, 1.9);
				
				\draw[->, thick, blue] (5.35, 1) -- (6.25, 1) -- (6.25, 1.9);
				\draw[->, thick, red] (6.25, .1) -- (6.25, 1) -- (7.15, 1); 
				
				\draw[->, thick, red] (-3.75, .1) -- (-3.75, 1.9);
				\draw[->, thick, red] (-4.65, 1) -- (-2.85, 1);		
				
				\filldraw[fill=gray!50!white, draw=black] (-2.85, 1) circle [radius=0] node [black, right = -1, scale = .75] {$i$};
				\filldraw[fill=gray!50!white, draw=black] (-.35, 1) circle [radius=0] node [black, right = -1, scale = .75] {$i$};
				\filldraw[fill=gray!50!white, draw=black] (2.15, 1) circle [radius=0] node [black, right = -1, scale = .75] {$j$};
				\filldraw[fill=gray!50!white, draw=black] (4.65, 1) circle [radius=0] node [black, right = -1, scale = .75] {$j$};
				\filldraw[fill=gray!50!white, draw=black] (7.15, 1) circle [radius=0] node [black, right = -1, scale = .75] {$i$};
				
				\filldraw[fill=gray!50!white, draw=black] (5.35, 1) circle [radius=0] node [black, left = -1, scale = .75] {$j$};
				\filldraw[fill=gray!50!white, draw=black] (2.85, 1) circle [radius=0] node [black, left = -1, scale = .75] {$i$};
				\filldraw[fill=gray!50!white, draw=black] (.35, 1) circle [radius=0] node [black, left = -1, scale = .75] {$j$};
				\filldraw[fill=gray!50!white, draw=black] (-2.15, 1) circle [radius=0] node [black, left = -1, scale = .75] {$i$};
				\filldraw[fill=gray!50!white, draw=black] (-4.65, 1) circle [radius=0] node [black, left = -1, scale = .75] {$i$};
				
				\filldraw[fill=gray!50!white, draw=black] (-3.75, 1.9) circle [radius=0] node [black, above = -1, scale = .75] {$i$};
				\filldraw[fill=gray!50!white, draw=black] (-1.25, 1.9) circle [radius=0] node [black, above = -1, scale = .75] {$j$};
				\filldraw[fill=gray!50!white, draw=black] (1.25, 1.9) circle [radius=0] node [black, above = -1, scale = .75] {$i$};
				\filldraw[fill=gray!50!white, draw=black] (3.75, 1.9) circle [radius=0] node [black, above = -1, scale = .75] {$i$};
				\filldraw[fill=gray!50!white, draw=black] (6.25, 1.9) circle [radius=0] node [black, above = -1, scale = .75] {$j$};
				
				\filldraw[fill=gray!50!white, draw=black] (-3.75, .1) circle [radius=0] node [black, below = -1, scale = .75] {$i$};
				\filldraw[fill=gray!50!white, draw=black] (-1.25, .1) circle [radius=0] node [black, below = -1, scale = .75] {$j$};
				\filldraw[fill=gray!50!white, draw=black] (1.25, .1) circle [radius=0] node [black, below = -1, scale = .75] {$i$};
				\filldraw[fill=gray!50!white, draw=black] (3.75, .1) circle [radius=0] node [black, below = -1, scale = .75] {$j$};
				\filldraw[fill=gray!50!white, draw=black] (6.25, .1) circle [radius=0] node [black, below = -1, scale = .75] {$i$};	
				
				\filldraw[fill=gray!50!white, draw=black] (-3.75, .1) circle [radius=0] node [black, below = -1, scale = .75] {$i$};
				\filldraw[fill=gray!50!white, draw=black] (-1.25, .1) circle [radius=0] node [black, below = -1, scale = .75] {$j$};
				\filldraw[fill=gray!50!white, draw=black] (1.25, .1) circle [radius=0] node [black, below = -1, scale = .75] {$i$};
				\filldraw[fill=gray!50!white, draw=black] (3.75, .1) circle [radius=0] node [black, below = -1, scale = .75] {$j$};
				\filldraw[fill=gray!50!white, draw=black] (6.25, .1) circle [radius=0] node [black, below = -1, scale = .75] {$i$};
				
				\filldraw[fill=gray!50!white, draw=black] (-3.75, 2.75) circle [radius=0] node [black] {$i \ge 0$};
				\filldraw[fill=gray!50!white, draw=black] (2.5, 2.75) circle [radius=0] node [black] {$j > i \ge 0$}; 
				
				\filldraw[fill=gray!50!white, draw=black] (-3.75, -.75) circle [radius=0] node [black] {$(i, i; i, i)$};
				\filldraw[fill=gray!50!white, draw=black] (-1.25, -.75) circle [radius=0] node [black] {$(j, i; j, i)$};
				\filldraw[fill=gray!50!white, draw=black] (1.25, -.75) circle [radius=0] node [black] {$(i, j; i, j)$};
				\filldraw[fill=gray!50!white, draw=black] (3.75, -.75) circle [radius=0] node [black] {$(j, i; i, j)$};
				\filldraw[fill=gray!50!white, draw=black] (6.25, -.75) circle [radius=0] node [black] {$(i, j; j, i)$};
				
				\filldraw[fill=gray!50!white, draw=black] (-3.75, -1.8) circle [radius=0] node [black] {$1$};
				\filldraw[fill=gray!50!white, draw=black] (-1.25, -1.8) circle [radius=0] node [black] {$\displaystyle\frac{q (1 - z)}{1 - qz}$};
				\filldraw[fill=gray!50!white, draw=black] (1.25, -1.8) circle [radius=0] node [black] {$\displaystyle\frac{1 - z}{1 - qz}$};
				\filldraw[fill=gray!50!white, draw=black] (3.75, -1.8) circle [radius=0] node [black] {$\displaystyle\frac{1 - q}{1 - qz}$};
				\filldraw[fill=gray!50!white, draw=black] (6.25, -1.8) circle [radius=0] node [black] {$\displaystyle\frac{z (1 - q)}{1 - qz}$};
			\end{tikzpicture}
		\end{center}
		
		\caption{\label{rz} The $R_z$ weights are depicted above.}
	\end{figure}

	Now let us describe how to sample a random colored six-vertex ensemble on $\mathbb{Z}_{> 0}^2$, using the $R_z$ weights from \eqref{rzabcd}. We will first define probability measures $\mathbb{P}_n$ on the set of colored six-vertex ensembles whose vertices are all contained in triangles of the form $\mathbb{T}_n = \{ (x, y) \in \mathbb{Z}_{> 0}^2: x + y \le n \}$, and then we will take a limit as $n$ tends to infinity to obtain the vertex models in infinite volume. The first measure $\mathbb{P}_0$ is supported by the empty ensemble (that has no paths).
	
	For each integer $n \ge 1$, we will define $\mathbb{P}_{n + 1}$ from $\mathbb{P}_n$ through the following Markovian update rules. Use $\mathbb{P}_n$ to sample a colored six-vertex ensemble $\mathcal{E}_n$ on $\mathbb{T}_n$. This yields arrow configurations for all vertices in the triangle $\mathbb{T}_{n - 1}$. To extend this to a colored six-vertex ensemble on $\mathbb{T}_{n + 1}$, we must prescribe arrow configurations to all vertices $(x, y)$ on the complement $\mathbb{T}_n \setminus\mathbb{T}_{n - 1}$, which is the diagonal $\mathbb{D}_n = \big\{ (x, y) \in \mathbb{Z}_{> 0}^2: x + y = n \big\}$. Since any incoming arrow to $\mathbb{D}_n$ is an outgoing arrow from $\mathbb{D}_{n - 1}$, $\mathcal{E}_n$ and the initial data prescribe the first two coordinates $(a, b)$ of the arrow configuration to each vertex in $\mathbb{D}_n$. Thus, it remains to explain how to assign the second two coordinates $(c, d)$ of the arrow configuration at any vertex $(i, j) \in \mathbb{D}_n$, given its first two coordinates $(a, b)$. This is done by producing $(c, d)$ from $(a, b)$ according to the transition probability 
	\begin{flalign}
		\label{configurationprobabilities} 
		& \mathbb{P}_n \big[ (c, d)  \big| (a, b) \big] = R_{y_i / x_j} (a, b; c, d).
	\end{flalign}
	
	\noindent We assume that the parameters $(\bm{x}; \bm{y}; q)$ are chosen so that the probabilities \eqref{configurationprobabilities} are all nonnegative; the stochasticity of the $R_z$ weights (\Cref{1rsum}) then ensures that \eqref{configurationprobabilities} indeed defines a probability measure.
	
	Choosing $(c, d)$ according to the above transition probabilities yields a random colored six-vertex ensemble $\mathcal{E}_{n + 1}$, now defined on $\mathbb{T}_{n + 1}$; the probability distribution of $\mathcal{E}_{n + 1}$ is then denoted by $\mathbb{P}_{n + 1}$. Taking the limit as $n$ tends to $\infty$ yields a probability measure on colored six-vertex ensembles $\mathcal{E}$ on the quadrant. We refer to it as the \emph{colored stochastic six-vertex model}; observe that it may also be sampled on any rectangle $\mathcal{D} \subset \mathbb{Z}^2$ in the same way as it was above on the quadrant.

	\subsection{Colored Line Ensembles}
	
	\label{EnsembleCoupled} 
	
	In this section we introduce terminology for colored families of line ensembles. We first define the notion of a line ensemble. Those that we consider here will be discrete, and their paths will be non-increasing, which is related to the fact that the associated stochastic models we analyze are discrete. By taking certain limit degenerations, one can obtain continuous line ensembles associated with non-discrete stochastic systems, but we will not pursue that in this work.
	
	\begin{definition} 
		
		\label{l} 
		
		Fix an interval $I \subseteq \mathbb{Z}$. A (discrete, down-right) \emph{line ensemble} (on $I$) is an infinite sequence $(\mathsf{L}_1, \mathsf{L}_2, \ldots )$ of functions $\mathsf{L}_k : I \rightarrow \mathbb{Z}$ such that 
		\begin{flalign}
			\label{lki2} 
			\mathsf{L}_k (i) \ge \mathsf{L}_{k+1} (i); \qquad \mathsf{L}_k (i) \ge \mathsf{L}_k (i+1), 
		\end{flalign} 
		
		\noindent for each $(k, i) \in \mathbb{Z}_{>0} \times I$ (where we must have $i+1 \in I$ in the second inequality of \eqref{lki2}). We call this line ensemble \emph{simple} if $\mathsf{L}_k (i) - \mathsf{L}_k (i+1) \in \{ 0, 1 \}$ for all $(k, i) \in \mathbb{Z}_{>0} \times I$ with $i+1 \in I$. 
	
	\end{definition} 
		
		We next define colored families of line ensembles, which are sequences of line ensembles whose differences are also line ensembles. 
		
	\begin{definition} 
	
	\label{cl} 
		
		Fix an integer $n \ge 1$ and an interval $I \subseteq \mathbb{Z}$. A \emph{colored family of line ensembles}, which we often abbreviate to a \emph{colored line ensemble}, on $I$ is a sequence $\bm{\mathsf{L}} = \big( \bm{\mathsf{L}}^{(1)}, \bm{\mathsf{L}}^{(2)}, \ldots , \bm{\mathsf{L}}^{(n)} \big)$ of line ensembles $\bm{\mathsf{L}}^{(c)} = \big( \mathsf{L}_1^{(c)}, \mathsf{L}_2^{(c)}, \ldots \big)$, such that $\bm{\Lambda}^{(c)} = \big( \Lambda_1^{(c)}, \Lambda_2^{(c)}, \ldots \big)$ is a line ensemble for each $c \in \llbracket 1, n \rrbracket$, where 
		\begin{flalign*}
			\Lambda_k^{(c)} (i) = \mathsf{L}_k^{(c)} (i) - \mathsf{L}_k^{(c+1)} (i), \qquad \text{for each $(k, i) \in \mathbb{Z}_{>0} \times I$}.
		\end{flalign*}
		
		\noindent Here, we have for convenience defined the constant function $\mathsf{L}_k^{(n+1)} : I \rightarrow \mathbb{Z}$ by setting $\mathsf{L}_k^{(n+1)} (i) = 0$ for any $(k, i) \in \mathbb{Z}_{>0} \times I$. We further call  $\bm{\mathsf{L}}$ \emph{simple} if $\bm{\mathsf{L}}^{(c)}$ is simple for each $c \in \llbracket 1, n \rrbracket$. 
		
	\end{definition}
	
	\begin{figure} 
		
		\begin{center}
			\begin{tikzpicture}[
				>=stealth, 
				scale = .65]{	
					
					\draw[dotted] (16, 0) -- (23, 0) node[right, scale = .7]{$0$};
					\draw[dotted] (16, 1) -- (23, 1) node[right, scale = .7]{$1$};
					\draw[dotted] (16, 2) -- (23, 2) node[right, scale = .7]{$2$};
					\draw[dotted] (16, 3) -- (23, 3) node[right, scale = .7]{$3$};
					\draw[dotted] (16, 4) -- (23, 4) node[right, scale = .7]{$4$};
					\draw[dotted] (16, 0) node[below = 1, scale = .7]{$0$} -- (16, 4);
					\draw[dotted] (17, 0) node[below = 1, scale = .7]{$1$} -- (17, 4);
					\draw[dotted] (18, 0) node[below = 1, scale = .7]{$2$} -- (18, 4);
					\draw[dotted] (19, 0) node[below = 1, scale = .7]{$3$} -- (19, 4);
					\draw[dotted] (20, 0) node[below = 1, scale = .7]{$4$} -- (20, 4);
					\draw[dotted] (21, 0) node[below = 1, scale = .7]{$5$} -- (21, 4);
					\draw[dotted] (22, 0) node[below = 1, scale = .7]{$6$} -- (22, 4);
					\draw[dotted] (23, 0) node[below = 1, scale = .7]{$7$} -- (23, 4); 
					
					\draw[dotted] (16, 5) -- (23, 5) node[right, scale = .7]{$0$};
					\draw[dotted] (16, 6) -- (23, 6) node[right, scale = .7]{$1$};
					\draw[dotted] (16, 7) -- (23, 7) node[right, scale = .7]{$2$};
					\draw[dotted] (16, 5) node[below = 1, scale = .7]{$0$} -- (16, 7);
					\draw[dotted] (17, 5) node[below = 1, scale = .7]{$1$} -- (17, 7);
					\draw[dotted] (18, 5) node[below = 1, scale = .7]{$2$} -- (18, 7);
					\draw[dotted] (19, 5) node[below = 1, scale = .7]{$3$} -- (19, 7);
					\draw[dotted] (20, 5) node[below = 1, scale = .7]{$4$} -- (20, 7);
					\draw[dotted] (21, 5) node[below = 1, scale = .7]{$5$} -- (21, 7);
					\draw[dotted] (22, 5) node[below = 1, scale = .7]{$6$} -- (22, 7);
					\draw[dotted] (23, 5) node[below = 1, scale = .7]{$7$} -- (23, 7);
					
					\draw[ thick, blue] (16, 7.1) -- (19, 7.1) -- (20, 6.1) -- (22, 6.1) -- (23, 5.1);
					\draw[ thick, blue] (16, 7) -- (18, 7) -- (19, 6) -- (22, 6) -- (23, 5);
					\draw[ thick, blue] (16, 6.9) -- (18, 6.9) -- (19, 5.9) -- (21, 5.9) -- (22, 4.9) -- (23, 4.9);

					\draw[ thick, violet] (16, 4.1) -- (17, 4.1) -- (18, 3.1) -- (19, 3.1) -- (20, 2.1) -- (21, 2.1) -- (22, 1.1) -- (23, .1);
					\draw[ thick, violet] (16, 4) -- (17, 4) -- (18, 3) -- (19, 2) -- (21, 2) -- (22, 1) -- (23, 0);
					\draw[ thick, violet] (16, 3.9) -- (17, 2.9) -- (18, 2.9) -- (19, 1.9) -- (20, 1.9) -- (21, .9) -- (22, -.1) -- (23, -.1);
					
					\draw[] (19, 7.1) circle [radius = 0] node[above, scale = .65]{$\mathsf{L}_1^{(2)}$};
					\draw[] (19, 6.5) circle [radius = 0] node[scale = .65]{$\mathsf{L}_2^{(2)}$};
					\draw[] (19, 5.9) circle [radius = 0] node[below, scale = .65]{$\mathsf{L}_3^{(2)}$};
					
					\draw[] (19, 3.1) circle [radius = 0] node[above, scale = .65]{$\mathsf{L}_1^{(1)}$};
					\draw[] (19, 2.5) circle [radius = 0] node[scale = .65]{$\mathsf{L}_2^{(1)}$};
					\draw[] (19, 1.9) circle [radius = 0] node[below, scale = .65]{$\mathsf{L}_3^{(1)}$};

					\draw[dotted] (26, 0) -- (33, 0) node[right, scale = .7]{$0$};
					\draw[dotted] (26, 1) -- (33, 1) node[right, scale = .7]{$1$};
					\draw[dotted] (26, 2) -- (33, 2) node[right, scale = .7]{$2$};
					\draw[dotted] (26, 3) -- (33, 3) node[right, scale = .7]{$3$};
					\draw[dotted] (26, 4) -- (33, 4) node[right, scale = .7]{$4$};
					\draw[dotted] (26, 0) node[below = 1, scale = .7]{$0$} -- (26, 4);
					\draw[dotted] (27, 0) node[below = 1, scale = .7]{$1$} -- (27, 4);
					\draw[dotted] (28, 0) node[below = 1, scale = .7]{$2$} -- (28, 4);
					\draw[dotted] (29, 0) node[below = 1, scale = .7]{$3$} -- (29, 4);
					\draw[dotted] (30, 0) node[below = 1, scale = .7]{$4$} -- (30, 4);
					\draw[dotted] (31, 0) node[below = 1, scale = .7]{$5$} -- (31, 4);
					\draw[dotted] (32, 0) node[below = 1, scale = .7]{$6$} -- (32, 4);
					\draw[dotted] (33, 0) node[below = 1, scale = .7]{$7$} -- (33, 4); 
					
					\draw[dotted] (26, 5) -- (33, 5) node[right, scale = .7]{$0$};
					\draw[dotted] (26, 6) -- (33, 6) node[right, scale = .7]{$1$};
					\draw[dotted] (26, 7) -- (33, 7) node[right, scale = .7]{$2$};
					\draw[dotted] (26, 5) node[below = 1, scale = .7]{$0$} -- (26, 7);
					\draw[dotted] (27, 5) node[below = 1, scale = .7]{$1$} -- (27, 7);
					\draw[dotted] (28, 5) node[below = 1, scale = .7]{$2$} -- (28, 7);
					\draw[dotted] (29, 5) node[below = 1, scale = .7]{$3$} -- (29, 7);
					\draw[dotted] (30, 5) node[below = 1, scale = .7]{$4$} -- (30, 7);
					\draw[dotted] (31, 5) node[below = 1, scale = .7]{$5$} -- (31, 7);
					\draw[dotted] (32, 5) node[below = 1, scale = .7]{$6$} -- (32, 7);
					\draw[dotted] (33, 5) node[below = 1, scale = .7]{$7$} -- (33, 7);
					
					\draw[ thick, blue] (26, 7.1) -- (29, 7.1) -- (30, 6.1) -- (32, 6.1) -- (33, 5.1);
					\draw[ thick, blue] (26, 7) -- (28, 7) -- (29, 6) -- (32, 6) -- (33, 5);
					\draw[ thick, blue] (26, 6.9) -- (28, 6.9) -- (29, 5.9) -- (31, 5.9) -- (32, 4.9) -- (33, 4.9);

					\draw[ thick, violet] (26, 4.1) -- (27, 4.1) -- (28, 3.1) -- (29, 3.1) -- (30, 2.1) -- (32, 2.1) -- (33, .1);
					\draw[ thick, violet] (26, 4) -- (27, 4) -- (28, 3) -- (29, 2) -- (30, 2) -- (31, 1) -- (32, 1) -- (33, 0);
					\draw[ thick, violet] (26, 3.9) -- (27, 2.9) -- (28, 2.9) -- (29, 1.9) -- (30, 1.9) -- (31, .9) -- (32, -.1) -- (33, -.1);
					
					\draw[] (29, 7.1) circle [radius = 0] node[above, scale = .65]{$\widehat{\mathsf{L}}_1^{(2)}$};
					\draw[] (29, 6.5) circle [radius = 0] node[scale = .65]{$\widehat{\mathsf{L}}_2^{(2)}$};
					\draw[] (29, 5.9) circle [radius = 0] node[below, scale = .65]{$\widehat{\mathsf{L}}_3^{(2)}$};
					
					\draw[] (29, 3.1) circle [radius = 0] node[above, scale = .65]{$\widehat{\mathsf{L}}_1^{(1)}$};
					\draw[] (29, 2.5) circle [radius = 0] node[scale = .65]{$\widehat{\mathsf{L}}_2^{(1)}$};
					\draw[] (29, 1.9) circle [radius = 0] node[below, scale = .65]{$\widehat{\mathsf{L}}_3^{(1)}$};
					
				}
			\end{tikzpicture}
		\end{center}
		
		\caption{\label{1lmu} Shown above are two colored line ensembles. To the left, $\bm{\mathsf{L}}$ is simple; to the right, $\widehat{\bm{\mathsf{L}}}$ is not simple and is $\llbracket 1, 2 \rrbracket \times \llbracket 5, 6 \rrbracket$-compatible with $\bm{\mathsf{L}}$.}
		
	\end{figure}
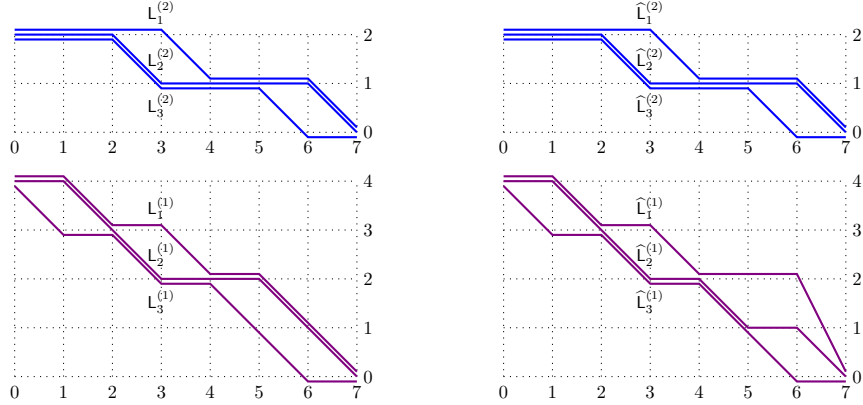 
	
	See \Cref{1lmu} for examples. Before proceeding, we require the notion of compatibility for colored line ensembles. 
	
	\begin{definition} 
		
	\label{lluvij} 
	
	Fix colored line ensembles $\bm{\mathsf{L}} = \big( \bm{\mathsf{L}}^{(1)}, \bm{\mathsf{L}}^{(2)}, \ldots , \bm{\mathsf{L}}^{(n)} \big)$ and $\bm{\mathsf{l}} = \big( \bm{\mathsf{l}}^{(1)}, \bm{\mathsf{l}}^{(2)}, \ldots , \bm{\mathsf{l}}^{(n)} \big)$ on an interval $I \subseteq \mathbb{Z}$, and integers $j \ge i \ge 1$ and $u, v \in I$ with $u \le v$. We say that $\bm{\mathsf{l}}$ is \emph{$\llbracket i, j \rrbracket \times \llbracket u, v \rrbracket$-compatible} with $\bm{\mathsf{L}}$ if $\mathsf{L}_k^{(c)} (m) = \mathsf{l}_k^{(c)} (m)$ for each $c \in \llbracket 1, n \rrbracket$ and $(k, m) \in (\mathbb{Z}_{> 0} \times I) \setminus \big( \llbracket i, j \rrbracket \times \llbracket u, v \rrbracket \big)$. 

	\end{definition} 

	Observe under the notation of \Cref{lluvij} that there are only finitely many colored line ensembles that are $\llbracket i, j \rrbracket \times \llbracket u, v \rrbracket$-compatible with a given one.

	\subsection{Colored Line Ensembles for the $q=0$ Stochastic Six-Vertex Model}
	
	\label{Vertexq02} 
	
	In this section we state a special case of our main results (see \Cref{L0} and \Cref{L0Fused} below for the more general ones), by associating a colored family of line ensembles to the colored stochastic six-vertex model at $q = 0$. This is provided by the following theorem, which is proven in \Cref{Lq0Example}. Its first part indicates that the top curves of the line ensembles in the colored family have the same joint law as the colored height functions for the stochastic six-vertex model. Its third part provides a Gibbs property (which is well posed by its second part) for the colored line ensemble. In what follows, we fix integers $M, N, n \ge 1$; real numbers $x, y \in (0, 1)$ with $y < x$; and a function $\sigma : \llbracket 1, N \rrbracket \rightarrow \llbracket 1, n \rrbracket$. We also define the rectangle $\mathcal{D}_{M;N} = \llbracket 1, M \rrbracket \times \llbracket 1, N \rrbracket \subset \mathbb{Z}^2$.

	\begin{thm} 
		
		\label{ensemblevertexq0} 
		
		Sample a colored six-vertex ensemble $\mathcal{E}$ on $\mathcal{D}_{M;N}$ according to the stochastic six-vertex model with $q = 0$; all parameters of $\bm{x}$ equal to $x$ and of $\bm{y}$ equal to $y$; and $\sigma$-entrance data. For each $c \in \llbracket 1, n \rrbracket$ define $H_c : \llbracket 0, M + N \rrbracket \rightarrow \mathbb{Z}$ by setting
		\begin{flalign*} 
			H_c (k) = \mathfrak{h}_{\ge c}^{\leftarrow} (M, k), \quad \text{if $k \in \llbracket 0, N \rrbracket$}; \qquad H_c (k) = \mathfrak{h}_{\ge c}^{\leftarrow} (M + N - k, N), \quad \text{if $k \in \llbracket N, M + N \rrbracket$},
		\end{flalign*}
	
		\noindent  where $\mathfrak{h}_{\ge c}^{\leftarrow}$ is the colored height function with respect to $\mathcal{E}$.  There exists a random simple colored line ensemble $\bm{\mathsf{L}} = \big( \bm{\mathsf{L}}^{(1)}, \bm{\mathsf{L}}^{(2)}, \ldots , \bm{\mathsf{L}}^{(n)} \big)$ on $\llbracket 0, M + N \rrbracket$ satisfying the following properties.  
		
		\begin{enumerate}
			\item The joint law of the functions $\big( \mathsf{L}_1^{(1)}, \mathsf{L}_1^{(2)}, \ldots , \mathsf{L}_1^{(n)} \big)$ is the same as that of $(H_1, H_2, \ldots , H_n)$.
			
			\item For any integers $c \in \llbracket 1, n \rrbracket$; $k \ge 1$; and $m \in \llbracket 1, M + N \rrbracket$ such that $\mathsf{L}_k^{(c+1)} (m) > \mathsf{L}_{k+1}^{(c+1)} (m)$, we almost surely have 
			\begin{flalign}
				\label{lkcq0} 
				\mathsf{L}_k^{(c)} (m-1) - \mathsf{L}_k^{(c)} (m) = \mathsf{L}_k^{(c+1)} (m-1) - \mathsf{L}_k^{(c+1)} (m). 
			\end{flalign}
			\item \label{3q0} Fix integers $j \ge i \ge 0$ and $u, v \in \llbracket 0, M + N \rrbracket$ such that $v \ge u$ and $N \notin \llbracket u, v \rrbracket$. Condition on the curves $\mathsf{L}_k^{(c)} (m)$ for all $c \in \llbracket 1, n \rrbracket$ and $(k, m) \notin \llbracket i, j \rrbracket \times \llbracket u, v \rrbracket$. Then the law of $\bm{\mathsf{L}}$ is uniform over all simple colored line ensembles $\bm{\mathsf{l}}$ that are $\llbracket i, j \rrbracket \times \llbracket u, v \rrbracket$-compatible with $\bm{\mathsf{L}}$ such that the following holds. For any integers $c \in \llbracket 1, n \rrbracket$; $k \ge 1$; and $m \in \llbracket 1, M + N \rrbracket$ such that $\mathsf{l}_k^{(c+1)} (m) > \mathsf{l}_{k+1}^{(c+1)} (m)$, we have
			\begin{flalign*}
				\mathsf{l}_k^{(c)} (m-1) - \mathsf{l}_k^{(c)} (m) = \mathsf{l}_k^{(c+1)} (m-1) - \mathsf{l}_k^{(c+1)} (m).
			\end{flalign*}
		\end{enumerate}	
	\end{thm} 
	
	Let us make several comments on this theorem. First, the Gibbs property for the line ensemble $\bm{\mathsf{L}}$ (\Cref{3q0} of \Cref{ensemblevertexq0}) does not depend on the initial data $\sigma$ for the stochastic six-vertex model; $\sigma$ instead will eventually appear as a sort of boundary condition for $\bm{\mathsf{L}}$. One cannot directly use this fact to obtain line ensembles for the single-color ($n=1$ case of the) stochastic six-vertex model under general initial data. Indeed, since $0$ is not in the range of $\sigma$, each site on the west boundary $\partial_{\we} \mathcal{D}_{M;N} = \{ 0 \} \times \llbracket 1, N \rrbracket$ of $\mathcal{D}_{M;N}$ is an entrance site for a path of some positive color. Thus, $\sigma$ necessarily gives rise to step (wedge) initial data if $n = 1$. However, one can instead pass to a $n=2$ color stochastic six-vertex model; use $\sigma$ to prescribe an arbitrary boundary condition for where the color $2$ arrows enter along $\partial_{\we} \mathcal{D}_{M;N}$ (having the color $1$ arrows enter at all other sites of $\partial_{\we} \mathcal{D}_{M;N}$); and then project to the color $2$ arrows\footnote{We emphasize, however, that the Gibbs property for $\bm{\mathsf{L}}$ does not seem to persist under this projection, that is, $\bm{\mathsf{L}}^{(2)}$ alone does not satisfy a Gibbs property (though $\bm{\mathsf{L}}^{(1)}$ does; see \Cref{lmerge} and \Cref{lmergefused} below). From this perspective, to treat general initial data (even only for stochastic vertex models with a single color), one must pass to a $n \ge 2$ colored line ensemble.} to yield a single-color stochastic six-vertex model with general boundary conditions along $\partial_{\we} \mathcal{D}_{M;N}$.

	Second, if $n=1$, the constraint that \eqref{lkcq0} holds whenever $\mathsf{L}_k^{(c+1)} (m) > \mathsf{L}_{k+1}^{(c+1)} (m)$ is irrelevant, since $\mathsf{L}_k^{(2)} (m) = 0$ for all $k \ge 1$ and $m \in \llbracket 0, M + N \rrbracket$. The Gibbs property for $\bm{\mathsf{L}}$ then becomes that of non-intersecting, down-right, discrete random paths conditioned to remain ordered. For $n \ge 2$ colors, this constraint is present and must be taken into account. Similar constraints have implicitly (in the language of vertex models) appeared previously in the context of stationary measures for colored interacting particle systems (see the $q=0$ case of \cite[Section 4.2]{CIPSR}). Their presence in our colored line ensembles therefore suggests that the latter may be useful in proving convergence to local stationarity \cite{CIPSR} or the stationary horizon \cite{SHG,SLM} for colored stochastic vertex models. An alternative explanation for the constraint \eqref{lkcq0} (by examining the law of the line ensembles $\big( \bm{\mathsf{L}}^{(2)}, \bm{\mathsf{L}}^{(3)} , \ldots , \bm{\mathsf{L}}^{(c)} \big)$ upon conditioning on the first one $\bm{\mathsf{L}}^{(1)}$) is found in forthcoming work \cite{PP}, where it will be used to prove scaling limit results for the multi-species asymmetric simple exclusion process and colored stochastic six-vertex model.

	\subsection{Outline}
	
	The remainder of this work is organized as follows. In \Cref{ZFunction} we recall the Yang--Baxter equation and certain families of (non)symmetric functions similar to those in \cite{CSVMST}. We use them in \Cref{Line} to produce probability measures related to the colored stochastic six-vertex model. In \Cref{L0} we reinterpret these results to associate colored line ensembles with the stochastic six-vertex model, special cases of which are analyzed in \Cref{ExampleL}. We then generalize this framework to the fused setting in \Cref{Fusion}, producing the associated colored line ensembles for stochastic fused vertex models in \Cref{L0Fused}. Finally, we explain these colored line ensembles in the example of the multi-species discrete time $q$-boson model in \Cref{L2}. The appendices are not directly related to line ensembles and instead include results about degenerating colored vertex models to other systems (along the lines of \Cref{qmodel}). Specifically, in \Cref{VertexPolymer} we explain how to degenerate the colored stochastic fused vertex model to the log-gamma polymer. In \Cref{Converge} we provide an effective rate of convergence to the colored stochastic six-vertex model to the colored ASEP.
	
	\subsection{Notation} 
	
	\label{AB}
		
		For any integers $n \ge 1$ and $i \in \llbracket 1, n \rrbracket$, we let $\bm{e}_i = \bm{e}_i^{(n)} \in \mathbb{R}^n$ denote the coordinate vector whose $i$-th entry is equal to $1$ and whose remaining entries are $0$; we also let $\bm{e}_0 = \bm{e}_0^{(n)} \in \mathbb{R}^n$ denote the vector with all entries equal to $0$. We denote the entries of any vector $\bm{X} \in \mathbb{R}^n$ by $\bm{X} = (X_1, X_2, \ldots , X_n)$, and we set $|\bm{X}| = \sum_{k=1}^n X_k$. For any integers $1 \le i \le j \le n$, we also denote $X_{[i,j]} = \sum_{k=i}^j X_k$. We further write $\bm{X} \ge \bm{Y}$ for any $\bm{X}, \bm{Y} \in \mathbb{R}^n$ if $X_i \ge Y_i$ for each $i \in \llbracket 1, n \rrbracket$. For any $k$-tuple $\bm{w} = (w_1, w_2, \ldots , w_k)$, let $\overleftarrow{w} = (w_k, w_{k-1}, \ldots , w_1)$ denote the order reversal of $\bm{w}$. Throughout this work, we fix a real number $q \in \mathbb{R}$. For any complex number $a \in \mathbb{C}$, we also denote the $q$-Pochhammer symbol $(a; q)_k = \prod_{j=0}^{k-1} (1 - aq^j)$ for each integer $k \ge 0$ and $(a; q)_k = \prod_{j=1}^{-k} (1 - a q^{-j})$ for each integer $j < 0$.
		
		A \emph{signature} is a sequence $\lambda = (\lambda_1, \lambda_2, \ldots , \lambda_{\ell})$ of integers such that $\lambda_1 \ge \lambda_2 \ge \cdots \ge \lambda_{\ell} \ge 0$. A \emph{composition} $\mu = (\mu_1, \mu_2, \ldots , \mu_{\ell}) \in \mathbb{Z}_{\ge 0}^{\ell}$ of some integer $K \ge 0$ is an $\ell$-tuple of nonnegative integers such that $\sum_{j = 1}^{\ell} \mu_j = K$ (in particular, any signature is a composition). The integer $\ell = \ell (\mu)$ is called the \emph{length} of $\mu$, and $K = |\mu|$ is its \emph{size}. Given a composition $\mu$, we let $\mathfrak{m}_k (\mu) = \# \big\{ j \in \llbracket 1, \ell \rrbracket : \mu_j = k \big\}$ denote the multiplicity of $k$ in $\mu$, for any integer $k \ge 0$; we also let $\mathfrak{m}_{\le k} (\mu) = \sum_{j=0}^k \mathfrak{m}_j (\mu)$ denote the number of entries in $\mu$ that are at most equal to $k$.

	\subsection*{Acknowledgements}
	
	The authors thank Ivan Corwin, Milind Hegde, and Shirshendu Ganguly for very valuable conversations. Amol Aggarwal was partially supported by a Packard Fellowship for Science and Engineering, a Clay Research Fellowship, NSF grant DMS-1926686, and the IAS School of Mathematics. Alexei Borodin was partially supported by the NSF grants DMS-1664619, DMS-1853981, and the Simons Investigator program.

	\section{Yang--Baxter Equation and Partition Functions} 
	
	\label{ZFunction}
	
	In this section we collect (largely from \cite{CSVMST}) several results on the Yang--Baxter equation and certain families of (non)symmetric functions. In \Cref{RLL} we recall a certain family of weights and the Yang--Baxter equation they satisfy. In \Cref{FunctionsZ} we provide notation for partition functions and height functions. This will be used to define certain (non)symmetric functions $f$ and $G$ in \Cref{FunctionfG}, whose properties we recall in \Cref{IdentitiesfG}. Throughout this section, we fix an integer $n \ge 1$.

	\subsection{Yang--Baxter Equation} 
	
	\label{RLL}
	
	In this section we introduce further classes of weights (in addition to the $R_z$ ones given by \Cref{rzabcd}), denoted by $L_{z;s}$ and $\widehat{L}_{z;s}$, and state the Yang--Baxter equation that they satisfy.
	
	Associated with an $L$-weight is a \emph{colored higher spin arrow configuration}, which is a quadruple $(\bm{A}, b; \bm{C}, d)$ with $b, d \in \llbracket 0, n \rrbracket$ and $\bm{A}, \bm{C} \in \mathbb{Z}_{\ge 0}^n$. We view this as an assignment of directed up-right colored arrows to a vertex $v \in \mathbb{Z}^2$; the horizontal edges incident to $v$ again accommodate one arrow,\footnote{We will remove this restriction in \Cref{FusedPath} below, through fusion.} but now the vertical edges incident to $v$ can accommodate arbitrarily many arrows. In particular, $b$ and $d$ denote the colors of the arrows horizontally entering and exiting $v$, respectively, and $A_k$ and $C_k$ denote the number of arrows of color $k$ vertically entering and exiting $v$, respectively, for each $k \in \llbracket 0, n \rrbracket$. In what follows, for any $i, j \in \llbracket 0, n \rrbracket$ and $\bm{X} \in \mathbb{R}^n$, we set
	\begin{flalign*}
		 \bm{X}_i^+ = \bm{X} + \bm{e}_i; \qquad \bm{X}_j^- = \bm{X} - \bm{e}_j; \qquad \bm{X}_{ij}^{+-} = \bm{X} + \bm{e}_i - \bm{e}_j.
	\end{flalign*}

	\begin{figure} 
		\begin{center}
			
			\begin{tikzpicture}[
				>=stealth,
				scale = .9
				]	
				\draw[-, black] (-7.5, 3.1) -- (7.5, 3.1);
				\draw[-, black] (-7.5, -2.1) -- (7.5, -2.1);
				\draw[-, black] (-7.5, -1.1) -- (7.5, -1.1);
				\draw[-, black] (-7.5, -.4) -- (7.5, -.4);
				\draw[-, black] (-7.5, 2.4) -- (7.5, 2.4);
				\draw[-, black] (-7.5, -2.1) -- (-7.5, 3.1);
				\draw[-, black] (7.5, -2.1) -- (7.5, 3.1);
				\draw[-, black] (-5, -2.1) -- (-5, 2.4);
				\draw[-, black] (5, -2.1) -- (5, 3.1);
				\draw[-, black] (-2.5, -2.1) -- (-2.5, 2.4);
				\draw[-, black] (2.5, -2.1) -- (2.5, 2.4);
				\draw[-, black] (0, -2.1) -- (0, 3.1);
				\draw[->, thick, blue] (-6.3, .1) -- (-6.3, 1.9); 
				\draw[->, thick, green] (-6.2, .1) -- (-6.2, 1.9); 
				\draw[->, thick, blue] (-3.8, .1) -- (-3.8, 1) -- (-2.85, 1);
				\draw[->, thick, green] (-3.7, .1) -- (-3.7, 1.9);
				\draw[->, thick, blue] (-1.35, .1) -- (-1.35, 1.9);
				\draw[->, thick, green] (-1.25, .1) -- (-1.25, 1.9);
				\draw[->, thick,  orange] (-2.15, 1.1) -- (-1.15, 1.1) -- (-1.15, 1.9);
				
				\draw[->, thick, red] (.35, 1) -- (1.15, 1) -- (1.15, 1.9);
				\draw[->, thick, blue] (1.25, .1) -- (1.25, 1.9);
				\draw[->, thick, green] (1.35, .1) -- (1.35, 1.1) -- (2.15, 1.1);
				\draw[->, thick, blue] (3.65, .1) -- (3.65, 1) -- (4.65, 1);
				\draw[->, thick, green] (3.75, .1) -- (3.75, 1.9);
				\draw[->, thick, orange] (2.85, 1.1) -- (3.85, 1.1) -- (3.85, 1.9); 
				\draw[->, thick, red] (5.35, 1) -- (7.15, 1); 
				\draw[->, thick, blue] (6.2, .1) -- (6.2, 1.9);
				\draw[->, thick, green] (6.3, .1) -- (6.3, 1.9); 
				\filldraw[fill=gray!50!white, draw=black] (-2.85, 1) circle [radius=0] node [black, right = -1, scale = .7] {$i$};
				\filldraw[fill=gray!50!white, draw=black] (2.15, 1) circle [radius=0] node [black, right = -1, scale = .7] {$j$};
				\filldraw[fill=gray!50!white, draw=black] (4.65, 1) circle [radius=0] node [black, right = -1, scale = .7] {$i$};
				\filldraw[fill=gray!50!white, draw=black] (7.15, 1) circle [radius=0] node [black, right = -1, scale = .7] {$i$};
				\filldraw[fill=gray!50!white, draw=black] (5.35, 1) circle [radius=0] node [black, left = -1, scale = .7] {$i$};
				\filldraw[fill=gray!50!white, draw=black] (2.85, 1) circle [radius=0] node [black, left = -1, scale = .7] {$j$};
				\filldraw[fill=gray!50!white, draw=black] (.35, 1) circle [radius=0] node [black, left = -1, scale = .7] {$i$};
				\filldraw[fill=gray!50!white, draw=black] (-2.15, 1) circle [radius=0] node [black, left = -1, scale = .7] {$i$};
				\filldraw[fill=gray!50!white, draw=black] (-6.25, 1.9) circle [radius=0] node [black, above = -1, scale = .7] {$\bm{A}$};
				\filldraw[fill=gray!50!white, draw=black] (-3.75, 1.9) circle [radius=0] node [black, above = -1, scale = .7] {$\bm{A}_i^-$};
				\filldraw[fill=gray!50!white, draw=black] (-1.25, 1.9) circle [radius=0] node [black, above = -1, scale = .7] {$\bm{A}_i^+$};
				\filldraw[fill=gray!50!white, draw=black] (1.25, 1.9) circle [radius=0] node [black, above = -1, scale = .65] {$\bm{A}_{ij}^{+-}$};
				\filldraw[fill=gray!50!white, draw=black] (3.75, 1.9) circle [radius=0] node [black, above = -1, scale = .65] {$\bm{A}_{ji}^{+-}$};
				\filldraw[fill=gray!50!white, draw=black] (6.25, 1.9) circle [radius=0] node [black, above = -1, scale = .7] {$\bm{A}$};
				\filldraw[fill=gray!50!white, draw=black] (-6.25, .1) circle [radius=0] node [black, below = -1, scale = .7] {$\bm{A}$};
				\filldraw[fill=gray!50!white, draw=black] (-3.75, .1) circle [radius=0] node [black, below = -1, scale = .7] {$\bm{A}$};
				\filldraw[fill=gray!50!white, draw=black] (-1.25, .1) circle [radius=0] node [black, below = -1, scale = .7] {$\bm{A}$};
				\filldraw[fill=gray!50!white, draw=black] (1.25, .1) circle [radius=0] node [black, below = -1, scale = .7] {$\bm{A}$};
				\filldraw[fill=gray!50!white, draw=black] (3.75, .1) circle [radius=0] node [black, below = -1, scale = .7] {$\bm{A}$};
				\filldraw[fill=gray!50!white, draw=black] (6.25, .1) circle [radius=0] node [black, below = -1, scale = .7] {$\bm{A}$};
				\filldraw[fill=gray!50!white, draw=black] (-3.75, 2.75) circle [radius=0] node [black] {$1 \le i \le n$};
				\filldraw[fill=gray!50!white, draw=black] (2.5, 2.75) circle [radius=0] node [black] {$1 \le i < j \le n$}; 
				\filldraw[fill=gray!50!white, draw=black] (6.25, 2.75) circle [radius=0] node [black] {};
				\filldraw[fill=gray!50!white, draw=black] (-6.25, -.75) circle [radius=0] node [black, scale = .8] {$(\bm{A}, 0; \bm{A}, 0)$};
				\filldraw[fill=gray!50!white, draw=black] (-3.75, -.75) circle [radius=0] node [black, scale = .8] {$\big( \bm{A}, 0; \bm{A}_i^-, i \big)$};
				\filldraw[fill=gray!50!white, draw=black] (-1.25, -.75) circle [radius=0] node [black, scale = .8] {$\big( \bm{A}, i; \bm{A}_i^+, 0 \big)$};
				\filldraw[fill=gray!50!white, draw=black] (1.25, -.75) circle [radius=0] node [black, scale = .8] {$\big( \bm{A}, i; \bm{A}_{ij}^{+-}, j \big)$};
				\filldraw[fill=gray!50!white, draw=black] (3.75, -.75) circle [radius=0] node [black, scale = .8] {$\big( \bm{A}, j; \bm{A}_{ji}^{+-}, i \big)$};
				\filldraw[fill=gray!50!white, draw=black] (6.25, -.75) circle [radius=0] node [black, scale = .8] {$(\bm{A}, i; \bm{A}, i)$};
				\filldraw[fill=gray!50!white, draw=black] (-6.25, -1.6) circle [radius=0] node [black, scale = .8] {$\displaystyle\frac{1 - s x q^{A_{[1, n]}}}{1 - sx}$};
				\filldraw[fill=gray!50!white, draw=black] (-3.75, -1.6) circle [radius=0] node [black, scale = .75] {$\displaystyle\frac{x (1 - q^{A_i}) q^{A_{[i+1, n]}}}{1-sx}$};
				\filldraw[fill=gray!50!white, draw=black] (-1.25, -1.6) circle [radius=0] node [black, scale = .8] {$\displaystyle\frac{1 - s^2 q^{A_{[1, n]}}}{1-sx}$};
				\filldraw[fill=gray!50!white, draw=black] (1.25, -1.6) circle [radius=0] node [black, scale = .75] {$\displaystyle\frac{x (1 - q^{A_j}) q^{A_{[j+1,n]}}}{1-sx}$};
				\filldraw[fill=gray!50!white, draw=black] (3.75, -1.6) circle [radius=0] node [black, scale = .75] {$\displaystyle\frac{s (1 - q^{A_i}) q^{A_{[i+1,n]}}}{1-sx}$};
				\filldraw[fill=gray!50!white, draw=black] (6.25, -1.6) circle [radius=0] node [black, scale = .75] {$\displaystyle\frac{(x-sq^{A_i})q^{A_{[i+1,n]}}}{1-sx}$};
			\end{tikzpicture}
		\end{center}	
		\caption{\label{lz} Depicted above are the $L_{x;s}$ weights.} 
	\end{figure}

	\begin{definition}
		
		\label{lzdefinition}
		
		Fix complex numbers $x, s \in \mathbb{C}$; define the $L_{x;s} = L_{x;s}^{(n)}$ vertex weight as follows. For any $i \in \llbracket 1, n \rrbracket$ and $\bm{A} \in \mathbb{Z}_{\ge 0}^n$, set
		\begin{flalign}
			\label{lzi}
			\begin{aligned}
				& L_{x;s} (\bm{A}, 0; \bm{A}, 0) = \displaystyle\frac{1 - sx q^{A_{[1,n]}}}{1-sx}; \qquad L_{x;s} (\bm{A}, 0; \bm{A}_i^-, i) = \displaystyle\frac{x (1 - q^{A_i}) q^{A_{[i + 1, n]}}}{1-sx}; \\ 
				& L_{x;s} (\bm{A}, i; \bm{A}_i^+, 0) = \displaystyle\frac{1-s^2 q^{A_{[1,n]}}}{1-sx}; \qquad L_{x;s} (\bm{A}, i; \bm{A}, i) = \displaystyle\frac{(x-sq^{A_i}) q^{A_{[i+1,n]}}}{1-sx}.
			\end{aligned} 
		\end{flalign} 
		
		\noindent Moreover, for any $1 \le i < j \le n$, set
		\begin{flalign}
			\label{lzij}
			L_{x;s} (\bm{A}, i; \bm{A}_{ij}^{+-}, j) = \displaystyle\frac{x (1 - q^{A_j}) q^{A_{[j + 1, n]}}}{1-sx}; \qquad L_{x;s} (\bm{A}, j; \bm{A}_{ji}^{+-}, i) = \displaystyle\frac{s(1-q^{A_i}) q^{A_{[i+1,n]}}}{1-sx}.
		\end{flalign}	
		
		\noindent We also set $L_{x;s} (\bm{A}, b; \bm{C}, d) = 0$ if $(\bm{A}, b; \bm{C}, d)$ is not of the above form (with $\bm{A}, \bm{C} \in \mathbb{Z}_{\ge 0}^n$); see \Cref{lz} for a depiction. Also define a normalization $\widehat{L}_{x;s} = \widehat{L}_{x;s}^{(n)}$ of the $L_{x;s}$ weights, by setting 
		\begin{flalign}
			\label{lzij2} 
			\widehat{L}_{x;s} (\bm{A}, b; \bm{C}, d) = \displaystyle\frac{1-sx}{x-s} \cdot L_{x;s} (\bm{A}, b; \bm{C}, d),
		\end{flalign}
		
		\noindent for any $b, d \in \llbracket 0, n \rrbracket$ and $\bm{A}, \bm{C} \in \mathbb{Z}_{\ge 0}^n$. In particular, we have $\widehat{L}_x (\bm{e}_0, i; \bm{e}_0, i) = 1$, for any $i \in \llbracket 1, n \rrbracket$. 
	\end{definition}

	The following proposition indicates that the $R$-weights and $L$-weights from \Cref{rzabcd} and \Cref{lzdefinition} satisfy the Yang--Baxter equation. It was originally due to \cite{QMGS,TSTECA,ERT} (though we adopt the notation of \cite{CSVMST}), but it can also be verified directly from the explicit forms of these weights. 
	
	\begin{lem}[{\cite[Proposition 2.3.1]{CSVMST}}]
		
		\label{rrr1} 
		
		Fix any complex numbers $s, x, y, z \in \mathbb{C}$ with $x, y \ne 0$, and indices $i_1, j_1, k_1, i_3, j_3, k_3 \in \llbracket 0, n \rrbracket$. We have  
		\begin{flalign}
			\label{rrrijk}
			\begin{aligned}
				\displaystyle\sum_{0 \le i_2, j_2, k_2 \le n} & R_{y / x} (i_1, j_1; i_2, j_2) R_{z / x} (k_1, j_2; k_2, j_3) R_{z / y} (k_2, i_2; k_3, i_3) \\
				& = \displaystyle\sum_{0 \le i_2, j_2, k_2 \le n} R_{z / y} (k_1, i_1; k_2, i_2) R_{z / x} (k_2, j_1; k_3, j_2) R_{y / x} (i_2, j_2; i_3, j_3).
			\end{aligned}
		\end{flalign}
		
		\noindent Further fixing integer sequences $\bm{K}_1, \bm{K}_3 \in \mathbb{Z}_{\ge 0}^n$, we have 
		\begin{flalign}
			\label{llrijk}
			\begin{aligned}
				\displaystyle\sum_{i_2, j_2, \bm{K}_2} & R_{y / x} (i_1, j_1; i_2, j_2) L_{x;s} (\bm{K}_1, j_2; \bm{K}_2, j_3) L_{y;s} (\bm{K}_2, i_2; \bm{K}_3, i_3) \\
				& = \displaystyle\sum_{i_2, j_2, \bm{K}_2} L_{y;s} (\bm{K}_1, i_1; \bm{K}_2, i_2) L_{x;s} (\bm{K}_2, j_1; \bm{K}_3, j_2) R_{y / x} (i_2, j_2; i_3, j_3),
			\end{aligned}
		\end{flalign}
		
		\noindent and 
		\begin{flalign}
			\label{llrijk2}
			\begin{aligned}
				\displaystyle\sum_{i_2, j_2, \bm{K}_2} & R_{y / x} (i_1, j_1; i_2, j_2) \widehat{L}_{x;s} (\bm{K}_1, j_2; \bm{K}_2, j_3) L_{y;s} (\bm{K}_2, i_2; \bm{K}_3, i_3) \\
				& = \displaystyle\sum_{i_2, j_2, \bm{K}_2} L_{y;s} (\bm{K}_1, i_1; \bm{K}_2, i_2) \widehat{L}_{x;s} (\bm{K}_2, j_1; \bm{K}_3, j_2) R_{y / x} (i_2, j_2; i_3, j_3),
			\end{aligned}
		\end{flalign}
		
		\noindent where in both equations $i_2, j_2$ are ranged over $\llbracket 0, n \rrbracket$, and $\bm{K}_2$ is ranged over $\mathbb{Z}_{\ge 0}^n$.
		
	\end{lem} 
	
	It will often be useful to interpret such equations diagrammatically; the diagrammatic interpretation of \eqref{rrrijk} is given by
	
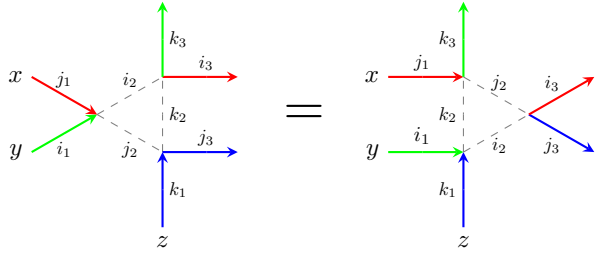
\begin{center}		
		\begin{tikzpicture}[
			>=stealth,
			auto,
			style={
				scale = 1
			}
			]
			
			\draw[->, thick, green] (-.87, -.5) -- (0, 0);
			\draw[->, thick, red] (-.87, .5) -- (0, 0);
			\draw[->, thick, blue] (.87, -1.5) -- (.87, -.5); 
			
			\draw[] (-.87, .5) circle[radius = 0]  node[left, scale = .95]{$x$};
			\draw[] (-.87, -.5) circle[radius = 0]  node[left, scale = .95]{$y$};
			\draw[] (.87, -1.5) circle[radius = 0]  node[below, scale = .95]{$z$};
			
			\draw[gray, dashed] (0, 0) -- (.87, -.5); 
			\draw[gray, dashed] (0, 0) -- (.87, .5); 
			\draw[gray, dashed] (.87, -.5) -- (.87, .5);
			
			\draw[->, thick, red] (.87, .5) -- (1.87, .5); 
			\draw[->, thick, blue] (.87, -.5) -- (1.87, -.5); 
			\draw[->, thick, green] (.87, .5) -- (.87, 1.5);

			\draw[] (3.87, .5) circle[radius = 0]  node[left, scale = .95]{$x$};
			\draw[] (3.87, -.5) circle[radius = 0]  node[left, scale = .95]{$y$};
			\draw[] (4.87, -1.5) circle[radius = 0]  node[below, scale = .95]{$z$};
			
			\draw[->, thick, blue] (4.87, -1.5) -- (4.87, -.5); 
			\draw[->, thick, red] (3.87, .5) -- (4.87, .5); 
			\draw[->, thick, green] (3.87, -.5) -- (4.87, -.5); 
			\draw[->, thick, green] (4.87, .5) -- (4.87, 1.5); 
			\draw[->, thick, blue] (5.74, 0) -- (6.61, -.5); 
			\draw[->, thick, red] (5.74, 0) -- (6.61, .5); 
			
			\draw[gray, dashed] (4.87, -.5) -- (4.87, .5);
			\draw[gray, dashed] (4.87, -.5) -- (5.74, 0); 
			\draw[gray, dashed] (4.87, .5) -- (5.74, 0);

			\filldraw[fill=white, draw=black] (2.75, 0) circle [radius=0] node[scale = 2]{$=$};
			
			\filldraw[fill=white, draw=black] (-.44, -.275) circle [radius=0] node[below, scale = .7]{$i_1$};
			\filldraw[fill=white, draw=black] (.44, .275) circle [radius=0] node[above, scale = .7]{$i_2$};
			\filldraw[fill=white, draw=black] (1.45, .5) circle [radius=0] node[above, scale = .7]{$i_3$};
			\filldraw[fill=white, draw=black] (-.44, .275) circle [radius=0] node[above, scale = .7]{$j_1$};
			\filldraw[fill=white, draw=black] (.44, -.275) circle [radius=0] node[below, scale = .7]{$j_2$};
			\filldraw[fill=white, draw=black] (1.45, -.5) circle [radius=0] node[above, scale = .7]{$j_3$};
			\filldraw[fill=white, draw=black] (.87, -1) circle [radius=0] node[right, scale = .7]{$k_1$};
			\filldraw[fill=white, draw=black] (.87, 0) circle [radius=0] node[right, scale = .7]{$k_2$};
			\filldraw[fill=white, draw=black] (.87, 1) circle [radius=0] node[right, scale = .7]{$k_3$};
			
			\filldraw[fill=white, draw=black] (4.32, -.5) circle [radius=0] node[above, scale =.8]{$i_1$};
			\filldraw[fill=white, draw=black] (5.35, -.225) circle [radius=0] node[below, scale = .7]{$i_2$};
			\filldraw[fill=white, draw=black] (6.05, .225) circle [radius=0] node[above, scale = .7]{$i_3$};
			\filldraw[fill=white, draw=black] (4.32, .5) circle [radius=0] node[above, scale = .7]{$j_1$};
			\filldraw[fill=white, draw=black] (5.35, .225) circle [radius=0] node[above, scale = .7]{$j_2$};
			\filldraw[fill=white, draw=black] (6.05, -.225) circle [radius=0] node[below, scale = .7]{$j_3$};
			\filldraw[fill=white, draw=black] (4.87, -1) circle [radius=0] node[left, scale = .7]{$k_1$};
			\filldraw[fill=white, draw=black] (4.87, 0) circle [radius=0] node[left, scale = .7]{$k_2$};
			\filldraw[fill=white, draw=black] (4.87, 1) circle [radius=0] node[left, scale = .7]{$k_3$};

		\end{tikzpicture}
		
	\end{center}
	
	\noindent where on each side of the equation is a family of vertices, and we view the weight of each family as the product of the weights of its constituent vertices. Along the solid edges the colors are fixed, and along the dashed ones they are summed over. The equations \eqref{llrijk} and \eqref{llrijk2} similarly have diagrammatic interpretations (which we do not depict here).

	\subsection{Height Functions and Partition Functions}
	
	\label{FunctionsZ}
	
	In this section we introduce several partition functions (that is, sums of weights of colored path ensembles), which will be of use to us. Similarly to the notion of a colored six-vertex ensemble from \Cref{ModelVertex}, a \emph{colored higher spin path ensemble} on a domain $\mathcal{D} \subseteq \mathbb{Z}^2$ is a consistent assignment of a colored higher spin arrow configuration $\big( \bm{A}(v), b(v); \bm{C}(v), d(v) \big)$ to each vertex $v \in \mathcal{D}$. 
	
	Associated with a colored higher spin path ensemble are \emph{height functions}, which count how many paths of specified colors are to the right of (equivalently, below) or to the left of (equivalently, above) a given location. More specifically, given a colored higher spin path ensemble $\mathcal{E}$ on a domain $\mathcal{D} \subseteq \mathbb{Z}^2$, for any integer $c \ge 0$, define $\mathfrak{h}_c^{\rightarrow} : \mathbb{Z}^2 \rightarrow \mathbb{Z}$ and $\mathfrak{h}_c^{\leftarrow} : \mathbb{Z}^2 \rightarrow \mathbb{Z}$ as follows. For any $(i, j) \in \mathbb{Z}^2$, set $\mathfrak{h}_c^{\rightarrow} (i, j)$ to be the number of arrows of color $c$ in $\mathcal{E}$ that intersect the vertical ray from $(i + 1 / 2, j + 1 / 2)$ to $(i + 1 / 2, -\infty)$; similarly set $\mathfrak{h}_c^{\leftarrow} (i, j)$ to be the number of arrows of color $c$ in $\mathcal{E}$ that do not intersect the vertical ray from $(i + 1 / 2, j + 1 / 2)$ and $(i + 1 / 2, -\infty)$. Further define $\mathfrak{h}_{\ge c}^{\rightarrow} : \mathbb{Z}^2 \rightarrow \mathbb{Z}_{\ge 0}$ and $\mathfrak{h}_{\ge c}^{\leftarrow} : \mathbb{Z}^2 \rightarrow \mathbb{Z}_{\ge 0}$ by setting $\mathfrak{h}_{\ge c}^{\rightarrow} (i, j) = \sum_{k=c}^{\infty} \mathfrak{h}_{\ge c}^{\rightarrow} (i, j)$ and $\mathfrak{h}_{\ge c}^{\leftarrow} (i, j) = \sum_{k=c}^{\infty} \mathfrak{h}_{\ge c}^{\leftarrow} (i, j)$, for each $(i, j) \in \mathbb{Z}^2$. Since the colored stochastic six-vertex model gives rise to a random colored six-vertex ensemble on $\mathcal{D}$, it also gives rise to a family of random height functions.

	We next introduce notation for weights of path ensembles on negative half-strips of the form $\mathbb{Z}_{\le 0} \times \llbracket 1, N \rrbracket$ (which will frequently be the domain for our models). Observe in what follows that, in the ``bulk'' $\mathbb{Z}_{< 0} \times \llbracket 1, N \rrbracket$ of the half-strip, we take the spin $s$ to be generic and use the weights $L_{x;s}$ (or $\widehat{L}_{x;s}$). However, on the $y$-axis boundary of the half-strip, we take $s = 0$ and use the weights $L_{x;0}$. This will later be relevant for producing stochastic matchings, such as \Cref{fgsv} below. 
	
	\begin{definition}
		
		\label{dwe} 
		
		Fix an integer $N \ge 1$; a complex number $s \in \mathbb{C}$; a sequence of complex numbers $\bm{x} = (x_1, x_2, \ldots , x_N)$; and a colored higher spin path ensemble $\mathcal{E}$ on $\mathcal{D}_N = \mathbb{Z}_{\le 0} \times \llbracket 1, N \rrbracket$, whose arrow configuration at any vertex $v \in \mathcal{D}_N$ is denoted by $\big( \bm{A}(v), b(v); \bm{C}(v), d(v) \big)$. Set 
		\begin{flalign}
			\label{lele} 
			\begin{aligned} 
			& L_{\bm{x}; s} (\mathcal{E}) = \displaystyle\prod_{k = 1}^{\infty} \displaystyle\prod_{j=1}^N L_{z_j; s} \big(\bm{A}(-k, j), b(-k, j); \bm{C}(-k, j), d(-k, j) \big) \\
			& \qquad \qquad \quad \times \displaystyle\prod_{j=1}^N L_{x_j; 0} \big( \bm{A} (0, j), b(0, j); \bm{C} (0, j), d(0, j) \big); \\
			& \widehat{L}_{\bm{x}; s} (\mathcal{E}) = \displaystyle\prod_{k=1}^{\infty} \displaystyle\prod_{j=1}^N \widehat{L}_{x_j; s} \big(\bm{A}(-k, j), b(-k, j); \bm{C}(-k, j), d(-k, j) \big) \\
			& \qquad \qquad \quad \times \displaystyle\prod_{j=1}^N L_{x_j; 0} \big( \bm{A} (0, j), b(0, j); \bm{C} (0, j), d(0, j) \big) .
			\end{aligned} 
		\end{flalign}
	
		\noindent The above notation implicitly assumes that, in each infinite product, all but finitely many factors are equal to $1$; this will always be the case below.
		
	\end{definition}

	We next have the following definition for certain types of compositions; below, we recall from \Cref{AB} the notation $X_{[i,j]} = \sum_{k=i}^j X_k$ for any $\bm{X} \in \mathbb{R}^n$.

	\begin{definition}
		\label{compositionj} 
		
		Let $N \ge 0$ be an integer and $\bm{\ell} = (\ell_1, \ell_2, \ldots , \ell_n) \in \mathbb{Z}_{\ge 0}^n$ be a composition of $N$. A composition $\mu = (\mu_1, \mu_2, \ldots , \mu_N)$ is called \emph{\bm{$\ell$}-colored} if $\mu_i \ge \mu_j$ whenever $\ell_{[1,c-1]} + 1 \le i \le j \le \ell_{[1,c]}$, for each $c \in \llbracket 1, n \rrbracket$; we then denote the signature $\mu^{(c)} = \big(\mu_{\ell_{[1,c-1]}+1}, \mu_{\ell_{[1,c-1]} + 2}, \ldots , \mu_{\ell_{[1,c]}} \big)$. 
		
		If $\mu$ is $\bm{\ell}$-colored for some $\bm{\ell} \in \mathbb{Z}_{\ge 0}^n$, then we call $\mu$ an \emph{$n$-composition}, and we write $\mu = \big( \mu^{(1)} \mid \mu^{(2)} \mid \cdots \mid \mu^{(n)} \big)$. Let $\Comp_n$ denote the set of $n$-compositions, and let $\Comp_n (N) \subseteq \Comp_n$ denote the set of $n$-compositions of length $N$. For any $n$-composition $\mu \in \Comp_n$, and integers $k \ge 0$ and $c \in \llbracket 1, n \rrbracket$, define the (sums of) multiplicities $\mathfrak{m}_k^{\ge c} (\mu) = \sum_{i=c}^n \mathfrak{m}_k \big( \mu^{(i)} \big)$ and $\mathfrak{m}_{\le k}^{\ge c} (\mu) = \sum_{i=c}^n \mathfrak{m}_{\le k} \big( \mu^{(i)} \big)$. 
			
	\end{definition}
	
	\begin{rem} 
		
	\label{mumuistate} 
	
	Any $n$-composition $\mu \in \Comp_n (N)$ indexes a family of $N$ colored arrows vertically exiting the row $\mathbb{Z}_{\le 0}$, in which $\mathfrak{m}_k \big( \mu^{(c)} \big)$ arrows of color $c$ exit through site $-k$, for all integers $c \in \llbracket 1, n \rrbracket$ and $k \ge 0$.
	
	\end{rem}

	\subsection{Nonsymmetric Functions}
	
	\label{FunctionfG}
	
	In this section we define a family of nonsymmetric functions $f$, and of symmetric ones $G$, as partition functions for the vertex model with weights given by \Cref{lzdefinition}. They are similar to those from \cite[Definition 3.5.1]{CSVMST} and \cite[Definition 4.4.1]{CSVMST}, respectively.

	\begin{definition}
		\label{fg} 
		
		Fix an integer $N \ge 0$; two $n$-compositions $\mu = \big( \mu^{(1)} \mid \mu^{(2)} \mid \cdots \mid \mu^{(n)} \big)$ and $\nu = \big( \nu^{(1)} \mid \nu^{(2)} \mid \cdots \mid \mu^{(n)} \big)$; and a function  $\sigma : \llbracket 1, N \rrbracket \rightarrow \llbracket 1, n \rrbracket$. 
	
		If $\ell (\mu) = \ell (\nu) + N$, then let $\mathfrak{P}_f (\mu/\nu; \sigma)$ denote the set of colored higher spin path ensembles on $\mathcal{D}_N = \mathbb{Z}_{\le 0} \times \llbracket 1, N \rrbracket$ with the following boundary data.  
		
		\begin{enumerate}
			\item For each $j \in \llbracket 1, N \rrbracket$, an arrow of color $\sigma (j)$ horizontally enters $\mathcal{D}_N$ through\footnote{This means that, for sufficiently large $i$, each edge between $(-i-1, j)$ and $(-i, j)$ contains an arrow of color $\sigma(j)$.} $(-\infty, j)$. 
			\item For each $k \ge 0$ and $c \in \llbracket 1, n \rrbracket$, $\mathfrak{m}_k \big( \nu^{(c)} \big)$ arrows of color $c$ vertically enter $\mathcal{D}_N$ through $(-k, 1)$. 
			\item For each $k \ge 0$ and $c \in \llbracket 1, n \rrbracket$, $\mathfrak{m}_k \big(\mu^{(c)} \big)$ arrows of color $c$ vertically exit $\mathcal{D}_N$ through $(-k, N)$.  
		\end{enumerate}
		
		\noindent See the left side of \Cref{fgfunctions} for a depiction when $\mu = (7, 5 \mid 5, 4, 1 \mid 3, 2, 2)$; $\nu = (\emptyset \mid 6 \mid 6, 5)$; and $\big( \sigma (1), \sigma(2), \sigma(3), \sigma(4), \sigma(5) \big) = (1, 3, 2, 1, 2)$. There, red, green, and blue are colors $1$, $2$, and $3$, respectively.
		
		Similarly, if $\ell (\mu) = \ell (\nu)$, then let $\mathfrak{P}_G (\mu / \nu)$ denote the set of colored higher spin path ensembles on $\mathcal{D}_N = \mathbb{Z}_{\le 0} \times \llbracket 1, N \rrbracket$, with the following boundary data. 
		
		\begin{enumerate} 
			\item For each integer $j \in \llbracket 1, N \rrbracket$, no arrow horizontally enters or exits $\mathcal{D}_N$ through the $j$-th row.
			\item For each $k \ge 0$ and $c \in \llbracket 1, n \rrbracket$, $\mathfrak{m}_k \big( \mu^{(c)} \big)$ arrows of color $c$ vertically enter $\mathcal{D}_N$ through $(-k, 1)$.
			\item For each $k \ge 0$ and $c \in \llbracket 1, n \rrbracket$, $\mathfrak{m}_k \big( \nu^{(c)} \big)$ arrows of color $c$ vertically exit $\mathcal{D}_N$ through $(-k, N)$.
		\end{enumerate} 
		
		\noindent  See the right side of \Cref{fgfunctions} for a depiction when $\mu = (7, 5 \mid 7, 6 \mid 6, 4)$ and $\nu = (5, 2 \mid 5, 1 \mid 3, 2)$.  
		
		For any complex number $s \in \mathbb{C}$ and sequence of complex numbers $\bm{x} = (x_1, x_2, \ldots , x_N)$, let
		\begin{flalign}
			\label{fgmunu} 
			f_{\mu / \nu; s}^{\sigma} (\bm{x}) = \displaystyle\sum_{\mathfrak{P}_f (\mu/ \nu; \sigma)} \widehat{L}_{\bm{x}; s}(\mathcal{E}); \qquad G_{\mu/\nu; s} (\bm{x}) = \displaystyle\sum_{\mathfrak{P}_G (\mu/\nu)} L_{\bm{x}; s} (\mathcal{E}).
		\end{flalign}
		
		\noindent If $\nu = \emptyset$ is empty, we write $f_{\mu; s}^{\sigma} (\bm{x}) = f_{\mu/\emptyset; s}^{\sigma} (\bm{x})$ and $G_{\mu; s} (\bm{x}) = G_{\mu/0^N; s} (\bm{x})$. If $N = 1$, we may write $f_{\mu/\nu; s}^{\sigma(1)}$ in place of $f_{\mu/\nu; s}^{\sigma}$. 
	\end{definition}
	
	Observe that the quantity $\widehat{L}_{\bm{x}; s} (\mathcal{E})$ appearing as the summand in \eqref{fgmunu} defining $f_{\mu/\nu; s}^{\sigma}$ is bounded, since all but finitely many of the vertices in any ensemble $\mathcal{E} \in \mathfrak{P}_f (\mu/\nu; \sigma)$ have arrow configurations of the form $(\bm{e}_0, i; \bm{e}_0, i)$ for some integer $i \in \llbracket 1, n \rrbracket$, and we have $\widehat{L}_{x_j} (\bm{e}_0, i; \bm{e}_0, i) = 1$ by \eqref{lzij2} and \eqref{lzi}. Similarly, $L_{\bm{x}; s} (\mathcal{E})$ appearing as the summand in \eqref{fgmunu} defining $G_{\mu/\nu; s}$ is bounded, since all but finitely many vertices in any $\mathcal{E} \in \mathfrak{P}_G (\mu/\nu)$ have arrow configurations of the form $(\bm{e}_0, 0; \bm{e}_0, 0)$, and we have $L_x (\bm{e}_0, 0; \bm{e}_0, 0) = 1$ by \eqref{lzi}.

	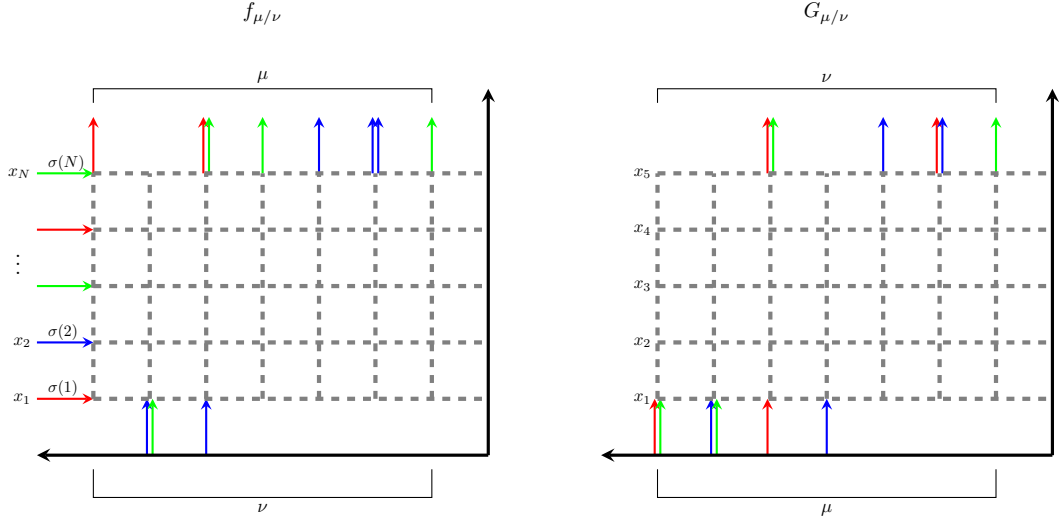
\begin{figure}
		\begin{center}
			\begin{tikzpicture}[
				>=stealth, 
				scale = .75]{				
					\draw[->, thick, red] (0, 1) node[above = 4, right = 2, black,  scale = .6]{$\sigma(1)$} -- (1, 1);
					\draw[->, thick, blue] (0, 2) node[above = 4, right = 2, black,  scale = .6]{$\sigma(2)$} -- (1, 2);
					\draw[->, thick, green] (0, 3) -- (1, 3);
					\draw[->, thick, red] (0, 4) -- (1, 4);
					\draw[->, thick, green] (0, 5) node[above = 4, right = 2, black,  scale = .6]{$\sigma(N)$} -- (1, 5);
					
					\draw[->, thick, blue] (1.95, 0) -- (1.95, 1);
					\draw[->, thick, green] (2.05, 0) -- (2.05, 1);
					\draw[->, thick, blue] (3, 0) -- (3, 1);
					
					\draw[->, thick, red] (1, 5) -- (1, 6);
					\draw[->, thick, red] (2.95, 5) -- (2.95, 6); 
					\draw[->, thick, green] (4, 5) -- (4, 6);
					\draw[->, thick, blue] (5.95, 5) -- (5.95,6);
					\draw[->, thick, blue] (5, 5) -- (5, 6);
					\draw[->, thick, blue] (6.05, 5) -- (6.05, 6);
					\draw[->, thick, green] (3.05, 5) -- (3.05, 6); 
					\draw[->, thick, green] (7, 5) -- (7, 6);
					
					\draw[ultra thick, gray, dashed] (1, 1) -- (1, 5);
					\draw[ultra thick, gray, dashed] (2, 1) -- (2, 5);
					\draw[ultra thick, gray, dashed] (3, 1) -- (3, 5);
					\draw[ultra thick, gray, dashed] (4, 1) -- (4, 5);
					\draw[ultra thick, gray, dashed] (5, 1) -- (5, 5);
					\draw[ultra thick, gray, dashed] (6, 1) -- (6, 5);
					\draw[ultra thick, gray, dashed] (7, 1) -- (7, 5);
					
					\draw[ultra thick, gray, dashed] (1, 1) -- (8, 1);
					\draw[ultra thick, gray, dashed] (1, 2) -- (8, 2);
					\draw[ultra thick, gray, dashed] (1, 3) -- (8, 3);
					\draw[ultra thick, gray, dashed] (1, 4) -- (8, 4);
					\draw[ultra thick, gray, dashed] (1, 5) -- (8, 5);
					
					\draw[->, very thick] (8, 0) -- (8, 6.5);
					\draw[->, very thick] (8, 0) -- (0, 0);
					
					\draw[]  (0, 1) circle [radius = 0] node[left, scale = .65]{$x_1$};
					\draw[]  (0, 2) circle [radius = 0] node[left, scale = .65]{$x_2$};
					\draw[]  (-.15, 3.5) circle [radius = 0] node[left, scale = .8]{$\vdots$};
					\draw[]  (0, 5) circle [radius = 0] node[left, scale = .65]{$x_N$};
					
					\draw[-] (1, -.25) -- (1, -.75) -- (7, -.75) -- (7, -.25);
					\draw[-] (1, 6.25) -- (1, 6.5) -- (7, 6.5) -- (7, 6.25);
					
					\draw[] (4, -.75) circle[radius = 0] node[below, scale = .7]{$\nu$};
					\draw[] (4, 6.5) circle[radius = 0] node[above, scale = .7]{$\mu$};
					\draw[] (4, 7.5) circle[radius = 0] node[above, scale = .8]{$f_{\mu / \nu}$};
					
					\draw[->, thick, red] (10.95, 0) -- (10.95, 1);
					\draw[->, thick, red] (12.95, 0) -- (12.95, 1);
					\draw[->, thick, blue] (11.95, 0) -- (11.95, 1);
					\draw[->, thick, blue] (14, 0) -- (14, 1);
					\draw[->, thick, green] (11.05, 0) -- (11.05, 1); 
					\draw[->, thick, green] (12.05, 0) -- (12.05, 1);
					
					\draw[->, thick, red] (12.95, 5) -- (12.95, 6); 
					\draw[->, thick, red] (15.95, 5) -- (15.95,6);
					\draw[->, thick, blue] (15, 5) -- (15, 6);
					\draw[->, thick, blue] (16.05, 5) -- (16.05, 6);
					\draw[->, thick, green] (13.05, 5) -- (13.05, 6); 
					\draw[->, thick, green] (17, 5) -- (17, 6);
					
					\draw[ultra thick, gray, dashed] (11, 1) -- (11, 5);
					\draw[ultra thick, gray, dashed] (12, 1) -- (12, 5);
					\draw[ultra thick, gray, dashed] (13, 1) -- (13, 5);
					\draw[ultra thick, gray, dashed] (14, 1) -- (14, 5);
					\draw[ultra thick, gray, dashed] (15, 1) -- (15, 5);
					\draw[ultra thick, gray, dashed] (16, 1) -- (16, 5);
					\draw[ultra thick, gray, dashed] (17, 1) -- (17, 5);
					
					\draw[ultra thick, gray, dashed] (11, 1) -- (18, 1);
					\draw[ultra thick, gray, dashed] (11, 2) -- (18, 2);
					\draw[ultra thick, gray, dashed] (11, 3) -- (18, 3);
					\draw[ultra thick, gray, dashed] (11, 4) -- (18, 4);
					\draw[ultra thick, gray, dashed] (11, 5) -- (18, 5);
					
					\draw[->, very thick] (18, 0) -- (18, 6.5);
					\draw[->, very thick] (18, 0) -- (10, 0);
					
					\draw[]  (11, 1) circle [radius = 0] node[left, scale = .65]{$x_1$};
					\draw[]  (11, 2) circle [radius = 0] node[left, scale = .65]{$x_2$};
					\draw[]  (11, 3) circle [radius = 0] node[left, scale = .65]{$x_3$};
					\draw[]  (11, 4) circle [radius = 0] node[left, scale = .65]{$x_4$};
					\draw[]  (11, 5) circle [radius = 0] node[left, scale = .65]{$x_5$};

					\draw[-] (11, -.25) -- (11, -.75) -- (17, -.75) -- (17, -.25);
					\draw[-] (11, 6.25) -- (11, 6.5) -- (17, 6.5) -- (17, 6.25);
					
					\draw[] (14, -.75) circle[radius = 0] node[below, scale = .7]{$\mu$};
					\draw[] (14, 6.5) circle[radius = 0] node[above, scale = .7]{$\nu$};
					\draw[] (14, 7.5) circle[radius = 0] node[above, scale = .8]{$G_{\mu / \nu}$};
				}
			\end{tikzpicture}
		\end{center}		
		\caption{\label{fgfunctions} Depicted to the left and right are vertex models for $f_{\mu / \nu; s}^{\sigma}$ and $G_{\mu / \nu; s}$, respectively.} 
		
	\end{figure}

	\subsection{Properties of $f$ and $G$} 
	
	\label{IdentitiesfG}
	
	In this section we provide properties (that are minor variants of those in \cite{CSVMST}) of the $f$ and $G$ functions from \Cref{fg}. The first is the symmetry of $G$ in its arguments; we omit its proof, which follows quickly from the Yang--Baxter equation \eqref{llrijk} (see also \cite[Definition 4.4.1]{CSVMST} or \cite[Proposition 4.7]{HSVMSRF}).
	
	\begin{lem} 
		
	\label{gsymmetric} 
	
	Adopt the notation of \Cref{fg}, and let $\varsigma : \llbracket 1, M \rrbracket \rightarrow \llbracket 1, M \rrbracket$ denote a permutation. We have $G_{\mu/\nu; s} (\bm{y}) = G_{\mu/\nu; s} \big( \varsigma(\bm{y}) \big)$, where $\varsigma(\bm{y}) = (y_{\varsigma(1)}, y_{\varsigma(2)}, \ldots , y_{\varsigma(M)} \big)$. 
	
	\end{lem}    
	
	The second is a \emph{branching identity}; we omit its proof, which is very similar to that of \cite[Proposition 4.2.1]{CSVMST} (and quickly follows from ``cutting'' the vertex models shown in \Cref{fgfunctions} at the line $\{ y = k \}$).
	
	\begin{lem}
		
		\label{ffgg} 
		
		Adopt the notation of \Cref{fg}; let $\ell = \ell (\nu)$; and fix $k \in \llbracket 1, N \rrbracket$. We have  
		\begin{flalign*}
			& f_{\mu / \nu; s}^{\sigma} (\bm{x}) = \displaystyle\sum_{\kappa \in \Comp_n (\ell + k)} f_{\kappa / \nu; s}^{\sigma |_{\llbracket 1,k \rrbracket}} \big(\bm{x}_{[1, k]} \big) f_{\mu / \kappa; s}^{\sigma |_{\llbracket k+1,N \rrbracket}} \big( \bm{x}_{[k+1, N]} \big); \\ 
			& G_{\mu / \nu; s} (\bm{x}) = \displaystyle\sum_{\kappa \in \Comp_n (\ell)} G_{\mu / \kappa; s} \big( \bm{x}_{[1, k]} \big) G_{\kappa / \nu; s} \big( \bm{x}_{[k+1, N]} \big).
		\end{flalign*}
		
		\noindent Here, we have defined the variable sets $\bm{x}_{[1, k]} = (x_1, x_2, \ldots , x_k)$ and $\bm{x}_{[k+1,N]} = (x_{k+1}, x_{k+2}, \ldots , x_N)$. For any interval $I = \big\llbracket i_0 + 1, i_0 + |I| \big\rrbracket \subset \llbracket 1, N \rrbracket$, we have also defined the function $\sigma |_I : \big\llbracket 1, |I| \big\rrbracket \rightarrow \llbracket 1, n \rrbracket$ by setting $\sigma |_I (i) = \sigma (i + i_0)$ for each $i \in \big\llbracket 1, |I| \big\rrbracket$. 
		
	\end{lem}

	The third is a \emph{Cauchy identity}. Its proof is similar to \cite[Proposition 4.5.1]{CSVMST}, following as a consequence of the Yang--Baxter equation \eqref{llrijk2}, though we include it here (since some results below, such as \Cref{fgsv}, will amount to mild modifications of it).

	\begin{lem}
		
		\label{fg2} 
		
		Fix integers $n, M, N \ge 1$; sequences of complex numbers $\bm{x} = (x_1, x_2, \ldots , x_N)$ and $\bm{y} = (y_1, y_2, \ldots , y_M)$; and a function $\sigma : \llbracket 1, N \rrbracket \rightarrow \llbracket 1, n \rrbracket$. If
		\begin{flalign}
			\label{sxsy} 
			\displaystyle\max_{\substack{1 \le i \le M \\ 1 \le j \le N}} \bigg| \displaystyle\frac{1 - sx_j}{x_j - s} \bigg| \cdot \bigg| \displaystyle\frac{y_i - s}{1 - sy_i} \bigg| < 1,
		\end{flalign}
	
		\noindent then  
		\begin{flalign*}
			\displaystyle\sum_{\mu \in \Comp_n (N)} f_{\mu; s}^{\sigma} (\bm{x}) G_{\mu; s} (\bm{y}) = \displaystyle\prod_{i=1}^M \displaystyle\prod_{j=1}^N \displaystyle\frac{x_j - q y_i}{x_j - y_i}.
		\end{flalign*}
	
	\end{lem}
	
	\begin{proof} 
		
	For each integer $i \in \llbracket 1, n \rrbracket$, let $\ell_i = \# \big\{ \sigma^{-1} (i) \big\}$ denote the number of preimages of $i$ under $\sigma$, and set $\bm{e}_{\bm{\ell}} = (\ell_1, \ell_2, \ldots , \ell_n)$. We begin by considering the partition function $\mathcal{Z}$ for the vertex model shown in \Cref{z1xy}.  
		
		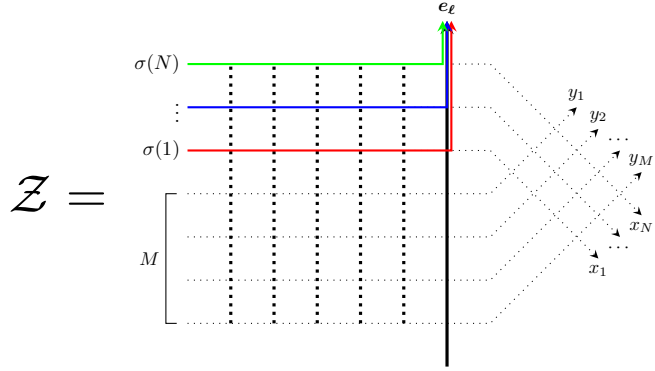
\begin{figure}
			\begin{center}
				\begin{tikzpicture}[
					>=stealth, 
					scale = .575]{
						\draw[] circle [radius = 0] (-3, 3) node[scale = 2]{$\mathcal{Z} =$};
						
						\draw[->, dotted] (0, 0) -- (7, 0) -- (10.5, 3.5) node[above, scale = .7]{$y_M$};
						\draw[->, dotted] (0, 1) -- (7, 1) -- (10, 4) node[above, scale = .7]{$\cdots$};
						\draw[->, dotted] (0, 2) -- (7, 2) -- (9.5, 4.5) node[above, scale = .7]{$y_2$};
						\draw[->, dotted] (0, 3) -- (7, 3) -- (9, 5) node[above, scale = .7]{$y_1$};
						
						\draw[dotted, very thick] (5, 0) -- (5, 6);
						\draw[dotted, very thick] (4, 0) -- (4, 6);
						\draw[dotted, very thick] (3, 0) -- (3, 6);
						\draw[dotted, very thick] (2, 0) -- (2, 6);
						\draw[dotted, very thick] (1, 0) -- (1, 6);
						
						\draw[->, very thick] (6, -1) -- (6, 7);
						
						\draw[->, dotted] (5.1, 4) -- (7, 4) -- (9.5, 1.5) node[below, scale = .7]{$x_1$};
						\draw[->, dotted] (5.1, 5) -- (7, 5) -- (10, 2) node[below, scale = .7]{$\cdots$};
						\draw[->, dotted] (5.1, 6) -- (7, 6) -- (10.5, 2.5) node[below, scale = .7]{$x_N$};
						
						\draw[red, thick, ->] (0, 4) node[left, black, scale = .75]{$\sigma(1)$} -- (6.1, 4) -- (6.1, 7);
						\draw[blue, thick, ->] (0, 5) node[left, black, scale = .75]{$\vdots$} -- (6, 5) -- (6, 7) node[above, black, scale = .7]{$\bm{e}_{\bm{\ell}}$}; 
						\draw[green, thick, ->] (0, 6) node[left, black, scale = .75]{$\sigma(N)$} -- (5.9, 6) -- (5.9, 7);

						\draw[-] (-.25, 0) -- (-.5, 0) -- (-.5, 3) -- (-.25, 3);
						\draw[] (-.5, 1.5) circle [radius = 0] node[left, scale = .7]{$M$};
					}
						\end{tikzpicture}
					\end{center} 
				\caption{\label{z1xy} Shown above is a vertex model used in the proof of \Cref{fg2}.}
				
				\end{figure}
				
				This model consists of three regions that we denote by $\mathcal{R}_1$, $\mathcal{R}_2$, and $\mathcal{R}_3$. The first region $\mathcal{R}_1 = \mathbb{Z}_{\le 0} \times \llbracket 1, M \rrbracket$ constitutes the bottom $M$ rows (weakly) to the left of the $y$-axis. The second region $\mathcal{R}_2 = \mathbb{Z}_{\le 0} \times \llbracket M+1, M+N \rrbracket$ constitutes the remaining $N$ rows (weakly) to the left of the $y$-axis. The third region $\mathcal{R}_3$ is the $M\times N$ ``cross'' to the right of the $y$-axis. Different vertex weights (recall \Cref{rzabcd} and \Cref{lzdefinition}) are used in these regions. In $\mathcal{R}_1$, for each $i \in \llbracket 1, M \rrbracket$, we use the weight $L_{y_i; s}$ at $(-k, i)$ if $k \ge 1$ and $L_{y_i; 0}$ at $(0, i)$; in $\mathcal{R}_2$, for each $j \in \llbracket 1, N \rrbracket$, we use the weight $\widehat{L}_{x_j; s}$ at $(-k, M+j)$ if $k \ge 1$ and $L_{x_j; 0}$ at $(0, M+j)$; and in $\mathcal{R}_3$ we use the weight $R_{y_i/x_j}$ at the intersection of $i$-th column (from the left) and $j$-th row (from the bottom) of the cross.

				The boundary conditions for the model in \Cref{z1xy} are prescribed as follows. The entrance data is defined by having no arrows vertically enter any column of the model; having no arrow horizontally enter through the bottom $M$ rows of the model; and having an arrow of color $\sigma(j)$ enter through the $(M+j)$-th row of the model, for each $j \in \llbracket 1, N \rrbracket$. The exit data is defined by having $\ell_i$ arrows of color $i$ exit the $y$-axis, for each index $i \in \llbracket 1, n \rrbracket$; having no arrows exit through any other column to the left of the $y$-axis; and having no arrows exit the cross to the right of the $y$-axis.
				
				Observe that this vertex model is \emph{frozen}, that is, there is at most one colored higher spin path ensemble with this boundary data with nonzero weight. It is the one in which, for each $j \in \llbracket 1, N \rrbracket$, the path of color $\sigma (j)$ in the $(M+j)$-th row travels horizontally until it reaches the $y$-axis, and then proceeds vertically until it exits $y$-axis. Recalling from \eqref{rzij}, \eqref{lzi}, and \eqref{lzij2} that 
				\begin{flalign*}
					& L_{y_i; s} (\bm{e}_0, 0; \bm{e}_0, 0) = 1; \qquad \qquad \qquad \qquad \qquad \qquad \widehat{L}_{x_j; s} \big( \bm{e}_0, \sigma(j); \bm{e}_0, \sigma(j) \big) = 1; \\ 
					& L_{x_j; 0} \bigg( \displaystyle\sum_{k=1}^{j-1} \bm{e}_k, \sigma(j); \displaystyle\sum_{k=1}^j \bm{e}_k, 0 \bigg) = 1; \qquad \qquad \qquad R_{y_i/x_j} (0, 0; 0, 0) = 1,
				\end{flalign*} 
			
				\noindent it follows that the partition function for the vertex model from \Cref{z1xy} is given by
				\begin{flalign}
					\label{z1}
					\mathcal{Z} = 1.
				\end{flalign} 
				
				Next, by $MN$ sequences of applications of the Yang--Baxter equation \eqref{llrijk2}, the partition function $\mathcal{Z}$ of this vertex model is unchanged if the cross originally in region $\mathcal{R}_3$ is moved to the left of $\mathcal{R}_1 \cup \mathcal{R}_2 = \mathbb{Z}_{\le 0} \times \llbracket 1, M+N \rrbracket$. In particular, $\mathcal{Z}$ is also the partition function of the vertex model in \Cref{z2xy}.
			
				\begin{figure}
						\begin{center}
							\begin{tikzpicture}[
								>=stealth, 
								scale = .575]{
						
						\draw[] circle [radius = 0] (.5, 3.5) node[scale=2]{$\mathcal{Z} = $};
						\draw[dotted] (4.5, 1.5) -- (6.5, 3.5) -- (7, 3.5);
						\draw[dotted] (12, 3.5) -- (13.5, 3.5) node[right, scale = .7]{$y_M$};
						\draw[dotted] (4, 2) -- (6.5, 4.5) -- (7, 4.5);
						\draw[dotted] (12, 4.5) -- (13.5, 4.5) node[right, scale = .7]{$\vdots$};
						\draw[dotted] (3.5, 2.5) -- (6.5, 5.5) -- (7, 5.5);
						\draw[dotted] (12, 5.5) -- (13.5, 5.5) node[right, scale = .7]{$y_2$};
						\draw[dotted] (3, 3) -- (6.5, 6.5) -- (7, 6.5);
						\draw[dotted] (12, 6.5) -- (13.5, 6.5) node[right, scale = .7]{$y_1$};

						\draw[->, red, thick] (3, 4) node[left, black, scale = .6]{$\sigma(1)$} -- (6, .5) -- (7, .5);
						\draw[->, blue, thick] (3.5, 4.5) node[left, black, scale = .6]{$\cdots$} -- (6, 1.5) -- (7, 1.5);
						\draw[->, green, thick] (4, 5) node[left, black, scale = .6]{$\sigma(N)$} -- (6, 2.5) -- (7, 2.5);
						
						\draw[dotted] (12, .5) -- (13.5, .5) node[right, scale = .7]{$x_1$};
						\draw[dotted] (12, 1.5) -- (13.5, 1.5) node[right, scale = .7]{$\vdots$};
						\draw[dotted] (12, 2.5) -- (13.5, 2.5) node[right, scale = .7]{$x_N$};
						
						\draw[->, very thick] (13, 0) -- (13, 7.5);
						
						\draw[-] (4.25, 1.25) -- (4, 1) -- (2.5, 2.5) -- (2.75, 2.75);
						\draw[] (3.25, 1.75) circle [radius = 0] node[below = 3, left, scale = .7]{$M$};
						
						\draw[red, thick, ->] (13.1, 6.5) -- (13.1, 7.5);
						\draw[blue, thick, ->] (13, 6.5) -- (13, 7.5) node[above, scale = .7, black]{$\bm{e}_{\bm{\ell}}$};
						\draw[green, thick, ->]	(12.9, 6.5) -- (12.9, 7.5);
						
						\draw[gray, ultra thick, dashed] (7, .5) -- (7, 6.5);
						\draw[gray, ultra thick, dashed] (8, .5) -- (8, 6.5);
						\draw[gray, ultra thick, dashed] (9, .5) -- (9, 6.5);
						\draw[gray, ultra thick, dashed] (10, .5) -- (10, 6.5);
						\draw[gray, ultra thick, dashed] (11, .5) -- (11, 6.5);
						\draw[gray, ultra thick, dashed] (12, .5) -- (12, 6.5);
						
						\draw[gray, thick, dashed] (7, .5) -- (13, .5);
						\draw[gray, thick, dashed] (7, 1.5) -- (13, 1.5);
						\draw[gray, thick, dashed] (7, 2.5) -- (13, 2.5);
						\draw[gray, thick, dashed] (7, 3.5) -- (13, 3.5);
						\draw[gray, thick, dashed] (7, 4.5) -- (13, 4.5);
						\draw[gray, thick, dashed] (7, 5.5) -- (13, 5.5);
						\draw[gray, thick, dashed] (7, 6.5) -- (13, 6.5);
						
					}
				\end{tikzpicture}
			\end{center}	
		
			\caption{\label{z2xy} Shown above is the vertex model from \Cref{z1xy} after using the Yang--Baxter equation to move the cross to the left of $\mathbb{Z}_{\le 0} \times \llbracket 1, M+N \rrbracket$.} 
			\end{figure} 
			
			This model also consists of three regions $\mathcal{R}_1'$, $\mathcal{R}_2'$, and $\mathcal{R}_3'$. The third $\mathcal{R}_3'$ is an $M \times N$ cross, that is now to the left of $\mathbb{Z}_{\le 0} \times \llbracket 1, M+N \rrbracket$. The second $\mathcal{R}_2' = \mathbb{Z}_{\le 0} \times \llbracket 1, N \rrbracket$ consists of the bottom $N$ rows to the right of the  cross. The first $\mathcal{R}_1'$ consists of the remaining $M$ rows to the right of the cross. Again different vertex weights are used in these regions. In $\mathcal{R}_1'$, for each $i \in \llbracket 1, M \rrbracket$, we use the weight $L_{y_i; s}$ at $(-k, i+N)$ for $k \ge 1$ and $L_{y_i; 0}$ at $(0, i+N)$. In $\mathcal{R}_2'$, for $j \in \llbracket 1, N \rrbracket$, we use the weight $\widehat{L}_{x_j; s}$ at $(-k, j)$ if $k \ge 1$ and $L_{x_j; 0}$ at $(0, j)$. In $\mathcal{R}_3'$, we use the weight $R_{y_i / x_j}$ at the intersection of the $i$-th column and $j$-th row of the cross.		
		
			The boundary data for the model in \Cref{z2xy} is prescribed as follows (it must match that of \Cref{z1xy}). The entrance data is defined by having no arrows vertically enter any column in the model, either in or to the right of the cross, and having an arrow of color $\sigma(j)$ enter through the $j$-th row (from the bottom) of the cross, for each $j \in \llbracket 1, N \rrbracket$. The exit data is defined by having no arrows horizontally exit through any row of the model; having $\ell_i$ arrows of color $i$ exit through the $y$-axis, for each $i \in \llbracket 1, n \rrbracket$; and having no arrows exit through any other column to the left of the $y$-axis.
			
			Let us now analyze this vertex model. Using the fact from \eqref{sxsy} (and \eqref{lzi} and \eqref{lzij2}) that, for any $k \in \llbracket 1, n \rrbracket$,
			\begin{flalign*}
				\displaystyle\max_{\substack{1 \le i \le M \\ 1 \le j \le N}} \bigg| \displaystyle\frac{\widehat{L}_{x_j; s} \big( \bm{e}_0, 0; \bm{e}_0, 0 \big) \cdot L_{y_i; s} (\bm{e}_0, k; \bm{e}_0, k)}{\widehat{L}_{x_j; s} (\bm{e}_0, k; \bm{e}_0, k) \cdot L_{y_i; s} (\bm{e}_0, 0; \bm{e}_0, 0)} \bigg| = \displaystyle\max_{\substack{1 \le i \le M \\ 1 \le j \le N}}\bigg| \displaystyle\frac{1 - sx_j}{x_j - s} \cdot \displaystyle\frac{y_i - s}{1 - sy_i} \bigg| < 1,
			\end{flalign*} 
		
			\noindent it is quickly verified (see the proof of \cite[Theorem 3.2.3]{CSVMST}) that a colored higher spin path ensemble on $\mathbb{Z}_{\le 0} \times \llbracket 1, M+N \rrbracket$ has nonzero weight only if all but finitely many vertices in $\mathcal{R}_1'$ have arrow configurations of the form $(\bm{e}_0, 0; \bm{e}_0, 0)$, and all but finitely many vertices in $\mathcal{R}_2'$ have arrow configurations of the form $(\bm{e}_0, k; \bm{e}_0, k)$ for some $k \in \llbracket 1, n \rrbracket$ (which may depend on the vertex). This means that an arrow must horizontally enter $\mathcal{R}_2'$ through $(-\infty, j)$ for each $j \in \llbracket 1, N \rrbracket$, and no arrow can horizontally enter $\mathcal{R}_1'$. Since each edge of the cross $\mathcal{R}_3'$ can accommodate at most one arrow, it follows that this cross is frozen; the vertex in its $i$-th column and $j$-th row must have arrow configuration $\big( 0, \sigma (j); 0, \sigma(j) \big)$. The weight of $\mathcal{R}_3'$ is therefore
			\begin{flalign}
				\label{productr3} 
				\displaystyle\prod_{i=1}^M \displaystyle\prod_{j=1}^N R_{y_i / x_j} \big( 0, \sigma(j); 0, \sigma(j) \big) = \displaystyle\prod_{i=1}^M \displaystyle\prod_{j=1}^N \displaystyle\frac{x_j - y_i}{x_j - qy_i},
			\end{flalign}
		
			\noindent where in the last equality we used \eqref{rzij}.
			
			The colored higher spin path ensembles in $\mathcal{R}_1'$ and $\mathcal{R}_2'$ can be arbitrary elements of $\mathfrak{P}_G (\mu/0^N)$ and $\mathfrak{P}_f (\mu / \emptyset; \sigma)$ for any $\mu \in \Comp_n (N)$ that is shared between $\mathcal{R}_1'$ and $\mathcal{R}_2'$ (this $n$-composition $\mu$ prescribes the $x$-coordinates where paths in the ensemble vertically exit $\mathcal{R}_2'$ and enter $\mathcal{R}_1'$). Hence, the weight of $\mathcal{R}_1' \cup \mathcal{R}_2'$ is 
			\begin{flalign*}
				\displaystyle\sum_{\mu \in \Comp_n (N)} f_{\mu;s}^{\sigma} (\bm{x}) G_{\mu;s} (\overleftarrow{\bm{y}}),
			\end{flalign*}
			
			\noindent Together with the weight \eqref{productr3} of $\mathcal{R}_3'$ (and the symmetry of $G$ from \Cref{gsymmetric}), it follows that the weight of the vertex model in \Cref{z2xy} is 
			\begin{flalign}
				\label{zsum2} 
				\mathcal{Z} = \displaystyle\prod_{i=1}^M \displaystyle\prod_{j=1}^N \displaystyle\frac{x_j - y_i}{x_j - qy_i} \cdot \displaystyle\sum_{\mu \in \Comp_n (N)} f_{\mu; s}^{\sigma} (\bm{x}) G_{\mu; s} (\bm{y}). 
			\end{flalign}
		
			\noindent The lemma then follows from \eqref{z1} and \eqref{zsum2}. 
	\end{proof}

	\section{Probability Measures and Matchings} 
	
	\label{Line} 
	
	In this section we use the functions $f$ and $G$ from \Cref{fg} to produce probability measures on sequences of compositions, and explain how such measures are related to the stochastic six-vertex model. The former is done in \Cref{AscendingfG}, and the latter is done in \Cref{EqualityCD} and \Cref{ProofE}. Throughout this section, we fix integers $n, M, N \ge 1$; a composition $\bm{\ell} = (\ell_1, \ell_2, \ldots , \ell_n)$ of $N$; a function $\sigma : \llbracket 1, N \rrbracket \rightarrow \llbracket 1, n \rrbracket$, such that for each $i \in \llbracket 1, n \rrbracket$ we have $\ell_i = \# \big\{ \sigma^{-1} (i) \big\}$; a complex number $s \in \mathbb{C}$; and sequences of complex numbers $\bm{x} = (x_1, x_2, \ldots , x_N)$ and $\bm{y} = (y_1, y_2, \ldots , y_M)$, such that \eqref{sxsy} holds.

	\subsection{Ascending $fG$ Measures} 
	
	\label{AscendingfG} 
	
	In this section we introduce probability measures that arise from the branching and Cauchy identities (\Cref{ffgg} and \Cref{fg2}), which are similar to those appearing in \cite[Equation (10.3.1)]{CSVMST}. We begin with the following definition for certain families of compositions.

	\begin{definition} 
		
		\label{mnmu}

		 A sequence $\bm{\mu} = \big( \mu (0), \mu (1), \ldots , \mu (M+N) \big)$ of $n$-compositions is called \emph{$(M; \sigma)$-ascending} if the following hold, using the notation $\mu (i) = \big( \mu^{(1)} (i) \mid \cdots \mid \mu^{(n)} (i) \big)$ below.
	
		\begin{enumerate}
			\item We have $\mu(0) = (\emptyset \mid \cdots \mid \emptyset)$ and $\mu(M+N) = (0^{\ell_1} \mid \cdots \mid 0^{\ell_n})$. 
			\item 
			\begin{enumerate} 
				\item For all $j \in \llbracket 0, N \rrbracket$ and $c \in \llbracket 1, n \rrbracket$, we have $\ell \big( \mu^{(c)} (j) \big) = \sum_{k=1}^j \mathbbm{1}_{\sigma(k) = c}$. Thus, $\ell \big( \mu(j) \big) = j$. 
			\item For all $i \in \llbracket N, M+N \rrbracket$ and $c \in \llbracket 1, n \rrbracket$, we have $\ell \big( \mu^{(c)} (i) \big) = \ell_c$. Thus, $\ell \big( \mu(i) \big) = N$.
			\end{enumerate} 
			\item For any $i \in \llbracket 1, M +N \rrbracket$ and $k \in \mathbb{Z}_{\ge 0}$, there is at most one index $\mathfrak{q} = \mathfrak{q}_{\bm{\mu}} (k, i) \in \llbracket 1, n \rrbracket$ so that 
			\begin{flalign}
				\label{moq} 
				\mathfrak{m}_{\le k-1} \big( \mu^{(\mathfrak{q})} (i) \big) = \mathfrak{m}_{\le k-1} \big( \mu^{(\mathfrak{q})} (i-1) \big) + 1. 
			\end{flalign}
		
			\noindent We set $\mathfrak{q}_{\bm{\mu}} (k, i) = 0$ if no index in $\mathfrak{q} \in \llbracket 1, n \rrbracket$ satisfying \eqref{moq} exists. Moreover, for all $c \in \llbracket 1, n \rrbracket$ with $c \ne \mathfrak{q}_{\bm{\mu}} (k, i)$, we have $\mathfrak{m}_{\le k-1} \big( \mu^{(c)} (i) \big) = \mathfrak{m}_{\le k-1} \big( \mu^{(c)} (i-1) \big)$. 
		\end{enumerate}
		
		\noindent Let us also define the $(M+N)$-tuple $\bm{\mathfrak{q}} (\bm{\mu}) = \big( \mathfrak{q}_{\bm{\mu}} (1, 1), \mathfrak{q}_{\bm{\mu}} (1, 2), \ldots , \mathfrak{q}_{\bm{\mu}} (1, M+N) \big)$. 
	\end{definition} 
	
	\begin{rem} 
		
		\label{mnmurow2}
	
		Given an $(M; \sigma)$-ascending sequence of compositions $\bm{\mu}$ as in \Cref{mnmu}, we will often view the $n$-composition $\mu(i)$ as indexing the positions (in the sense of \Cref{mumuistate}) of the colored arrows exiting the row $\{ y = i \}$, in a vertex model on $\mathbb{Z}_{\le 0} \times \llbracket 1, M + N \rrbracket$ (of the form arising in the dashed part of \Cref{z2xy}; see also \Cref{00lmu}). This gives rise to a colored higher spin path ensemble on $\mathbb{Z}_{\le 0} \times \llbracket 1, M + N \rrbracket$, that we will denote by $\mathcal{E}_{\bm{\mu}}$. In this way, $\mathfrak{q}_{\bm{\mu}} (k, i)$ denotes the color of the arrow in $\mathcal{E}_{\bm{\mu}}$ along the edge connecting $(-k, i)$ to $(1-k, i)$. Therefore, the $(M+N)$-tuple $\bm{\mathfrak{q}} (\bm{\mu})$ records the colors of the arrows (from bottom to top) along the horizontal edges in $\mathcal{E}_{\bm{\mu}}$ joining the $(-1)$-st column to the $0$-th one. 
		
		The boundary data for this ensemble is described as follows. For each $j \in \llbracket 1, N \rrbracket$, it has an arrow of color $\sigma(j)$ horizontally entering the row $\{ y = j \}$, and it has no other arrows horizontally entering or exiting any other row of the model. For each $c \in \llbracket 1, n \rrbracket$, it has $\ell_c$ arrows of color $c$ vertically exiting the $y$-axis $\{ x = 0 \}$, and it has no other arrows horizontally entering or exiting any other column of the model. We denote by $\mathfrak{P}_{\fG} (M; \sigma)$ the set of colored higher spin path ensembles on $\mathbb{Z}_{\le 0} \times \llbracket 1, M + N \rrbracket$ with these boundary conditions, as any $\mathcal{E}_{\bm{\mu}} \in \mathfrak{P}_{\fG} (M; \sigma)$ can be thought of an ensemble from $\mathfrak{P}_G (\mu / \emptyset)$ that is juxtaposed above one from $\mathfrak{P}_f (\mu / \emptyset; \sigma)$ (recall \Cref{fg}), for some $n$-composition $\mu \in \Comp_n (N)$. It is quickly verified that the above procedure is a bijection between $\mathfrak{P}_{\fG} (M; \sigma)$ and $(M; \sigma)$-ascending sequences $\bm{\mu}$ of $n$-compositions. 
		
		See \Cref{00lmu} for a depiction, where there $(n, M, N) = (2, 3, 4)$ and 
		\begin{flalign*} 
			&  \sigma (1) = 1, \quad \sigma(2) = 2, \quad \sigma(3) = 2, \quad \sigma (4) = 1; \qquad \qquad \mu (1) = (2 \mid \emptyset), \quad \mu (2) = (0 \mid 3), \\
			&  \quad \mu (3) = (0 \mid 3, 1), \quad \mu (4) = (3, 0 \mid 3, 0), \quad \mu (5) = (2, 0 \mid 3, 0), \quad \mu (6) = (0, 0 \mid 2, 0).
		\end{flalign*} 
	
	\end{rem}

	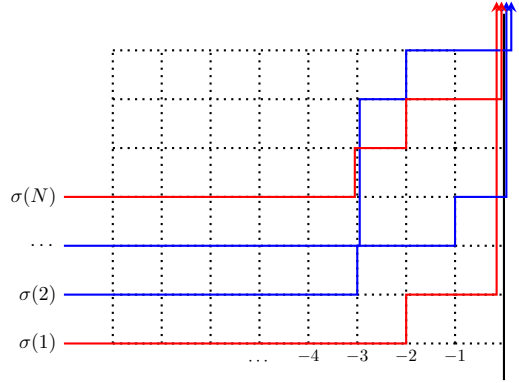
\begin{figure} 
		
		\begin{center}
			\begin{tikzpicture}[
				>=stealth, 
				scale = .65]{

					\draw[thick] (23, -.25) -- (23, 7.25);
					
					\draw[black, thick, dotted] (15, .5) -- (15, 6.5);
					\draw[black, thick, dotted] (16, .5) -- (16, 6.5);
					\draw[black, thick, dotted] (17, .5)  -- (17, 6.5);
					\draw[black, thick, dotted] (18, .5) node[below = 3, scale = .65]{$\cdots$} -- (18, 6.5);
					\draw[black, thick, dotted] (19, .5) node[below, scale = .65]{$-4$}  -- (19, 6.5);
					\draw[black, thick, dotted] (20, .5) node[below, scale = .65]{$-3$} -- (20, 6.5);
					\draw[black, thick, dotted] (21, .5) node[below, scale = .65]{$-2$} -- (21, 6.5);
					\draw[black, thick, dotted] (22, .5) node[below, scale = .65]{$-1$} -- (22, 6.5);
					
					\draw[black, thick, dotted] (15, .5) -- (22, .5);
					\draw[black, thick, dotted] (15, 1.5) -- (22, 1.5);
					\draw[black, thick, dotted] (15, 2.5) -- (22, 2.5);
					\draw[black, thick, dotted] (15, 3.5) -- (22, 3.5);
					\draw[black, thick, dotted] (15, 4.5) -- (22, 4.5);
					\draw[black, thick, dotted] (15, 5.5) -- (22, 5.5);
					\draw[black, thick, dotted] (15, 6.5) -- (22, 6.5);
					
					\draw[thick, dotted] (22, 2.5) -- (23, 2.5);
					\draw[thick, dotted] (22, 4.5) -- (23, 4.5);
					\draw[thick, dotted] (22, .5) -- (23, .5);
					\draw[thick, dotted] (22, 1.5) -- (23, 1.5);
					
					\draw[->, red, thick] (14, .5) node[left, black, scale = .75]{$\sigma(1)$} -- (15, .5) -- (21, .5) -- (21, 1.5) -- (22.85, 1.5) -- (22.85, 7.5);
					\draw[->, blue, thick] (14, 1.5) node[left, black, scale = .75]{$\sigma(2)$} -- (15, 1.5) -- (20, 1.5) -- (20, 2.5) -- (22, 2.5) -- (22, 3.5) -- (23.05, 3.5) -- (23.05, 7.5);
					\draw[->, blue, thick] (14, 2.5) node[left, black, scale = .75]{$\cdots$} -- (15, 2.5) -- (20.05, 2.5) -- (20.05, 5.5) -- (21, 5.5) -- (21, 6.5) -- (22, 6.5) -- (23.15, 6.5) -- (23.15, 7.5);
					\draw[->, red, thick] (14, 3.5) node[left, black, scale = .75]{$\sigma(N)$} -- (15, 3.5) -- (19.95, 3.5) -- (19.95, 4.5) -- (21, 4.5) -- (21, 5.5) -- (22.95, 5.5) -- (22.95, 7.5);
				}
			\end{tikzpicture}
		\end{center}
		
		\caption{\label{00lmu} Depicted above is the colored higher spin path ensemble associated with the sequence $\bm{\mu}$ in the example at the end of \Cref{mnmurow2}. Here, red and blue are colors $1$ and $2$, respectively.} 
		
	\end{figure}

	Next we define the following probability measure on sequences of ascending compositions.

	\begin{definition}
		
		\label{measurefg}

		Define the probability measure $\mathbb{P}_{\fG}^{\sigma} = \mathbb{P}_{\fG; n; s; \bm{x}; \bm{y}}^{\sigma}$ on $(M; \sigma)$-ascending sequences of $n$-compositions, by setting
		\begin{flalign}
			\label{fgprobabilitymu} 
			\mathbb{P}_{\fG}^{\sigma} [\bm{\mu}] = \mathcal{Z}_{\bm{x}; \bm{y}}^{-1} \cdot \displaystyle\prod_{j=1}^N f_{\mu (j) / \mu (j-1); s}^{\sigma (j)} (x_j) \displaystyle\prod_{i=N+1}^{M+N} G_{\mu (i-1) / \mu (i); s} (y_{i-N}),
		\end{flalign}
		
		\noindent for each $(M; \sigma)$-ascending sequence $\bm{\mu} = \big( \mu (0), \mu (1), \ldots , \mu (M+N) \big)$, where 
		\begin{flalign}
			\label{zxys}
			\mathcal{Z}_{\bm{x}; \bm{y}} =  \displaystyle\prod_{i=1}^M \displaystyle\prod_{j=1}^N \displaystyle\frac{x_j - q y_i}{x_j -  y_i}.
		\end{flalign}
	
		\noindent Here, we implicitly assume that $s$, $\bm{x}$, and $\bm{y}$ are such that the left side of \eqref{fgprobabilitymu} is nonnegative. The fact that these probabilities sum to one follows from the following lemma. 
		
	\end{definition}

	\begin{lem}
		
		\label{sumfg1}
		
		Under the notation and assumptions of \Cref{measurefg}, we have 
		\begin{flalign*}
			\displaystyle\sum_{\bm{\mu}} \displaystyle\prod_{j=1}^N f_{\mu(j) / \mu(j-1); s}^{\sigma (j)} (x_j) \cdot \displaystyle\prod_{i=N+1}^{M+N} G_{\mu(i-1) / \mu(i); s} (y_{i-N}) = \mathcal{Z}_{\bm{x}; \bm{y}},
		\end{flalign*} 
		
		\noindent where the sum on the left side is over all $(M; \sigma)$-ascending sequences of $n$-compositions $\bm{\mu} = \big( \mu  (0), \mu(1), \ldots , \mu (M+N) \big)$.
	\end{lem}
	
	\begin{proof}
		
		By the branching identity \Cref{ffgg}, we have for any $\mu^{(N)} \in \Comp_n (N)$ that
		\begin{flalign*}
			\displaystyle\sum_{\bm{\mu}^{[0,N-1]}} \displaystyle\prod_{j=1}^N f_{\mu (j) / \mu (j-1); s}^{\sigma (j)} (x_j) = f_{\mu (N); s}^{\sigma} (\bm{x}); \quad \displaystyle\sum_{\bm{\mu}^{[N+1,M+N]}} \displaystyle\prod_{i=N+1}^{M+N} G_{\mu (i-1) / \mu (i); s} (y_{i-N}) = G_{\mu(N); s} (\bm{y}).
		\end{flalign*}
		
		\noindent Here, the sums are over all sequences of $n$-compositions $\bm{\mu}^{[0,N-1]} = \big( \mu(0), \mu(1), \ldots , \mu(N-1) \big)$ and $\bm{\mu}^{[N+1,M+N]} = \big( \mu(N+1), \mu(N+2), \ldots , \mu (M+N) \big)$ satisfying the  constraints of \Cref{mnmu}. Hence,  
		\begin{flalign*}
			\displaystyle\sum_{\bm{\mu}} \displaystyle\prod_{j=1}^N f_{\mu(j) / \mu(j-1); s}^{\sigma (j)} (x_j) \cdot \displaystyle\prod_{i=N+1}^{M+N} G_{\mu(i-1) / \mu(i); s} (y_{i-N}) = \displaystyle\sum_{\mu(N) \in \Comp_n (N)} f_{\mu(N); s}^{\sigma} (\bm{x}) G_{\mu(N); s} (\bm{y}).
		\end{flalign*}
		
		\noindent This, together with the Cauchy identity \Cref{fg2}, yields the lemma.
	\end{proof}

	\subsection{Matching Between Colored Stochastic Six-Vertex Models and $\mathbb{P}_{\fG}^{\sigma}$} 
	
	\label{EqualityCD}
	
	In this section we establish a matching between the law of the $(M+N)$-tuple $\bm{\mathfrak{q}} (\bm{\mu})$ (recall \Cref{mnmu}) associated with a sequence of compositions sampled from $\mathbb{P}_{\fG}^{\sigma}$ (from \Cref{measurefg}), with a certain random variable associated with the stochastic six-vertex model (from \Cref{ModelVertex}). We begin by defining the latter. 
	
	\begin{definition} 
		
	\label{dece}
	
	Let $\mathcal{E}$ denote a six-vertex ensemble on the rectangular domain $\mathcal{D}_{M;N} = \llbracket 1, M \rrbracket \times \llbracket 1, N \rrbracket$. For each integer $i \in \llbracket 1, M \rrbracket$, let $c_i = c_i (\mathcal{E}) \in \llbracket 0, n \rrbracket$ denote the color of the path in $\mathcal{E}$ vertically exiting $\mathcal{D}_{M;N}$ through $(i, N)$; for each integer $j \in \llbracket 1, N \rrbracket$, let $d_j = d_j (\mathcal{E}) \in \llbracket 0, n \rrbracket$ denote the color of the path in $\mathcal{E}$ horizontally exiting $\mathcal{D}_{M;N}$ through $(M, j)$. Then set $\mathfrak{C} (\mathcal{E}) = (c_1, c_2, \ldots,  c_M) \in \llbracket 1, n \rrbracket^M$ and $\mathfrak{D} (\mathcal{E}) = (d_1, d_2, \ldots , d_N) \in \llbracket 1, n \rrbracket^N$. 
	
	\end{definition} 

	We next require notation for the colored stochastic six-vertex model (defined at the end of \Cref{ModelVertex}) with $\sigma$-entrance data introduced in \Cref{ModelVertex}.
		
	\begin{definition} 
		
	\label{probabilityvertex}
	
	Let $\mathbb{P}_{\SV}^{\sigma} = \mathbb{P}_{\SV; \bm{x}; \bm{y}}^{\sigma}$ denote the measure on colored six-vertex ensembles on $\mathcal{D}_{M;N}$ obtained by running the colored stochastic six-vertex model on $\mathcal{D}_{M;N}$ under $\sigma$-entrance data, with weight $R_{y_i / x_j}$ (recall \Cref{rzabcd}) at any vertex $(i, j) \in \mathcal{D}_{M;N}$.
	
	\end{definition} 
	
	We refer to \Cref{zxy2} for depictions of \Cref{dece} and \Cref{probabilityvertex}. The following proposition now provides a matching between the $(M+N)$-tuple $\bm{\mathfrak{q}}(\bm{\mu})$ sampled under $\mathbb{P}_{\fG}^{\sigma}$ of \Cref{measurefg} and the $(M+N)$-tuple $\mathfrak{D}(\mathcal{E}) \cup \overleftarrow{\mathfrak{C}}(\mathcal{E})$ sampled under $\mathbb{P}_{\SV}^{\sigma}$ of \Cref{probabilityvertex}.\footnote{Observe that the left side of \eqref{ekmuk} is independent of $s$, while the right side seems to involve $s$; such a phenomenon had already been observed in the uncolored case in \cite[Proposition 7.14]{RFSVM}, and in the different setting of colored stationary measures \cite[Remark 4.6]{CIPSR}. In our context, this fact will later enable us to freely choose $s$ as we see fit, which will be useful in producing the simplest looking line ensembles.} It is a colored generalization of \cite[Theorem 5.5]{SSVMP}, though its proof is similar. We establish it in \Cref{ProofE} below.

	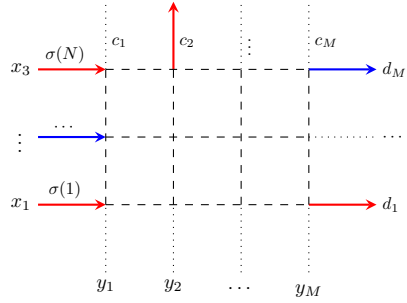
\begin{figure}
		\begin{center}			
			\begin{tikzpicture}[
				>=stealth,
				auto,
				style={
					scale = .9
				}
				]
				\draw[->, thick, red] (-1, 0) -- (0, 0);
				\draw[->, thick, blue] (-1, 1) -- (0, 1);
				\draw[->, thick, red] (-1, 2) -- (0, 2);
				
				\draw[->, thick, red] (3, 0) -- (4, 0);
				\draw[dotted ] (3, 1) -- (4, 1);
				\draw[->, thick, blue] (3, 2) -- (4, 2);
				
				\draw[dotted] (0, -1) -- (0, 0);
				\draw[dotted] (1, -1) -- (1, 0);
				\draw[dotted] (2, -1) -- (2, 0);
				\draw[dotted] (3, -1) -- (3, 0);
				
				\draw[dotted] (0, 2) -- (0, 3);
				\draw[->, thick, red] (1, 2) -- (1, 3);
				\draw[dotted] (2, 2) -- (2, 3);
				\draw[dotted] (3, 2) -- (3, 3);
				
				\draw[-, dashed] (0, 0) -- (3, 0);
				\draw[-, dashed] (0, 1) -- (3, 1);
				\draw[-, dashed] (0, 2) -- (3, 2);
				
				\draw[-, dashed] (0, 0) -- (0, 2);
				\draw[-, dashed] (1, 0) -- (1, 2);
				\draw[-, dashed] (2, 0) -- (2, 2);
				\draw[-, dashed] (3, 0) -- (3, 2);
				
				\draw[] (0, -1) circle[radius = 0] node[below, scale = .75]{$y_1$};
				\draw[] (1, -1) circle[radius = 0] node[below, scale = .75]{$y_2$};
				\draw[] (2, -1) circle[radius = 0] node[below = 2, scale = .75]{$\cdots$};
				\draw[] (3, -1) circle[radius = 0] node[below = 2, scale = .75]{$y_M$};
				
				\draw[] (-1, 0) circle[radius = 0] node[left, scale = .75]{$x_1$};
				\draw[] (-1.125, 1) circle[radius = 0] node[left, scale = .75]{$\vdots$};
				\draw[] (-1, 2) circle[radius = 0] node[left, scale = .75]{$x_3$};
				
				\draw[] (-.6, 0) circle[radius = 0] node[above, scale = .65]{$\sigma(1)$};
				\draw[] (-.6, 1) circle[radius = 0] node[above, scale = .65]{$\cdots$};
				\draw[] (-.6, 2) circle[radius = 0] node[above, scale = .65]{$\sigma(N)$};
				
				\draw[] (4, 0) circle[radius = 0] node[right, scale = .65]{$d_1$};
				\draw[] (4, 1) circle[radius = 0] node[right, scale = .65]{$\cdots$};
				\draw[] (4, 2) circle[radius = 0] node[right, scale = .65]{$d_M$};
				
				\draw[] (0, 2.4) circle[radius = 0] node[right, scale = .65]{$c_1$};
				\draw[] (1, 2.4) circle[radius = 0] node[right, scale = .65]{$c_2$};
				\draw[] (2, 2.4) circle[radius = 0] node[right, scale = .65]{$\vdots$};
				\draw[] (3, 2.4) circle[radius = 0] node[right, scale = .65]{$c_M$};

			\end{tikzpicture}		
		\end{center}	
		\caption{\label{zxy2} Shown above is the entrance and exit data for a colored six-vertex ensemble on $\mathcal{D}_{4;3} = \llbracket 1, 4 \rrbracket \times \llbracket 1, 3 \rrbracket$, under $\sigma$-entrance data for $\big( \sigma(1), \sigma(2), \sigma(3) \big)= (1, 2, 1)$ (with red being color $1$ and blue being color $2$).}	
	\end{figure}

	\begin{prop} 
		
		\label{fgsv}

		 Fix an index sequence $\bm{\mathfrak{q}} = (\mathfrak{q}_1, \mathfrak{q}_2, \ldots , \mathfrak{q}_{M+N}) \in \llbracket 0, n \rrbracket^{M+N}$; and define the $M$-tuple $\mathfrak{C} = (\mathfrak{q}_{M+N}, \mathfrak{q}_{M+N-1}, \ldots,  \mathfrak{q}_{N+1})$ and $N$-tuple $\mathfrak{D} = (\mathfrak{q}_1, \mathfrak{q}_2, \ldots , \mathfrak{q}_N)$. Then,  
		\begin{flalign}
			\label{ekmuk} 
			\mathbb{P}_{\SV}^{\sigma} \Big[ \{ \mathfrak{C} (\mathcal{E}) = \mathfrak{C} \big\} \cap \big\{ \mathfrak{D} (\mathcal{E}) = \mathfrak{D} \big\} \Big] = \mathbb{P}_{\fG}^{\sigma} \big[ \bm{\mathfrak{q}} (\bm{\mu}) = \bm{\mathfrak{q}} \big]. 
		\end{flalign}	
	
		\noindent Here, on the left side of \eqref{ekmuk}, the colored six-vertex ensemble $\mathcal{E}$ is sampled under the colored stochastic six-vertex measure $\mathbb{P}_{\SV; \bm{x}; \bm{y}}^{\sigma}$. On the right side of \eqref{ekmuk}, the $(M; \sigma)$-ascending sequence $\bm{\mu}$ of colored compositions is sampled under the measure $\mathbb{P}_{\fG; n; s; \bm{x}; \bm{y}}^{\sigma}$.
		
	\end{prop}

	Before establishing \Cref{fgsv}, we deduce the following corollary. It equates the joint law of the height functions (recall \Cref{FunctionsZ}) evaluated along the exit sites of an $M \times N$ rectangle, sampled under the colored stochastic six-vertex model, with the joint law of the number of zero entries in a family $\bm{\mu}$ of $n$-compositions, sampled under the $\mathbb{P}_{\fG}^{\sigma}$ measure.
	
	\begin{cor} 
		
		\label{hfg2}
		
		The joint law of all the height functions 
		\begin{flalign}
			\label{hcij}
			\bigcup_{c = 1}^n \big( \mathfrak{h}_{\ge c}^{\rightarrow} (M, 1), \mathfrak{h}_{\ge c}^{\rightarrow} (M, 2), \ldots , \mathfrak{h}_{\ge c}^{\rightarrow} (M, N), \mathfrak{h}_{\ge c}^{\rightarrow} (M-1, N), \ldots , \mathfrak{h}_{\ge c}^{\rightarrow} (0, N) \big),
		\end{flalign} 
	
		\noindent is equal to the joint law of all zero-entry counts 
		\begin{flalign}
			\label{mcj}
			\bigcup_{c = 1}^n \Big( \mathfrak{m}_0^{\ge c} \big(\mu(1) \big), \mathfrak{m}_0^{\ge c} \big( \mu (2) \big), \ldots , \mathfrak{m}_0^{\ge c} \big( \mu (N) \big), \mathfrak{m}_0^{\ge c} \big( \mu(N+1) \big), \ldots , \mathfrak{m}_0^{\ge c} \big( \mu(M+N) \big) \Big).
		\end{flalign} 
	
		\noindent Here, the height functions in \eqref{hcij} are associated with a colored six-vertex ensemble sampled under $\mathbb{P}_{\SV; \bm{x}; \bm{y}}^{\sigma}$, and the zero-entry counts in \eqref{mcj} are associated with a $(M; \sigma)$-ascending sequence of $n$-compositions $\bm{\mu} = \big( \mu (0), \mu(1), \ldots , \mu(M+N) \big)$ sampled under $\mathbb{P}_{\fG; n; s; \bm{x}; \bm{y}}^{\sigma}$.
		 
	\end{cor} 
		
	\begin{proof}
		
		Adopt the notation of \Cref{fgsv}, and denote the $M$-tuple $\mathfrak{C} (\mathcal{E}) = (\mathfrak{c}_1, \mathfrak{c}_2, \ldots , \mathfrak{c}_M)$ and $N$-tuple $\mathfrak{D} (\mathcal{E}) = (\mathfrak{d}_1, \mathfrak{d}_2, \ldots , \mathfrak{d}_N)$. Then for each $c \in \llbracket 1, n \rrbracket$, $i \in \llbracket 1, M-1 \rrbracket$, and $j \in \llbracket 1, N \rrbracket$, we have (from the definition of the height function) that
		\begin{flalign}
			\label{hcmj} 
			\mathfrak{h}_{\ge c}^{\rightarrow} (M, j) = \displaystyle\sum_{a=1}^j \mathbbm{1}_{\mathfrak{d}_a \ge c}; \qquad \mathfrak{h}_{\ge c}^{\rightarrow} (M-i, N) = \displaystyle\sum_{a=1}^N \mathbbm{1}_{\mathfrak{d}_a \ge c} + \displaystyle\sum_{b = 1}^i \mathbbm{1}_{\mathfrak{c}_{M-b+1} \ge c}.
		\end{flalign}
		
		\noindent Furthermore, we have 
		\begin{flalign}
			\label{m0cmuj}
			\mathfrak{m}_0^{\ge c} \big( \mu(j) \big) = \displaystyle\sum_{a=1}^j \mathbbm{1}_{\mathfrak{q}_a \ge c}; \qquad \mathfrak{m}_0^{\ge c} \big( \mu(N+i) \big) = \displaystyle\sum_{a=1}^{N+i} \mathbbm{1}_{\mathfrak{q}_a \ge c} = \displaystyle\sum_{a=1}^N \mathbbm{1}_{\mathfrak{q}_a \ge c} + \displaystyle\sum_{b=1}^i \mathbbm{1}_{\mathfrak{q}_{N+b} \ge c}. 
		\end{flalign}
		
		\noindent Since \Cref{fgsv} implies that the $(M+N)$-tuple $( \mathfrak{d}_1, \mathfrak{d}_2, \ldots , \mathfrak{d}_N; \mathfrak{c}_M, \mathfrak{c}_{M-1}, \ldots , \mathfrak{c}_1)$ has the same law as $(\mathfrak{q}_1, \mathfrak{q}_2, \ldots , \mathfrak{q}_N; \mathfrak{q}_{N+1}, \mathfrak{q}_{N+2}, \ldots , \mathfrak{q}_{M+N})$, the lemma follows from \eqref{hcmj} and \eqref{m0cmuj}. 
	\end{proof}
		
	\subsection{Proof of the Matching} 
	
	\label{ProofE}
	
	In this section we establish \Cref{fgsv}. Before proceeding, it will be useful to set some notation. For any sequences $\mathfrak{A} = (a_1, a_2, \ldots , a_M)$, $\mathfrak{B} = (b_1, b_2, \ldots , b_N)$, $\mathfrak{C} = (c_1, c_2, \ldots , c_M)$, and $\mathfrak{D} = (d_1, d_2, \ldots , d_N)$ of indices in $\llbracket 0, n \rrbracket$, define 
		\begin{flalign*}
			R_{\bm{x}; \bm{y}} (\mathfrak{A}, \mathfrak{B}; \mathfrak{C}, \mathfrak{D}) = \displaystyle\sum \displaystyle\prod_{i=1}^M \displaystyle\prod_{j=1}^N R_{y_i / x_j} (v_{i,j}, u_{i,j}; v_{i+1,j}, u_{i,j+1}),
		\end{flalign*} 
		
		\noindent where the sum is over all sequences $(u_{i,j})$ and $(v_{i,j})$ of indices in $\llbracket 0, n \rrbracket$, with $(v_{k,1}, v_{k,N+1}) = (a_k, c_k)$ for each $k \in \llbracket 1, M \rrbracket$ and $(u_{1,k}, u_{M+1,k}) = (b_k, d_k)$ for each $k \in \llbracket 1, N \rrbracket$. See \Cref{zxy} for a depiction.
		
		Setting $\mathfrak{S} (\sigma) = \big( \sigma(1), \sigma(2), \ldots , \sigma (N) \big)$ and $\bm{0}_M = (0, 0, \ldots , 0) \in \mathbb{Z}_{\ge 0}^M$, and recalling the notation of \Cref{dece} and \Cref{probabilityvertex}, observe for any $\mathfrak{C} \in \llbracket 0, n \rrbracket^M$ and $\mathfrak{D} \in \llbracket 0, n \rrbracket^N$ that 
		\begin{flalign}
			\label{zvertex2}
			\mathbb{P}_{\SV; \bm{x}; \bm{y}}^{\sigma} \Big[ \big\{ \mathfrak{C} (\mathcal{E}) = \mathfrak{C} \big\} \cap \big\{ \mathfrak{D} (\mathcal{E}) = \mathcal{D} \big\} \Big] = R_{\bm{x}; \bm{y}} \big( \bm{0}_M, \mathfrak{S} (\sigma); \mathfrak{C}, \mathfrak{D} \big).
		\end{flalign}

	\begin{figure}
		\begin{center}			
			\begin{tikzpicture}[
				>=stealth,
				auto,
				style={
					scale = .9
				}
				]
				\draw[->, thick, red] (-1, 0) -- (0, 0);
				\draw[->, thick, blue] (-1, 1) -- (0, 1);
				\draw[->, thick, red] (-1, 2) -- (0, 2);
				
				\draw[->, thick, green] (3, 0) -- (4, 0);
				\draw[->, thick, blue] (3, 1) -- (4, 1);
				\draw[->, thick, orange] (3, 2) -- (4, 2);
				
				\draw[->, thick, orange] (0, -1) -- (0, 0);
				\draw[->, thick, green] (1, -1) -- (1, 0);
				\draw[->, thick, orange] (2, -1) -- (2, 0);
				\draw[->, thick, blue] (3, -1) -- (3, 0);
				
				\draw[->, thick, red] (0, 2) -- (0, 3);
				\draw[->, thick, blue] (1, 2) -- (1, 3);
				\draw[->, thick, red] (2, 2) -- (2, 3);
				\draw[->, thick, orange] (3, 2) -- (3, 3);
				
				\draw[-, dashed] (0, 0) -- (3, 0);
				\draw[-, dashed] (0, 1) -- (3, 1);
				\draw[-, dashed] (0, 2) -- (3, 2);
				
				\draw[-, dashed] (0, 0) -- (0, 2);
				\draw[-, dashed] (1, 0) -- (1, 2);
				\draw[-, dashed] (2, 0) -- (2, 2);
				\draw[-, dashed] (3, 0) -- (3, 2);
				
				\draw[] (0, -1) circle[radius = 0] node[below, scale = .75]{$y_1$};
				\draw[] (1, -1) circle[radius = 0] node[below, scale = .75]{$y_2$};
				\draw[] (2, -1) circle[radius = 0] node[below = 2, scale = .75]{$y_3$};
				\draw[] (3, -1) circle[radius = 0] node[below = 2, scale = .75]{$y_4$};
				
				\draw[] (-1, 0) circle[radius = 0] node[left, scale = .75]{$x_1$};
				\draw[] (-1, 1) circle[radius = 0] node[left, scale = .75]{$x_2$};
				\draw[] (-1, 2) circle[radius = 0] node[left, scale = .75]{$x_3$};
				
				\draw[] (-.6, 0) circle[radius = 0] node[above, scale = .65]{$b_1$};
				\draw[] (-.6, 1) circle[radius = 0] node[above, scale = .65]{$b_2$};
				\draw[] (-.6, 2) circle[radius = 0] node[above, scale = .65]{$b_3$};
				
				\draw[] (3.56, 0) circle[radius = 0] node[above, scale = .65]{$d_1$};
				\draw[] (3.6, 1) circle[radius = 0] node[above, scale = .65]{$d_2$};
				\draw[] (3.6, 2) circle[radius = 0] node[above, scale = .65]{$d_3$};
				
				\draw[] (0, -.6) circle[radius = 0] node[left, scale = .65]{$a_1$};
				\draw[] (1, -.6) circle[radius = 0] node[left, scale = .65]{$a_2$};
				\draw[] (2, -.6) circle[radius = 0] node[left, scale = .65]{$a_3$};
				\draw[] (3, -.6) circle[radius = 0] node[left, scale = .65]{$a_4$};
				
				\draw[] (0, 2.4) circle[radius = 0] node[left, scale = .65]{$c_1$};
				\draw[] (1, 2.4) circle[radius = 0] node[left, scale = .65]{$c_2$};
				\draw[] (2, 2.4) circle[radius = 0] node[left, scale = .65]{$c_3$};
				\draw[] (3, 2.4) circle[radius = 0] node[left, scale = .65]{$c_4$};		
				
				\draw[] (1.5, 1) circle[radius = 0] node[above, scale = .65]{$u_{3, 2}$};
				\draw[] (2, .5) circle[radius = 0] node[left, scale = .65]{$v_{3, 2}$};
				\draw[] (2.5, 1) circle[radius = 0] node[above, scale = .65]{$u_{4, 2}$};
				\draw[] (2, 1.5) circle[radius = 0] node[left, scale = .65]{$v_{3, 3}$};			
				
				\draw[] (0, -1.5) -- (0, -1.625) -- (3, -1.625) -- (3, -1.5); 
				\draw[] (-2, 0) -- (-2.125, 0) -- (-2.125, 2) -- (-2, 2);
				
				\draw[] (1.5, -1.675) circle[radius = 0] node[below, scale = .7]{$M$};
				\draw[] (-2.175, 1) circle[radius = 0] node[left, scale = .7]{$N$};	
			\end{tikzpicture}		
		\end{center}	
		\caption{\label{zxy} Shown above is a diagrammatic interpretation for $R_{\bm{x}; \bm{y}} (\mathfrak{A}, \mathfrak{B}; \mathfrak{C}, \mathfrak{D})$.}	
	\end{figure}
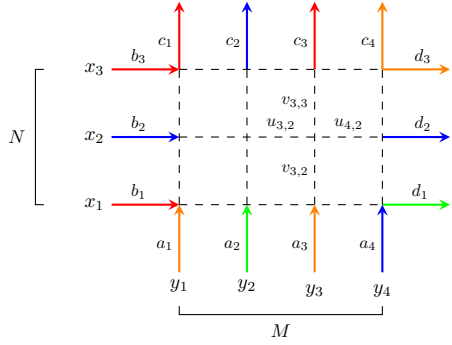

	Now we can establish \Cref{fgsv}. 
	
	\begin{proof}[Proof of \Cref{fgsv}]
		
		The proof of this proposition will be close to that of \Cref{fg2}. We begin by considering the partition function $\mathcal{Z} (\bm{\mathfrak{q}})$ for the vertex model depicted in \Cref{fg3}.
		
		This model consists of three regions, denoted by $\mathcal{R}_1$, $\mathcal{R}_2$, and $\mathcal{R}_3$. The first region $\mathcal{R}_1$ is the $M \times N$ ``cross'' to the left of $\mathbb{Z}_{< 0} \times \llbracket 1, M + N \rrbracket$. The second $\mathcal{R}_2 = \mathbb{Z}_{< 0} \times \llbracket 1, N \rrbracket$ constitutes the bottom $N$ rows to the right of the cross. The third $\mathcal{R}_3 = \mathbb{Z}_{< 0} \times \llbracket N+1, M+N \rrbracket$ constitutes the remaining $M$ rows to the right of the cross. Different vertex weights (recall \Cref{rzabcd} and \Cref{lzdefinition}) are used in these regions. In $\mathcal{R}_1$, we use $R_{y_i / x_j}$ at the intersection of the $i$-th column (from the left) and $j$-th row (from the bottom) of the cross; in $\mathcal{R}_2$, for each $j \in \llbracket 1, N \rrbracket$, we use $\widehat{L}_{x_j; s}$ at each $(-k, j)$; and in $\mathcal{R}_3$, for each $i \in \llbracket 1, M \rrbracket$, we use $L_{y_i; s}$ at each $(-k, i+M)$. 
		
		\begin{figure} 
			
		\begin{center}
			\begin{tikzpicture}[
				>=stealth, 
				scale = .75]{
					
					\draw[] (10, 3.25) circle [radius = 0] node[scale = 2]{$\mathcal{Z} (\bm{\mathfrak{q}}) = $}; 
					\draw[-, very thick] (23, 0) node[below, scale = .7]{$0$} -- (23, 7.25);
					
					\draw[dotted] (14.5, 1.5) -- (16.5, 3.5) -- (17, 3.5);
					\draw[dotted] (22, 3.5) -- (23, 3.5) node[right = 3, scale = .7]{$y_1$};
					\draw[dotted] (14, 2) -- (16.5, 4.5) -- (17, 4.5);
					\draw[->, thick, red] (22, 4.5) -- (23, 4.5) node[right = 5, scale = .7, black]{$y_2$};
					\draw[dotted] (13.5, 2.5) -- (16.5, 5.5) -- (17, 5.5);
					\draw[->, thick, green] (22, 5.5) -- (23, 5.5) node[right = 3, scale = .7, black]{$\vdots$}; 
					\draw[dotted] (13, 3) -- (16.5, 6.5) -- (17, 6.5);
					\draw[dotted] (22, 6.5) -- (23, 6.5) node[right = 3, scale = .7]{$y_M$};
					
					\draw[->, red, thick] (13, 4) node[left, black, scale = .75]{$\sigma(1)$} -- (16, .5) -- (17, .5);
					\draw[->, blue, thick] (13.5, 4.5) node[left, black, scale = .75]{$\cdots$} -- (16, 1.5) -- (17, 1.5);
					\draw[->, green, thick] (14, 5) node[left, black, scale = .75]{$\sigma(N)$} -- (16, 2.5) -- (17, 2.5);
					
					\draw[dotted] (22, .5) -- (23, .5) node[right = 3, scale = .7]{$x_1$};
					\draw[dotted] (22, 1.5) -- (23, 1.5) node[right = 3, scale = .7]{$\vdots$};
					\draw[->, thick, blue] (22, 2.5) -- (23, 2.5) node[right = 3, scale = .7, black]{$x_N$};
					
					\draw[] (22.5, .5) circle [radius = 0] node[above, scale = .7]{$\mathfrak{q}_1$};
					\draw[] (22.5, 1.5) circle [radius = 0] node[above, scale = .7]{$\cdots$};
					\draw[] (22.5, 2.5) circle [radius = 0] node[above, scale = .7]{$\mathfrak{q}_N$};
					\draw[] (22.5, 3.5) circle [radius = 0] node[above, scale = .7]{$\mathfrak{q}_{N+1}$};
					\draw[] (22.5, 4.5) circle [radius = 0] node[left = 1, above, scale = .7]{$\mathfrak{q}_{N+2}$};
					\draw[] (22.5, 5.5) circle [radius = 0] node[above, scale = .7]{$\cdots$};
					\draw[] (22.5, 6.5) circle [radius = 0] node[left = 2, above, scale = .7]{$\mathfrak{q}_{M+N}$};

					\draw[-] (14.25, 1.25) -- (14, 1) -- (12.5, 2.5) -- (12.75, 2.75);
					\draw[] (13.25, 1.75) circle [radius = 0] node[below = 3, left, scale = .7]{$M$};
					
					\draw[gray, ultra thick, dashed] (17, .5) node[below, black, scale = .7]{$-6$} -- (17, 6.5);
					\draw[gray, ultra thick, dashed] (18, .5) node[below, black, scale = .7]{$-5$} -- (18, 6.5);
					\draw[gray, ultra thick, dashed] (19, .5) node[below, black, scale = .7]{$-4$} -- (19, 6.5);
					\draw[gray, ultra thick, dashed] (20, .5) node[below, black, scale = .7]{$-3$} -- (20, 6.5);
					\draw[gray, ultra thick, dashed] (21, .5) node[below, black, scale = .7]{$-2$} -- (21, 6.5);
					\draw[gray, ultra thick, dashed] (22, .5) node[below, black, scale = .7]{$-1$} -- (22, 6.5);
					
					\draw[gray, thick, dashed] (17, .5) -- (22, .5);
					\draw[gray, thick, dashed] (17, 1.5) -- (22, 1.5);
					\draw[gray, thick, dashed] (17, 2.5) -- (22, 2.5);
					\draw[gray, thick, dashed] (17, 3.5) -- (22, 3.5);
					\draw[gray, thick, dashed] (17, 4.5) -- (22, 4.5);
					\draw[gray, thick, dashed] (17, 5.5) -- (22, 5.5);
					\draw[gray, thick, dashed] (17, 6.5) -- (22, 6.5);
				}
			\end{tikzpicture}
		\end{center}
	
		\caption{\label{fg3} Shown above is the vertex model used in the proof of \Cref{fgsv}.} 
		
		\end{figure}
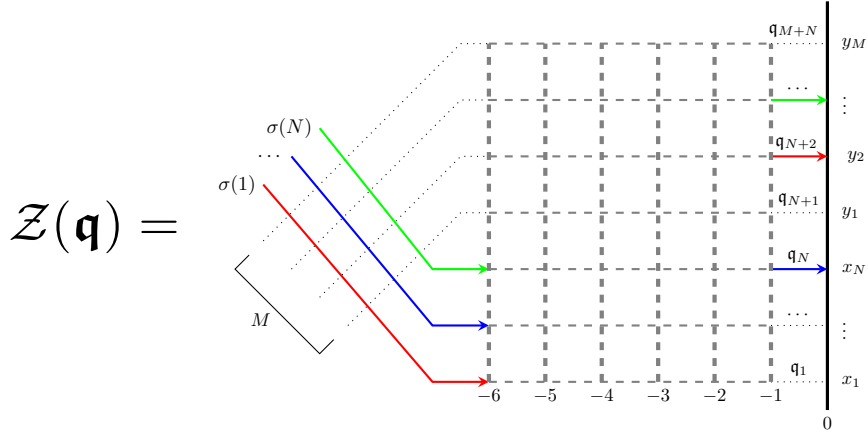

		  The boundary conditions for the model in \Cref{fg3} are prescribed as follows. The entrance data is defined in the same way as for \Cref{z2xy} in the proof of \Cref{fg2}. Specifically, we have no arrows vertically enter any column of the model, either in or to the right of the cross, and we have an arrow of color $\sigma (j)$ enter through the $j$-th row (from the bottom) of the cross, for each $j \in \llbracket 1, N \rrbracket$. The exit data is defined differently, by having an arrow of color $\mathfrak{q}_i$ horizontally exit through $(-1, i)$ (the $i$-th row of the model) for each $i \in \llbracket 1, M+N \rrbracket$, and having no arrow vertically exit through any column of the model.
		
		Let us now analyze the partition function for this vertex model. Under the constraint \eqref{sxsy}, it is quickly verified (as in the proof of \Cref{fg2}) that a colored higher spin path ensemble in $\mathbb{Z}_{< 0} \times \llbracket 1, M+N \rrbracket$ has nonzero weight only if all but finitely many vertices in $\mathcal{R}_2$ have arrow configurations of the form $(\bm{e}_0, i; \bm{e}_0, i)$ for some integer $i \in \llbracket 1, n \rrbracket$, and all but finitely many vertices in $\mathcal{R}_3$ have arrow configurations of the form $(\bm{e}_0, 0; \bm{e}_0, 0)$. This forces the cross $\mathcal{R}_1$ to freeze, with the arrow configuration $\big( 0, \sigma (j); 0, \sigma (j) \big)$ at the intersection of its $i$-th column and $j$-th row, for each $(i, j) \in \llbracket 1, M \rrbracket \times \llbracket 1, N \rrbracket$. By \eqref{rzij}, this means that the weight of $\mathcal{R}_1$ is  
		\begin{flalign}
			\label{r1r} 
			\displaystyle\prod_{i=1}^M \displaystyle\prod_{j=1}^N R_{y_i / x_j} \big( 0, \sigma(j); 0, \sigma(j) \big) = \displaystyle\prod_{i=1}^M \displaystyle\prod_{j=1}^N \displaystyle\frac{x_j - y_i}{x_j - qy_i}.
		\end{flalign}
			
		We next evaluate the partition function $\breve{\mathcal{Z}} (\bm{\mathfrak{q}})$ of $\mathcal{R}_2 \cup \mathcal{R}_3$. To that end, we modify this part of the vertex model shown in \Cref{fg3}, by having all of the colored paths exit vertically through the $y$-axis; see \Cref{fg5} for a depiction. Here, the weights used for the $y$-axis are $L_{x_j; 0}$ at $(0, j)$ for $j \in \llbracket 1, N \rrbracket$ and $L_{y_i; 0}$ at $(0, i + N)$ for $i \in \llbracket 1, M \rrbracket$. Since by \Cref{lzdefinition} we have for each $i \in \llbracket 1, M \rrbracket$, $j \in \llbracket 1, N \rrbracket$, and $\bm{A} \in \mathbb{Z}_{\ge 0}^n$ that 
		\begin{flalign}
			\label{lyi0}
			L_{y_i; 0} (\bm{A}, \mathfrak{q}_{i+N}; \bm{A}_{\mathfrak{q}_{i+N}}^+, 0) = 1; \qquad L_{x_j; 0} (\bm{A}, \mathfrak{q}_j; \bm{A}_{\mathfrak{q}_j}^+, 0) = 1,
		\end{flalign} 
	
		\noindent the weight of the $y$-axis in \Cref{fg5} is $1$. Hence, denoting the partition function of the vertex model depicted in \Cref{fg5} by $\mathcal{Z}' (\bm{\mathfrak{q}})$, we have $\breve{\mathcal{Z}} (\bm{\mathfrak{q}}) = \mathcal{Z}' (\bm{\mathfrak{q}})$. Together with \eqref{r1r}, this yields 
		\begin{flalign}
			\label{zkzk} 
			\mathcal{Z} (\bm{\mathfrak{q}}) = \breve{\mathcal{Z}} (\bm{\mathfrak{q}}) \cdot \displaystyle\prod_{i=1}^M \displaystyle\prod_{j=1}^N \displaystyle\frac{x_j - y_i}{x_j - qy_i} = \mathcal{Z}' (\bm{\mathfrak{q}}) \cdot \displaystyle\prod_{i=1}^M \displaystyle\prod_{j=1}^N \displaystyle\frac{x_j - y_i}{x_j - qy_i}. 
		\end{flalign}

				\begin{figure} 
			
			\begin{center}
				\begin{tikzpicture}[
					>=stealth, 
					scale = .75]{
						
						\draw[] (12.75, 3.5) circle [radius = 0] node[scale = 2]{$\mathcal{Z}' (\bm{\mathfrak{q}}) = $};
						\draw[->, very thick] (23, 0) -- (23, 7.5);
						
						\draw[dotted] (16.5, 3.5) -- (17, 3.5);
						\draw[dotted] (22, 3.5) -- (23, 3.5) node[right = 3, scale = .7]{$y_1$};
						\draw[dotted]  (16.5, 4.5) -- (17, 4.5);
						\draw[->, thick, red] (22, 4.5) -- (22.9, 4.5) node[right = 7, scale = .7, black]{$y_2$} -- (22.9, 7.5);
						\draw[dotted]  (16.5, 5.5) -- (17, 5.5);
						\draw[->, thick, green] (22, 5.5) -- (23, 5.5) node[right = 3, scale = .7, black]{$\vdots$} -- (23.1, 5.5) -- (23.1, 7.5); 
						\draw[dotted] (16.5, 6.5) -- (17, 6.5);
						\draw[dotted] (22, 6.5) -- (23, 6.5) node[right = 3, scale = .7]{$y_M$};
						
						\draw[->, red, thick] (16, .5) node[left, black, scale = .75]{$\sigma(1)$} -- (17, .5);
						\draw[->, blue, thick] (16, 1.5) node[left, black, scale = .75]{$\cdots$} -- (17, 1.5);
						\draw[->, green, thick] (16, 2.5) node[left, black, scale = .75]{$\sigma(N)$} -- (17, 2.5);
						
						\draw[dotted] (22, .5) -- (23, .5) node[right = 3, scale = .7]{$x_1$};
						\draw[dotted] (22, 1.5) -- (23, 1.5) node[right = 3, scale = .7]{$\vdots$};
						\draw[->, thick, blue] (22, 2.5) -- (23, 2.5) node[right = 3, scale = .7, black]{$x_N$} -- (23, 7.5);
						
						\draw[] (22.5, .5) circle [radius = 0] node[above, scale = .7]{$\mathfrak{q}_1$};
						\draw[] (22.5, 1.5) circle [radius = 0] node[above, scale = .7]{$\cdots$};
						\draw[] (22.5, 2.5) circle [radius = 0] node[above, scale = .7]{$\mathfrak{q}_N$};
						\draw[] (22.5, 3.5) circle [radius = 0] node[above, scale = .7]{$\mathfrak{q}_{N+1}$};
						\draw[] (22.5, 4.5) circle [radius = 0] node[left = 1, above, scale = .7]{$\mathfrak{q}_{N+2}$};
						\draw[] (22.5, 5.5) circle [radius = 0] node[above, scale = .7]{$\cdots$};
						\draw[] (22.5, 6.5) circle [radius = 0] node[left = 2, above, scale = .7]{$\mathfrak{q}_{M+N}$};

						\draw[-] (16.25, 3.5) -- (16, 3.5) -- (16, 6.5) -- (16.25, 6.5);
						\draw[] (16, 5) circle [radius = 0] node[left, scale = .7]{$M$};
						
						\draw[gray, ultra thick, dashed] (17, .5) -- (17, 6.5);
						\draw[gray, ultra thick, dashed] (18, .5) -- (18, 6.5);
						\draw[gray, ultra thick, dashed] (19, .5) -- (19, 6.5);
						\draw[gray, ultra thick, dashed] (20, .5) -- (20, 6.5);
						\draw[gray, ultra thick, dashed] (21, .5) -- (21, 6.5);
						\draw[gray, ultra thick, dashed] (22, .5) -- (22, 6.5);
						
						\draw[gray, thick, dashed] (17, .5) -- (22, .5);
						\draw[gray, thick, dashed] (17, 1.5) -- (22, 1.5);
						\draw[gray, thick, dashed] (17, 2.5) -- (22, 2.5);
						\draw[gray, thick, dashed] (17, 3.5) -- (22, 3.5);
						\draw[gray, thick, dashed] (17, 4.5) -- (22, 4.5);
						\draw[gray, thick, dashed] (17, 5.5) -- (22, 5.5);
						\draw[gray, thick, dashed] (17, 6.5) -- (22, 6.5);
					}
				\end{tikzpicture}
			\end{center}
			
			\caption{\label{fg5} Shown above is a modification of the $\mathcal{R}_2 \cup \mathcal{R}_3$ region of the vertex model from \Cref{fg3}, in which all paths exit through the $y$-axis.} 
			
		\end{figure}
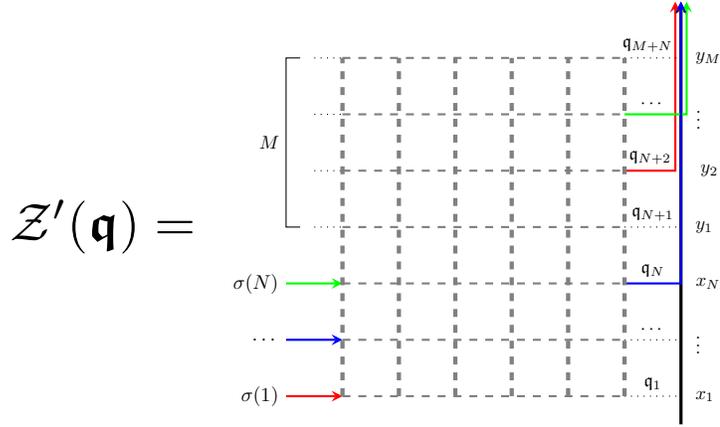 
		
		To evaluate $\mathcal{Z}' (\bm{\mathfrak{q}})$, observe that any colored higher spin path ensemble with boundary data as depicted in \Cref{fg5} is determined by a sequence $\bm{\mu} = \big( \mu(0), \mu(1), \ldots , \mu(M+N) \big)$ of $n$-compositions, where $\mu(i)$ indexes the locations of the colored arrows exiting the $i$-th row $\{ y = i \}$ of the model, for each $i \in \llbracket 0, M+N \rrbracket$ (in the sense of \Cref{mumuistate}). It is quickly verified that any such sequence $\bm{\mu}$ is $(M; \sigma)$-ascending (as in \Cref{mnmu}) and satisfies $\bm{\mathfrak{q}} (\bm{\mu}) = \bm{\mathfrak{q}}$. Moreover \Cref{fg} implies that, for any $j \in \llbracket 1, N \rrbracket$, the $j$-th row of the model in \Cref{fg5} has weight $f_{\mu(j) / \mu(j-1)}^{\sigma(j)} (x_j)$. Simiarly, for any $i \in \llbracket N+1, M + N \rrbracket$, its $i$-th row has weight $G_{\mu(i-1) / \mu(i)} (y_{i-N})$. Thus,
		\begin{flalign*}
			\mathcal{Z}' (\bm{\mathfrak{q}}) = \displaystyle\sum_{\bm{\mu}} \mathbbm{1}_{\bm{\mathfrak{q}} (\bm{\mu}) = \bm{\mathfrak{q}}} \cdot  \displaystyle\prod_{j=1}^N f_{\mu(j) / \mu(j-1); s}^{\sigma(j)} (x_j) \displaystyle\prod_{i=N+1}^{M+N} G_{\mu(i-1) / \mu(i); s} (y_{i-N}),
		\end{flalign*} 
	
		\noindent where the sum is over all $(M; \sigma)$-ascending sequences $\bm{\mu}$ of $n$-compositions. Together with \eqref{zkzk} and \Cref{measurefg}, this gives
		\begin{flalign}
			\label{zkprobability}
			\begin{aligned}
			\mathcal{Z} (\bm{\mathfrak{q}}) & = \displaystyle\prod_{i=1}^M \displaystyle\prod_{j=1}^N \displaystyle\frac{x_j - y_i}{x_j - qy_i} \cdot \displaystyle\sum_{\bm{\mu}} \mathbbm{1}_{\bm{\mathfrak{q}} (\bm{\mu}) = \bm{\mathfrak{q}}} \cdot  \displaystyle\prod_{j=1}^N f_{\mu(j) / \mu(j-1); s}^{\sigma(j)} (x_j) \displaystyle\prod_{i=N+1}^{M+N} G_{\mu(i-1) / \mu(i); s} (y_{i-N}) \\
			& = \mathbb{P}_{\fG}^{\sigma} \big[ \bm{\mathfrak{q}} (\bm{\mu}) = \bm{\mathfrak{q}} \big].
			\end{aligned}
		\end{flalign}
	  
		  Next, by $MN$ sequences of applications of the Yang--Baxter equation \eqref{llrijk2}, the partition function $\mathcal{Z} (\bm{\mathfrak{q}})$ of the model from \Cref{fg3} is unchanged if the cross originally in region $\mathcal{R}_1$ is moved to the right of $\mathcal{R}_2 \cup \mathcal{R}_3 = \mathbb{Z}_{< 0} \times \llbracket 1, M+N \rrbracket$. In particular, $\mathcal{Z} (\bm{\mathfrak{q}})$ is also the partition function of the vertex model in \Cref{fg4}.

		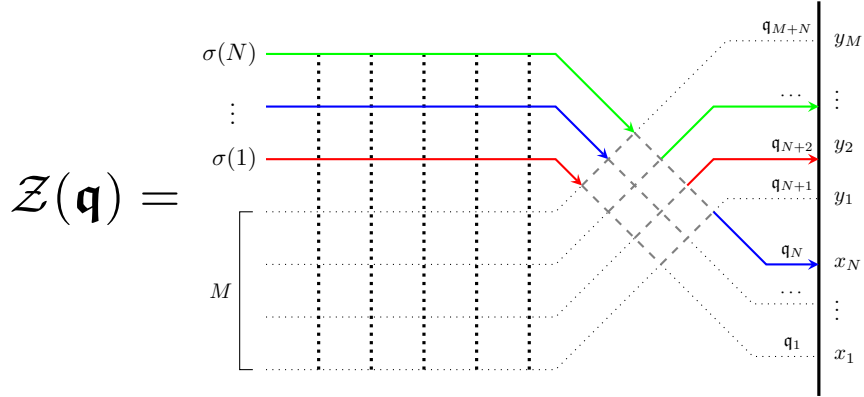
\begin{figure} 
		\begin{center}
			\begin{tikzpicture}[
				>=stealth, 
				scale = .7]{
					\draw[-, very thick] (9, -.5) -- (9, 7);
					
					\draw[very thick, dotted] (-.5, 0) -- (-.5, 6);
					\draw[very thick, dotted] (.5, 0) -- (.5, 6);
					\draw[very thick, dotted] (1.5, 0) -- (1.5, 6);
					\draw[very thick, dotted] (2.5, 0) -- (2.5, 6);
					\draw[very thick, dotted] (3.5, 0) -- (3.5, 6);
					
					\draw[] (-4.75, 3.25) circle [radius = 0] node[scale=2]{$\mathcal{Z} (\bm{\mathfrak{q}}) = $};
					\draw[dotted] (7, 3) -- (7.25, 3.25) -- (9, 3.25) node[right = 3, scale = .8]{$y_1$};
					\draw[] (9, 4.25) circle [radius = 0] node[right = 3, scale = .8]{$y_2$};
					\draw[]  (9, 5.25)circle [radius = 0] node[right = 3, scale = .8]{$\vdots$};
					\draw[dotted] (5.5, 4.5) -- (7.25, 6.25) -- (9, 6.25) node[right = 3, scale = .8]{$y_M$};
					
					\draw[dotted] (6, 2) -- (7.75, .25) -- (9, .25) node[right = 3, scale = .8]{$x_1$};
					\draw[dotted] (6.5, 2.5) -- (7.75, 1.25) -- (9, 1.25) node[right = 3, scale = .8]{$\vdots$};
					\draw[]  (9, 2) circle [radius = 0] node[right = 3, scale = .8]{$x_N$};
					
					\draw[->, red, thick] (-1.5, 4) node[left, black, scale = .9]{$\sigma(1)$} -- (4, 4) -- (4.5, 3.5);
					\draw[->, blue, thick] (-1.5, 5) node[left = 8, black, scale = .9]{$\vdots$} -- (4, 5) -- (5, 4); 
					\draw[->, green, thick] (-1.5, 6) node[left, black, scale = .9]{$\sigma(N)$} -- (4, 6) -- (5.5, 4.5);
					
					\draw[dotted] (-1.5, 0) -- (4, 0) -- (6, 2);
					\draw[dotted] (-1.5, 1) -- (4, 1) -- (5.5, 2.5);
					\draw[dotted] (-1.5, 2) -- (4, 2) -- (5, 3);
					\draw[dotted] (-1.5, 3) -- (4, 3) -- (4.5, 3.5);

					\draw[gray, dashed, thick] (4.5, 3.5) -- (5.5, 4.5); 
					\draw[gray, dashed, thick] (6, 2) -- (7, 3); 
					\draw[->, thick, red] (6.5, 3.5) -- (7, 4) -- (9, 4); 
					\draw[gray, dashed, thick] (5.5, 2.5) -- (6.5, 3.5); 
					\draw[->, thick, green] (6, 4) -- (7, 5) -- (9, 5); 
					\draw[gray, dashed, thick] (5, 3) -- (6, 4);
					
					\draw[gray, dashed, thick] (4.5, 3.5) -- (6, 2);
					\draw[gray, dashed, thick] (5, 4) -- (6.5, 2.5);
					\draw[gray, dashed, thick] (5.5, 4.5) -- (7, 3); 
					\draw[->, thick, blue] (7, 3) -- (8, 2) -- (9, 2);
					
					\draw[-] (-1.75, 0) -- (-2, 0) -- (-2, 3) -- (-1.75, 3);
					\draw[] (-2, 1.5) circle [radius = 0] node[left, scale = .8]{$M$};
					
					\draw[] (8.5, .25) circle [radius = 0] node[above, scale = .7]{$\mathfrak{q}_1$};
					\draw[] (8.5, 1.25) circle [radius = 0] node[above, scale = .7]{$\cdots $};
					\draw[] (8.5, 2) circle [radius = 0] node[above, scale = .7]{$\mathfrak{q}_N$};
					\draw[] (8.5, 3.25) circle [radius = 0] node[above, scale = .7]{$\mathfrak{q}_{N+1}$};
					\draw[] (8.5, 4) circle [radius = 0] node[above, scale = .7]{$\mathfrak{q}_{N+2}$};
					\draw[] (8.5, 5) circle [radius = 0] node[above, scale = .7]{$\cdots$};
					\draw[] (8.4, 6.25) circle [radius = 0] node[above, scale = .7]{$\mathfrak{q}_{M+N}$};
					
				}
			\end{tikzpicture}
		\end{center}
		
		\caption{\label{fg4} Shown above is the vertex model from \Cref{fg3} after using the Yang--Baxter equation to move the cross to the right of $\mathbb{Z}_{< 0} \times \llbracket 1, M + N \rrbracket$.}
		\end{figure}

		This model also consists of three regions $\mathcal{R}_1'$, $\mathcal{R}_2'$, and $\mathcal{R}_3'$. The first $\mathcal{R}_1'$ is an $M \times N$ cross, that is now to the right of $\mathbb{Z}_{< 0} \times \llbracket 1, M+N \rrbracket$. The third $\mathcal{R}_3' = \mathbb{Z}_{\le 0} \times \llbracket 1, M \rrbracket$ consists of the bottom $M$ rows to the left of the  cross. The second $\mathcal{R}_2'$ consists of the remaining $N$ rows to the left of the cross. Again different vertex weights are used in these regions. In $\mathcal{R}_1'$, we use the weight $R_{y_i / x_j}$ at the intersection of the $i$-th column and $j$-th row of the cross, for each $(i, j) \in \llbracket 1, M \rrbracket \times \llbracket 1, N \rrbracket$; in $\mathcal{R}_2'$, for each $j \in \llbracket 1, N \rrbracket$, we use the weight $\widehat{L}_{x_j; s}$ at every $(-k, M+j)$; and in $\mathcal{R}_3'$, for each $i \in \llbracket 1, M \rrbracket$, we use the weight $L_{y_i; s}$ at every $(-k, i)$. 
		
		The boundary data for the model in \Cref{fg4} is prescribed as follows (it must match that of \Cref{fg3}). The entrance data is defined by having no arrows vertically enter any column of the model; no arrow horizontally enter through the bottom $M$ rows of the model; and having an arrow of color $\sigma (j)$ enter through the $(M+j)$-th row of the model, for each $j \in \llbracket 1, N \rrbracket$. The exit data is defined by having no arrows horizontally exit through any column of the model; having an arrow of color $\mathfrak{q}_j$ exit through the $j$-th row of the cross, for each $j \in \llbracket 1, N \rrbracket$; and having an arrow of color $\mathfrak{q}_{i+N}$ exit through the $(M-i+1)$-th column of the cross, for each $j \in \llbracket 1, M \rrbracket$.  
		
		Observe that the $\mathcal{R}_2' \cup \mathcal{R}_3' = \mathbb{Z}_{< 0} \times \llbracket 1, M + N \rrbracket$ part of this vertex model is frozen. The one colored higher spin path ensemble there (with that boundary data) that has nonzero weight is the one in which, for each $j \in \llbracket 1, N \rrbracket$, the path of color $\sigma (j)$ in the $(M+j)$-th row travels horizontally until it reaches the cross. Recalling from \eqref{rzij}, \eqref{lzi}, and \eqref{lzij2} that 
	\begin{flalign*}
	& L_{y_i; s} (\bm{e}_0, 0; \bm{e}_0, 0) = 1, \qquad \widehat{L}_{x_j; s} \big( \bm{e}_0, \sigma(j); \bm{e}_0, \sigma(j) \big) = 1,
\end{flalign*} 

	\noindent it follows that the weight of $\mathcal{R}_2' \cup \mathcal{R}_3'$ in \Cref{fg4} is equal to $1$.  The partition function of $\mathcal{R}_1'$ is given by $R_{\bm{x}; \bm{y}} \big( \bm{0}_M, \mathfrak{S} (\sigma); \mathfrak{C}, \mathfrak{D})$. Together with \eqref{zvertex2}, this gives 
	\begin{flalign*} 
		\mathcal{Z} (\bm{\mathfrak{q}}) = \mathbb{P}_{\SV}^{\sigma} \Big[ \big\{ \mathfrak{C} (\mathcal{E}) = \mathfrak{C} \big\} \cap \big\{ \mathfrak{D} (\mathcal{E}) = \mathfrak{D} \big\} \Big], 
	\end{flalign*} 

	\noindent which together with \eqref{zkprobability} yields the proposition. 	
	\end{proof}

	\begin{rem}
		
		Observe that the proof of \Cref{fgsv} used \eqref{zvertex2}, which required that the $R$-weights from \Cref{rzabcd} were stochastic. It is also possible to formulate a version of \Cref{fgsv} for non-stochastic $R$-weights, in which the stochastic weights of the vertex model describing the left side of \eqref{ekmuk} (equivalently, the stochastic weights of the cross in the proof of \Cref{fgsv}) would be determined by the $R$-weights through the stochasticization procedure of \cite{SSE}. 
		
	\end{rem}

	\begin{rem}

	Our reason for using the domain $\mathbb{Z}_{< 0} \times \llbracket 1, M + N \rrbracket$ above is to avoid having to ``reflect'' the stochastic cross (to direct its paths up-left instead of up-right) in the proof of \Cref{fgsv}, which would have been necessary had we instead used the more standard domain $\mathbb{Z}_{> 0} \times \llbracket 1, M + N \rrbracket$ from previous works \cite{HSVMSRF,SSVMP,CSVMST,RFSVM}. 
	 		
	\end{rem} 

	\begin{rem} 

	It is possible to formulate a generalization of \Cref{fgsv} when its domain $\mathcal{D}_{M;N}$ is not necessarily rectangular but instead ``jagged,'' that is, bounded by an up-left directed path $\mathcal{P}$. In this case, the associated measures $\mathbb{P}_{\fG}$ would no longer necessarily ascending, but would rather have ascents and descents (depending on whether a corresponding step of $\mathcal{P}$ is directed north or west); see \cite[Theorem 5.6]{SSVMP} for such a statement in the colorless ($n=1$) case.
	
	\end{rem}

	\section{Colored Line Ensembles for Colored Six-Vertex Models} 
	
	\label{L0}
	
	In this section we explain how the results of \Cref{Line} can be reformulated in terms of colored line ensembles. We first associate colored line ensembles with ascending sequences of compositions in \Cref{Ensemble} and then discuss properties of random colored line ensembles (associated with random ascending sequences of compositions sampled according to $\mathbb{P}_{\fG}^{\sigma}$) in \Cref{ConditionW}. We then describe color merging properties for these random colored line ensembles in \Cref{ColorL}. Throughout this section, we fix integers $n, M, N \ge 1$; a composition $\bm{\ell} = (\ell_1, \ell_2, \ldots , \ell_n)$ of $N$; a function $\sigma : \llbracket 1, N \rrbracket \rightarrow \llbracket 1, n \rrbracket$, such that for each $i \in \llbracket 1, n \rrbracket$ we have $\ell_i = \# \big\{ \sigma^{-1} (i) \big\}$; a complex number $s \in \mathbb{C}$; and sequences of complex numbers $\bm{x} = (x_1, x_2, \ldots , x_N)$ and $\bm{y} = (y_1, y_2, \ldots , y_M)$, such that \eqref{sxsy} holds.

	\subsection{Colored Line Ensembles and Ascending Sequences} 
	
	\label{Ensemble}
	
	In this section we associate a colored line ensemble (see \Cref{cl}) with a given $(M; \sigma)$-ascending sequence $\bm{\mu}$ of $n$-compositions. This is done through the following definition.    
	
	 \begin{definition} 
	 	
	 \label{lmu} 
	 
	 Let $\bm{\mu} = \big( \mu (0), \mu(1), \ldots , \mu(M+N) \big)$ denote an $(M; \sigma)$-ascending sequence of $n$-compositions. The associated simple colored line ensemble $\bm{\mathsf{L}} = \bm{\mathsf{L}}_{\bm{\mu}} = \big( \bm{\mathsf{L}}^{(1)}, \bm{\mathsf{L}}^{(2)}, \ldots , \bm{\mathsf{L}}^{(n)} \big)$ on $\llbracket 0, M+N \rrbracket$ is defined as follows. For each $c \in \llbracket 1, n \rrbracket$ let $\bm{\mathsf{L}}^{(c)} = \bm{\mathsf{L}}_{\bm{\mu}}^{(c)} = \big( \mathsf{L}_1^{(c)}, \mathsf{L}_2^{(c)}, \ldots \big)$, where for each $k \ge 1$ the function $\mathsf{L}_k^{(c)} = \mathsf{L}_{k;\bm{\mu}}^{(c)} : \llbracket 0, M+N \rrbracket \rightarrow \mathbb{Z}$ prescribed by setting  
	\begin{flalign}
		\label{kclmu}
		\mathsf{L}_k^{(c)} (i) = \ell_{[c,n]}  -\mathfrak{m}_{\le k - 1}^{\ge c} \big( \mu (i) \big), \qquad \text{for each  $i \in \llbracket 0, M+N \rrbracket$}.
	\end{flalign}
	
	\noindent The fact that this defines a simple colored line ensemble follows from \Cref{lmu1} below. We moreover set the differences $\bm{\Lambda}^{(c)} = \big( \Lambda_1^{(c)}, \Lambda_2^{(c)}, \ldots \big)$ of $\bm{\mathsf{L}}$ by 
	\begin{flalign*} 
		\Lambda_k^{(c)} (i) = \mathsf{L}_k^{(c)} (i) - \mathsf{L}_k^{(c+1)} (i), \qquad \text{for each $(k, i) \in \mathbb{Z}_{> 0} \times \llbracket 0, M + N \rrbracket$}, 
	\end{flalign*} 

	\noindent where $\mathsf{L}_k^{(n+1)} : \llbracket 0, M + N \rrbracket \rightarrow \mathbb{Z}$ is  defined by setting $\mathsf{L}_k^{(n+1)} (i) = 0$ for each $(k, i) \in \mathbb{Z}_{>0} \times  \llbracket 0, M + N \rrbracket$.

	\end{definition}

	\begin{figure} 
		
		\begin{center}
			\begin{tikzpicture}[
				>=stealth, 
				scale = .65]{	
					
					\draw[thick] (23, -.25) -- (23, 7.25);
					
					\draw[black, thick, dotted] (15, .5) -- (15, 6.5);
					\draw[black, thick, dotted] (16, .5) -- (16, 6.5);
					\draw[black, thick, dotted] (17, .5)  -- (17, 6.5);
					\draw[black, thick, dotted] (18, .5) node[below = 3, scale = .65]{$\cdots$} -- (18, 6.5);
					\draw[black, thick, dotted] (19, .5) node[below, scale = .65]{$-4$}  -- (19, 6.5);
					\draw[black, thick, dotted] (20, .5) node[below, scale = .65]{$-3$} -- (20, 6.5);
					\draw[black, thick, dotted] (21, .5) node[below, scale = .65]{$-2$} -- (21, 6.5);
					\draw[black, thick, dotted] (22, .5) node[below, scale = .65]{$-1$} -- (22, 6.5);
					
					\draw[black, thick, dotted] (15, .5) -- (22, .5);
					\draw[black, thick, dotted] (15, 1.5) -- (22, 1.5);
					\draw[black, thick, dotted] (15, 2.5) -- (22, 2.5);
					\draw[black, thick, dotted] (15, 3.5) -- (22, 3.5);
					\draw[black, thick, dotted] (15, 4.5) -- (22, 4.5);
					\draw[black, thick, dotted] (15, 5.5) -- (22, 5.5);
					\draw[black, thick, dotted] (15, 6.5) -- (22, 6.5);
					
					\draw[thick, dotted] (22, 2.5) -- (23, 2.5);
					\draw[thick, dotted] (22, 4.5) -- (23, 4.5);
					\draw[thick, dotted] (22, .5) -- (23, .5);
					\draw[thick, dotted] (22, 1.5) -- (23, 1.5);
					
					\draw[->, red, thick] (14, .5) node[left, black, scale = .75]{$\sigma(1)$} -- (15, .5) -- (21, .5) -- (21, 1.5) -- (22.85, 1.5) -- (22.85, 7.5);
					\draw[->, blue, thick] (14, 1.5) node[left, black, scale = .75]{$\sigma(2)$} -- (15, 1.5) -- (20, 1.5) -- (20, 2.5) -- (22, 2.5) -- (22, 3.5) -- (23.05, 3.5) -- (23.05, 7.5);
					\draw[->, blue, thick] (14, 2.5) node[left, black, scale = .75]{$\cdots$} -- (15, 2.5) -- (20.05, 2.5) -- (20.05, 5.5) -- (21, 5.5) -- (21, 6.5) -- (22, 6.5) -- (23.15, 6.5) -- (23.15, 7.5);
					\draw[->, red, thick] (14, 3.5) node[left, black, scale = .75]{$\sigma(N)$} -- (15, 3.5) -- (19.95, 3.5) -- (19.95, 4.5) -- (21, 4.5) -- (21, 5.5) -- (22.95, 5.5) -- (22.95, 7.5);
					
					\draw[-] (14.75, 4.5) -- (14.5, 4.5) -- (14.5, 6.5) -- (14.75, 6.5);
					\draw[] (14.5, 5.5) circle [radius = 0] node[left, scale = .7]{$M$};
					
					\draw[] (23, 0) circle [radius = 0] node[right, scale = .7]{$\mu(0)$};	
					\draw[] (23, 1) circle [radius = 0] node[right, scale = .7]{$\mu(1)$};	
					\draw[] (23, 2) circle [radius = 0] node[right, scale = .7]{$\vdots$};	
					\draw[] (23, 3) circle [radius = 0] node[right, scale = .7]{$\mu(N)$};	
					\draw[] (23, 4) circle [radius = 0] node[right, scale = .7]{$\mu(N+1)$};	
					\draw[] (23, 5) circle [radius = 0] node[right, scale = .7]{$\mu(N+2)$};		
					\draw[] (23, 6) circle [radius = 0] node[right = 2, scale = .7]{$\vdots$};	
					\draw[] (23, 7) circle [radius = 0] node[right = 3, scale = .7]{$\mu(M+N)$};		
					
					\draw[dotted] (26, 0) -- (33, 0) node[right, scale = .7]{$0$};
					\draw[dotted] (26, 1) -- (33, 1) node[right, scale = .7]{$1$};
					\draw[dotted] (26, 2) -- (33, 2) node[right, scale = .7]{$2$};
					\draw[dotted] (26, 3) -- (33, 3) node[right, scale = .7]{$3$};
					\draw[dotted] (26, 4) -- (33, 4) node[right, scale = .7]{$4$};
					\draw[dotted] (26, 0) node[below = 1, scale = .7]{$0$} -- (26, 4);
					\draw[dotted] (27, 0) node[below = 1, scale = .7]{$1$} -- (27, 4);
					\draw[dotted] (28, 0) node[below = 1, scale = .7]{$2$} -- (28, 4);
					\draw[dotted] (29, 0) node[below = 1, scale = .7]{$3$} -- (29, 4);
					\draw[dotted] (30, 0) node[below = 1, scale = .7]{$4$} -- (30, 4);
					\draw[dotted] (31, 0) node[below = 1, scale = .7]{$5$} -- (31, 4);
					\draw[dotted] (32, 0) node[below = 1, scale = .7]{$6$} -- (32, 4);
					\draw[dotted] (33, 0) node[below = 1, scale = .7]{$7$} -- (33, 4); 
					
					\draw[dotted] (26, 5) -- (33, 5) node[right, scale = .7]{$0$};
					\draw[dotted] (26, 6) -- (33, 6) node[right, scale = .7]{$1$};
					\draw[dotted] (26, 7) -- (33, 7) node[right, scale = .7]{$2$};
					\draw[dotted] (26, 5) node[below = 1, scale = .7]{$0$} -- (26, 7);
					\draw[dotted] (27, 5) node[below = 1, scale = .7]{$1$} -- (27, 7);
					\draw[dotted] (28, 5) node[below = 1, scale = .7]{$2$} -- (28, 7);
					\draw[dotted] (29, 5) node[below = 1, scale = .7]{$3$} -- (29, 7);
					\draw[dotted] (30, 5) node[below = 1, scale = .7]{$4$} -- (30, 7);
					\draw[dotted] (31, 5) node[below = 1, scale = .7]{$5$} -- (31, 7);
					\draw[dotted] (32, 5) node[below = 1, scale = .7]{$6$} -- (32, 7);
					\draw[dotted] (33, 5) node[below = 1, scale = .7]{$7$} -- (33, 7);
					
					\draw[ thick, blue] (26, 7.1) -- (29, 7.1) -- (30, 6.1) -- (32, 6.1) -- (33, 5.1);
					\draw[ thick, blue] (26, 7) -- (28, 7) -- (29, 6) -- (32, 6) -- (33, 5);
					\draw[ thick, blue] (26, 6.9) -- (28, 6.9) -- (29, 5.9) -- (31, 5.9) -- (32, 4.9) -- (33, 4.9);

					\draw[ thick, violet] (26, 4.1) -- (27, 4.1) -- (28, 3.1) -- (29, 3.1) -- (30, 2.1) -- (31, 2.1) -- (32, 1.1) -- (33, .1);
					\draw[ thick, violet] (26, 4) -- (27, 4) -- (28, 3) -- (29, 2) -- (31, 2) -- (32, 1) -- (33, 0);
					\draw[ thick, violet] (26, 3.9) -- (27, 2.9) -- (28, 2.9) -- (29, 1.9) -- (30, 1.9) -- (31, .9) -- (32, -.1) -- (33, -.1);
					
					\draw[] (29, 7.1) circle [radius = 0] node[above, scale = .65]{$\mathsf{L}_1^{(2)}$};
					\draw[] (29, 6.5) circle [radius = 0] node[scale = .65]{$\mathsf{L}_2^{(2)}$};
					\draw[] (29, 5.9) circle [radius = 0] node[below, scale = .65]{$\mathsf{L}_3^{(2)}$};
					
					\draw[] (29, 3.1) circle [radius = 0] node[above, scale = .65]{$\mathsf{L}_1^{(1)}$};
					\draw[] (29, 2.5) circle [radius = 0] node[scale = .65]{$\mathsf{L}_2^{(1)}$};
					\draw[] (29, 1.9) circle [radius = 0] node[below, scale = .65]{$\mathsf{L}_3^{(1)}$};
				}
			\end{tikzpicture}
		\end{center}
		
		\caption{\label{0lmu} Shown to the left is a depiction for an ascending sequence $\bm{\mu}$ of $2$-compositions through a colored higher spin path ensemble (where color $1$ is red and color $2$ is blue). Shown to the right are the two associated simple line ensembles $\bm{\mathsf{L}}^{(1)}$ (in purple, as it counts both red and blue paths) and $\bm{\mathsf{L}}^{(2)}$ (in blue).} 
		
	\end{figure}
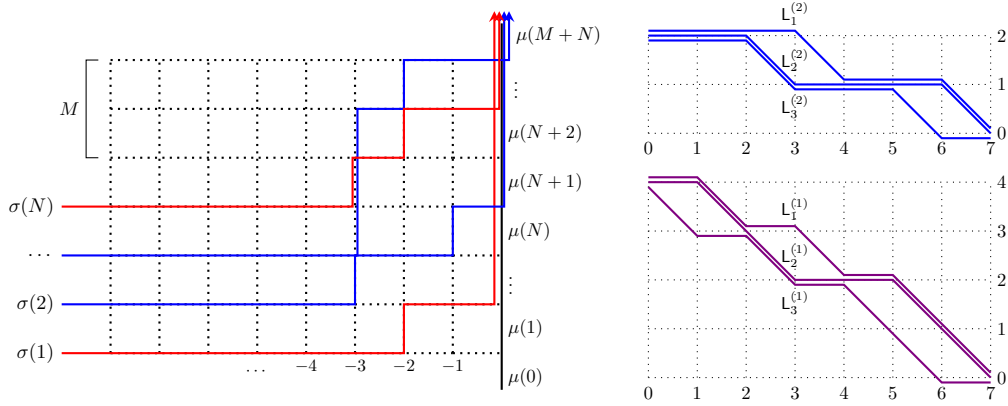

	\begin{lem} 
		
	\label{lmu1} 
	
	Adopting the notation and assumptions of \Cref{lmu}, $\bm{\mathsf{L}}$ is a simple colored line ensemble, which satisfies the following three properties for any $c \in \llbracket 1, n \rrbracket$, $k \in \mathbb{Z}_{>0}$, and $i \in \llbracket 0, M + N \rrbracket$.   
	\begin{enumerate}
		\item We have $\mathsf{L}_1^{(c)} (i) \ge \mathsf{L}_2^{(c)} (i) \ge \cdots $ and $\mathsf{L}_k^{(1)} (i) \ge \mathsf{L}_k^{(2)} (i) \ge \cdots$.  
		\item We have $\Lambda_k^{(c)} (i) - \Lambda_{k+1}^{(c)} (i) = \big( \mathsf{L}_{k}^{(c)} (i) - \mathsf{L}_{k}^{(c+1)} (i) \big) - \big( \mathsf{L}_{k+1}^{(c)} (i) - \mathsf{L}_{k+1}^{(c+1)} (i) \big) = \mathfrak{m}_k \big( \mu^{(c)} (i) \big)$. 
		\item If $i \ge 1$, we have $\mathsf{L}_k^{(c)} (i-1) - \mathsf{L}_k^{(c)} (i) = \mathbbm{1}_{\mathfrak{q}_{\bm{\mu}} (k, i) \ge c}$. 
	\end{enumerate} 

	\end{lem} 
	
	\begin{proof} 
		
		Let us first confirm that the three properties in the lemma hold for $\bm{\mathsf{L}}$. The first follows from the facts that 
		\begin{flalign}
			\label{lkci}
			\begin{aligned}  
			& \mathsf{L}_k^{(c)} (i) - \mathsf{L}_{k+1}^{(c)} (i) = \mathfrak{m}_{\le k}^{\ge c} \big( \mu(i) \big) - \mathfrak{m}_{\le k-1}^{\ge c} \big( \mu(i) \big) = \mathfrak{m}_k^{\ge c} \big( \mu(i) \big) = \displaystyle\sum_{c' = c}^n \mathfrak{m}_k \big( \mu^{(c')} (i) \big) \ge 0; \\
			& \mathsf{L}_k^{(c)} (i) - \mathsf{L}_k^{(c+1)} (i)= \ell_{[c, n]} - \ell_{[c+1, n]} + \mathfrak{m}_{\le k-1}^{\ge c+1} \big( \mu (i) \big) - \mathfrak{m}_{\le k-1}^{\ge c} \big( \mu(i) \big) = \ell_{c} - \mathfrak{m}_{\le k-1} \big( \mu^{(c)} (i) \big) \ge 0,
			\end{aligned} 
		\end{flalign} 
	
		\noindent where in the last bound we used the inequality $\ell_{c} \ge \ell \big(\mu^{(c)} (i) \big) \ge \mathfrak{m}_{\le k-1} \big( \mu^{(c)} (i) \big)$ (by the second property in \Cref{mnmu}). The second property follows from the fact that 
		\begin{flalign*}
			\Lambda_k^{(c)} (i) - \Lambda_{k+1}^{(c)} (i) = \big( \mathsf{L}_{k}^{(c)} & (i) - \mathsf{L}_{k}^{(c+1)} (i) \big) - \big( \mathsf{L}_{k+1}^{(c)} (i) - \mathsf{L}_{k+1}^{(c+1)} (i) \big) \\ 
			& = \Big( \ell_{c} - \mathfrak{m}_{\le k-1} \big( \mu^{(c)} (i) \big) \Big) - \Big( \ell_{c} - \mathfrak{m}_{\le k} \big( \mu^{(c)} (i) \big) \Big) = \mathfrak{m}_k \big( \mu^{(c)} (i) \big) \ge 0,
		\end{flalign*}
		
		\noindent where in the second equality we used the second statement in \eqref{lkci}. The third holds by the equality
		\begin{flalign}
			\label{lkc}  
			\mathsf{L}_k^{(c)} (i-1) - \mathsf{L}_k^{(c)} (i) = \mathfrak{m}_{\le k-1}^{\ge c} \big( \mu(i) \big) - \mathfrak{m}_{\le k-1}^{\ge c} \big( \mu(i-1) \big),
		\end{flalign}
	
		\noindent and the fact that (by arrow conservation) the right side of \eqref{lkc} counts the number of arrows with color at least $c$ in $\mathcal{E}_{\bm{\mu}}$ that horizontally enter the vertex $(k-1, i)$, or equivalently that horizontally exit the vertex $(k, i)$; this is $\mathbbm{1}_{\mathfrak{q}_{\bm{\mu}} (k, i) \ge c}$, by \Cref{mnmurow2}. 
		
		The first and third properties of the lemma verify that each $\bm{\mathsf{L}}^{(c)}$ is a simple line ensemble. Moreover, each $\bm{\Lambda}^{(c)}$ is also a line ensemble, since $\Lambda_k^{(c)} \ge \Lambda_{k+1}^{(c)}$ by the second property in the lemma, and 
		\begin{flalign*} 
			\Lambda_k^{(c)} (i-1) - \Lambda_k^{(c)} (i) & = \mathsf{L}_k^{(c)} (i-1) - \mathsf{L}_k^{(c)} (i) - \big( \mathsf{L}_k^{(c+1)} (i-1) - \mathsf{L}_k^{(c+1)} (i) \big) \\ 
			& =  \mathbbm{1}_{\mathfrak{q}_{\bm{\mu}} (k, i) \ge c - 1} - \mathbbm{1}_{\mathfrak{q}_{\bm{\mu}} (k, i) \ge c} \ge 0,
		\end{flalign*} 
		
		\noindent by the third. This means that $\bm{\mathsf{L}}$ is a simple colored line ensemble. 
	\end{proof}

	\begin{rem} 
		
		As in \Cref{mnmurow2}, we may interpret $\bm{\mu}$ as associated with a colored higher spin path ensemble $\mathcal{E}_{\bm{\mu}} \in \mathfrak{P}_{\fG} (M; \sigma)$ on $\mathbb{Z}_{\le 0} \times \llbracket 1, M + N \rrbracket$. Then $\mathsf{L}_k^{(c)} (i) = \mathfrak{h}_{\ge c}^{\leftarrow} (-k, i)$, where the height function $\mathfrak{h}_{\ge c}^{\leftarrow}$ is with respect to $\mathcal{E}_{\bm{\mu}}$; stated alternatively, $\mathsf{L}_k^{(c)} (i)$ denotes the number of arrows with color at least $c$ that horizontally exit the column $\{ x = -k \}$ strictly above the vertex $(-k, i)$. See \Cref{0lmu} for a depiction when $(n, M, N) = (2, 3, 4)$.  
		
	\end{rem}

	By \Cref{lmu1}, \Cref{lmu}  associates a simple colored line ensemble to a given $(M; \sigma)$-ascending sequence of $n$-compositions. Since the latter are in bijection with colored higher spin path ensembles in $\mathfrak{P}_{\fG} (M; \sigma)$ by \Cref{mnmurow2}, this associates a simple colored line ensemble with any element of $\mathfrak{P}_{\fG} (M; \sigma)$. The following definition is towards the reverse direction; it associates a colored higher spin path ensemble with a simple colored line ensemble $\bm{\mathsf{L}}$.

	\begin{definition} 
		
		\label{le2}
		
		Adopt the notation from \Cref{cl} and assume that $\bm{\mathsf{L}}$ is simple. For any $v = (-k, i) \in \mathbb{Z}_{\le 0} \times \llbracket 1, M + N \rrbracket$, define the arrow configuration $\big( \bm{A}^{\bm{\mathsf{L}}} (v), b^{\bm{\mathsf{L}}} (v); \bm{C}^{\bm{\mathsf{L}}} (v), d^{\bm{\mathsf{L}}} (v) \big)$ as follows. The $n$-tuples $\bm{A}^{\bm{\mathsf{L}}} (v) = \big( A_1^{\bm{\mathsf{L}}} (v), A_2^{\bm{\mathsf{L}}} (v), \ldots , A_n^{\bm{\mathsf{L}}} (v) \big) \in \mathbb{Z}_{\ge 0}^n$ and $\bm{C}^{\bm{\mathsf{L}}} = \big( C_1^{\bm{\mathsf{L}}} (v), C_2^{\bm{\mathsf{L}}} (v), \ldots , C_n^{\bm{\mathsf{L}}} (v) \big) \in \mathbb{Z}_{\ge 0}^n$ are prescribed by setting   
		\begin{flalign*}
			& A_c^{\bm{\mathsf{L}}} (v) = \Lambda_k^{(c)} (i-1) - \Lambda_{k+1}^{(c)} (i-1) = \mathsf{L}_k^{(c)} (i-1) - \mathsf{L}_{k+1}^{(c)} (i-1) - \big( \mathsf{L}_k^{(c+1)} (i-1) - \mathsf{L}_{k+1}^{(c+1)} (i-1) \big); \\ 
			& C_c^{\bm{\mathsf{L}}} (v) = \Lambda_k^{(c)} (i) - \Lambda_{k+1}^{(c)} (i) = \mathsf{L}_k^{(c)} (i) - \mathsf{L}_{k+1}^{(c)} (i) - \big( \mathsf{L}_k^{(c+1)} (i) - \mathsf{L}_{k+1}^{(c+1)} (i) \big),
		\end{flalign*}
	
		\noindent for each $c \in \llbracket 1, n \rrbracket$ (observe that both are nonnegative since $\bm{\Lambda}^{(c)}$ is a line ensemble), and the indices $b^{\bm{\mathsf{L}}} (v), d^{\bm{\mathsf{L}}} (v) \in \llbracket 0, n \rrbracket$ are prescribed by setting 
		\begin{flalign*}
			& b^{\bm{\mathsf{L}}} (v) = \max \big\{ c \in \llbracket 1, n \rrbracket : \mathsf{L}_{k+1}^{(c)} (i-1) - \mathsf{L}_{k+1}^{(c)} (i) = 1 \big\}; \\
			& d^{\bm{\mathsf{L}}} (v) = \max \big\{ c \in \llbracket 1, n \rrbracket : \mathsf{L}_k^{(c)} (i-1) - \mathsf{L}_k^{(c)} (i) = 1 \big\},
		\end{flalign*}
	
		\noindent where the maxima are by definition set to $0$ if such an index $c \in \llbracket 1, n \rrbracket$ does not exist. This assignment of arrow configurations is consistent and satisfies arrow conservation; therefore, it defines a colored higher spin path ensemble $\mathcal{E}^{\bm{\mathsf{L}}}$ associated with the simple colored line ensemble $\bm{\mathsf{L}}$. 
			
		\end{definition} 
	
		The following lemma indicates that the associations from \Cref{lmu} and \Cref{le2} are compatible; we omit its proof, which is a quick verification using the second and third statements of \Cref{lmu1}. 
		
		\begin{lem} 
			
		\label{llmu}

		 If $\mathcal{E}^{\bm{\mathsf{L}}} = \mathcal{E}_{\bm{\mu}}$ for some $(M; \sigma)$-ascending sequence $\bm{\mu}$ of $n$-compositions, then $\bm{\mathsf{L}}$ is associated with $\bm{\mu}$ in the sense of \Cref{lmu}. 
	
		\end{lem}

	\subsection{Properties of Random Colored Line Ensembles} 
	
	\label{ConditionW}
	
	In this section we discuss some properties of colored line ensembles $\bm{\mathsf{L}}_{\bm{\mu}}$ associated with an $(M; \sigma)$-ascending sequence $\bm{\mu}$ of $n$-compositions sampled from the measure $\mathbb{P}_{\fG}^{\sigma}$ (recall \Cref{measurefg}). Let us first assign notation to this law on colored line ensembles.
	
	\begin{definition} 
		
	\label{llmu0} 
	
	Let $\mathbb{P}_{\scL}^{\sigma} = \mathbb{P}_{\scL; n; s; \bm{x}; \bm{y}}^{\sigma}$ denote the law of a simple colored line ensemble $\bm{\mathsf{L}}_{\bm{\mu}}$ associated with a random $(M; \sigma)$-ascending sequence $\bm{\mu}$ of $n$-compositions (as in \Cref{lmu}) sampled from the measure $\mathbb{P}_{\fG; n; s; \bm{x}; \bm{y}}^{\sigma}$. 
	
	\end{definition}

	The following result, which is a quick consequence of \Cref{hfg2}, provides under this setup a matching in law between the top curves of $\bm{\mathsf{L}}$ (under $\mathbb{P}_{\scL}^{\sigma}$) and the height functions for a colored stochastic six-vertex model (recall \Cref{probabilityvertex}). 
	
	\begin{thm} 
		
	\label{lmu2} 
	 
	 Sample a simple colored line ensemble $\bm{\mathsf{L}}$ on $\llbracket 0, M+N \rrbracket$ from the measure $\mathbb{P}_{\scL; n; s; \bm{x}; \bm{y}}^{\sigma}$, and sample a random colored six-vertex ensemble $\mathcal{E}$ under $\mathbb{P}_{\SV; \bm{x}; \bm{y}}^{\sigma}$.  For each $c \in \llbracket 1, n \rrbracket$, define the function $H_c : \llbracket 0, M + N \rrbracket \rightarrow \mathbb{Z}$ by setting  
	 \begin{flalign*} 
	 	H_c (k) = \mathfrak{h}_{\ge c}^{\leftarrow} (M, k), \quad \text{if $k \in \llbracket 0, N \rrbracket$}; \qquad H_c (k) = \mathfrak{h}_{\ge c}^{\leftarrow} (M+N-k,  N), \quad \text{if $k \in \llbracket N, M+N \rrbracket$},
	\end{flalign*} 

	\noindent where $\mathfrak{h}_{\ge c}^{\leftarrow}$ is the height function associated with $\mathcal{E}$. Then, the joint law of $\big( \mathsf{L}_1^{(1)}, \mathsf{L}_1^{(2)}, \ldots , \mathsf{L}_1^{(n)} \big)$ is the same as that of $(H_1, H_2, \ldots,  H_n)$.
			
	\end{thm} 

	\begin{proof}
		
		Since $\mathfrak{h}_{\ge c}^{\leftarrow} (i, j) = \ell_{[c, n]} - \mathfrak{h}_{\ge c}^{\rightarrow} (i, j)$ holds for any integer $c \in \llbracket 1, n \rrbracket$ and vertex $(i, j) \in \big\{ (M, 0), (M, 1), \ldots , (M,N), (M, N-1), \ldots , (0, N) \big\}$ along the northeast boundary of $\llbracket 0, M \rrbracket \times \llbracket 0, N \rrbracket$, this theorem follows from \eqref{kclmu} and \Cref{hfg2}. 
	\end{proof}

	The next theorem explains the effect of conditioning on some of the curves in $\bm{\mathsf{L}}$, if $\bm{\mu}$ is sampled under the $\mathbb{P}_{\scL}^{\sigma}$ measure; see \Cref{lconditionmu} for a depiction. We will use this as a Gibbs property for the line ensemble $\bm{\mathsf{L}}$. In the below, we recall the notion of compatibility for line ensembles from \Cref{lluvij}; the vertex weights $L_{x;s}$ from \Cref{lzdefinition}; the association of a colored line ensemble with an ascending sequence of $n$-compositions from \Cref{lmu}; and the notation from \Cref{le2}.

	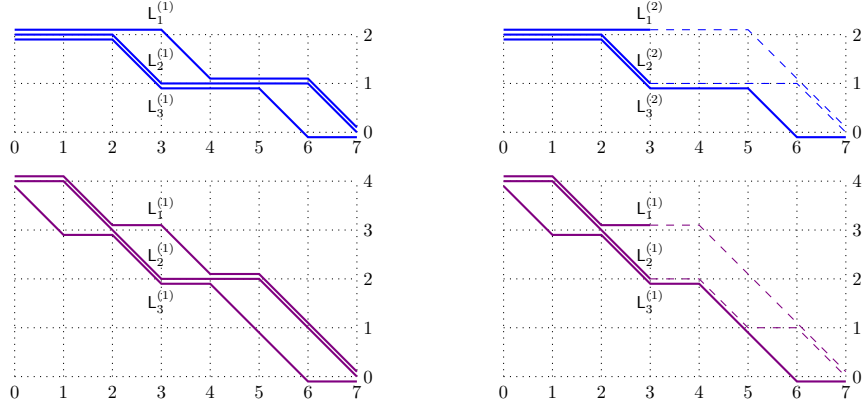
\begin{figure} 
		
		\begin{center}
			\begin{tikzpicture}[
				>=stealth, 
				scale = .65]{	
					
					\draw[dotted] (26, 0) -- (33, 0) node[right, scale = .7]{$0$};
					\draw[dotted] (26, 1) -- (33, 1) node[right, scale = .7]{$1$};
					\draw[dotted] (26, 2) -- (33, 2) node[right, scale = .7]{$2$};
					\draw[dotted] (26, 3) -- (33, 3) node[right, scale = .7]{$3$};
					\draw[dotted] (26, 4) -- (33, 4) node[right, scale = .7]{$4$};
					\draw[dotted] (26, 0) node[below = 1, scale = .7]{$0$} -- (26, 4);
					\draw[dotted] (27, 0) node[below = 1, scale = .7]{$1$} -- (27, 4);
					\draw[dotted] (28, 0) node[below = 1, scale = .7]{$2$} -- (28, 4);
					\draw[dotted] (29, 0) node[below = 1, scale = .7]{$3$} -- (29, 4);
					\draw[dotted] (30, 0) node[below = 1, scale = .7]{$4$} -- (30, 4);
					\draw[dotted] (31, 0) node[below = 1, scale = .7]{$5$} -- (31, 4);
					\draw[dotted] (32, 0) node[below = 1, scale = .7]{$6$} -- (32, 4);
					\draw[dotted] (33, 0) node[below = 1, scale = .7]{$7$} -- (33, 4); 
					
					\draw[dotted] (26, 5) -- (33, 5) node[right, scale = .7]{$0$};
					\draw[dotted] (26, 6) -- (33, 6) node[right, scale = .7]{$1$};
					\draw[dotted] (26, 7) -- (33, 7) node[right, scale = .7]{$2$};
					\draw[dotted] (26, 5) node[below = 1, scale = .7]{$0$} -- (26, 7);
					\draw[dotted] (27, 5) node[below = 1, scale = .7]{$1$} -- (27, 7);
					\draw[dotted] (28, 5) node[below = 1, scale = .7]{$2$} -- (28, 7);
					\draw[dotted] (29, 5) node[below = 1, scale = .7]{$3$} -- (29, 7);
					\draw[dotted] (30, 5) node[below = 1, scale = .7]{$4$} -- (30, 7);
					\draw[dotted] (31, 5) node[below = 1, scale = .7]{$5$} -- (31, 7);
					\draw[dotted] (32, 5) node[below = 1, scale = .7]{$6$} -- (32, 7);
					\draw[dotted] (33, 5) node[below = 1, scale = .7]{$7$} -- (33, 7);
					
					\draw[ thick, blue] (26, 7.1) -- (29, 7.1);
					\draw[dashed, blue] (29, 7.1) -- (31, 7.1) -- (33, 5.1);
					\draw[ thick, blue] (26, 7) -- (28, 7) -- (29, 6);
					\draw[dashed, blue] (29, 6)  -- (32, 6) -- (33, 5);
					\draw[ thick, blue] (26, 6.9) -- (28, 6.9) -- (29, 5.9) -- (31, 5.9) -- (32, 4.9) -- (33, 4.9);
					
					\draw[ thick, violet] (26, 4.1) -- (27, 4.1) -- (28, 3.1) -- (29, 3.1);
					\draw[dashed, violet] (29, 3.1) -- (30, 3.1) -- (33, .1);
					\draw[ thick, violet] (26, 4) -- (27, 4) -- (28, 3) -- (29, 2);
					\draw[dashed, violet] (29, 2) -- (30, 2) -- (31, 1) -- (32, 1) -- (33, 0);
					\draw[ thick, violet] (26, 3.9) -- (27, 2.9) -- (28, 2.9) -- (29, 1.9) -- (30, 1.9) -- (31, .9) -- (32, -.1) -- (33, -.1);
					
					\draw[] (29, 7.1) circle [radius = 0] node[above, scale = .65]{$\mathsf{L}_1^{(2)}$};
					\draw[] (29, 6.5) circle [radius = 0] node[scale = .65]{$\mathsf{L}_2^{(2)}$};
					\draw[] (29, 5.9) circle [radius = 0] node[below, scale = .65]{$\mathsf{L}_3^{(2)}$};
					
					\draw[] (29, 3.1) circle [radius = 0] node[above, scale = .65]{$\mathsf{L}_1^{(1)}$};
					\draw[] (29, 2.5) circle [radius = 0] node[scale = .65]{$\mathsf{L}_2^{(1)}$};
					\draw[] (29, 1.9) circle [radius = 0] node[below, scale = .65]{$\mathsf{L}_3^{(1)}$};

					\draw[dotted] (16, 0) -- (23, 0) node[right, scale = .7]{$0$};
					\draw[dotted] (16, 1) -- (23, 1) node[right, scale = .7]{$1$};
					\draw[dotted] (16, 2) -- (23, 2) node[right, scale = .7]{$2$};
					\draw[dotted] (16, 3) -- (23, 3) node[right, scale = .7]{$3$};
					\draw[dotted] (16, 4) -- (23, 4) node[right, scale = .7]{$4$};
					\draw[dotted] (16, 0) node[below = 1, scale = .7]{$0$} -- (16, 4);
					\draw[dotted] (17, 0) node[below = 1, scale = .7]{$1$} -- (17, 4);
					\draw[dotted] (18, 0) node[below = 1, scale = .7]{$2$} -- (18, 4);
					\draw[dotted] (19, 0) node[below = 1, scale = .7]{$3$} -- (19, 4);
					\draw[dotted] (20, 0) node[below = 1, scale = .7]{$4$} -- (20, 4);
					\draw[dotted] (21, 0) node[below = 1, scale = .7]{$5$} -- (21, 4);
					\draw[dotted] (22, 0) node[below = 1, scale = .7]{$6$} -- (22, 4);
					\draw[dotted] (23, 0) node[below = 1, scale = .7]{$7$} -- (23, 4); 
					
					\draw[dotted] (16, 5) -- (23, 5) node[right, scale = .7]{$0$};
					\draw[dotted] (16, 6) -- (23, 6) node[right, scale = .7]{$1$};
					\draw[dotted] (16, 7) -- (23, 7) node[right, scale = .7]{$2$};
					\draw[dotted] (16, 5) node[below = 1, scale = .7]{$0$} -- (16, 7);
					\draw[dotted] (17, 5) node[below = 1, scale = .7]{$1$} -- (17, 7);
					\draw[dotted] (18, 5) node[below = 1, scale = .7]{$2$} -- (18, 7);
					\draw[dotted] (19, 5) node[below = 1, scale = .7]{$3$} -- (19, 7);
					\draw[dotted] (20, 5) node[below = 1, scale = .7]{$4$} -- (20, 7);
					\draw[dotted] (21, 5) node[below = 1, scale = .7]{$5$} -- (21, 7);
					\draw[dotted] (22, 5) node[below = 1, scale = .7]{$6$} -- (22, 7);
					\draw[dotted] (23, 5) node[below = 1, scale = .7]{$7$} -- (23, 7);
					
					\draw[ thick, blue] (16, 7.1) -- (19, 7.1) -- (20, 6.1) -- (22, 6.1) -- (23, 5.1);
					\draw[ thick, blue] (16, 7) -- (18, 7) -- (19, 6) -- (22, 6) -- (23, 5);
					\draw[ thick, blue] (16, 6.9) -- (18, 6.9) -- (19, 5.9) -- (21, 5.9) -- (22, 4.9) -- (23, 4.9);
					
					\draw[ thick, violet] (16, 4.1) -- (17, 4.1) -- (18, 3.1) -- (19, 3.1) -- (20, 2.1) -- (21, 2.1) -- (22, 1.1) -- (23, .1);
					\draw[ thick, violet] (16, 4) -- (17, 4) -- (18, 3) -- (19, 2) -- (21, 2) -- (22, 1) -- (23, 0);
					\draw[ thick, violet] (16, 3.9) -- (17, 2.9) -- (18, 2.9) -- (19, 1.9) -- (20, 1.9) -- (21, .9) -- (22, -.1) -- (23, -.1);
					
					\draw[] (19, 7.1) circle [radius = 0] node[above, scale = .65]{$\mathsf{L}_1^{(1)}$};
					\draw[] (19, 6.5) circle [radius = 0] node[scale = .65]{$\mathsf{L}_2^{(1)}$};
					\draw[] (19, 5.9) circle [radius = 0] node[below, scale = .65]{$\mathsf{L}_3^{(1)}$};
					
					\draw[] (19, 3.1) circle [radius = 0] node[above, scale = .65]{$\mathsf{L}_1^{(1)}$};
					\draw[] (19, 2.5) circle [radius = 0] node[scale = .65]{$\mathsf{L}_2^{(1)}$};
					\draw[] (19, 1.9) circle [radius = 0] node[below, scale = .65]{$\mathsf{L}_3^{(1)}$};
				}
			\end{tikzpicture}
		\end{center}
		
		\caption{\label{lconditionmu} To the left is a simple colored line ensemble $\bm{\mathsf{L}}$. To the right, we have conditioned on $\mathsf{L}_k^{(c)} (m)$ for $(k, m) \notin \llbracket 1, 2 \rrbracket \times \llbracket 4, 6 \rrbracket$ and resampled the first two curves on $\llbracket 4, 6 \rrbracket$ (shown as dashed).}
		
	\end{figure}

	\begin{thm} 
		
		\label{conditionl}
		
		Sample $\bm{\mathsf{L}} = \bm{\mathsf{L}}_{\bm{\mu}}$ under $\mathbb{P}_{\scL; n; s; \bm{x}; \bm{y}}^{\sigma}$. Fix integers $j > i \ge 0$ and $u, v \in \llbracket 0, M + N \rrbracket$ with $u < v$; set $i_0 = \max \{ i, 1 \}$; and condition on the curves $\mathsf{L}_k^{(c)} (m)$ for all $c \in \llbracket 1, n \rrbracket$ and $(k, m) \in \big( \mathbb{Z}_{> 0} \times \llbracket 0, M+N \rrbracket \big) \setminus \big( \llbracket i+1, j \rrbracket \times \llbracket u, v-1 \rrbracket \big)$. For any simple colored line ensemble $\bm{\mathsf{l}}$ that is $\llbracket i+1, j \rrbracket \times \llbracket u, v-1 \rrbracket$-compatible with $\bm{\mathsf{L}}$, we have
		\begin{flalign}
			\label{llprobability}
			\begin{aligned}
				\mathbb{P} [\bm{\mathsf{L}} = \bm{\mathsf{l}}] & = \mathcal{Z}^{-1} \cdot \displaystyle\prod_{k=i_0}^j \displaystyle\prod_{\substack{m \in \llbracket u, v \rrbracket \\ m \le N}} L_{x_m; s} \big( \bm{A}^{\bm{\mathsf{l}}} (-k, m), b^{\bm{\mathsf{l}}} (-k, m); \bm{C}^{\bm{\mathsf{l}}} (-k, m), d^{\bm{\mathsf{l}}} (-k, m) \big) \\
				& \qquad \times \displaystyle\prod_{k=i_0}^j \displaystyle\prod_{\substack{m \in \llbracket u, v \rrbracket \\ m > N}} L_{y_{m-N}; s} \big( \bm{A}^{\bm{\mathsf{l}}} (-k, m), b^{\bm{\mathsf{l}}} (-k, m); \bm{C}^{\bm{\mathsf{l}}} (-k, m), d^{\bm{\mathsf{l}}} (-k, m) \big).
			\end{aligned}
		\end{flalign}
		
		\noindent Here, the probability on the left side of \eqref{llprobability} is with respect to the conditional law of $\bm{\mathsf{L}}$. Moreover, $\mathcal{Z}$ is a normalizing constant defined so that the sum of the right side of \eqref{llprobability}, over all simple colored line ensembles $\bm{\mathsf{l}}$ that are $\llbracket i+1, j \rrbracket \times \llbracket u, v-1 \rrbracket$-compatible with $\bm{\mathsf{L}}_{\bm{\mu}}$, is equal to $1$.
		
	\end{thm}
	
	\begin{proof}
		
		Letting $\bm{\mu}$ be distributed according to $\mathbb{P}_{\fG; n; s; \bm{x}; \bm{y}}^{\sigma}$, \eqref{lele}, \eqref{fgmunu}, and \eqref{fgprobabilitymu} together imply that
		\begin{flalign}
			\label{zmul} 
			\begin{aligned}
			\mathbb{P} [\bm{\mu}] & = \mathcal{Z}_{\bm{x}; \bm{y}}^{-1} \cdot \displaystyle\prod_{k=1}^{\infty} \displaystyle\prod_{m=1}^N \widehat{L}_{x_m; s} \big( \bm{A}_{\mathcal{E}_{\bm{\mu}}} (-k, m), b_{\mathcal{E}_{\bm{\mu}}} (-k, m); \bm{C}_{\mathcal{E}_{\bm{\mu}}} (-k, m), d_{\mathcal{E}_{\bm{\mu}}} (-k, m) \big) \\
			& \qquad \qquad \qquad \times \displaystyle\prod_{k=1}^{\infty} \displaystyle\prod_{m=N+1}^{M+N} L_{y_{m-N}; s} \big( \bm{A}_{\mathcal{E}_{\bm{\mu}}} (-k, m), b_{\mathcal{E}_{\bm{\mu}}} (-k, m); \bm{C}_{\mathcal{E}_{\bm{\mu}}} (-k, m), d_{\mathcal{E}_{\bm{\mu}}} (-k, m) \big),
			\end{aligned} 
		\end{flalign}
	
		\noindent where $\big( \bm{A}_{\mathcal{E}_{\bm{\mu}}} (-k, m), b_{\mathcal{E}_{\bm{\mu}}} (-k, m); \bm{C}_{\mathcal{E}_{\bm{\mu}}} (-k, m), d_{\mathcal{E}_{\bm{\mu}}} (-k, m) \big)$ denotes the arrow configuration in the colored higher spin path ensemble $\mathcal{E}_{\bm{\mu}}$ (recall \Cref{mnmurow2}) at $(-k, m) \in \mathbb{Z}_{\le 0} \times \llbracket 1, M+N \rrbracket$, and we recall the normalization constant $\mathcal{Z}(\bm{x}; \bm{y})$ from \eqref{zxys}.\footnote{Here, to restrict the products on the right side of \eqref{zmul} to terms with $k \ne 0$, we implicitly used the fact that $L_{z;0} (\bm{A}, b; \bm{A}_b^+, 0) = 1$ for any $\bm{A} \in \mathbb{Z}_{\ge 0}^n$ and $b \in \llbracket 0, n \rrbracket$ (see \eqref{lyi0}), and that $d_{\mathcal{E}_{\bm{\mu}}} (0, m) = 0$ for all $m \in \llbracket 1, M + N \rrbracket$.} Next, by \Cref{llmu}, each arrow configuration $\big( \bm{A}_{\mathcal{E}_{\bm{\mu}}} (-k, m), b_{\mathcal{E}_{\bm{\mu}}} (-k, m); \bm{C}_{\mathcal{E}_{\bm{\mu}}} (-k, m), d_{\mathcal{E}_{\bm{\mu}}} (-k, m) \big)$ appearing in \eqref{zmul} coincides with $\big( \bm{A}^{\bm{\mathsf{L}}} (-k, m), b^{\bm{\mathsf{L}}} (-k, m); \bm{C}^{\bm{\mathsf{L}}} (-k, m), d^{\bm{\mathsf{L}}} (-k, m) \big)$. Together with \eqref{zmul} and the fact that $\widehat{L}_{x;s} = (1-sx) (x-s)^{-1} \cdot L_{x;s}$ (by \Cref{lzdefinition}), this yields 
		\begin{flalign}
			\label{lmu0}
			\begin{aligned} 
			\mathbb{P} [\bm{\mathsf{L}} = \bm{\mathsf{l}}] & = \mathcal{Z}_{\bm{x}; \bm{y}}^{-1} \cdot \displaystyle\prod_{k=1}^{\infty} \displaystyle\prod_{m=1}^N \displaystyle\frac{1-sx_m}{x_m - s} \cdot L_{x_m;s} \big( \bm{A}^{\bm{\mathsf{l}}} (-k, m), b^{\bm{\mathsf{l}}} (-k, m); \bm{C}^{\bm{\mathsf{l}}} (-k, m), d^{\bm{\mathsf{l}}} (-k, m) \big) \\
			& \qquad \qquad \qquad \times \displaystyle\prod_{k=1}^{\infty} \displaystyle\prod_{m=N+1}^{M+N} L_{y_{m-N}; s} \big( \bm{A}^{\bm{\mathsf{l}}} (-k, m), b^{\bm{\mathsf{l}}} (-k, m); \bm{C}^{\bm{\mathsf{l}}} (-k, m), d^{\bm{\mathsf{l}}} (-k, m) \big),
			\end{aligned}
		\end{flalign}
		
		\noindent where on the left side $\bm{\mathsf{L}}$ is sampled under the measure $\mathbb{P}_{\fG; n; s; \bm{x}; \bm{y}}^{\sigma}$, without any conditioning yet.
		
		Now, as in the statement of the theorem, we condition on the curves $ \mathsf{L}_k^{(c)} (m)$ for $c \in \llbracket 1, n \rrbracket$ and $(k, m) \notin \llbracket i+1, j \rrbracket \times \llbracket u, v-1 \rrbracket$. By \Cref{le2}, this amounts to conditioning on the restriction of $\mathcal{E}^{\bm{\mathsf{L}}}$ to the complement of $\llbracket -j,-i \rrbracket \times \llbracket u, v \rrbracket$. Hence, the factors on the right side of \eqref{lmu0} corresponding to $(k, m) \notin \llbracket i, j \rrbracket \times \llbracket u, v \rrbracket$ are deterministic and can thus be incorporated into the normalization constant, which gives (recalling $i_0 = \max \{ i, 1 \}$)
		\begin{flalign*}
			\mathbb{P} \big[ \bm{\mathsf{L}} & = \bm{\mathsf{l}} \big] = \widetilde{\mathcal{Z}}^{-1} \cdot\displaystyle\prod_{k=i_0}^j \displaystyle\prod_{\substack{m \in \llbracket u, v \rrbracket \\ m \le N}} \displaystyle\frac{1-sx_m}{x_m - s} \cdot L_{x_m; s} \big( \bm{A}^{\bm{\mathsf{l}}} (-k, m), b^{\bm{\mathsf{l}}} (-k, m); \bm{C}^{\bm{\mathsf{l}}} (-k, m), d^{\bm{\mathsf{l}}} (-k, m) \big) \\
			& \qquad \qquad \qquad \times \displaystyle\prod_{k=i_0}^j \displaystyle\prod_{\substack{m \in \llbracket u, v \rrbracket \\ m > N}}  L_{y_{m-N};s} \big( \bm{A}^{\bm{\mathsf{l}}} (-k, m), b^{\bm{\mathsf{l}}} (-k, m); \bm{C}^{\bm{\mathsf{l}}} (-k, m), d^{\bm{\mathsf{l}}} (-k, m) \big), 
		\end{flalign*} 
	
		\noindent for some normalization constant $\widetilde{\mathcal{Z}}$. Similarly incorporating the product $\prod_{m=1}^N (1-sx_m) (x_m - s)^{-1}$ into the normalization constant gives \eqref{llprobability}. 
	\end{proof}

	\subsection{Color Merging} 
	
	\label{ColorL} 
	
	In this section we describe several \emph{color merging} properties, which enable us to obtain a system with $n-1$ colors by merging two colors (say $1$ and $2$) in a corresponding system with $n$ colors. Some of the proofs in this section will only be outlined, since analogous color merging phenomena have been discussed extensively in the literature already; see, for example, \cite[Section 2.4]{CSVMST}, \cite[Proposition 4.11]{SVMP}, and \cite[Sections 2.3 and 5.2]{CFVMSF}. Throughout this section, we will define several functions that all have the effect of merging colors $1$ and $2$. They will each be denoted by $\vartheta$, which should not cause confusion since they act on different spaces. 
	
	First define $\vartheta : \llbracket 0, n \rrbracket \rightarrow \llbracket 0, n-1 \rrbracket$ by setting 
	\begin{flalign}
		\label{121}
		\vartheta (0) = 0, \quad \vartheta (1) = 1, \quad \vartheta(2) = 1, \quad \text{and} \quad \vartheta (i) = i - 1, \quad \text{for each $i \in \llbracket 3, n \rrbracket$}.
	\end{flalign}

	\noindent Also define its action on $n$-tuples of integers $\vartheta : \mathbb{Z}_{\ge 0}^n \rightarrow \mathbb{Z}_{\ge 0}^{n-1}$ by setting
	\begin{flalign}
		\label{12i1}
		\vartheta (\bm{I}) = (I_1 + I_2, I_3, \ldots , I_n), \qquad \text{for any $\bm{I} = (I_1, I_2, \ldots , I_n) \in \mathbb{Z}_{\ge 0}^n$}.
	\end{flalign}

	With this notation, it is quickly verified\footnote{This also follows from applying fusion (see \Cref{Fusion}) below to \cite[Proposition 4.3]{SVMP}.} that the weights from \Cref{lzdefinition} satisfy the following color-merging property for $n \ge 2$. Fix integers $\breve{b}, \breve{d} \in \llbracket 0, n - 1 \rrbracket$ and $(n-1)$-tuples $\breve{\bm{A}}, \breve{\bm{C}} \in \mathbb{Z}_{\ge 0}^{n-1}$, as well as an integer $b \in \llbracket 0, n \rrbracket$ and an $n$-tuple $\bm{A} \in \mathbb{Z}_{\ge 0}^n$, such that $\vartheta (b) = \breve{b}$ and $\vartheta (\bm{A}) = \breve{\bm{A}}$. Then, 
		\begin{flalign}
			\label{2lmerge} 
			\displaystyle\sum_{\substack{\bm{C} \in \mathbb{Z}_{\ge 0}^n \\ \vartheta(\bm{C}) = \breve{\bm{C}}}} \displaystyle\sum_{\substack{d \in \llbracket 0, n \rrbracket \\ \vartheta(d) = \breve{d}}} L_{x;s}^{(n)} (\bm{A}, b; \bm{C}, d) = L_{x;s}^{(n-1)} ( \breve{\bm{A}}, \breve{b}; \breve{\bm{C}}, \breve{d});
		\end{flalign}
		
		\noindent the analogous statement also holds for the $\widehat{L}$-weights. The equality \eqref{2lmerge} indicates that identifying colors $1$ and $2$ in an $n$-color $L$-weight yields an $(n-1)$-color one.
		
		The next lemma states that merging colors $1$ and $2$ in either the $n$-color $f$ or $G$ yields the same function, but on $n-1$ colors. 
		In the below, we define the action of $\vartheta$ on $n$-compositions $\vartheta : \Comp_n \rightarrow \Comp_{n-1}$ by for any $\bm{\mu} = \big( \mu^{(1)} \mid \mu^{(2)} \mid \cdots \mid \mu^{(n)} \big) \in \Comp_n$, setting 
		\begin{flalign}
			\label{mu121}
			\vartheta (\bm{\mu}) = \big( \mu^{(1)} \cup \mu^{(2)} \mid \mu^{(3)} \mid \cdots \mid \mu^{(n)} \big),
		\end{flalign}
		
		\noindent where $\mu^{(1)} \cup \mu^{(2)}$ is the signature obtained by taking the disjoint union of $\mu^{(1)}$ and $\mu^{(2)}$, and sorting its parts in non-increasing order. Moreover, given any function $\varsigma : \llbracket 1, N \rrbracket \rightarrow \llbracket 1, n \rrbracket$ further define the function $\vartheta (\varsigma) : \llbracket 1, N \rrbracket \rightarrow \llbracket 1, n-1 \rrbracket$ by setting 
		\begin{flalign}
			\label{1210} 
			\vartheta(\varsigma) (i) = \vartheta \big( \varsigma(i) \big), \quad \text{for each $i \in \llbracket 1, N \rrbracket$, and set} \quad \breve{\sigma} = \vartheta(\sigma),
		\end{flalign}
		
		\noindent where we recall $\sigma : \llbracket 1, N \rrbracket \rightarrow \llbracket 1, n \rrbracket$ that was fixed in the beginning of \Cref{L0}. 	
		
		\begin{lem} 
		
		\label{mergemunu}
	
		Fix an integer $k \ge 1$; $(n-1)$-compositions $\breve{\mu}, \breve{\nu}, \breve{\kappa} \in \Comp_{n-1}$; and an $n$-composition $\nu \in \Comp_n$, such that $\vartheta (\nu) = \breve{\nu}$. We have 
		\begin{flalign*}
			\displaystyle\sum_{\substack{\mu \in \Comp_n \\ \vartheta (\mu) = \breve{\mu}}} f_{\mu/\nu; s}^{\sigma} (\bm{x}) = f_{\breve{\mu} / \breve{\nu}; s}^{\breve{\sigma}} (\bm{x}); \qquad \displaystyle\sum_{\substack{\kappa \in \Comp_n \\ \vartheta (\kappa) = \breve{\kappa}}} G_{\nu / \kappa; s} (\bm{y}) = G_{\breve{\nu} / \breve{\kappa}; s} (\bm{y}).
		\end{flalign*} 
	
		\end{lem} 
	
		\begin{proof}[Proof (Outline)] 
			
			Recalling the definitions \eqref{fgmunu} of $f$ and $G$ as partition functions (under the $\widehat{L}$-weights and $L$-weights, respectively), this follows quickly from inductively applying the color merging \eqref{2lmerge} of vertex weights. See also \cite[Proposition 4.11]{SVMP} (which can in fact be seen to directly imply \Cref{mergemunu} using fusion, defined in \Cref{Fusion}), for a very similar argument; we omit further details.
		\end{proof} 
		
	The next lemma describes the effect of merging colors $1$ and $2$ in a random $n$-composition sampled from the measure $\mathbb{P}_{\fG; n; s; \bm{x}; \bm{y}}^{\sigma}$ (recall \Cref{measurefg}). In the following, given a sequence of $n$-compositions $\bm{\mu} = \big( \mu(0), \mu(1), \ldots , \mu(M+N) \big)$, we set 
	\begin{flalign}
		\label{12mu0} 
		\vartheta (\bm{\mu}) = \Big( \vartheta \big( \mu(0) \big), \vartheta \big( \mu(1) \big), \ldots , \vartheta \big( \mu(M+N) \big) \Big), 
	\end{flalign}

	\noindent which is an $(M; \breve{\sigma})$-ascending sequence of $(n-1)$-compositions.
	\begin{lem}
	
	\label{nn1mu} 
	
	 If an $(M; \sigma)$-ascending sequence of $n$-compositions $\bm{\mu}$ is sampled from $\mathbb{P}_{\fG; n; s; \bm{x}; \bm{y}}^{\sigma}$, then the $(M; \breve{\sigma})$-ascending sequence $\vartheta(\bm{\mu})$ of $(n-1)$-compositions has law $\mathbb{P}_{\fG; n-1; s; \bm{x}; \bm{y}}^{\breve{\sigma}}$. 

	\end{lem} 
	
	\begin{proof}
		
		Fix an $(M; \breve{\sigma})$-ascending sequence $\breve{\bm{\mu}} = \big( \breve{\mu}(0), \breve{\mu}(1), \ldots,  \breve{\mu} (M+N) \big)$ of $(n-1)$-compositions. For any integers $j \in \llbracket 1, N \rrbracket$ and $i \in \llbracket N+1, M + N \rrbracket$, and $n$-compositions $\mu(j-1), \mu(i-1) \in \Comp_n$ such that $\vartheta \big( \mu(j-1) \big) = \breve{\mu}(j-1)$ and $\vartheta \big( \mu(i-1) \big) = \breve{\mu}(i-1)$, we have by \Cref{mergemunu} that
		\begin{flalign*}
		& \displaystyle\sum_{\substack{\mu(j) \in \Comp_n \\ \vartheta(\mu(j)) = \breve{\mu}(j)}} f_{\mu(j) / \mu(j-1); s}^{\sigma(j)} (x_j) = f_{\breve{\mu}(j) / \breve{\mu}(j-1); s}^{\breve{\sigma}(j)} (x_j); \\ 
		& \displaystyle\sum_{\substack{\mu(i) \in \Comp_n \\ \vartheta (\mu(i)) = \breve{\mu}(i)}} G_{\mu(i-1) / \mu(i); s} (y_{i-N}) = G_{\breve{\mu}(i-1) / \breve{\mu} (i); s} (y_{i-N}). 
		\end{flalign*}
		
		\noindent Inductively applying these equalities, and using \Cref{measurefg}, it follows that 
		\begin{flalign*}
			\mathbb{P}_{\fG; n; s; \bm{x}; \bm{y}}^{\sigma} \big[ \vartheta(\bm{\mu}) =  \breve{\bm{\mu}} \big] & = \mathcal{Z}_{\bm{x}; \bm{y}}^{-1} \cdot \displaystyle\sum_{\vartheta (\bm{\mu}) = \breve{\bm{\mu}}} \displaystyle\prod_{j=1}^N f_{\mu(j) / \mu(j-1); s}^{\sigma(j)} (x_j) \displaystyle\prod_{i=N+1}^{M+N} G_{\mu(i-1)/\mu(i); s} (y_{i-N}) \\
			& = \mathcal{Z}_{\bm{x}; \bm{y}}^{-1} \cdot \displaystyle\prod_{j=1}^N f_{\breve{\mu}(j) / \breve{\mu} (j-1); s}^{\breve{\sigma}(j)}(x_j) \displaystyle\prod_{i=N+1}^{M+N} G_{\breve{\mu}(i-1) / \breve{\mu}(i); s} (y_{i-N}) \\
			& = \mathbb{P}_{\fG; n-1; s; \bm{x}; \bm{y}}^{\breve{\sigma}} [\breve{\bm{\mu}}],
		\end{flalign*}
		
		\noindent which yields the lemma. 	
	\end{proof} 

	We conclude this section with the following proposition indicating that the marginal joint law of all line ensembles but the second\footnote{This could be replaced by the $k$-th, for any $k \in \llbracket 2, n \rrbracket$, through a similar proof.} in a colored line ensemble $\bm{\mathsf{L}}$ sampled under $\mathbb{P}_{\scL; n; \bm{x}; \bm{y}}^{\sigma}$ (recall \Cref{llmu0}) is equal to that of a colored line ensemble (with $n-1$ colors) sampled under $\mathbb{P}_{\scL; n-1; \bm{x}; \bm{y}}^{\breve{\sigma}}$ (where we recall $\breve{\sigma} = \vartheta(\sigma)$ from \eqref{1210}).  

	\begin{prop}
		
		\label{lmerge}
		
		Sample $\bm{\mathsf{L}} = \big( \bm{\mathsf{L}}^{(1)}, \bm{\mathsf{L}}^{(2)}, \ldots , \bm{\mathsf{L}}^{(n)} \big)$ under $\mathbb{P}_{\scL; n; s; \bm{x}; \bm{y}}^{\sigma}$. Then the joint law of the colored line ensemble $\breve{\bm{\mathsf{L}}} = \big( \bm{\mathsf{L}}^{(1)}, \bm{\mathsf{L}}^{(3)}, \bm{\mathsf{L}}^{(4)}, \ldots , \bm{\mathsf{L}}^{(n)} \big)$ (with $n-1$ colors) is given by $\mathbb{P}_{\scL; n-1; s; \bm{x}; \bm{y}}^{\breve{\sigma}}$. 
			
	\end{prop}
	
	\begin{proof}
		
		Sample $\bm{\mu}$ under $\mathbb{P}_{\fG; n; s; \bm{x}; \bm{y}}^{\sigma}$ (recall \Cref{measurefg}). By \Cref{llmu0}, we may identify $\bm{\mathsf{L}}$ as the colored line ensemble $\bm{\mathsf{L}}_{\bm{\mu}}$ associated with $\bm{\mu}$. By \Cref{lmu}, the $n-1$ line ensembles $\breve{\bm{\mathsf{L}}} = \big( \bm{\mathsf{L}}^{(1)}, \bm{\mathsf{L}}^{(3)}, \bm{\mathsf{L}}^{(4)}, \ldots , \bm{\mathsf{L}}^{(n)} \big)$ in $\bm{\mathsf{L}} = \bm{\mathsf{L}}_{\bm{\mu}}$ constitute the colored line ensemble associated with $\vartheta (\bm{\mu})$, which has law $\mathbb{P}_{\fG; n-1; s; \bm{x}; \bm{y}}^{\breve{\sigma}}$ by \Cref{nn1mu}. Hence, again by \Cref{llmu0}, $\breve{\bm{\mathsf{L}}}$ has law  $\mathbb{P}_{\scL; n-1; s; \bm{x}; \bm{y}}^{\breve{\sigma}}$, thereby establishing the proposition.
	\end{proof}

	\begin{rem} 
		
		\label{12interval} 
		
		Throughout this section, we have merged the colors $1$ and $2$. It is more generally possible to merge several (disjoint) intervals of colors; see \cite[Sections 4.4 and 4.9]{SVMP} and \cite[Sections 2.3 and 5.2]{CFVMSF} for similar setups. This would correspond in \Cref{lmerge} to omitting line ensembles $\bm{\mathsf{L}}^{(i)}$, for $i \in \llbracket 2, n \rrbracket$ arbitrary (depending on the corresponding merged color intervals), in $\bm{\mathsf{L}}$. 
		
	\end{rem}

	\section{Examples} 
	
	\label{ExampleL} 
	
	In this section we examine the colored line ensembles introduced in \Cref{L0} under two specializations; in both, we take $s=0$ and refer to \Cref{tz} for a depiction of the $L_{x;0}$ weights. We investigate the case $n=1$ in \Cref{LExamplen1} and the case $q=0$ in \Cref{Lq0Example} (where we also prove \Cref{ensemblevertexq0}); we also restrict to the \emph{homogeneous} cases of both, when the entries in $\bm{x}$ (and in $\bm{y}$) are all equal. Throughout this section, we fix integers $n, M, N \ge 1$; a composition $\bm{\ell} = (\ell_1, \ell_2, \ldots , \ell_n)$ of $N$; a function $\sigma : \llbracket 1, N \rrbracket \rightarrow \llbracket 1, n \rrbracket$, such that for each $i \in \llbracket 1, n \rrbracket$ we have $\ell_i = \# \big\{ \sigma^{-1} (i) \big\}$; and real numbers $x, y \in (0, 1)$ with $y < x$. We also assume that $q \in (0, 1)$ and define the sequences $\bm{x} = (x, x, \ldots , x)$ and $\bm{y} = (y, y, \ldots , y)$, where $x$ and $y$ appear with multiplicities $N$ and $M$, respectively. We also recall the notation from \Cref{Ensemble} throughout.

 	\begin{figure}
 		\begin{center}
 			\begin{tikzpicture}[
 				>=stealth,
 				scale = .95
 				]	
 				\draw[-, black] (-7.5, 3.1) -- (7.5, 3.1);
 				\draw[-, black] (-7.5, -2.1) -- (7.5, -2.1);
 				\draw[-, black] (-7.5, -1.1) -- (7.5, -1.1);
 				\draw[-, black] (-7.5, -.4) -- (7.5, -.4);
 				\draw[-, black] (-7.5, 2.4) -- (7.5, 2.4);
 				\draw[-, black] (-7.5, -2.1) -- (-7.5, 3.1);
 				\draw[-, black] (7.5, -2.1) -- (7.5, 3.1);
 				\draw[-, black] (-5, -2.1) -- (-5, 2.4);
 				\draw[-, black] (5, -2.1) -- (5, 3.1);
 				\draw[-, black] (-2.5, -2.1) -- (-2.5, 2.4);
 				\draw[-, black] (2.5, -2.1) -- (2.5, 2.4);
 				\draw[-, black] (0, -2.1) -- (0, 3.1);
 				\draw[->, thick, blue] (-6.3, .1) -- (-6.3, 1.9); 
 				\draw[->, thick, green] (-6.2, .1) -- (-6.2, 1.9); 
 				\draw[->, thick, blue] (-3.8, .1) -- (-3.8, 1) -- (-2.85, 1);
 				\draw[->, thick, green] (-3.7, .1) -- (-3.7, 1.9);
 				\draw[->, thick, blue] (-1.35, .1) -- (-1.35, 1.9);
 				\draw[->, thick, green] (-1.25, .1) -- (-1.25, 1.9);
 				\draw[->, thick,  orange] (-2.15, 1.1) -- (-1.15, 1.1) -- (-1.15, 1.9);
 				
 				\draw[->, thick, red] (.35, 1) -- (1.15, 1) -- (1.15, 1.9);
 				\draw[->, thick, blue] (1.25, .1) -- (1.25, 1.9);
 				\draw[->, thick, green] (1.35, .1) -- (1.35, 1.1) -- (2.15, 1.1);
 				\draw[->, thick, blue] (3.65, .1) -- (3.65, 1) -- (4.65, 1);
 				\draw[->, thick, green] (3.75, .1) -- (3.75, 1.9);
 				\draw[->, thick, orange] (2.85, 1.1) -- (3.85, 1.1) -- (3.85, 1.9); 
 				\draw[->, thick, red] (5.35, 1) -- (7.15, 1); 
 				\draw[->, thick, blue] (6.2, .1) -- (6.2, 1.9);
 				\draw[->, thick, green] (6.3, .1) -- (6.3, 1.9); 
 				\filldraw[fill=gray!50!white, draw=black] (-2.85, 1) circle [radius=0] node [black, right = -1, scale = .7] {$i$};
 				\filldraw[fill=gray!50!white, draw=black] (2.15, 1) circle [radius=0] node [black, right = -1, scale = .7] {$j$};
 				\filldraw[fill=gray!50!white, draw=black] (4.65, 1) circle [radius=0] node [black, right = -1, scale = .7] {$i$};
 				\filldraw[fill=gray!50!white, draw=black] (7.15, 1) circle [radius=0] node [black, right = -1, scale = .7] {$i$};
 				\filldraw[fill=gray!50!white, draw=black] (5.35, 1) circle [radius=0] node [black, left = -1, scale = .7] {$i$};
 				\filldraw[fill=gray!50!white, draw=black] (2.85, 1) circle [radius=0] node [black, left = -1, scale = .7] {$j$};
 				\filldraw[fill=gray!50!white, draw=black] (.35, 1) circle [radius=0] node [black, left = -1, scale = .7] {$i$};
 				\filldraw[fill=gray!50!white, draw=black] (-2.15, 1) circle [radius=0] node [black, left = -1, scale = .7] {$i$};
 				\filldraw[fill=gray!50!white, draw=black] (-6.25, 1.9) circle [radius=0] node [black, above = -1, scale = .7] {$\bm{A}$};
 				\filldraw[fill=gray!50!white, draw=black] (-3.75, 1.9) circle [radius=0] node [black, above = -1, scale = .7] {$\bm{A}_i^-$};
 				\filldraw[fill=gray!50!white, draw=black] (-1.25, 1.9) circle [radius=0] node [black, above = -1, scale = .7] {$\bm{A}_i^+$};
 				\filldraw[fill=gray!50!white, draw=black] (1.25, 1.9) circle [radius=0] node [black, above = -1, scale = .65] {$\bm{A}_{ij}^{+-}$};
 				\filldraw[fill=gray!50!white, draw=black] (3.75, 1.9) circle [radius=0] node [black, above = -1, scale = .65] {$\bm{A}_{ji}^{+-}$};
 				\filldraw[fill=gray!50!white, draw=black] (6.25, 1.9) circle [radius=0] node [black, above = -1, scale = .7] {$\bm{A}$};
 				\filldraw[fill=gray!50!white, draw=black] (-6.25, .1) circle [radius=0] node [black, below = -1, scale = .7] {$\bm{A}$};
 				\filldraw[fill=gray!50!white, draw=black] (-3.75, .1) circle [radius=0] node [black, below = -1, scale = .7] {$\bm{A}$};
 				\filldraw[fill=gray!50!white, draw=black] (-1.25, .1) circle [radius=0] node [black, below = -1, scale = .7] {$\bm{A}$};
 				\filldraw[fill=gray!50!white, draw=black] (1.25, .1) circle [radius=0] node [black, below = -1, scale = .7] {$\bm{A}$};
 				\filldraw[fill=gray!50!white, draw=black] (3.75, .1) circle [radius=0] node [black, below = -1, scale = .7] {$\bm{A}$};
 				\filldraw[fill=gray!50!white, draw=black] (6.25, .1) circle [radius=0] node [black, below = -1, scale = .7] {$\bm{A}$};
 				\filldraw[fill=gray!50!white, draw=black] (-3.75, 2.75) circle [radius=0] node [black] {$1 \le i \le n$};
 				\filldraw[fill=gray!50!white, draw=black] (2.5, 2.75) circle [radius=0] node [black] {$1 \le i < j \le n$}; 
 				\filldraw[fill=gray!50!white, draw=black] (6.25, 2.75) circle [radius=0] node [black] {};
 				\filldraw[fill=gray!50!white, draw=black] (-6.25, -.75) circle [radius=0] node [black, scale = .8] {$(\bm{A}, 0; \bm{A}, 0)$};
 				\filldraw[fill=gray!50!white, draw=black] (-3.75, -.75) circle [radius=0] node [black, scale = .8] {$\big( \bm{A}, 0; \bm{A}_i^-, i \big)$};
 				\filldraw[fill=gray!50!white, draw=black] (-1.25, -.75) circle [radius=0] node [black, scale = .8] {$\big( \bm{A}, i; \bm{A}_i^+, 0 \big)$};
 				\filldraw[fill=gray!50!white, draw=black] (1.25, -.75) circle [radius=0] node [black, scale = .8] {$\big( \bm{A}, i; \bm{A}_{ij}^{+-}, j \big)$};
 				\filldraw[fill=gray!50!white, draw=black] (3.75, -.75) circle [radius=0] node [black, scale = .8] {$\big( \bm{A}, j; \bm{A}_{ji}^{+-}, i \big)$};
 				\filldraw[fill=gray!50!white, draw=black] (6.25, -.75) circle [radius=0] node [black, scale = .8] {$(\bm{A}, i; \bm{A}, i)$};
 				\filldraw[fill=gray!50!white, draw=black] (-6.25, -1.6) circle [radius=0] node [black, scale = .9] {$1$};
 				\filldraw[fill=gray!50!white, draw=black] (-3.75, -1.6) circle [radius=0] node [black, scale = .8] {$(1 - q^{A_i}) q^{A_{[i+1, n]}} x$};
 				\filldraw[fill=gray!50!white, draw=black] (-1.25, -1.6) circle [radius=0] node [black, scale = .9] {$1$};
 				\filldraw[fill=gray!50!white, draw=black] (1.25, -1.6) circle [radius=0] node [black, scale = .8] {$(1 - q^{A_j}) q^{A_{[j+1,n]}} x$};
 				\filldraw[fill=gray!50!white, draw=black] (3.75, -1.6) circle [radius=0] node [black, scale = .9] {$0$};
 				\filldraw[fill=gray!50!white, draw=black] (6.25, -1.6) circle [radius=0] node [black, scale = .9] {$q^{A_{[i+1,n]}} x$};
 			\end{tikzpicture}
 		\end{center}
 		
 		\caption{\label{tz} The $L_{x;0}$ weights are depicted above.} 
 	\end{figure}

	\subsection{The Case $n=1$} 
	
	\label{LExamplen1}
	
	Throughout this section, we set $n=1$, $s = 0$, and sample $\bm{\mathsf{L}} = \bm{\mathsf{L}}^{(1)}$ under the probability measure $\mathbb{P}_{\scL; 1}^{\sigma} = \mathbb{P}_{\scL; 1; 0; \bm{x}; \bm{y}}^{\sigma}$. We will analyze how $\bm{\mathsf{L}}$, and in particular its Gibbs property, behaves. Since $n=1$, we omit the superscripts indexing the color $1$ from the notation in what follows, writing $\bm{\mathsf{L}} = (\mathsf{L}_1, \mathsf{L}_2, \ldots )$. Define for any integers $(k, i) \in \mathbb{Z}_{>0} \times \llbracket 0, M+N \rrbracket$ and line ensemble $\bm{\mathsf{l}} = (\mathsf{l}_1, \mathsf{l}_2, \ldots )$ on $\llbracket 0, M+N \rrbracket$ the quantity
	\begin{flalign}
		\label{deltai} 
			\Delta_{k; \bm{\mathsf{l}}} (i) = \mathsf{l}_k (i) - \mathsf{l}_{k+1} (i),
	\end{flalign}

	\noindent where by definition $\mathsf{L}_0 (i) = \mathsf{l}_0 (i) = \infty$ for each $i \in \llbracket 1, M+N \rrbracket$.
	
	Given this notation, we have the following proposition explaining the Gibbs property (\Cref{conditionl}) in this setting when we resample one curve (say that $i$-th one $\mathsf{L}_i$) of $\bm{\mathsf{L}}$. 
	
	\begin{prop} 
		
		\label{n1conditionl} 
		
		Adopt the above notation and assumptions. Let $j \ge i \ge 1$ and $u, v \in \llbracket 0, M + N \rrbracket$ be integers with $v \ge u$, and condition on the curves $\mathsf{L}_k (m)$ for $(k, m) \notin \llbracket i, j \rrbracket \times \llbracket u, v \rrbracket$; set $u_0 = \max \{ u-1, N \}$ and $v_0 = \min \{ v+1, N \}$. For any simple line ensemble $\bm{\mathsf{l}} = (\mathsf{l}_1, \mathsf{l}_2, \ldots )$ that is $\{ i \} \times \llbracket u, v \rrbracket$-compatible with $\bm{\mathsf{L}}$, we have 
		\begin{flalign} 
			\label{llv0u0}
			\begin{aligned} 
			\mathbb{P} [\bm{\mathsf{L}} = \bm{\mathsf{l}}] = \mathcal{Z}^{-1}  \cdot \displaystyle\prod_{k=i}^j x^{ - \mathsf{l}_k (v_0)} y^{\mathsf{l}_k (u_0)} & \displaystyle\prod_{k = i-1}^j \displaystyle\prod_{m = u}^{v+1} \big( 1 - q^{\Delta_k (m-1)} \cdot \mathbbm{1}_{\Delta_k m-1) - \Delta_k (m) = 1} \big),
			\end{aligned} 
		\end{flalign}
		
		\noindent for some normalization constant $\mathcal{Z} > 0$, where we have abbreviated $\Delta_k = \Delta_{k;\bm{\mathsf{l}}}$. 
		
	\end{prop}

	\begin{proof} 
		
		To make use of \Cref{conditionl}, we must first understand how the quantities on the right side of \eqref{llprobability} behave. By \Cref{le2} (and the fact that $n=1$), we have for any vertex $(-k, m) \in \mathbb{Z}_{\le -1} \times \llbracket 1, M + N \rrbracket$ that 
	\begin{flalign}
		\label{abcdn1} 
		\begin{aligned}
		& A^{\bm{\mathsf{l}}} (-k, m) = \mathsf{l}_k (m-1) - \mathsf{l}_{k+1} (m-1) = \Delta_k (m-1); \qquad b^{\bm{\mathsf{l}}} (-k, m) = \mathsf{l}_{k+1} (m-1) - \mathsf{l}_{k+1} (m); \\
		& C^{\bm{\mathsf{l}}} (-k, m) = \mathsf{l}_k (m) - \mathsf{l}_{k+1} (m) = \Delta_k (m); \qquad \qquad \qquad \quad d^{\bm{\mathsf{l}}} (-k, m) = \mathsf{l}_k (m-1) - \mathsf{l}_k (m),
		\end{aligned} 
	\end{flalign}
	
	\noindent where we omit the subscript indexing the color $1$ (as $n=1$). Further observe (see \Cref{tz}) that, if $A+b = C + d$, then
	\begin{flalign*} 
		L_{x;0} (A, b; C, d) = x^d \cdot L_{1;0} (A, b; C, d); \qquad L_{1;0} (A, b; C, d) = 1 - q^A \cdot \mathbbm{1}_{d-b = 1},
	\end{flalign*} 

	\noindent and (by \eqref{abcdn1} and \eqref{deltai}) that
	\begin{flalign*} 
		d^{\bm{\mathsf{l}}} (-k, m) - b^{\bm{\mathsf{l}}} (-k, m) = \mathsf{l}_k (m-1) - \mathsf{l}_k (m) - \big( \mathsf{l}_{k+1} (m-1) - \mathsf{l}_{k+1} (m) \big) = \Delta_k (m-1) - \Delta_k (m).
	\end{flalign*} 

	\noindent Inserting these into \Cref{conditionl} (with the $(i,j; u, v)$ there equal to $(i-1, j; u, v+1)$ here), yields 
	\begin{flalign}
		\label{lldeltaprobability}
		\begin{aligned}
		\mathbb{P} [\bm{\mathsf{L}} = \bm{\mathsf{l}}] = \mathcal{Z}^{-1} & \cdot  \displaystyle\prod_{k=i-1}^j \displaystyle\prod_{\substack{m \in \llbracket u, v+1 \rrbracket \\ m \le N}} x^{\mathsf{l}_k (m-1) - \mathsf{l}_k (m)} \displaystyle\prod_{\substack{m \in \llbracket u, v+1 \rrbracket \\ m > N}} y^{\mathsf{l}_k (m-1) - \mathsf{l}_k (m)} \\
		&  \times \displaystyle\prod_{k=i-1}^j \displaystyle\prod_{m = u}^{v+1} \big( 1 - q^{\Delta_k (m-1)} \cdot \mathbbm{1}_{\Delta_k (m-1) - \Delta_k (m) = 1} \big),
		\end{aligned} 
	\end{flalign}

	\noindent for some normalization constant $\mathcal{Z} > 0$. 
	
	Next observe since $u_0 = \max \{ u - 1, N \}$ and $v_0 = \min \{ v+1, N \}$ that 
	\begin{flalign*} 
		& \displaystyle\sum_{m = u}^{\min \{ v+1, N \}} \big( \mathsf{l}_k (m-1) - \mathsf{l}_k (m) \big) = \big( \mathsf{l}_k (u-1) - \mathsf{l}_k  (v_0) \big) \cdot \mathbbm{1}_{u \le N}; \\ 
		& \displaystyle\sum_{m=\max \{ N + 1, u \}}^{v+1} \big(\mathsf{l}_k (m-1) - \mathsf{l}_k (m) \big) = \big( \mathsf{l}_k (u_0) - \mathsf{l}_k (v+1) \big) \cdot \mathbbm{1}_{v \ge N}.
	\end{flalign*} 

	\noindent As we conditioned on $\mathsf{L}_{i-1} (m) = \mathsf{l}_{i-1} (m)$ for all $m$, and on $\mathsf{L}_k (m) = \mathsf{l}_k (m)$ and $\mathsf{L}_k (m) = \mathsf{l}_k (m)$ for all $k$ and $m \notin \llbracket u, v \rrbracket$, we may incorporate the factors $x^{(\mathsf{l}_{i-1} (u-1) - \mathsf{l}_{i-1} (v_0)) \cdot \mathbbm{1}_{u \le N}} \cdot y^{(\mathsf{l}_{i-1} (u_0) - \mathsf{l}_{i-1} (v+1)) \cdot \mathbbm{1}_{v \ge N}}$ and $x^{\mathsf{l}_k (u-1) \cdot \mathbbm{1}_{u \le N} + \mathsf{l}_k (v_0) \cdot \mathbbm{1}_{u > N}} \cdot y^{- \mathsf{l}_i (v+1) \cdot \mathbbm{1}_{v \ge N} - \mathsf{l}_i (u_0) \cdot \mathbbm{1}_{v < N}}$ for $k \in \llbracket i, j \rrbracket$ appearing in the right side of \eqref{lldeltaprobability} into the normalization constant $\mathcal{Z}$ (where we used the facts that if $u > N$ then $v > N$ and so $v_0 = N \le u-1$, and that if $v < N$ then $u < N$ and so $u_0 = N \ge v+1$). This gives  
	\begin{flalign*}
		\mathbb{P} [\bm{\mathsf{L}} = \bm{\mathsf{l}}] = \mathcal{Z}^{-1} & \cdot \displaystyle\prod_{k=i}^j x^{ - \mathsf{l}_k (v_0)} y^{\mathsf{l}_k (u_0)} \displaystyle\prod_{k=i-1}^j \displaystyle\prod_{m = u}^{v+1} \big( 1 - q^{\Delta_k (m-1)} \cdot \mathbbm{1}_{\Delta_k (m-1) - \Delta_k (m) = 1} \big),
	\end{flalign*}

	\noindent after altering $\mathcal{Z}$ if necessary, which yields the proposition.
	\end{proof}

	The Gibbs property described by \Cref{n1conditionl} coincides the Hall--Littlewood Gibbs property introduced in \cite[Definition 3.4]{TFSVM}.\footnote{To properly observe the match, one must read first our line ensemble in reverse (from right to left), and then apply a gauge transformation to its weights that does not affect its Gibbs property (namely, multiply them by $(q; q)_{\Delta_i (m-1)}^{-1} (q; q)_{\Delta_{i-1} (m-1)}^{-1} (q; q)_{\Delta_i (m)} (q; q)_{\Delta_{i-1} (m)}$).} Moreover, setting $n = 1$ in \Cref{lmu2} yields that the law of $\mathsf{L}_1$ has the same law as the height function of the (uncolored) stochastic six-vertex model introduced in \cite{SVMRS}; this had been shown earlier in \cite[Theorem 5.5]{SSVMP}. More generally, it can be shown (although we do not do so here) that the $n=1$ case of $\bm{\mathsf{L}}$ coincides (up to an affine transformation) with the Hall--Littlewood Gibbsian line ensemble introduced in \cite[Section 3.2]{TFSVM}. This is to be expected, since taking the $n=1$ case of our arguments in the preceding sections would essentially yield the content described in \cite[Section 5.7]{SSVMP} and \cite[Proposition 3.9]{TFSVM}.

	\begin{rem}
	
	\label{xyn1} 
	
	Observe that \eqref{llv0u0} only depends on $x$ and $y$ through their ratio $x^{-1} y$. Indeed, if $u-1 \ge N$ or $N \ge v + 1$, then the factor $x^{-\mathsf{l}_i (v_0)} y^{\mathsf{l}_i (u_0)} = x^{-\mathsf{l}_i (v+1)} y^{\mathsf{l}_i (u-1)}$ on the right side of \eqref{llv0u0} is fixed by the conditioning and can therefore be incorporated into the normalization constant. Otherwise, $N \in \llbracket u, v \rrbracket$, and so the factor $x^{-\mathsf{l}_i (v_0)} y^{\mathsf{l}_i (u_0)} = (x^{-1} y)^{\mathsf{l}_i (N)}$ only depends on $x^{-1} y$.
	\end{rem}

	\subsection{The Case $q=0$}
	
	\label{Lq0Example}
	
	Throughout this section, we set $q=0$, $s = 0$, and sample the colored line ensemble $\bm{\mathsf{L}} = \big( \bm{\mathsf{L}}^{(1)}, \bm{\mathsf{L}}^{(2)}, \ldots , \bm{\mathsf{L}}^{(n)} \big)$ from the measure $\mathbb{P}_{\scL}^{\sigma} = \mathbb{P}_{\scL; n; 0; \bm{x}; \bm{y}}^{\sigma}$. The following proposition explains the Gibbs property (\Cref{conditionl}) for $\bm{\mathsf{L}}$. Below, we restrict to the scenario when $\llbracket u, v \rrbracket$ does not contain $N$, since the Gibbs property will be simplest to state in this situation (as, analogously to \Cref{xyn1}, it will not depend on $x$ or $y$).
	
	\begin{prop} 
		
		\label{q0l} 
		
	Adopt the above notation and assumptions, and let $j \ge i \ge 1$ and $u, v \in \llbracket 0, M+N \rrbracket$ be integers such that $u \le v$ and $N \notin \llbracket u, v \rrbracket$. Then, the following two statements hold. 
	\begin{enumerate}
		\item For any $(c, k, m) \in \llbracket 1, n \rrbracket \times \mathbb{Z}_{>0} \times \llbracket 1, M+N \rrbracket$ such that $\mathsf{L}_k^{(c+1)} (m) > \mathsf{L}_{k+1}^{(c+1)} (m)$, we almost surely have
	\begin{flalign}
		\label{l23}
		\mathsf{L}_k^{(c)} (m-1) - \mathsf{L}_k^{(c)} (m) = \mathsf{L}_k^{(c+1)} (m-1) - \mathsf{L}_k^{(c+1)} (m),
	\end{flalign} 
	
		\item Condition on the curves $\mathsf{L}_k^{(c)} (m)$ for all $c \in \llbracket 1, n \rrbracket$ and $(k, m) \notin \llbracket i, j \rrbracket \times \llbracket u, v \rrbracket$. Then the law of $\bm{\mathsf{L}}$ is uniform over all simple colored line ensembles $\bm{\mathsf{l}} = \big( \bm{\mathsf{l}}^{(1)}, \bm{\mathsf{l}}^{(2)}, \ldots , \bm{\mathsf{l}}^{(n)} \big)$ that are $\llbracket i, j \rrbracket \times \llbracket u, v \rrbracket$-compatible with $\bm{\mathsf{L}}$ such that, for any $(c, k, m) \in \llbracket 1, n \rrbracket \times \mathbb{Z}_{>0} \times \llbracket 1, M+N \rrbracket$ with $\mathsf{l}_k^{(c+1)} (m) > \mathsf{l}_{k+1}^{(c+1)} (m)$, we have
		\begin{flalign}
			\label{lkc0} 
			\mathsf{l}_k^{(c)} (m-1) - \mathsf{l}_k^{(c)} (m) = \mathsf{l}_k^{(c+1)} (m-1) - \mathsf{l}_k^{(c+1)} (m).
		\end{flalign}
	\end{enumerate}

	\end{prop} 

	To prove \Cref{q0l}, we require the next lemma that explains why $\mathsf{l}_k^{(c+1)} (m) > \mathsf{l}_{k+1}^{(c+1)} (m)$ should imply \eqref{lkc0}. 
	
	\begin{lem} 
		
		\label{lweight0e}
	
		Let $\bm{\mathsf{l}} = \big( \bm{\mathsf{l}}^{(1)}, \bm{\mathsf{l}}^{(2)}, \ldots , \bm{\mathsf{l}}^{(n)} \big)$ denote a simple colored line ensemble on $\llbracket 0, M + N \rrbracket$; $\mathcal{E}^{\bm{\mathsf{l}}}$ denote the associated higher spin path ensemble; and $\big( A^{\bm{\mathsf{l}}} (-k, m), b^{\bm{\mathsf{l}}} (-k, m); \bm{C}^{\bm{\mathsf{l}}} (-k, m), d^{\bm{\mathsf{l}}} (-k, m) \big)$ denote the arrow configuration in $\mathcal{E}^{\bm{\mathsf{l}}}$ at any $(-k, m) \in \mathbb{Z}_{\ge 0} \times \llbracket 1, M + N \rrbracket$. Then, for any $(-k, m) \in \mathbb{Z}_{\ge 0} \times \llbracket 1, M + N \rrbracket$, we have that $L_{1;0} \big( \bm{A}^{\bm{\mathsf{l}}} (-k, m), b^{\bm{\mathsf{l}}} (-k, m); \bm{C}^{\bm{\mathsf{l}}} (-k, m), d^{\bm{\mathsf{l}}} (-k, m) \big) = 1$ if and only if, for any $c \in \llbracket 1, n \rrbracket$ with $\mathsf{l}_k^{(c+1)} (m) > \mathsf{l}_{k+1}^{(c+1)} (m)$, \eqref{lkc0} holds. Otherwise, we have that the weight $L_{1;0} \big( \bm{A}^{\bm{\mathsf{l}}} (-k, m), b^{\bm{\mathsf{l}}} (-k, m); \bm{C}^{\bm{\mathsf{l}}} (-k, m), d^{\bm{\mathsf{l}}} (-k, m) \big) = 0$.
	
	\end{lem}

	\begin{proof} 
		
		By \Cref{le2}, we have for any $c \in \llbracket 1, n \rrbracket$ that 
	\begin{flalign}
		\label{cd} 
		\begin{aligned} 
		& C_c^{\bm{\mathsf{l}}} (-k, m) = \mathsf{l}_k^{(c)} (m) - \mathsf{l}_{k+1}^{(c)} (m) - \big( \mathsf{l}_k^{(c+1)} (m) - \mathsf{l}_{k+1}^{(c+1)} (m) \big);  \\
		& d^{\bm{\mathsf{l}}} = \max \big\{ d \in \llbracket 1, n \rrbracket : \mathsf{l}_k^{(d)} (m-1) - \mathsf{l}_k^{(d)} (m) = 1 \big\},
		\end{aligned} 
	\end{flalign}
		
	\noindent with $d^{\bm{\mathsf{l}}} = 0$ if no such $d \in \llbracket 1, n \rrbracket$ exists. Further observe (see \Cref{tz}) that for $q = 0$, if $\bm{A} + \bm{e}_b = \bm{C} + \bm{e}_d$, then
	\begin{flalign*}
		 L_{1;0} (\bm{A}, b; \bm{C}, d) =  1  -  \mathbbm{1}_{d > 0} \cdot \mathbbm{1}_{C_{[d+1,n]} > 0}.
	\end{flalign*} 
	
	\noindent By \eqref{cd}, we have for any $d \in \llbracket 1, n \rrbracket$ that 
	\begin{flalign*} 
		C_{[d+1, n]}^{\bm{\mathsf{l}}} (-k, m) = \mathsf{l}_k^{(d+1)} (m) - \mathsf{l}_{k+1}^{(d+1)} (m).
	\end{flalign*} 

	\noindent Also by \eqref{cd}, we have $d^{\bm{\mathsf{l}}} (-k, m) = d \ge 1$ if and only if $\mathsf{l}_k^{(d)} (m-1) - \mathsf{l}_k^{(d)} (m) - \big( \mathsf{l}_k^{(d+1)} (m-1) - \mathsf{l}_k^{(d+1)} (m) \big) > 0$, that is, if and only if \eqref{lkc0} holds. 
	
	Hence, $L_{1;0} \big( \bm{A}^{\bm{\mathsf{l}}} (-k, m), b^{\bm{\mathsf{l}}} (-k, m); \bm{C}^{\bm{\mathsf{l}}} (-k, m), d^{\bm{\mathsf{l}}} (-k, m) \big) = 1$ if and only if, for any $c \in \llbracket 1, n \rrbracket$ such that $\mathsf{l}_k^{(c+1)} (m) > \mathsf{l}_{k+1}^{(c+1)} (m)$, \eqref{lkc0} holds. Otherwise, this weight is equal to $0$. 
	\end{proof} 

	Now we can establish \Cref{q0l}. 
	
	\begin{proof}[Proof of \Cref{q0l}]
		
		Observe (see \Cref{tz}) for any $\bm{A}, \bm{C} \in \mathbb{Z}_{\ge 0}^n$ and $b, d \in \llbracket 0, n \rrbracket$ with $\bm{A} + \bm{e}_b = \bm{C} + \bm{e}_d$ that
		\begin{flalign}
			\label{l2x} 
		\begin{aligned} 
		& L_{x;0} (\bm{A}, b; \bm{C}, d) = x^{\mathbbm{1}_{d > 0}} \cdot L_{1;0} (\bm{A}, b; \bm{C}, d); \qquad L_{1;0} (\bm{A}, b; \bm{C}, d) =  1  -  \mathbbm{1}_{d > 0} \cdot \mathbbm{1}_{C_{[d+1,n]} > 0}.
		 \end{aligned} 
		\end{flalign} 
	
		\noindent Thus, since $\bm{\mathsf{L}}$ was sampled according to the measure $\mathbb{P}_{\scL}^{\sigma}$, \eqref{lele}, \eqref{fgmunu}, \eqref{fgprobabilitymu}, and \Cref{llmu} together imply that 
		\begin{flalign*} 
			\mathbb{P}_{\scL}^{\sigma} [\bm{\mathsf{L}}] = \mathcal{Z}_{\bm{x}; \bm{y}}^{-1} & \cdot \displaystyle\prod_{k=1}^{\infty}  \displaystyle\prod_{m=1}^N \widehat{L}_{x_m;0} \big( \bm{A}^{\bm{\mathsf{L}}} (-k, m), b^{\bm{\mathsf{L}}} (-k, m); \bm{C}^{\bm{\mathsf{L}}} (-k, m), d^{\bm{\mathsf{L}}} (-k, m) \big) \\
			& \qquad \times \displaystyle\prod_{k=1}^{\infty} \displaystyle\prod_{m=N+1}^{M+N} L_{y_{m-N}; 0} \big( \bm{A}^{\bm{\mathsf{L}}} (-k, m), b^{\bm{\mathsf{L}}} (-k, m); \bm{C}^{\bm{\mathsf{L}}} (-k, m), d^{\bm{\mathsf{L}}} (-k, m) \big),
		\end{flalign*} 
	
		\noindent which by the first statement of \eqref{l2x} (and the fact that $x, y \ne 0$) is nonzero if and only if $L_{1;0} (\bm{A}^{\bm{\mathsf{L}}} (-k, m), b^{\bm{\mathsf{L}}} (-k,m); \bm{C}^{\bm{\mathsf{L}}} (-k, m), d^{\bm{\mathsf{L}}} (-k,m) \big) \ne 0$ for all $(-k, m) \in \mathbb{Z}_{\le 0} \times \llbracket 1, M + N \rrbracket$. By \Cref{lweight0e}, this is true if and only if, for any $(c, k, m) \in \llbracket 1, n \rrbracket \times \mathbb{Z}_{>0} \times \llbracket 1, M + N \rrbracket$ such that $\mathsf{L}_k^{(c+1)} (m) > \mathsf{L}_{k+1}^{(c+1)} (m)$, \eqref{l23} holds. This confirms the first statement of the proposition.	
		
		To verify the second, we first apply \Cref{conditionl} (with the $(i, j; u, v)$ there equal to $(i-1, j; u, v+1)$ here) and \eqref{l2x} to deduce for some normalization constant $\mathcal{Z} > 0$ that 
		\begin{flalign} 
			\label{llq0} 
			\begin{aligned} 
			\mathbb{P}_{\scL}^{\sigma} [\bm{\mathsf{L}} = \bm{\mathsf{l}}] = \mathcal{Z}^{-1} & \cdot \displaystyle\prod_{k=i-1}^j  \displaystyle\prod_{\substack{m \in \llbracket u, v+1 \rrbracket \\ m \le N}} x^{\mathbbm{1}_{d^{\bm{\mathsf{l}}} (-k, m) > 0}} \cdot L_{1;0} \big( \bm{A}^{\bm{\mathsf{l}}} (-k, m), b^{\bm{\mathsf{l}}} (-k, m); \bm{C}^{\bm{\mathsf{l}}} (-k, m), d^{\bm{\mathsf{l}}} (-k, m) \big) \\
			& \times \displaystyle\prod_{k=i-1}^j  \displaystyle\prod_{\substack{m \in \llbracket u, v+1 \rrbracket \\ m > N}} y^{\mathbbm{1}_{d^{\bm{\mathsf{l}}} (-k, m) > 0}} \cdot L_{1;0} \big( \bm{A}^{\bm{\mathsf{l}}} (-k, m), b^{\bm{\mathsf{l}}} (-k, m); \bm{C}^{\bm{\mathsf{l}}} (-k, m), d^{\bm{\mathsf{l}}} (-k, m) \big).
			\end{aligned}
		\end{flalign} 
	
		\noindent Moreover, by \Cref{le2} (with the fact that $\bm{\mathsf{l}}^{(c)}$ is simple), we have $\mathbbm{1}_{d^{\bm{\mathsf{l}}} (-k, m) > 0} = \mathsf{l}_{k+1}^{(1)} (m-1) - \mathsf{l}_{k+1}^{(1)} (m)$. Hence, setting $u_0 = \max \{ u-1, N \}$ and $v_0 = \min \{ v+1, N \}$, we have 
		\begin{flalign*}
			\displaystyle\prod_{\substack{m \in \llbracket u, v + 1 \rrbracket \\ m \le N}} x^{\mathbbm{1}_{d^{\bm{\mathsf{l}}} (-k, m) > 0}} = x^{(\mathsf{l}_k^{(1)} (u-1) - \mathsf{l}_k^{(1)} (v_0)) \cdot \mathbbm{1}_{u \le N}}; \quad \displaystyle\prod_{\substack{m \in \llbracket u, v + 1 \rrbracket \\ m > N}} y^{\mathbbm{1}_{d^{\bm{\mathsf{l}}} (-k, m) > 0}} = y^{(\mathsf{l}_k^{(1)} (u_0) - \mathsf{l}_k^{(1)} (v+1) ) \cdot \mathbbm{1}_{v \ge N}}.
		\end{flalign*}
	
		\noindent Since $u-1, u_0, v_0, v+1 \notin \llbracket u, v \rrbracket$ (as $N \notin \llbracket u, v \rrbracket$), the above factors are fixed by the conditioning. Thus, on the right side of \eqref{llq0}, we may incorporate them into the normalization constant $\mathcal{Z}$ to obtain (after altering $\mathcal{Z}$ if necessary) 
		\begin{flalign*} 
			\mathbb{P}_{\scL}^{\sigma} [\bm{\mathsf{L}} = \bm{\mathsf{l}}] = \mathcal{Z}^{-1} & \cdot \displaystyle\prod_{k=i-1}^j  \displaystyle\prod_{\substack{m \in \llbracket u, v+1 \rrbracket \\ m \le N}} L_{1;0} \big( \bm{A}^{\bm{\mathsf{l}}} (-k, m), b^{\bm{\mathsf{l}}} (-k, m); \bm{C}^{\bm{\mathsf{l}}} (-k, m), d^{\bm{\mathsf{l}}} (-k, m) \big) \\
			& \times \displaystyle\prod_{k=i-1}^j  \displaystyle\prod_{\substack{m \in \llbracket u, v+1 \rrbracket \\ m > N}} L_{1;0} \big( \bm{A}^{\bm{\mathsf{l}}} (-k, m), b^{\bm{\mathsf{l}}} (-k, m); \bm{C}^{\bm{\mathsf{l}}} (-k, m), d^{\bm{\mathsf{l}}} (-k, m) \big).
		\end{flalign*} 
		
		\noindent By \Cref{lweight0e}, the above product is equal to $1$ if and only if, for any $(c, k, m) \in \llbracket 1, n \rrbracket \times \mathbb{Z}_{>0} \times \llbracket 1, M + N \rrbracket$ with $\mathsf{l}_k^{(c+1)} (m) > \mathsf{l}_{k+1}^{(c+1)} (m)$, \eqref{lkc0} holds; otherwise it is equal to $0$. This confirms the second part of the proposition.
	\end{proof} 
	
	Now we can quickly establish \Cref{ensemblevertexq0}. 
	
	\begin{proof}[Proof of \Cref{ensemblevertexq0}]
		
		The first part of this theorem follows from the $q=0$ case of \Cref{lmu2}; the second and third follow from \Cref{q0l}.
	\end{proof}

	\section{Fusion} 
	
	\label{Fusion} 
	
	Until now, we have used colored six-vertex or higher spin path ensembles; they only allowed at most one arrow to occupy any horizontal edge. In this section we remove that restriction using the fusion procedure originally introduced in \cite{ERT}, and describe the counterparts to the statements from \Cref{ZFunction} and \Cref{Line} when horizontal edges may accommodate more than one arrow. Since such ideas have been used repeatedly throughout the literature, and since many statements in this section are similar to those in \Cref{ZFunction} and \Cref{Line}, we will sometimes only outline (or omit) the proofs of the below results. Throughout this section, we fix an integer $n \ge 1$.
	
	\subsection{Stochastic Fused Vertex Models} 
	
	\label{FusedPath}

	A \emph{colored fused arrow configuration} is a quadruple $(\bm{A}, \bm{B}; \bm{C}, \bm{D})$ of elements in $\mathbb{Z}_{\ge 0}^n$. We view this as an assignment of directed up-right colored arrows to a vertex $v \in \mathbb{Z}^2$, in which a (horizontal or vertical) edge can accommodate arbitrarily many arrows. In particular, for each $k \in \llbracket 1, n \rrbracket$, the numbers $A_k$, $B_k$, $C_k$, and $D_k$ denote the numbers of arrows of color $k$ vertically entering, horizontally entering, vertically exiting, and horizontally exiting $v$, respectively. 
	
	As in \Cref{ModelVertex} and \Cref{FunctionsZ}, a \emph{colored fused path ensemble} on a domain $\mathcal{D} \subseteq \mathbb{Z}^2$ is a consistent assignment of a colored fused arrow configuration $\big( \bm{A}(v), \bm{B} (v); \bm{C} (v), \bm{D} (v) \big)$ to each vertex $v \in \mathcal{D}$. Observe that the arrows in a colored fused path ensemble on $\mathcal{D}$ form colored up-right directed paths (that can share horizontal and vertical edges) connecting vertices of $\mathcal{D}$. Associated with a colored fused path ensemble $\mathcal{E}$ on some domain $\mathcal{D} \subseteq \mathbb{Z}^2$ are height functions $\mathfrak{h}_c^{\rightarrow}, \mathfrak{h}_c^{\leftarrow}, \mathfrak{h}_{\ge c}^{\rightarrow}, \mathfrak{h}_{\le c}^{\leftarrow}: \mathbb{Z}^2 \rightarrow \mathbb{Z}$, which are defined in the same way as in \Cref{FunctionsZ}.

	The probability measure on random colored fused path ensembles on $\mathbb{Z}_{> 0}^2$ that we next define depends on four sequences of complex parameters $\bm{x} = (x_1, x_2, \ldots )$; $\bm{r} = (r_1, r_2, \ldots )$; $\bm{y} = (y_1, y_2, \ldots )$; and $\bm{s} = (s_1, s_2, \ldots )$. We view $x_j$ and $r_j$ as associated with the $j$-th row, so they are called row rapidity and spin parameters, respectively; we view $y_i$ and $s_i$ as associated with the $i$-th column, so they are called column rapidity and fusion parameters, respectively. The specific forms of these probability measures are expressed through certain weights $U_{x_j / y_i; r_j, s_i} (\bm{A}, \bm{B}; \bm{C}, \bm{D})$ associated with each vertex $v = (i, j) \in \mathbb{Z}_{> 0}^2$ (analogously to \Cref{rzabcd}). We use the following ones, due to \cite[Equation (7.10)]{STRR} (though our notation is closer to \cite{CSVMST}\footnote{The $(r, s)$ in \Cref{wzabcd} are the $(q^{-\mathsf{L}/2}, q^{-\mathsf{M}/2})$ from \cite{CSVMST}.}), that satisfy the Yang--Baxter equation (see \Cref{equationrank} below). In the following, we define the function $\varphi : \mathbb{R}^n \times \mathbb{R}^n \rightarrow \mathbb{R}$ by setting
	\begin{flalign*}
		\varphi (\bm{X}, \bm{Y}) = \displaystyle\sum_{1 \le i < j \le n} X_i Y_j, \qquad \text{for any $\bm{X}, \bm{Y} \in \mathbb{R}^n$}. 
	\end{flalign*}

	\begin{definition}[{\cite[Equation (C.1.4)]{CSVMST}}] 
		
		\label{wzabcd}
		
		Fix $r, s, z \in \mathbb{C}$ and $\bm{A}, \bm{B}, \bm{C}, \bm{D} \in \mathbb{Z}_{\ge 0}^n$, and let $t = |\bm{T}|$ for each index $T \in \{ A, B, C, D \}$. Define the vertex weight
		\begin{flalign}
			\label{weightu} 
			\begin{aligned}
				& U_{z; r, s} (\bm{A}, \bm{B}; \bm{C}, \bm{D}) \\	
				& \quad = z^{d-b} r^{2(c-a)} s^{2d} q^{\varphi (\bm{D}, \bm{C})} \displaystyle\frac{(r^2; q)_d}{(r^2; q)_b} \displaystyle\sum_{p = 0}^{\min \{ b, c \}} (r^{-2} z)^p \displaystyle\frac{(r^{-2} s^2 z; q)_{c-p} (r^2 z^{-1}; q)_p (z; q)_{b-p}}{(s^2 z; q)_{c+d-p}}  \\
				& \qquad \qquad \quad \times \mathbbm{1}_{\bm{A} + \bm{B} = \bm{C} + \bm{D}} \cdot \displaystyle\sum_{\substack{\bm{P} \le \bm{B}, \bm{C} \\ |\bm{P}| = p}} q^{\varphi (\bm{B} - \bm{D} - \bm{P}, \bm{P})} \displaystyle\prod_{i=1}^n \displaystyle\frac{(q; q)_{C_i + D_i - P_i}}{(q; q)_{D_i} (q; q)_{C_i - P_i}} \displaystyle\frac{(q; q)_{B_i}}{(q; q)_{P_i} (q; q)_{B_i - P_i}},
			\end{aligned} 
		\end{flalign}
		
		\noindent where the sum is over all $\bm{P} = (P_1, \ldots , P_n) \in \mathbb{Z}_{\ge 0}^n$ such that $|\bm{P}| = p$ and $0 \le P_i \le \min \{ B_i, C_i \}$ for each $i \in \llbracket 1, n \rrbracket$. 
		
	\end{definition} 
	
	\begin{rem} 
		
		\label{stochasticu} 
		
		As in \Cref{1rsum}, the $U$-weights from \Cref{wzabcd} are stochastic in the sense that $\sum_{\bm{C}, \bm{D} \in \mathbb{Z}_{\ge 0}^n} U_{z; r, s} (\bm{A}, \bm{B}; \bm{C}, \bm{D}) = 1$, for each $r, s, z \in \mathbb{C}$ and $\bm{A}, \bm{B} \in \mathbb{Z}_{\ge 0}^n$; see \cite[Proposition C.1.2]{CSVMST}.
		
	\end{rem}

	We can now use these $U$-weights to describe (similarly to \Cref{ModelVertex}) how to sample a random colored fused path ensemble on $\mathbb{Z}_{> 0}^2$. We first define probability measures $\mathbb{P}_{\FV}^n$ on the set of colored fused ensembles whose vertices are all contained in the triangles $\mathbb{T}_N = \big\{ (x, y) \in \mathbb{Z}_{> 0}^2: x + y \le N \big\}$. The initial measure $\mathbb{P}_{\FV}^0$ is supported by the empty ensemble. 
	
	For each integer $N \ge 1$, we will define $\mathbb{P}_{\FV}^{N + 1}$ from $\mathbb{P}_{\FV}^N$ by first using $\mathbb{P}_{\FV}^N$ to sample a colored fused path ensemble $\mathcal{E}_N$ on $\mathbb{T}_N$. This yields arrow configurations for all vertices in the triangle $\mathbb{T}_{N - 1}$. To extend this to a colored six-vertex ensemble on $\mathbb{T}_{N + 1}$, we must prescribe arrow configurations to all vertices on the diagonal $\mathbb{D}_N = \big\{ (x, y) \in \mathbb{Z}_{> 0}^2: x + y = N \big\}$. Since $\mathcal{E}_N$ and the initial data prescribe the first two coordinates $(\bm{A}, \bm{B})$ of the arrow configuration to each vertex in $\mathbb{D}_N$, it remains to explain how to assign the second two coordinates $(\bm{C}, \bm{D})$ of the arrow configuration at any vertex $(i, j) \in \mathbb{D}_N$, given its first two coordinates $(\bm{A}, \bm{B})$. This is done according to the transition probability $\mathbb{P}_{\FV}^N \big[ (\bm{C}, \bm{D})  \big| (\bm{A}, \bm{B}) \big] = U_{y_i / x_j; r_j, s_i} (\bm{A}, \bm{B}; \bm{C}, \bm{D})$. We assume that the parameters $(\bm{x}; \bm{y}; \bm{r}; \bm{s}; q)$ are chosen so that these probabilities are all nonnegative; the stochasticity of the $U$-weights (\Cref{stochasticu}) then ensures that $\mathbb{P}_{\FV}^N$ indeed defines a probability measure.
	
	Choosing $(\bm{C}, \bm{D})$ according to the above transition probabilities yields a random colored fused path ensemble $\mathcal{E}_{N + 1}$, now defined on $\mathbb{T}_{N + 1}$; the probability distribution of $\mathcal{E}_{N + 1}$ is then denoted by $\mathbb{P}_{\FV}^{N + 1}$. Taking the limit as $n$ tends to $\infty$ yields a probability measure on colored fused path ensembles on the quadrant. We refer to it as the \emph{colored stochastic fused vertex model}; observe that it may also be sampled on any rectangle $\mathcal{D} \subset \mathbb{Z}^2$ in the same way as it was above on the quadrant.

	\subsection{Yang--Baxter Equation for  Fused Weights} 
	
	\label{UWW} 
	
	In this section we state the Yang--Baxter equation for the $U$-weights from \Cref{wzabcd}, with another family of weights given by the $W$ and $\widehat{W}$ ones below (which serve as analogs of the $L$ and $\widehat{L}$ weights from \Cref{lzdefinition}).
	
	\begin{definition} 
		
		\label{wabcdw} 
		
		Adopting the notation of \Cref{wzabcd}, define the weight 
		\begin{flalign*}
			W_{x; r, s} (\bm{A}, \bm{B}; \bm{C}, \bm{D}) = (-s)^{-d} \cdot U_{x/s; r, s} (\bm{A}, \bm{B}; \bm{C}, \bm{D}),
		\end{flalign*} 
	
		\noindent so that 
		\begin{flalign}
			\label{wweight0} 
			\begin{aligned}
			& W_{x; r, s} (\bm{A}, \bm{B}; \bm{C}, \bm{D}) \\
			& \quad = (-1)^d x^{d-b} r^{2(c-a)} s^b q^{\varphi (\bm{D}, \bm{C})} \displaystyle\frac{(r^2; q)_d}{(r^2; q)_b} \displaystyle\sum_{p = 0}^{\min \{ b, c \}} (r^2 s)^{-p} x^p \displaystyle\frac{(r^{-2} s x; q)_{c-p} (r^2 s x^{-1}; q)_p (s^{-1} x; q)_{b-p}}{(s x; q)_{c+d-p}}  \\
			& \qquad \qquad \quad \times \mathbbm{1}_{\bm{A} + \bm{B} = \bm{C} + \bm{D}} \cdot \displaystyle\sum_{\substack{\bm{P} \le \bm{B}, \bm{C} \\ |\bm{P}| = p}} q^{\varphi (\bm{B} - \bm{D} - \bm{P}, \bm{P})} \displaystyle\prod_{i=1}^n \displaystyle\frac{(q; q)_{C_i + D_i - P_i}}{(q; q)_{D_i} (q; q)_{C_i - P_i}} \displaystyle\frac{(q; q)_{B_i}}{(q; q)_{P_i} (q; q)_{B_i - P_i}}.
			\end{aligned} 
		\end{flalign}
		
		\noindent We also set $W_{x;r, 0} (\bm{A}, \bm{B}; \bm{C}, \bm{D}) = \lim_{s \rightarrow 0} W_{x; r, s} (\bm{A}, \bm{B}, \bm{C}, \bm{D})$, where the existence (and an explicit form) of this limit follows from \Cref{wweight} below. Additionally, if there exists an integer $\mathrm{R} \ge 1$ for which $r = q^{-\mathrm{R} /2}$, then set 
		\begin{flalign}
			\label{wzr2}
			\widehat{W}_{x; r, s} (\bm{A}, \bm{B}; \bm{C}, \bm{D}) = W_{x; r, s} (\bm{A}, \bm{B}; \bm{C}, \bm{D}) \cdot (-s)^{-\mathrm{R}} \displaystyle\frac{(sx; q)_{\mathrm{R}}}{(s^{-1} x; q)_{\mathrm{R}}}.
		\end{flalign}
		
	\end{definition}

	\begin{rem} 
		
		\label{uabcdac0} 
		
		Observe for any $r, s, z \in \mathbb{C}$ and $\bm{B} \in \mathbb{Z}_{\ge 0}^n$ that 
		\begin{flalign*}
			U_{z; r, s} (\bm{e}_0, \bm{B}; \bm{e}_0, \bm{B}) = s^{2b} \cdot \displaystyle\frac{(z; q)_b}{(s^2 z; q)_b}; \qquad W_{z; r, s} (\bm{e}_0, \bm{B}; \bm{e}_0, \bm{B}) = (-s)^b \cdot \displaystyle\frac{(s^{-1} z; q)_b}{(sz; q)_b},
		\end{flalign*}
		
		\noindent which quickly follow from the fact that the sums in \eqref{weightu} and \eqref{wweight0} are supported on the term $\bm{P} = \bm{e}_0$ (as $\bm{C} = \bm{e}_0$). 
		
	\end{rem} 

		\begin{rem} 
		
		\label{a0c0u} 
		
		Let us explain the sense in which the $\widehat{W}$-weight from \Cref{wabcdw} is analogous to the $\widehat{L}$ one from \eqref{lzij2}. The latter was chosen to be the (unique) normalization of $L_x$ such that $\widehat{L}_x (\bm{e}_0, k; \bm{e}_0, k) = 1$ for each $k \in \llbracket 1, n \rrbracket$; the arrow configuration $(\bm{e}_0, k; \bm{e}_0, k)$ could be viewed as ``horizontally saturated,'' since horizontal edges could accommodate at most one arrow under the $L$-weights. The analog of this constraint would be to make $\widehat{W}_{z; r, s}$ a normalization of $W_{z; r, s}$ such that $\widehat{W}_{z;r, s} = 1$ at a fused arrow configuration that is ``horizontally saturated'' in one color. One way to make sense of ``horizontal saturation'' for fused arrow configurations is to impose a threshold $\mathrm{R} \in \mathbb{Z}_{>0}$ for the number of arrows that can occupy a horizontal edge; this is done by setting $r = q^{-\mathrm{R}/2}$ (as then the factor of $(r^2; q)_d$ in the $W$-weight \eqref{wweight0} is equal to $0$ if $d > \mathrm{R}$). In this case, the normalization condition would be for 
		\begin{flalign}
			\label{rw0}
			\widehat{W}_{z; r, s} (\bm{e}_0, \mathrm{R} \bm{e}_k; \bm{e}_0, \mathrm{R} \bm{e}_k) = 1, \qquad \text{for each $k \in \llbracket 1, n \rrbracket$},
		\end{flalign} 
		
		\noindent so that $\widehat{W}_{z; r, s} (\bm{A}, \bm{B}; \bm{C}, \bm{D}) = W_{z; r, s} (\bm{A}, \bm{B}; \bm{C}, \bm{D}) \cdot W_{z; r, s} (\bm{e}_0, \mathrm{R} \bm{e}_k; \bm{e}_0, \mathrm{R} \bm{e}_k)^{-1}$, which by \Cref{uabcdac0} yields \eqref{wzr2}.
		
		Let us also briefly mention that another way of imposing ``horizontal saturation'' would be to have infinitely many arrows of some color $k \in \llbracket 1, n \rrbracket$ travel along rows of the model. One should then track how many arrows of color $k$ leave a row (as well as how many arrows of the other colors enter it), that is, one ``complements'' the arrow configuration in the color $k$. This should enable one to remove the restriction that $r = q^{-\mathrm{R}/2}$ for some integer $\mathrm{R} \ge 1$; similar ideas were also used in \cite[Section 3.1.3]{CFSSVM} and \cite[Section 17.7.3]{CFVMSF}. However, we will not pursue this direction here,\footnote{See, however, \Cref{VertexPolymer} below, which implements a version of this complementation to degenerate the colored stochastic fused vertex model to the log-gamma polymer.} and keep ourselves constrained to the case when $r^2 \in q^{\mathbb{Z}_{<0}}$ whenever using the $\widehat{W}$-weights. 
		
	\end{rem}

	The following proposition states that the $U$ and $W$ weights (of \Cref{wzabcd} and \Cref{a0c0u}) satisfy the Yang-Baxter equation; it is due to \cite[Equation (3.20)]{STRR} (though, as stated below, it appears in \cite{CSVMST}\footnote{In \cite[Theorem C.1.1]{CSVMST}, it was assumed that $r^2, s^2, t^2 \in q^{\mathbb{Z}_{< 0}}$, but this assumption can be removed by using analytic continuation (with the fact that the $U$-weights are rational in $r$, $s$, and $t$).}); it is a fused generalization of \Cref{rrr1}.
	
	\begin{lem}[{\cite[Theorem C.1.1]{CSVMST}}] 
		
		\label{equationrank} 
		
		Fix $x, y, z, r, s, t \in \mathbb{C}$ and $\bm{I}_1, \bm{J}_1, \bm{K}_1, \bm{I}_3, \bm{J}_3, \bm{K}_3 \in \mathbb{Z}_{\ge 0}$. Then, 
		\begin{flalign*}
			\begin{aligned}
				& \displaystyle\sum_{\bm{I}_2, \bm{J}_2, \bm{K}_2 \in \mathbb{Z}_{\ge 0}^n}  U_{x/y; r, s}  ( \bm{I}_1, \bm{J}_1; \bm{I}_2, \bm{J}_2 ) U_{x/z; r, t}  ( \bm{K}_1, \bm{J}_2; \bm{K}_2, \bm{J}_3) U_{y/z; s, t}  ( \bm{K}_2, \bm{I}_2; \bm{K}_3, \bm{I}_3) \\
				& \quad = \displaystyle\sum_{\bm{I}_2, \bm{J}_2, \bm{K}_2 \in \mathbb{Z}_{\ge 0}^n}   U_{y/z; s, t}  ( \bm{K}_1, \bm{I}_1; \bm{K}_2, \bm{I}_2 ) U_{x/z; r, t}  ( \bm{K}_2, \bm{J}_1; \bm{K}_3, \bm{J}_2 ) U_{x/y; r, s}  ( \bm{I}_2, \bm{J}_2; \bm{I}_3, \bm{J}_3).
			\end{aligned}
		\end{flalign*}	
		
		\noindent Therefore, if there exists an integer $\mathrm{R} > 0$ such that $r = q^{-\mathrm{R}/2}$, then 
		\begin{flalign*}
			\begin{aligned}
				& \displaystyle\sum_{\bm{I}_2, \bm{J}_2, \bm{K}_2 \in \mathbb{Z}_{\ge 0}^n}  U_{x/y; r, s}  ( \bm{I}_1, \bm{J}_1; \bm{I}_2, \bm{J}_2 ) \widehat{W}_{x/z; r, t}  ( \bm{K}_1, \bm{J}_2; \bm{K}_2, \bm{J}_3) W_{y/z; s}  ( \bm{K}_2, \bm{I}_2; \bm{K}_3, \bm{I}_3) \\
				& \quad = \displaystyle\sum_{\bm{I}_2, \bm{J}_2, \bm{K}_2 \in \mathbb{Z}_{\ge 0}^n}   U_{y/z; s, t}  ( \bm{K}_1, \bm{I}_1; \bm{K}_2, \bm{I}_2 ) \widehat{W}_{x/z; r, t}  ( \bm{K}_2, \bm{J}_1; \bm{K}_3, \bm{J}_2 ) W_{x/y; r}  ( \bm{I}_2, \bm{J}_2; \bm{I}_3, \bm{J}_3).
			\end{aligned}
		\end{flalign*}	
		
	\end{lem}

	Before proceeding, we record the following results that evaluate the $W_{x;r,0}$ weights.

	\begin{lem} 
		
		\label{wweight} 
		
		Adopting the notation of \Cref{wabcdw}, we have
		\begin{flalign*}
			\begin{aligned} 
				W_{z; r, 0} (\bm{A}, \bm{B}; \bm{C}, \bm{D}) & = (-1)^{b-d} z^d r^{2(c-a)} q^{\varphi (\bm{D}, \bm{C}) + \binom{b}{2}} \displaystyle\frac{(r^2; q)_d}{(r^2; q)_b} \cdot \mathbbm{1}_{\bm{A} + \bm{B} = \bm{C} + \bm{D}} \\
				& \qquad \times \displaystyle\prod_{i=1}^n \Bigg( \displaystyle\sum_{p=0}^{\min \{ B_i, C_i \}} (-r^{-2})^p q^{\binom{p+1}{2} -p (B_{[i,n]} + D_{[1,i-1]})} \\
				&  \qquad \qquad \qquad \times \displaystyle\frac{(q; q)_{C_i + D_i - p}}{(q; q)_{C_i - p} (q; q)_{D_i}} \displaystyle\frac{(q; q)_{B_i}}{(q; q)_{B_i - p} (q; q)_p} \Bigg).
			\end{aligned} 
		\end{flalign*}	
		
	\end{lem} 
	
	\begin{proof} 
		
		Throughout this proof, we assume that $\bm{A} + \bm{B} = \bm{C} + \bm{D}$, as otherwise $W_{z; r, 0} (\bm{A}, \bm{B}; \bm{C}, \bm{D})= 0$ (as $U_{z; r, s} (\bm{A}, \bm{B}; \bm{C}, \bm{D}) = 0$ for any $s \in \mathbb{C}$, by \Cref{wzabcd}). Inserting the equalities 
		\begin{flalign*}
			\displaystyle\lim_{s \rightarrow 0} s^{b-p} (s^{-1} z; q)_{b-p} = (-1)^{b-p} q^{\binom{b-p}{2}} z^{b-p}; \qquad \displaystyle\lim_{s \rightarrow 0} \displaystyle\frac{(r^{-2} sz; q)_{c-p} (r^2 sz^{-1}; q)_p}{(sz; q)_{c+d-p}} = 1, 
		\end{flalign*}
		
		\noindent into \eqref{weightu}, we obtain 
		\begin{flalign*}
			W_{z; r, 0} (\bm{A}, \bm{B}; \bm{C}, \bm{D})  = (-1)^{b-d} z^d r^{2(c-a)}  q^{\varphi (\bm{D}, \bm{C})} \displaystyle\frac{(r^2; q)_d}{(r^2; q)_b} & \displaystyle\sum_{p = 0}^{\min \{ b, c \}} (-r^{-2})^p  q^{\binom{b-p}{2}}   \displaystyle\sum_{\substack{\bm{P} \le \bm{B}, \bm{C} \\ |\bm{P}| = p}} q^{\varphi (\bm{B} - \bm{D} - \bm{P}, \bm{P})} \\
			& \times \displaystyle\prod_{i=1}^n \displaystyle\frac{(q; q)_{C_i + D_i - P_i}}{(q; q)_{D_i} (q; q)_{C_i - P_i}} \displaystyle\frac{(q; q)_{B_i}}{(q; q)_{P_i} (q; q)_{B_i - P_i}}.
		\end{flalign*}
		
		\noindent Also since 
		\begin{flalign*}
			\binom{b-p}{2} & = \binom{b}{2} - bp + \binom{p+1}{2}; \qquad  \binom{p+1}{2} - \varphi (\bm{P}, \bm{P}) = \displaystyle\sum_{i=1}^n \binom{P_i+1}{2}; \\
			& \varphi (\bm{B}, \bm{P}) - bp = -\displaystyle\sum_{1 \le i \le j \le n} P_i B_j = -\displaystyle\sum_{i=1}^n B_i P_i - \varphi (\bm{P}, \bm{B}); 
		\end{flalign*} 
		
		\noindent and since $\varphi (\bm{B} - \bm{D} - \bm{P}, \bm{P}) = \varphi (\bm{B}, \bm{P}) - \varphi (\bm{D}, \bm{P}) - \varphi (\bm{P}, \bm{P})$ (by the bilinearity of $\varphi$), we have 
		\begin{flalign}
			\label{wzr0} 
			\begin{aligned} 
				W_{z; r, 0} & (\bm{A}, \bm{B}; \bm{C}, \bm{D}) \\	
				& = (-1)^{b-d} z^d r^{2(c-a)}  q^{\varphi (\bm{D}, \bm{C}) + \binom{b}{2}} \displaystyle\frac{(r^2; q)_d}{(r^2; q)_b} \displaystyle\prod_{i=1}^n \displaystyle\frac{(q;q)_{B_i}}{(q; q)_{D_i}} \\
				& \quad \times \displaystyle\sum_{p = 0}^{\min \{ b, c \}} (-r^{-2})^p \displaystyle\sum_{\substack{\bm{P} \le \bm{B}, \bm{C} \\ |\bm{P}| = p}} q^{- \varphi (\bm{P}, \bm{B}) - \varphi (\bm{D}, \bm{P})} \displaystyle\prod_{i=1}^n q^{\binom{P_i+1}{2} - B_i P_i} \displaystyle\frac{(q; q)_{C_i + D_i - P_i}}{(q; q)_{C_i - P_i} (q; q)_{P_i} (q; q)_{B_i - P_i}}.
			\end{aligned} 	
		\end{flalign}
		
		\noindent Now observe for any complex number $w \in \mathbb{C}$; $n$-tuples $\bm{X}, \bm{Y}, \bm{Z} \in \mathbb{Z}^n$; and functions $f_1, f_2, \ldots , f_n : \mathbb{Z}_{\ge 0} \rightarrow \mathbb{C}$ we have 
		\begin{flalign*}
			\displaystyle\sum_{\substack{\bm{P} \in \mathbb{Z}_{\ge 0}^n \\ \bm{P} \le \bm{Z}}} w^{|\bm{P}|} q^{\varphi (\bm{P}, \bm{X}) + \varphi (\bm{Y}, \bm{P})} \displaystyle\prod_{i=1}^n f_i (P_i) = \displaystyle\prod_{i=1}^n \Bigg( \displaystyle\sum_{p=0}^{Z_i} q^{p (X_{[i+1, n]} + Y_{[1, i-1]})} w^p f_i (p) \Bigg),
		\end{flalign*}
		
		\noindent by expanding the product on the right side. Applying this in \eqref{wzr0} with
		\begin{flalign*}
			& w = -r^{-2}; \qquad \bm{X} = -\bm{B}; \qquad \bm{Y} = -\bm{D}; \qquad \bm{Z} = \min \{ \bm{B}, \bm{C} \}; \\
			& \qquad \quad f_i (k) = q^{\binom{k+1}{2} - B_i k} \displaystyle\frac{(q; q)_{C_i + D_i - k}}{(q; q)_{C_i - k} (q; q)_k (q; q)_{B_i - k}},
		\end{flalign*} 
		
		\noindent yields the lemma (where $\min \{ \bm{B}, \bm{C} \}$ denotes the entry-wise minimum of $\bm{B}$ and $\bm{C}$).
	\end{proof}
	
	\begin{cor}
		
		\label{w01}
		
		Adopting the notation of \Cref{wabcdw}, we have $W_{z; r, 0} (\bm{A}, \bm{B}; \bm{A} + \bm{B}, \bm{e}_0) = 1$. 
	\end{cor}

	\begin{proof}
		
		By \Cref{wweight} and the facts that $(q; q)_{B_i} (q; q)_{B_i - p}^{-1} = (-1)^p q^{B_i p - \binom{p}{2}} (q^{-B_i}; q)_{B_i}$; that $B_i - B_{[i, n]} = -B_{[i+1, n]}$; and that $\binom{p+1}{2} - \binom{p}{2} = p$, we have 
		\begin{flalign}
			\label{wabab}
			W_{z; r, 0} (\bm{A}, \bm{B}; \bm{A} + \bm{B}, \bm{e}_0) = (-1)^b r^{2b} q^{\binom{b}{2}} (r^2; q)_b^{-1} \cdot \displaystyle\prod_{i=1}^n \Bigg( \displaystyle\sum_{p=0}^{B_i} (r^{-2} q^{1 - B_{[i + 1,n]}})^p \displaystyle\frac{(q^{-B_i}; q)_p}{(q; q)_p} \Bigg).
		\end{flalign}
	
		\noindent From the $q$-binomial theorem, we have for each $i \in \llbracket 1, n \rrbracket$ that 
		\begin{flalign*} 
				\displaystyle\sum_{p=0}^{B_i} (r^{-2} q^{1 - B_{[i+1,n]}})^p \displaystyle\frac{(q^{-B_i}; q)_p}{(q;q)_p} = (r^{-2} q^{1-B_{[i,n]}}; q)_{B_i},
		\end{flalign*}
	
		\noindent and hence 
		\begin{flalign*}
			\displaystyle\prod_{i=1}^n \Bigg( \displaystyle\sum_{p=0}^{B_i} (r^{-2} q^{1 - B_{[i + 1,n]}})^p \displaystyle\frac{(q^{-B_i}; q)_p}{(q; q)_p} \Bigg) & = \displaystyle\prod_{i=1}^n (r^{-2} q^{1 - B_{[i, n]}}; q)_{B_i} \\
			& = (r^{-2}; q^{-1})_b = (-1)^b r^{-2b} q^{-\binom{b}{2}} (r^2; q)_b.
		\end{flalign*}
	
		\noindent Inserting this into \eqref{wabab} yields the corollary.
	\end{proof}

	\begin{rem}
	
	An alternative (and perhaps more conceptual) way of proving \Cref{w01} would be through fusion, using the fact that it holds at $\mathrm{R} = 1$ by \eqref{lyi0}; we will not provide further details on that route here.
	
	\end{rem}

	\subsection{Fused Nonsymmetric Functions}
	
	\label{FusedfG}
	
	In this section we formulate the fused analogs of the functions $f$ and $G$ (from \Cref{fg}), as well as some of their properties. We begin with the following definition, which is parallel to \Cref{dwe}. 
	
	\begin{definition}
		
		\label{fusedde}
		
		Fix an integer $N \ge 1$; a complex number $s \in \mathbb{C}$; sequences of complex numbers $\bm{r} = (r_1, r_2, \ldots , r_N)$ and $\bm{x} = (x_1, x_2, \ldots , x_N)$; and a colored fused path ensemble $\mathcal{E}$ on $\mathcal{D}_N = \mathbb{Z}_{\le 0} \times \llbracket 1, N \rrbracket$, whose arrow configuration at any $v \in \mathcal{D}_N$ is denoted by $\big( \bm{A}(v), \bm{B} (v); \bm{C} (v), \bm{D} (v) \big)$. Set
		\begin{flalign*}
			& W_{\bm{x}; \bm{r}; s} (\mathcal{E}) = \displaystyle\prod_{k = 1}^{\infty} \displaystyle\prod_{j=1}^N  W_{x_j; r_j, s} \big( \bm{A} (-k, j), \bm{B} (-k, j); \bm{C} (-k, j), \bm{D} (-k, j) \big) \\
			& \qquad \qquad \qquad \times \displaystyle\prod_{j=1}^N W_{x_j; r_j, 0} \big( \bm{A} (0, j), \bm{B} (0, j); \bm{C} (0, j), \bm{D} (0, j) \big); \\
			& \widehat{W}_{\bm{x}; \bm{r}; s} (\mathcal{E}) = \displaystyle\prod_{k = 1}^{\infty} \displaystyle\prod_{j=1}^N \widehat{W}_{x_j; r_j, s} \big( \bm{A} (-k, j), \bm{B} (-k, j); \bm{C} (-k, j), \bm{D} (-k, j) \big) \\
			& \qquad \qquad \qquad \times \displaystyle\prod_{j=1}^N W_{x_j; r_j, 0} \big( \bm{A} (0, j), \bm{B} (0, j); \bm{C} (0, j), \bm{D} (0, j) \big), 
		\end{flalign*}
		
		\noindent where in the second equality it is assumed that $r_j^2 \in q^{\mathbb{Z}_{< 0}}$, for each $j \in \llbracket 1, N \rrbracket$.
		
	\end{definition}
	
	Now we can define the following (non)symmetric functions, in a way parallel to \Cref{fg}, but with three differences. First, the vertex models are fused path ensembles, instead of higher spin ones. Second, they have weights the $W$ and $\widehat{W}$, instead of the $L$ and $\widehat{L}$. Third, in defining $\mathbbm{f}$ below, we have multiple arrows entering all rows, instead of only one, with the purpose of fully saturating them (see \Cref{fgfunctionsfused}).

 	\begin{definition}
 		
 		\label{fgfused} 
 		
 		Fix an integer $N \ge 1$; a sequence of positive integers $\bm{\mathrm{R}} = (\mathrm{R}_1, \mathrm{R}_2, \ldots , \mathrm{R}_N)$; two $n$-compositions $\mu, \nu \in \Comp_n$; and a function  $\sigma : \llbracket 1, N \rrbracket \rightarrow \llbracket 1, n \rrbracket$. 
 		
 		If $\ell (\mu) = \ell (\nu) + \mathrm{R}_{[1,N]}$, then let $\mathfrak{P}_{\mathbbm{f}} (\mu/\nu; \sigma; \bm{\mathrm{R}})$ denote the set of colored fused path ensembles on $\mathcal{D}_N = \mathbb{Z}_{\le 0} \times \llbracket 1, N \rrbracket$ with the following boundary data.  
 		
 		\begin{enumerate}
 			\item For each $j \in \llbracket 1, N \rrbracket$, $\mathrm{R}_j$ arrows of color $\sigma (j)$ horizontally enters $\mathcal{D}_N$ through $(-\infty, j)$. 
 			\item For each $k \ge 0$ and $c \in \llbracket 1, n \rrbracket$, $\mathfrak{m}_k \big( \nu^{(c)} \big)$ arrows of color $c$ vertically enter $\mathcal{D}_N$ through $(-k, 1)$. 
 			\item For each $k \ge 0$ and $c \in \llbracket 1, n \rrbracket$, $\mathfrak{m}_k \big(\mu^{(c)} \big)$ arrows of color $c$ vertically exit $\mathcal{D}_N$ through $(-k, N)$.  
 		\end{enumerate}
 		
 		\noindent See the left side of \Cref{fgfunctionsfused} for a depiction when $\mu = (7, 2, 0 \mid 5, 5, 4, 1, 1, 0 \mid 5, 2, 2)$; $\nu = (\emptyset \mid 6 \mid 6, 5)$; $\big( \sigma (1), \sigma(2), \sigma(3), \sigma(4), \sigma(5) \big) = (1, 3, 2, 1, 2)$; and $\bm{\mathrm{R}} = (2, 1, 3, 1, 2)$. There, red, green, and blue are colors $1$, $2$, and $3$, respectively. 
 		
 		Similarly, if $\ell (\mu) = \ell (\nu)$, then let $\mathfrak{P}_{\mathbb{G}} (\mu / \nu)$ denote the set of colored fused path ensembles on $\mathcal{D}_N = \mathbb{Z}_{\le 0} \times \llbracket 1, N \rrbracket$, with the following boundary data. 
 		
 		\begin{enumerate} 
 			\item For each $j \in \llbracket 1, N \rrbracket$, no arrow horizontally enters or exits $\mathcal{D}_N$ through the $j$-th row.
 			\item For each $k \ge 0$ and $c \in \llbracket 1, n \rrbracket$, $\mathfrak{m}_k \big( \mu^{(c)} \big)$ arrows of color $c$ vertically enter $\mathcal{D}_N$ through $(-k, 1)$.
 			\item For each $k \ge 0$ and $c \in \llbracket 1, n \rrbracket$, $\mathfrak{m}_k \big( \nu^{(c)} \big)$ arrows of color $c$ vertically exit $\mathcal{D}_N$ through $(-k, N)$.
 		\end{enumerate} 
 		
 		\noindent  See the right side of \Cref{fgfunctionsfused} for a depiction when $\mu = (7, 5 \mid 7, 6 \mid 6, 4)$ and $\nu = (5, 2 \mid 5, 1 \mid 3, 2)$.  
 		
 		For any complex number $s \in \mathbb{C}$, and sequences of complex numbers $\bm{r} = (r_1, r_2, \ldots , r_N)$ and $\bm{x} = (x_1, x_2, \ldots , x_N)$, let
 		\begin{flalign}
 			\label{fgmunufused} 
 			\mathbbm{f}_{\mu / \nu; s}^{\sigma} (\bm{x}; \bm{r}) = \displaystyle\sum_{\mathfrak{P}_f (\mu/ \nu; \sigma; \bm{\mathrm{R}})} \widehat{W}_{\bm{x}; \bm{r}; s}(\mathcal{E}); \qquad \mathbbm{G}_{\mu/\nu; s} (\bm{x}; \bm{r}) = \displaystyle\sum_{\mathfrak{P}_G (\mu/\nu)} W_{\bm{x}; \bm{r}; s} (\mathcal{E}),
 		\end{flalign}
 		
 		\noindent where in the first equality it is assumed that $r_j^{-2} = q^{-\mathrm{R}_j}$ for each $j \in \llbracket 1, N \rrbracket$ (and $\bm{r}$ can be arbitrary in the second). If $\nu = \emptyset$ is empty, we write $\mathbbm{f}_{\mu; s}^{\sigma} (\bm{x}; \bm{r}) = \mathbbm{f}_{\mu/\emptyset; s}^{\sigma} (\bm{x}; \bm{r})$ and $\mathbb{G}_{\mu; s} (\bm{x}; \bm{r}) = \mathbb{G}_{\mu/0^N; s} (\bm{x}; \bm{r})$. If $N = 1$, we may write $\mathbbm{f}_{\mu/\nu; s}^{\sigma(1)}$ in place of $\mathbbm{f}_{\mu/\nu; s}^{\sigma}$. 
 	\end{definition}

 	\begin{figure}
 		\begin{center}
 			\begin{tikzpicture}[
 				>=stealth, 
 				scale = .75]{				
 					
 					\draw[->, very thick] (8, 0) -- (8, 6.5);
 					\draw[->, very thick] (8, 0) -- (0, 0);
 					
 					\draw[->, thick, red] (0, .95) node[above = 6, right = 2, black,  scale = .6]{$\sigma(1)$} -- (1, .95);
 					\draw[->, thick, red] (0, 1.05) -- (1, 1.05);
 					\draw[->, thick, blue] (0, 2) node[above = 4, right = 2, black,  scale = .6]{$\sigma(2)$} -- (1, 2);
 					\draw[->, thick, green] (0, 3) -- (1, 3);
 					\draw[->, thick, green] (0, 3.1) -- (1, 3.1);
 					\draw[->, thick, green] (0, 2.9) -- (1, 2.9);
 					\draw[->, thick, red] (0, 4) -- (1, 4);
 					\draw[->, thick, green] (0, 5.05) node[above = 4, right = 1, black,  scale = .6]{$\sigma(N)$} -- (1, 5.05);
 					\draw[->, thick, green] (0, 4.95) -- (1, 4.95);
 					
 					\draw[->, thick, blue] (1.95, 0) -- (1.95, 1);
 					\draw[->, thick, green] (2.05, 0) -- (2.05, 1);
 					\draw[->, thick, blue] (3, 0) -- (3, 1);
 					
 					\draw[->, thick, red] (1, 5) -- (1, 6);
 					\draw[->, thick, green] (8.05, 5) -- (8.05, 6);
 					\draw[->, thick, blue] (2.9, 5) -- (2.9, 6); 
 					\draw[->, thick, green] (3, 5) -- (3, 6);
 					\draw[->, thick, green] (4, 5) -- (4, 6);
 					\draw[->, thick, red] (5.9, 5)-- (5.9, 6);
 					\draw[->, thick, blue] (6, 5) -- (6, 6);
 					\draw[->, thick, red] (7.95, 5) -- (7.95, 6);
 					\draw[->, thick, blue] (6.1, 5) -- (6.1, 6);
 					\draw[->, thick, green] (3.1, 5) -- (3.1, 6); 
 					\draw[->, thick, green] (7.05, 5) -- (7.05, 6);
 					\draw[->, thick, green] (6.95, 5) -- (6.95, 6);
 					
 					\draw[ultra thick, gray, dashed] (1, 1) -- (1, 5);
 					\draw[ultra thick, gray, dashed] (2, 1) -- (2, 5);
 					\draw[ultra thick, gray, dashed] (3, 1) -- (3, 5);
 					\draw[ultra thick, gray, dashed] (4, 1) -- (4, 5);
 					\draw[ultra thick, gray, dashed] (5, 1) -- (5, 5);
 					\draw[ultra thick, gray, dashed] (6, 1) -- (6, 5);
 					\draw[ultra thick, gray, dashed] (7, 1) -- (7, 5);
 					
 					\draw[ultra thick, gray, dashed] (1, 1) -- (8, 1);
 					\draw[ultra thick, gray, dashed] (1, 2) -- (8, 2);
 					\draw[ultra thick, gray, dashed] (1, 3) -- (8, 3);
 					\draw[ultra thick, gray, dashed] (1, 4) -- (8, 4);
 					\draw[ultra thick, gray, dashed] (1, 5) -- (8, 5);

 					\draw[]  (0, 1) circle [radius = 0] node[left, scale = .65]{$(x_1, r_1)$};
 					\draw[]  (0, 2) circle [radius = 0] node[left, scale = .65]{$(x_2, r_2)$};
 					\draw[]  (-.15, 3.5) circle [radius = 0] node[left, scale = .8]{$\vdots$};
 					\draw[]  (0, 5) circle [radius = 0] node[left, scale = .65]{$(x_N, r_N)$};
 					
 					\draw[-] (1, -.25) -- (1, -.75) -- (7, -.75) -- (7, -.25);
 					\draw[-] (1, 6.25) -- (1, 6.5) -- (7, 6.5) -- (7, 6.25);
 					
 					\draw[] (4, -.75) circle[radius = 0] node[below, scale = .7]{$\nu$};
 					\draw[] (4, 6.5) circle[radius = 0] node[above, scale = .7]{$\mu$};
 					\draw[] (4, 7.5) circle[radius = 0] node[above, scale = .8]{$\mathbbm{f}_{\mu / \nu}$};
 					
 					\draw[->, thick, red] (10.95, 0) -- (10.95, 1);
 					\draw[->, thick, red] (12.95, 0) -- (12.95, 1);
 					\draw[->, thick, blue] (11.95, 0) -- (11.95, 1);
 					\draw[->, thick, blue] (14, 0) -- (14, 1);
 					\draw[->, thick, green] (11.05, 0) -- (11.05, 1); 
 					\draw[->, thick, green] (12.05, 0) -- (12.05, 1);
 					
 					\draw[->, thick, red] (12.95, 5) -- (12.95, 6); 
 					\draw[->, thick, red] (15.95, 5) -- (15.95,6);
 					\draw[->, thick, blue] (15, 5) -- (15, 6);
 					\draw[->, thick, blue] (16.05, 5) -- (16.05, 6);
 					\draw[->, thick, green] (13.05, 5) -- (13.05, 6); 
 					\draw[->, thick, green] (17, 5) -- (17, 6);
 					
 					\draw[ultra thick, gray, dashed] (11, 1) -- (11, 5);
 					\draw[ultra thick, gray, dashed] (12, 1) -- (12, 5);
 					\draw[ultra thick, gray, dashed] (13, 1) -- (13, 5);
 					\draw[ultra thick, gray, dashed] (14, 1) -- (14, 5);
 					\draw[ultra thick, gray, dashed] (15, 1) -- (15, 5);
 					\draw[ultra thick, gray, dashed] (16, 1) -- (16, 5);
 					\draw[ultra thick, gray, dashed] (17, 1) -- (17, 5);
 					
 					\draw[ultra thick, gray, dashed] (11, 1) -- (18, 1);
 					\draw[ultra thick, gray, dashed] (11, 2) -- (18, 2);
 					\draw[ultra thick, gray, dashed] (11, 3) -- (18, 3);
 					\draw[ultra thick, gray, dashed] (11, 4) -- (18, 4);
 					\draw[ultra thick, gray, dashed] (11, 5) -- (18, 5);
 					
 					\draw[->, very thick] (18, 0) -- (18, 6.5);
 					\draw[->, very thick] (18, 0) -- (10, 0);
 					
 					\draw[]  (11, 1) circle [radius = 0] node[left, scale = .65]{$(x_1, r_1)$};
 					\draw[]  (11, 2) circle [radius = 0] node[left, scale = .65]{$(x_2, r_2)$};
 					\draw[]  (11, 3.5) circle [radius = 0] node[left = 10, scale = .65]{$\vdots$};
 					\draw[]  (11, 5) circle [radius = 0] node[left, scale = .65]{$(x_N, r_N)$};

 					\draw[-] (11, -.25) -- (11, -.75) -- (17, -.75) -- (17, -.25);
 					\draw[-] (11, 6.25) -- (11, 6.5) -- (17, 6.5) -- (17, 6.25);
 					
 					\draw[] (14, -.75) circle[radius = 0] node[below, scale = .7]{$\mu$};
 					\draw[] (14, 6.5) circle[radius = 0] node[above, scale = .7]{$\nu$};
 					\draw[] (14, 7.5) circle[radius = 0] node[above, scale = .8]{$G_{\mu / \nu}$};
 				}
 			\end{tikzpicture}
 		\end{center}		
 		\caption{\label{fgfunctionsfused} Depicted to the left and right are vertex models for $\mathbbm{f}_{\mu / \nu; s}^{\sigma}$ and $\mathbb{G}_{\mu / \nu; s}$, respectively.} 
 		
 	\end{figure}

 	Observe that the quantity $\widehat{W}_{\bm{x}; \bm{r}; s} (\mathcal{E})$ appearing as the summand in \eqref{fgmunufused} defining $\mathbbm{f}_{\mu/\nu}^{\sigma}$ is bounded, since all but finitely many vertices in any ensemble $\mathcal{E} \in \mathfrak{P}_{\mathbbm{f}} (\mu/\nu; \sigma; \bm{\mathrm{R}})$ have arrow configurations of the form $(\bm{e}_0, \mathrm{R}_j  \bm{e}_{\sigma(j)}; \bm{e}_0, \mathrm{R}_j  \bm{e}_{\sigma(j)}) $ for some $j \in \llbracket 1, N \rrbracket$, and $\widehat{W}_{x_j; r_j, s} (\bm{e}_0, \mathrm{R}_j  \bm{e}_{\sigma(j)}; \bm{e}_0,  \mathrm{R}_j \bm{e}_{\sigma(j)}) = 1$ for $r_j = q^{-\mathrm{R}_j/2}$, by \eqref{rw0}. Similarly, $W_{\bm{x}; \bm{r}; s} (\mathcal{E})$ appearing as the summand in \eqref{fgmunufused} defining $\mathbb{G}_{\mu/\nu; s}$ is bounded, since all but finitely many vertices in any $\mathcal{E} \in \mathfrak{P}_{\mathbb{G}} (\mu/\nu)$ have arrow configurations of the form $(\bm{e}_0, \bm{e}_0; \bm{e}_0, \bm{e}_0)$, and we have $W_{x; r, s} (\bm{e}_0, 0; \bm{e}_0, 0) = 1$ by \Cref{uabcdac0}. 	
 	
 	\begin{rem} 
 		
 		\label{fr}
 		
 		It may be possible to analytically continue the $\mathbbm{f}$  functions in the parameters $\bm{r}$ (so as to avoid imposing the assumption that each $r_j \in q^{\mathbb{Z}_{<0}}$), by following the complementation procedure outlined at the end of \Cref{a0c0u}. However, we will not pursue this here. 
 	\end{rem} 
 	
 	\subsection{Properties of $\mathbbm{f}$ and $\mathbb{G}$} 
 	
 	\label{IdentitiesFusedfG}
 	
 	In this section we provide properties of the $\mathbbm{f}$ and $\mathbb{G}$ functions from \Cref{fgfused}. The first is the \emph{fusion} property that relates these to the $f$ and $G$ functions from \Cref{fg}, when the parameters of the latter are specialized to unions of geometric progressions. We omit its proof, as very similar statements have appeared repeatedly throughout the literature; see \cite[Section 6E]{CSVMPL}, \cite[Theorem 6.2]{SVMP}, and \cite[Proposition 7.2.3]{CFVMSF} for references in the colored case and \cite[Proposition 5.5]{HSVMSRF} in the uncolored one.
 	
 	\begin{lem}
 		
 		\label{fgfg}
 		
 		Adopt the notation of \Cref{fgfused}, and assume $r_j = q^{-\mathrm{R}_j/2}$ for each $j \in \llbracket 1, N \rrbracket$. Define $\omega : \llbracket 1, \mathrm{R}_{[1,N]} \rrbracket \rightarrow \llbracket 1, n \rrbracket$ by for each $j \in \llbracket 1, \mathrm{R}_{[1,N]} \rrbracket$ setting $\omega (j) = \sigma(j')$, where $j' \in \llbracket 1, N \rrbracket$ is the unique index satisfying $\mathrm{R}_{[1, j' - 1]} + 1 \le j \le \mathrm{R}_{[1, j']}$; also define the $\mathrm{R}_{[1,N]}$-tuple of complex numbers
 		\begin{flalign*} 
 			\bm{z} = (x_1, qx_1, \ldots ,q^{\mathrm{R}_1-1} x_1, x_2, qx_2, \ldots , q^{\mathrm{R}_2-1} x_2, \ldots , x_N, q x_N, \ldots , q^{\mathrm{R}_N-1} x_n).
 		\end{flalign*} 
 	
 		\noindent Then, we have $\mathbbm{f}_{\mu/\nu;s}^{\sigma} (\bm{x}; \bm{r}) = f_{\mu/\nu; s}^{\omega} (\bm{z})$ and $\mathbb{G}_{\mu/\nu; s} (\bm{x}; \bm{r}) = G_{\mu/\nu; s} (\bm{z})$.  
 		
 	\end{lem} 
 
 	We next provide symmetry properties, branching statements, and Cauchy identities for the $\mathbbm{f}$ and $\mathbb{G}$ functions, which are parallel to their counterparts (\Cref{gsymmetric}, \Cref{ffgg}, and \Cref{fg2}, respectively) for the $f$ and $g$ functions. They are quick consequences of the latter, together with \Cref{fgfg} and analytic continuation. 
 	
 	\begin{lem} 
 		
 		\label{gsymmetricfused} 
 		
 		Adopt the notation of \Cref{fgfused}, and let $\varsigma : \llbracket 1, M \rrbracket \rightarrow \llbracket 1, M \rrbracket$ denote a permutation. We have $\mathbb{G}_{\mu/\nu; s} (\bm{y}; \bm{r}) = \mathbb{G}_{\mu/\nu; s} \big( \varsigma(\bm{y}); \varsigma (\bm{r}) \big)$, where $\varsigma(\bm{y}) = (y_{\varsigma(1)}, y_{\varsigma(2)}, \ldots , y_{\varsigma(M)} \big)$ and $\varsigma (\bm{r}) = (r_{\varsigma(1)}, r_{\varsigma(2)}, \ldots , r_{\varsigma(N)})$. 
 		
 	\end{lem}    
 
 	\begin{proof}
 		
 		If there exist positive integers $(\mathrm{R}_1, \mathrm{R}_2, \ldots , \mathrm{R}_N)$ such that $r_j = q^{-\mathrm{R}_j/2}$ for each $j \in \llbracket 1, N \rrbracket$, then this follows from \Cref{fgfg} and \Cref{gsymmetric}. The fact that the lemma holds for an arbitrary set $\bm{r}$ of complex numbers then follows from uniqueness of analytic continuation, as $\mathbb{G}$ is a rational function in $\bm{r}$ (since the $W$-weights from \Cref{wabcdw} are).  		
 	\end{proof}
 	
 	\begin{lem}
 		
 		\label{ffggfused} 
 		
 		Adopting the notation of \Cref{fgfused}; letting $\ell = \ell (\nu)$; and fixing an integer $k \in \llbracket 1, N \rrbracket$, we have  
 		\begin{flalign*}
 			& \mathbbm{f}_{\mu / \nu; s}^{\sigma} (\bm{x}; \bm{r}) = \displaystyle\sum_{\kappa \in \Comp_n (\ell + k)} \mathbbm{f}_{\kappa / \nu; s}^{\sigma |_{\llbracket 1,k \rrbracket}} \big(\bm{x}_{[1, k]}; \bm{r}_{[1,k]} \big) \mathbbm{f}_{\mu / \kappa; s}^{\sigma |_{\llbracket k+1, N \rrbracket}} \big( \bm{x}_{[k+1, N]};  \bm{r}_{[k+1,N]} \big); \\ 
 			& \mathbb{G}_{\mu / \nu; s} (\bm{x}; \bm{r}) = \displaystyle\sum_{\kappa \in \Comp_n (\ell)} \mathbb{G}_{\mu / \kappa; s} \big( \bm{x}_{[1, k]}; \bm{r}_{[1,k]} \big) \mathbb{G}_{\kappa / \nu; s} \big( \bm{x}_{[k+1, N]}; \bm{r}_{[k+1, N]} \big).
 		\end{flalign*}
 		
 		\noindent where in the first equality we assume that $r_j = q^{-\mathrm{R}_j/2}$ for each $j \in \llbracket 1, N \rrbracket$ (and $\bm{r}$ can be arbitrary in the second). Here, we have defined the variable sets $\bm{x}_{[1, k]} = (x_1, x_2, \ldots , x_k)$, $\bm{x}_{[k+1,N]} = (x_{k+1}, x_{k+2}, \ldots , x_N)$, $\bm{r}_{[1,k]} = (r_1, r_2, \ldots , r_k)$, and $\bm{r}_{[k+1, N]} = (r_{k+1}, r_{k+2}, \ldots , r_N)$. For any interval $I = \big\llbracket i_0 + 1, i_0 + |I| \big\rrbracket \subset \llbracket 1, N \rrbracket$, we have also defined the function $\sigma |_I : \big\llbracket 1, |I| \big\rrbracket \rightarrow \llbracket 1, n \rrbracket$ by setting $\sigma |_I (i) = \sigma (i + i_0)$ for each $i \in \big\llbracket 1, |I| \big\rrbracket$. 
 		
 	\end{lem}
 
 	\begin{proof}
 		The first equality in the lemma follows from \Cref{ffgg} and \Cref{fgfg}. If $r_j = q^{-\mathrm{R}_j/2}$ for each $j \in \llbracket 1, N \rrbracket$, the second follows from the same two statements; for general $\bm{r}$, it then follows from uniqueness of analytic continuation, as $\mathbb{G}$ is rational in $\bm{r}$. 
 	\end{proof}

 	\begin{lem}
 		
 		\label{fg2fused}
 		
 		Fix integers $n, M, N \ge 1$; a sequence of positive integers $\bm{\mathrm{R}} = (\mathrm{R}_1, \mathrm{R}_2, \ldots , \mathrm{R}_N)$; sequences of complex variables $\bm{r} = (r_1, r_2, \ldots,  r_N)$, $\bm{t} = (t_1, t_2, \ldots , t_M)$, $\bm{x} = (x_1, x_2, \ldots , x_N)$, and $\bm{y} = (y_1, y_2, \ldots , y_M)$; and a function $\sigma : \llbracket 1, N \rrbracket \rightarrow \llbracket 1, n \rrbracket$. Assume for each $j \in \llbracket 1, N \rrbracket$ that $r_j = q^{-\mathrm{R}_j/2}$ and that 
 		\begin{flalign}
 			\label{rtxy} 
 			\displaystyle\max_{\substack{1 \le i \le M \\ 1 \le j \le N}} \displaystyle\sup_{\substack{(b, b') \in \mathbb{Z}_{\ge 0} \times \llbracket 0, \mathrm{R}_j \rrbracket \\ (b, b') \ne (0, \mathrm{R}_j)}} \bigg| (-s)^b \displaystyle\frac{(s^{-1} y_j; q)_b}{(sy_j; q)_b} \cdot (-s)^{b' - \mathrm{R}_j} \displaystyle\frac{(s^{-1} x_i; q)_{b'}}{(sx_i; q)_{b'}} \displaystyle\frac{(sx_i; q)_{\mathrm{R}_j}}{(s^{-1} x_i; q)_{\mathrm{R}_j}} \bigg|  < 1.
 		\end{flalign}
 		
 		\noindent Then,  
 		\begin{flalign*}
 			\displaystyle\sum_{\mu \in \Comp_n (\mathrm{R}_{[1,N]})} \mathbbm{f}_{\mu; s}^{\sigma} (\bm{x}; \bm{r}) \mathbb{G}_{\mu; s} (\bm{y}; \bm{t}) = \displaystyle\prod_{i=1}^M \displaystyle\prod_{j=1}^N \displaystyle\frac{(t_i^2 x_j y_i^{-1}; q)_{\mathrm{R}_j}}{ t_i^{2 \mathrm{R}_j} (x_j y_i^{-1}; q)_{\mathrm{R}_j}}.
 		\end{flalign*}
 		
 	\end{lem}
 	
 	\begin{proof}[Proof (Outline)]
 		
 		The proof is analogous to that of \Cref{fg2}, so we only briefly outline it. We will use the equality of partition functions depicted in \Cref{uvfused} (which is similar to the equality of the partition functions depicted in \Cref{z1xy} and \Cref{z2xy}). On both sides of that figure, $\mathrm{R}_j$ arrows of color $\sigma(j)$ enter horizontally through the $(M+j)$-th row (from the bottom) for each $j \in \llbracket 1, N \rrbracket$, and all arrows exit vertically through the $y$-axis. The weights on both sides are assigned as follows. In the crosses, we use $U_{x_j / y_i; r_j, t_i}$ at the intersection of the $j$-th row (from the bottom) and $i$-th column (from the left). Along the $y$-axis, we use $W_{x_j; r_j, 0}$ or $W_{y_i; t_i, 0}$, depending on whether the row is marked by $(x_j, t_j)$ or $(y_i, s_i)$ in \Cref{uvfused}. In $\mathbb{Z}_{< 0} \times \llbracket 1, M + N \rrbracket$, we use $\widehat{W}_{x_j; r_j, s}$ or $W_{y_i; t_i, s}$, again depending on the marking of the row. The equality of partition functions depicted in \Cref{uvfused} is then a consequence of the Yang--Baxter equation \Cref{equationrank}; denote this partition function by $\mathcal{Z}$. 
 		
 		The vertex model on the left side of \Cref{uvfused} is frozen; it is quickly verified that it has weight 
 		\begin{flalign}
 			\label{0z1} 
 			\mathcal{Z} = 1,
 		\end{flalign} 
 	
 		\noindent using \Cref{uabcdac0} (at $b=0$), \eqref{rw0}, and \Cref{w01}. To analyze the right side of \Cref{uvfused}, first observe by \eqref{rtxy}, \Cref{uabcdac0}, and \eqref{wzr2} that, for any $(i, j) \in \llbracket 1, M \rrbracket \times \llbracket 1, N \rrbracket$, we have 
 		\begin{flalign*}
 			  \displaystyle\sup_{\substack{(b, b') \in \mathbb{Z}_{\ge 0} \times \llbracket 0, \mathrm{R}_j \rrbracket \\ (b, b') \ne (0, \mathrm{R}_j)}} \displaystyle\max_{\substack{\bm{B}, \bm{B}' \in \mathbb{Z}_{\ge 0}^n \\ |\bm{B}| = b, |\bm{B}'| = b'}} \bigg| \displaystyle\frac{\widehat{W}_{x_j; r_j, s} (\bm{e}_0, \bm{B}'; \bm{e}_0, \bm{B}')}{\widehat{W}_{x_j; r_j, s} (\bm{e}_0, \mathrm{R}_j \bm{e}_{\sigma(j)}; \bm{e}_0, \mathrm{R}_j \bm{e}_{\sigma(j)})} \cdot \displaystyle\frac{W_{y_i; t_i, s} (\bm{e}_0, \bm{B}; \bm{e}_0, \bm{B})}{W_{y_i; t_i, s} (\bm{e}_0, \bm{e}_0; \bm{e}_0, \bm{e}_0)} \bigg| < 1.
 		\end{flalign*}
 		
 		\noindent Using this bound, one can verify (see, for example, the proof of \cite[Proposition 6.2.2]{CFVMSF} for a very similar argument) that the vertex model on the right side of \Cref{uvfused} has nonzero weight only if all but finitely many vertices in rows marked by $(y_i, t_i)$ for some $i \in \llbracket 1, M \rrbracket$ have arrow configuration $(\bm{e}_0, \bm{e}_0; \bm{e}_0, \bm{e}_0)$. This implies that the cross on the right side of \Cref{uvfused} is frozen to have arrow configuration $(\bm{e}_0, \mathrm{R}_j \bm{e}_{\sigma(j)}; \bm{e}_0, \mathrm{R}_j \bm{e}_{\sigma(j)} \big)$ at each vertex in its $j$-th row. 
 		
 		Hence, the weight of the cross is 
 		\begin{flalign*}
 			\displaystyle\prod_{i=1}^M \displaystyle\prod_{j=1}^N U_{x_j/y_i; r_j, t_i} (\bm{e}_0, \mathrm{R}_j \bm{e}_{\sigma(j)}; \bm{e}_0, \mathrm{R}_j \bm{e}_{\sigma(j)}) = \displaystyle\sum_{i=1}^M \displaystyle\prod_{j=1}^N \displaystyle\frac{t_i^{2 \mathrm{R}_j} (x_j y_i^{-1}; q)_{\mathrm{R}_j}}{(t_i^2 x_j y_i^{-1}; q)_{\mathrm{R}_j}}.
 		\end{flalign*}
 	
 		\noindent Moreover, by \eqref{fgmunufused}, the weight of the part of the right side of \Cref{uvfused} on $\mathbb{Z}_{\le 0} \times \llbracket 1, M + N \rrbracket$ is given by $\sum_{\mu \in \Comp_n (\mathrm{R}_{[1,N]})} \mathbbm{f}_{\mu}^{\sigma} (\bm{x}; \bm{r}) \mathbb{G}_{\mu} (\bm{y}; \bm{t})$. Hence,
 		\begin{flalign*}
 			\mathcal{Z} = \displaystyle\prod_{i=1}^M \displaystyle\prod_{j=1}^N \displaystyle\frac{t_i^{2 \mathrm{R}_j} (x_j y_i^{-1}; q)_{\mathrm{R}_j}}{(t_i^2 x_j y_i^{-1}; q)_{\mathrm{R}_j}} \cdot \displaystyle\sum_{\mu \in \Comp_n (\mathrm{R}_{[1,N]})} \mathbbm{f}_{\mu}^{\sigma} (\bm{x}; \bm{r}) \mathbb{G}_{\mu} (\bm{y}; \bm{t}).
 		\end{flalign*}
 	
 		\noindent This, together with \eqref{0z1}, yields the lemma.
 	\end{proof}

 	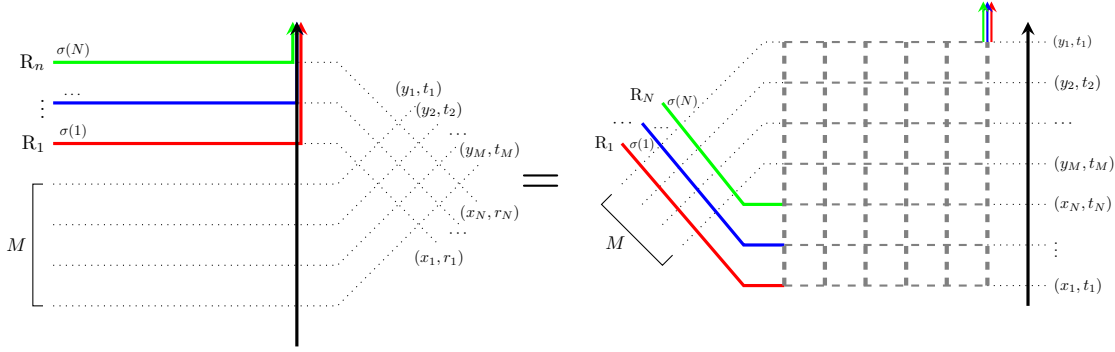
\begin{figure}
 		\begin{center}
 			\begin{tikzpicture}[
 				>=stealth, 
 				scale = .54]{
 					\draw[dotted] (-1, 0) -- (6, 0) -- (9.5, 3.5);
 					\draw[] (9.75, 3.5) circle[radius=0] node[above, scale = .6]{$(y_M, t_M)$};
 					\draw[dotted] (-1, 1) -- (6, 1) -- (9, 4) node[above, scale = .6]{$\cdots$};
 					\draw[dotted] (-1, 2) -- (6, 2) -- (8.5, 4.5) node[above, scale = .6]{$(y_2, t_2)$};
 					\draw[dotted] (-1, 3) -- (6, 3) -- (8, 5) node[above, scale = .6]{$(y_1, t_1)$};
 					
 					\draw[dotted] (4.1, 4) -- (6, 4) -- (8.5, 1.5) node[below, scale = .6]{$(x_1, r_1)$};
 					\draw[dotted] (4.1, 5) -- (6, 5) -- (9, 2) node[below, scale = .6]{$\cdots$};
 					\draw[dotted] (4.1, 6) -- (6, 6) -- (9.5, 2.5);
 					\draw[] (9.75, 2.625) circle[radius = 0] node[below, scale = .6]{$(x_N, r_N)$};
 					
 					\draw[red, very thick, ->] (-1, 4) node[left, black, scale = .75]{$\mathrm{R}_1$} -- (5.1, 4) -- (5.1, 7);
 					\draw[blue, very thick, ->] (-1, 5) node[left, black, scale = .75]{$\vdots$} -- (5, 5) -- (5, 7);  
 					\draw[green, very thick, ->] (-1, 6) node[left, black, scale = .75]{$\mathrm{R}_n$} -- (4.9, 6) -- (4.9, 7);
 					\draw[] (-.5, 4) circle[radius = 0] node[above, scale = .55]{$\sigma(1)$};
 					\draw[] (-.5, 5) circle[radius = 0] node[above, scale = .55]{$\cdots$};
 					\draw[] (-.5, 6) circle[radius = 0] node[above, scale = .55]{$\sigma(N)$};
 					\draw[->, very thick] (5, -1) -- (5, 7);
 					
 					\draw[-] (-1.25, 0) -- (-1.5, 0) -- (-1.5, 3) -- (-1.25, 3);
 					\draw[] (-1.5, 1.5) circle [radius = 0] node[left, scale = .7]{$M$};
 					
 					\draw[] (11, 3) circle[radius=0] node[scale = 2]{$=$};
 					
 					\draw[dotted] (14.5, 1.5) -- (16.5, 3.5) -- (17, 3.5);
 					\draw[dotted] (22, 3.5) -- (23.5, 3.5) node[right, scale = .6]{$(y_M, t_M)$};
 					\draw[dotted] (14, 2) -- (16.5, 4.5) -- (17, 4.5);
 					\draw[dotted] (22, 4.5) -- (23.5, 4.5) node[right, scale = .6]{$\cdots$};
 					\draw[dotted] (13.5, 2.5) -- (16.5, 5.5) -- (17, 5.5);
 					\draw[dotted] (22, 5.5) -- (23.5, 5.5) node[right, scale = .6]{$(y_2, t_2)$};
 					\draw[dotted] (13, 3) -- (16.5, 6.5) -- (17, 6.5);
 					\draw[dotted] (22, 6.5) -- (23.5, 6.5) node[right, scale = .5]{$(y_1, t_1)$};

 					\draw[red, very thick] (13, 4) node[left, black, scale = .65]{$\mathrm{R}_1$} -- (16, .5) -- (17, .5);
 					\draw[blue, very thick] (13.5, 4.5) node[left, black, scale = .65]{$\cdots$} -- (16, 1.5) -- (17, 1.5);
 					\draw[green, very thick] (14, 5) node[above = 3, left, black, scale = .65]{$\mathrm{R}_N$} -- (16, 2.5) -- (17, 2.5);
 					\draw[black] (13.5, 3.7) circle[radius = 0] node[scale = .5, above]{$\sigma(1)$};
 					\draw[black] (14, 4.2) circle[radius = 0] node[scale = .5, above]{$\cdots$};
 					\draw[black] (14.5, 4.75) circle[radius = 0] node[scale = .5, above]{$\sigma(N)$};
 					
 					\draw[dotted] (22, .5) -- (23.5, .5) node[right, scale = .6]{$(x_1, t_1)$};
 					\draw[dotted] (22, 1.5) -- (23.5, 1.5) node[right, scale = .6]{$\vdots$};
 					\draw[dotted] (22, 2.5) -- (23.5, 2.5) node[right, scale = .6]{$(x_N, t_N)$};
 					
 					\draw[->, very thick] (23, 0) -- (23, 7);
 					
 					\draw[-] (14.25, 1.25) -- (14, 1) -- (12.5, 2.5) -- (12.75, 2.75);
 					\draw[] (13.25, 1.75) circle [radius = 0] node[below = 3, left, scale = .7]{$M$};
 					
 					\draw[red, thick, ->] (22.1, 6.5) -- (22.1, 7.5);
 					\draw[blue, thick, ->] (22, 6.5) -- (22, 7.5); 
 					\draw[green, thick, ->]	(21.9, 6.5) -- (21.9, 7.5);
 					
 					\draw[gray, ultra thick, dashed] (17, .5) -- (17, 6.5);
 					\draw[gray, ultra thick, dashed] (18, .5) -- (18, 6.5);
 					\draw[gray, ultra thick, dashed] (19, .5) -- (19, 6.5);
 					\draw[gray, ultra thick, dashed] (20, .5) -- (20, 6.5);
 					\draw[gray, ultra thick, dashed] (21, .5) -- (21, 6.5);
 					\draw[gray, ultra thick, dashed] (22, .5) -- (22, 6.5);
 					
 					\draw[gray, thick, dashed] (17, .5) -- (22, .5);
 					\draw[gray, thick, dashed] (17, 1.5) -- (22, 1.5);
 					\draw[gray, thick, dashed] (17, 2.5) -- (22, 2.5);
 					\draw[gray, thick, dashed] (17, 3.5) -- (22, 3.5);
 					\draw[gray, thick, dashed] (17, 4.5) -- (22, 4.5);
 					\draw[gray, thick, dashed] (17, 5.5) -- (22, 5.5);
 					\draw[gray, thick, dashed] (17, 6.5) -- (22, 6.5);
 					
 				}
 			\end{tikzpicture}
 		\end{center}	
 		
 		\caption{\label{uvfused} Depicted above is the equality of partition functions used in the proof of \Cref{fg2fused}.}
 	\end{figure}

 	\subsection{Ascending $\mathbbm{f}\mathbb{G}$ Measures} 
 	
 	\label{AscendingfGFused}

 	In this section we introduce the fused variants of the probability measures from \Cref{measurefg}. Throughout, we fix integers $M, N \ge 1$; a sequence of positive integers $\bm{\mathrm{R}} = (\mathrm{R}_1, \mathrm{R}_2, \ldots , \mathrm{R}_N)$; a composition $\bm{\ell} = (\ell_1, \ell_2, \ldots , \ell_n)$ of $\mathrm{R}_{[1,N]}$; and a function $\sigma : \llbracket 1, N \rrbracket \rightarrow \llbracket 1, n \rrbracket$ such that $\ell_i = \sum_{j \in \sigma^{-1} (i)} \mathrm{R}_j$ for each $i \in \llbracket 1, n \rrbracket$. We must first introduce the relevant family of $n$-compositions.
 	
 	\begin{definition} 
 		
 		\label{mnmufused}

 		A sequence $\bm{\mu} = \big( \mu (0), \mu (1), \ldots , \mu (M+N) \big)$ of $n$-compositions is called \emph{$(M; \sigma; \bm{\mathrm{R}})$-ascending} if the following hold.  
 		
 		\begin{enumerate}
 			\item We have $\mu (0) = (\emptyset \mid \cdots \mid \emptyset)$ and $\mu(M+N) = (0^{\ell_1} \mid \cdots \mid 0^{\ell_n})$. 
 			\item 
 			\begin{enumerate} 
 				\item For all $j \in \llbracket 0, N \rrbracket$ and $c \in \llbracket 1, n \rrbracket$, we have $\ell \big( \mu^{(c)} (j) \big) = \sum_{k=1}^j \mathrm{R}_k \cdot \mathbbm{1}_{\sigma(k) = c}$. Thus, $\ell \big( \mu(j) \big) = \mathrm{R}_{[1,j]}$. 
 				\item For all $i \in \llbracket N, M+N \rrbracket$ and $c \in \llbracket 1, n \rrbracket$, we have $\ell \big( \mu^{(c)} (i) \big) = \ell_c$. Thus, $\ell \big( \mu(i) \big) = \mathrm{R}_{[1,N]}$.
 			\end{enumerate} 
 
 			\item For each $(c, k, i) \in \llbracket 1, n \rrbracket \times \mathbb{Z}_{>0} \times \llbracket 1, M+N \rrbracket$, we have 
 			\begin{flalign}
 				\label{qmucfused} 
 				 \mathfrak{Q}_c^{\bm{\mu}} (k, i) = \mathfrak{m}_{\le k-1} \big( \mu^{(c)} (i) \big) - \mathfrak{m}_{\le k-1} \big( \mu^{(c)} (i-1) \big) \ge 0.
 			\end{flalign}
 		
			\noindent For each $(k, i) \in \mathbb{Z}_{>0} \llbracket 1, M + N \rrbracket$, we set $\bm{\mathfrak{Q}}^{\bm{\mu}} (k, i) = \big( \mathfrak{Q}_1^{\bm{\mu}} (k, i), \mathfrak{Q}_2^{\bm{\mu}} (k, i), \ldots , \mathfrak{Q}_n^{\bm{\mu}} (k, i) \big) \in \mathbb{Z}_{\ge 0}^n$. 
		\end{enumerate} 
	
 		 \noindent Let us also define the sequence $\mathscr{Q} (\bm{\mu}) = \big( \bm{\mathfrak{Q}}^{\bm{\mu}} (1, 1), \bm{\mathfrak{Q}}^{\bm{\mu}} (1, 2), \ldots , \bm{\mathfrak{Q}}^{\bm{\mu}} (1, M+N) \big)$. 
 	\end{definition} 
 	
 	\begin{rem} 
 		
 		\label{mnmurow2fused}
 		
 		Given an $(M; \sigma; \bm{\mathrm{R}})$-ascending sequence of compositions $\bm{\mu}$ as in \Cref{mnmufused}, we will often view the $n$-composition $\mu(i)$ as indexing the positions (as in \Cref{mumuistate}) of the colored arrows exiting the row $\{ y = i \}$, in a vertex model on $\mathbb{Z}_{\le 0} \times \llbracket 1, M + N \rrbracket$. This yields a colored fused path ensemble on $\mathbb{Z}_{\le 0} \times \llbracket 1, M + N \rrbracket$, that we will denote by $\mathcal{E}_{\bm{\mu}}$. In this way, the $c$-th entry in $\bm{\mathfrak{Q}}^{\bm{\mu}} (k, i)$ denotes the number of color $c$ arrows in $\mathcal{E}_{\bm{\mu}}$ along the edge connecting $(-k, i)$ to $(1-k, i)$. Therefore, $\bm{\mathfrak{Q}}$ records the colors of the arrows (from top to bottom) along the horizontal edges in $\mathcal{E}_{\bm{\mu}}$ joining the $(-1)$-st column to the $0$-th one.
 		
 		The boundary data for this ensemble is described as follows. For each $j \in \llbracket 1, N \rrbracket$, it has $\mathrm{R}_j$ arrows of color $\sigma(j)$ horizontally entering the row $\{ y = j \}$, and it has no other arrows horizontally entering or exiting any other row of the model. For each $c \in \llbracket 1, n \rrbracket$, it has $\ell_c$ arrows of color $c$ vertically exiting the $y$-axis $\{ x = 0 \}$, and it has no other arrows horizontally entering or exiting any other column of the model. We denote by $\mathfrak{P}_{\mathbbm{f}\mathbb{G}} (M; \sigma; \bm{\mathrm{R}})$ the set of colored fused path ensembles on $\mathbb{Z}_{\le 0} \times \llbracket 1, M + N \rrbracket$ with these boundary conditions, as any $\mathcal{E}_{\bm{\mu}} \in \mathfrak{P}_{\mathbbm{f}\mathbb{G}} (M; \sigma; \bm{\mathrm{R}})$ can be thought of an ensemble from $\mathfrak{P}_{\mathbb{G}} (\mu / \emptyset)$ that is juxtaposed above one from $\mathfrak{P}_{\mathbbm{f}} (\mu / \emptyset; \sigma; \bm{\mathrm{R}})$ (recall \Cref{fgfused}), for some $n$-composition $\mu \in \Comp_n (\mathrm{R}_{[1, N]})$. It is quickly verified that the above procedure is a bijection between $\mathfrak{P}_{\mathbbm{f}\mathbb{G}} (M; \sigma; \bm{\mathrm{R}})$ and $(M; \sigma; \bm{\mathrm{R}})$-ascending sequences $\bm{\mu}$ of $n$-compositions. 
 		
 		See \Cref{00lmufused} for a depiction, where there $(n, M, N) = (2, 3, 4)$, $(\mathrm{R}_1, \mathrm{R}_2, \mathrm{R}_3, \mathrm{R}_4) = (3, 1, 2, 2)$, $(\ell_1, \ell_2) = (5, 3)$, $\big( \sigma (1), \sigma(2), \sigma(3), \sigma(4) \big) = (1, 2, 2, 1)$, and 
 		\begin{flalign*} 
 			&   \mu (1) = (2, 2, 0 \mid \emptyset), \qquad \mu (2) = (2, 0, 0 \mid 3), \qquad \mu (3) = (2, 0, 0 \mid 4, 1, 0), \\
 			\mu (4) = ( & 5, 3, 0, 0, 0 \mid 3, 0, 0), \quad \mu (5) = (2, 1, 0, 0, 0 \mid 3, 0, 0), \quad \mu (6) = (1, 0, 0, 0, 0 \mid 1, 0, 0).
 		\end{flalign*} 
 		
 	\end{rem}

 	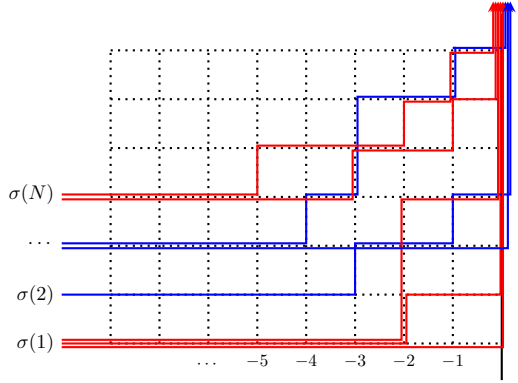
\begin{figure} 
 		
 		\begin{center}
 			\begin{tikzpicture}[
 				>=stealth, 
 				scale = .65]{

 					\draw[thick] (23, -.25) -- (23, 7.25);
 					
 					\draw[black, thick, dotted] (15, .5) -- (15, 6.5);
 					\draw[black, thick, dotted] (16, .5) -- (16, 6.5);
 					\draw[black, thick, dotted] (17, .5) node[below = 4, scale = .65]{$\cdots$} -- (17, 6.5);
 					\draw[black, thick, dotted] (18, .5) node[below = 2, scale = .65]{$-5$} -- (18, 6.5);
 					\draw[black, thick, dotted] (19, .5) node[below = 2, scale = .65]{$-4$}  -- (19, 6.5);
 					\draw[black, thick, dotted] (20, .5) node[below = 2, scale = .65]{$-3$} -- (20, 6.5);
 					\draw[black, thick, dotted] (21, .5) node[below = 2, scale = .65]{$-2$} -- (21, 6.5);
 					\draw[black, thick, dotted] (22, .5) node[below = 2, scale = .65]{$-1$} -- (22, 6.5);
 					
 					\draw[black, thick, dotted] (15, .5) -- (23, .5);
 					\draw[black, thick, dotted] (15, 1.5) -- (23, 1.5);
 					\draw[black, thick, dotted] (15, 2.5) -- (23, 2.5);
 					\draw[black, thick, dotted] (15, 3.5) -- (23, 3.5);
 					\draw[black, thick, dotted] (15, 4.5) -- (23, 4.5);
 					\draw[black, thick, dotted] (15, 5.5) -- (23, 5.5);
 					\draw[black, thick, dotted] (15, 6.5) -- (23, 6.5);
 					
 					\draw[->, red, thick] (14, .5) node[left, black, scale = .75]{$\sigma(1)$} -- (15, .5) -- (21.05, .5) -- (21.05, 1.5) -- (22.975, 1.5) -- (22.975, 7.5);
 					\draw[->, blue, thick] (14, 1.5) node[left, black, scale = .75]{$\sigma(2)$} -- (15, 1.5) -- (20, 1.5) -- (20, 2.55) -- (22, 2.55) -- (22, 3.55) -- (23.175, 3.55) -- (23.175, 7.5);
 					\draw[->, blue, thick] (14, 2.55) node[left, black, scale = .75]{$\cdots$} -- (15, 2.55) -- (19, 2.55) -- (19, 3.55) -- (20.05, 3.55) -- (20.05, 5.55) -- (22.05, 5.55) -- (22.05, 6.55) -- (22, 6.55) -- (23.075, 6.55) -- (23.075, 7.5);
 					\draw[->, red, thick] (14, .575) -- (20.95, .575) -- (20.95, 3.45) -- (22.925, 3.45) -- (22.925, 7.5); 
 					\draw[->, red, thick] (14, .425) -- (23.025, .425) -- (23.025, 7.5); 
 					\draw[->, blue, thick] (14, 2.45) -- (23.125, 2.45) -- (23.125, 7.5);  
 					\draw[->, red, thick] (14, 3.55) node[left, black, scale = .75]{$\sigma(N)$} -- (15, 3.55) -- (18, 3.55) -- (18, 4.55) -- (21, 4.55) -- (21, 5.45) -- (21.95, 5.45) -- (21.95, 6.45) -- (22.825, 6.45) -- (22.825, 7.5);
 					\draw[->, red, thick] (14, 3.45) -- (19.95, 3.45) -- (19.95, 4.45) -- (22, 4.45) -- (22, 5.5) -- (22.875, 5.5) -- (22.875, 7.5); 
 				}
 			\end{tikzpicture}
 		\end{center}
 		
 		\caption{\label{00lmufused} Depicted above is the colored fused path ensemble associated with the sequence $\bm{\mu}$ in the example at the end of \Cref{mnmurow2fused}. Here, red and blue are colors $1$ and $2$, respectively.} 
 		
 	\end{figure}

 	Next we define the following probability measure on sequences of ascending compositions.

 	\begin{definition}
 		
 		\label{measurefgfused}
 		
 		Fix a complex number $s \in \mathbb{C}$ and four sequences $\bm{r} = (r_1, r_2, \ldots , r_N)$; $\bm{t} = (t_1, t_2, \ldots , t_M)$; $\bm{x} = (x_1, x_2, \ldots , x_N)$; and $\bm{y} = (y_1, y_2, \ldots , y_M)$ of complex numbers, with $r_j = q^{-\mathrm{R}_j/2}$ for each $j \in \llbracket 1, n \rrbracket$. Define the probability measure $\mathbb{P}_{\mathbbm{f}\mathbb{G}}^{\sigma} = \mathbb{P}_{\mathbbm{f}\mathbb{G}; n; s; \bm{x}; \bm{y}; \bm{r}; \bm{t}}^{\sigma}$ on $(M; \sigma; \bm{\mathrm{R}})$-ascending sequences of $n$-compositions, by setting
 		\begin{flalign}
 			\label{fgprobabilitymufused}
 			\mathbb{P}_{\mathbbm{f}\mathbb{G}}^{\sigma} [\bm{\mu}] = \mathcal{Z}_{\bm{x}; \bm{y}; \bm{r}; \bm{t}}^{-1} \cdot \displaystyle\prod_{j=1}^N \mathbbm{f}_{\mu (j) / \mu (j-1); s}^{\sigma (j)} (x_j; r_j) \displaystyle\prod_{i=N+1}^{M+N} \mathbb{G}_{\mu (i-1) / \mu (i); s} (y_{i-N}; t_{i-N}).
 		\end{flalign}
 		
 		\noindent for each $(M; \sigma; \bm{\mathrm{R}})$-ascending sequence $\bm{\mu} = \big( \mu (0), \mu (1), \ldots , \mu (M+N) \big)$, where 
 		\begin{flalign}
 			\label{zxysfused}
 			\mathcal{Z}_{\bm{x}; \bm{y}; \bm{r}; \bm{t}} =  \displaystyle\prod_{i=1}^M \displaystyle\prod_{j=1}^N \displaystyle\frac{(t_i^2 x_j y_i^{-1}; q)_{\mathrm{R}_j}}{t_i^{2 \mathrm{R}_j} (x_j y_i^{-1}; q)_{\mathrm{R}_j}}.
 		\end{flalign}
 		
 		\noindent Here, we implicitly assume that $s$, $\bm{x}$, $\bm{y}$, $\bm{\mathrm{R}}$, and $\bm{t}$ are such that the left side of \eqref{fgprobabilitymufused} is nonnegative and \eqref{rtxy} holds. The fact that these probabilities sum to one follows from \Cref{sumfg1fused} below.
 		
 	\end{definition}

 	The proof of the below lemma, given \Cref{ffggfused} and \Cref{fg2fused}, is entirely analogous to that of \Cref{sumfg1}, given \Cref{ffgg} and \Cref{fg2}; it is therefore omitted.

 	\begin{lem}
 		
 		\label{sumfg1fused}
 		
 		Under the notation and assumptions of \Cref{measurefgfused}, we have 
 		\begin{flalign*}
 			\displaystyle\sum_{\bm{\mu}} \displaystyle\prod_{j=1}^N \mathbbm{f}_{\mu(j) / \mu(j-1); s}^{\sigma (j)} (x_j; r_j) \cdot \displaystyle\prod_{i=N+1}^{M+N} \mathbb{G}_{\mu(i-1) / \mu(i); s} (y_{i-N}; t_{i-N}) = \mathcal{Z}_{\bm{x}; \bm{y}; \bm{r}; \bm{s}},
 		\end{flalign*} 
 		
 		\noindent where the sum on the left side is over all $(M; \sigma; \bm{\mathrm{R}})$-ascending sequences of $n$-compositions $\bm{\mu}$.

 	\end{lem}

 	\subsection{Matching Between Colored Stochastic Six-Vertex Models and $\mathbb{P}_{\mathbbm{f}\mathbb{G}}^{\sigma}$} 
 	
 	\label{EqualityCDFused}
 	
 	In this section we provide a matching between the law of the $(M+N)$-tuple $\mathscr{Q} (\bm{\mu})$ of elements in $\mathbb{Z}_{\ge 0}^n$ (recall \Cref{mnmufused}) associated with a sequence of compositions sampled from $\mathbb{P}_{\mathbbm{f}\mathbb{G}}^{\sigma}$ (from \Cref{measurefgfused}), with a certain random variable associated with the colored fused vertex model (from \Cref{FusedPath}). We begin by defining the latter, which is analogous to \Cref{dece}. Throughout this section, we fix integers $M, N \ge 1$; a complex number $s \in \mathbb{C}$; sequence of positive integers $\bm{\mathrm{R}} = (\mathrm{R}_1, \mathrm{R}_2, \ldots , \mathrm{R}_N)$; sequences of complex numbers $\bm{r} = (r_1, r_2, \ldots , r_N)$, $\bm{t} = (t_1, t_2, \ldots , t_M)$, $\bm{x} = (x_1, x_2, \ldots , x_N)$, and $\bm{y} = (y_1, y_2, \ldots , y_M)$, satisfying \eqref{rtxy} and $r_j = q^{-\mathrm{R}_j/2}$ for each $j \in \llbracket 1, N \rrbracket$; and a function $\sigma : \llbracket 1, N \rrbracket \rightarrow \llbracket 1, n \rrbracket$.
 	 	
 	\begin{definition} 
 		
 		\label{decefused}
 		
 		Let $\mathcal{E}$ denote a colored fused path ensemble on the rectangle domain $\mathcal{D}_{M;N} = \llbracket 1, M \rrbracket \times \llbracket 1, N \rrbracket$. For each integer $i \in \llbracket 1, M \rrbracket$, let $\bm{C} (i) = \bm{C}^{\mathcal{E}} (i) \in \mathbb{Z}_{\ge 0}^n$ be such that the $k$-th entry $C_k (i)$ of $\bm{C} (i)$ denotes the number of color $k$ arrows in $\mathcal{E}$ vertically exiting $\mathcal{D}_{M;N}$ through $(i, N)$, for each $k \in \llbracket 1, n \rrbracket$. For each integer $j \in \llbracket 1, N \rrbracket$, let $\bm{D} (j) = \bm{D}^{\mathcal{E}} (j) \in \mathbb{Z}_{\ge 0}^n$ be such that the $k$-th entry $D_k (j)$ of $\bm{D} (j)$ denotes the number of color $k$ arrows in $\mathcal{E}$ horizontally exiting $\mathcal{D}_{M;N}$ through $(M, j)$. Then set $\mathscr{C} (\mathcal{E}) = \big( \bm{C}(1), \bm{C}(2), \ldots , \bm{C}(M) \big) \in (\mathbb{Z}_{\ge 0}^n)^M$ and $\mathscr{D} (\mathcal{E}) = \big(\bm{D}(1), \bm{D}(2), \ldots , \bm{D} (N) \big) \in (\mathbb{Z}_{\ge 0}^n)^N$. 
 		
 	\end{definition} 
 	
 	We next require notation for the colored stochastic fused vertex model (defined at the end of \Cref{FusedPath}) with specific boundary data.
 	
 	\begin{definition} 
 		
 		\label{probabilityvertexfused}

 		We say that a colored fused path ensemble on the rectangle domain $\mathcal{D}_{M;N} = \llbracket 1, M \rrbracket \times \llbracket 1, N \rrbracket$ has \emph{$(\sigma; \bm{\mathrm{R}})$-entrance data} if the following holds. For each $j \in \llbracket 1, N \rrbracket$, $\mathrm{R}_j$ paths of color $\sigma (j)$ horizontally enters $\mathcal{D}_{M;N}$ from the site $(j, 0)$ on the $y$-axis, and no path horizontally enters $\mathcal{D}_{M;N}$ from any site on the $x$-axis. Let $\mathbb{P}_{\FV}^{\sigma} = \mathbb{P}_{\FV; \bm{x}; \bm{y}; \bm{r}; \bm{t}}^{\sigma}$ denote the measure on colored fused path ensembles on $\mathcal{D}_{M;N}$ obtained by running the colored stochastic fused vertex model on $\mathcal{D}_{M;N}$ under $(\sigma; \bm{\mathrm{R}})$-boundary data, with weight $U_{y_i / x_j; r_j, t_i}$ (recall \Cref{wzabcd}) at any vertex $(i, j) \in \mathcal{D}_{M;N}$.
 		
 	\end{definition}

 	 The following proposition now provides a matching between the $(M+N)$-tuple $\mathscr{Q} (\bm{\mu})$ (recall \Cref{mnmufused}) sampled under $\mathbb{P}_{\mathbbm{f}\mathbb{G}}^{\sigma}$ of \Cref{measurefgfused} and the $(M+N)$-tuple $\mathscr{D}(\mathcal{E}) \cup \overleftarrow{\mathscr{C}}(\mathcal{E})$ sampled under $\mathbb{P}_{\FV}^{\sigma}$ of \Cref{probabilityvertexfused}. We omit its proof, which is very similar to that of \Cref{fgsv} (with the few necessary modifications already explained in the proof outline of \Cref{fg2fused}).

 	\begin{prop}
 		
 		\label{fgsvfused}
 		
 		Let $\mathscr{Q} = \big( \bm{\mathfrak{Q}} (1), \bm{\mathfrak{Q}} (2), \ldots , \bm{\mathfrak{Q}} (M+N) \big)$ denote an $(M+N)$-tuple of elements in $\mathbb{Z}_{\ge 0}^n$, and define the $M$-tuple $\mathscr{C} = \big( \bm{\mathfrak{Q}} (M+N), \bm{\mathfrak{Q}}(M+N-1), \ldots,  \bm{\mathfrak{Q}} (N+1) \big)$ and $N$-tuple $\mathscr{D} = \big( \bm{\mathfrak{Q}}(1), \bm{\mathfrak{Q}}(2), \ldots , \bm{\mathfrak{Q}} (N) \big)$. Then,  
 		\begin{flalign}
 			\label{ekmukfused} 
 			\mathbb{P}_{\FV}^{\sigma} \Big[ \{ \mathscr{C} (\mathcal{E}) = \mathscr{C} \big\} \cap \big\{ \mathscr{D} (\mathcal{E}) = \mathscr{D} \big\} \Big] = \mathbb{P}_{\mathbbm{f}\mathbb{G}}^{\sigma} \big[ \mathscr{Q}  (\bm{\mu}) = \mathscr{Q} \big]. 
 		\end{flalign}	
 		
 		\noindent Here, on the left side of \eqref{ekmukfused}, the colored fused path ensemble $\mathcal{E}$ is sampled under the colored stochastic fused vertex measure $\mathbb{P}_{\FV; \bm{x}; \bm{y}; \bm{r}; \bm{t}}^{\sigma}$. On the right side of \eqref{ekmukfused}, the $(M; \sigma; \bm{\mathrm{R}})$-ascending sequence $\bm{\mu}$ of colored compositions is sampled under the measure $\mathbb{P}_{\mathbbm{f}\mathbb{G}; n; s; \bm{x}; \bm{y}; \bm{r}; \bm{t}}^{\sigma}$.
 		
 	\end{prop}
 
 	The following corollary of \Cref{fgsvfused} equates the joint law of the height functions (recall \Cref{FunctionsZ}) evaluated along the exit sites of an $M \times N$ rectangle, sampled under the colored stochastic fused vertex model, with the joint law of the number of zero entries in a family $\bm{\mu}$ of $n$-compositions, sampled under the $\mathbb{P}_{\mathbbm{f}\mathbb{G}}^{\sigma}$ measure. We omit its proof, which given \Cref{fgsvfused} is entirely analogous to that of \Cref{hfg2} given \Cref{fgsv}.

 	\begin{cor} 
 		
 		\label{hfg2fused}
 		
 		The joint law of all the height functions 
 		\begin{flalign}
 			\label{hcijfused}
 			\bigcup_{c = 1}^n \big( \mathfrak{h}_{\ge c}^{\rightarrow} (M, 1), \mathfrak{h}_{\ge c}^{\rightarrow} (M, 2), \ldots , \mathfrak{h}_{\ge c}^{\rightarrow} (M, N), \mathfrak{h}_{\ge c}^{\rightarrow} (M-1, N), \ldots , \mathfrak{h}_{\ge c}^{\rightarrow} (0, N) \big),
 		\end{flalign} 
 		
 		\noindent is equal to the joint law of all zero-entry counts 
 		\begin{flalign}
 			\label{mcjfused}
 			\bigcup_{c = 1}^n \Big( \mathfrak{m}_0^{\ge c} \big(\mu(1) \big), \mathfrak{m}_0^{\ge c} \big( \mu (2) \big), \ldots , \mathfrak{m}_0^{\ge c} \big( \mu (N) \big), \mathfrak{m}_0^{\ge c} \big( \mu(N+1) \big), \ldots , \mathfrak{m}_0^{\ge c} \big( \mu(M+N) \big) \Big).
 		\end{flalign} 
 		
 		\noindent Here, the height functions in \eqref{hcijfused} are associated with a colored fused path ensemble sampled under $\mathbb{P}_{\FV; \bm{x}; \bm{y}; \bm{r}; \bm{t}}^{\sigma}$, and the zero-entry counts in \eqref{mcjfused} are associated with a $(M; \sigma; \bm{\mathrm{R}})$-ascending sequence of $n$-compositions $\bm{\mu} = \big( \mu (0), \mu(1), \ldots , \mu(M+N) \big)$ sampled under $\mathbb{P}_{\mathbbm{f}\mathbb{G}; n; s; \bm{x}; \bm{y}; \bm{r}; \bm{t}}^{\sigma}$.
 		
 	\end{cor}

 	\section{Colored Line Ensembles for Fused Vertex Models} 
 	
 	\label{L0Fused}
 	
 	This section may be viewed as the fused counterpart of \Cref{L0}, in which we describe the colored line ensembles associated with fused stochastic vertex models. Unlike in \Cref{Ensemble}, the line ensembles we obtain will no longer be simple, which is a manifestation of the fact that horizontal edges in the associated stochastic fused vertex model can accommodate more than one arrow. Outside of this difference, the content in this section will closely follow that in \Cref{L0}. Throughout this section, we fix integers $n, M, N \ge 1$; a complex number $s \in \mathbb{C}$; a sequence of positive integers $\bm{\mathrm{R}} = (\mathrm{R}_1, \mathrm{R}_2, \ldots , \mathrm{R}_N)$; sequences of complex numbers $\bm{r} = (r_1, r_2, \ldots , r_N)$, $\bm{t} = (t_1, t_2, \ldots , t_M)$, $\bm{x} = (x_1, x_2, \ldots , x_N)$ and $\bm{y} = (y_1, y_2, \ldots , y_M)$, satisfying \eqref{rtxy} and $r_j = q^{-\mathrm{R}_j/2}$ for each $j \in \llbracket 1, N \rrbracket$; a composition $\bm{\ell} = (\ell_1, \ell_2, \ldots , \ell_n)$ of $\mathrm{R}_{[1,N]}$; and a function $\sigma : \llbracket 1, N \rrbracket \rightarrow \llbracket 1, n \rrbracket$, such that $\ell_i = \sum_{j\in \sigma^{-1} (i)} \mathrm{R}_j$ for each $i \in \llbracket 1, n \rrbracket$.

 	\subsection{Fused Colored Line Ensembles and Ascending Sequences} 
 	
 	\label{EnsembleFused}
 	
 	In this section, given an $(M; \sigma; \bm{\mathrm{R}})$-ascending sequence $\bm{\mu}$ of $n$-compositions, we associate a colored line ensemble (which, unlike in \Cref{lmu} and \Cref{lmu1}, need not be simple). This is done through the following definition (where in the below we recall the notion of a colored line ensemble from \Cref{cl}).
 	
 	\begin{definition} 
 		
 		\label{lmufused} 
 		
 		Let $\bm{\mu} = \big( \mu (0), \mu(1), \ldots , \mu(M+N) \big)$ denote an $(M; \sigma; \bm{\mathrm{R}})$-ascending sequence of $n$-compositions. The associated colored line ensemble $\mathds{L} = \mathds{L}_{\bm{\mu}} = \big( \mathds{L}^{(1)}, \mathds{L}^{(2)}, \ldots , \mathds{L}^{(n)} \big)$ on $\llbracket 0, M+N \rrbracket$ is defined as follows. For each $c \in \llbracket 1, n \rrbracket$ let $\mathds{L}^{(c)} = \mathds{L}_{\bm{\mu}}^{(c)} = \big( \mathds{L}_1^{(c)}, \mathds{L}_2^{(c)}, \ldots \big)$, where for each $k \ge 1$ the function $\mathds{L}_k^{(c)} = \mathds{L}_{k;\bm{\mu}}^{(c)} : \llbracket 0, M+N \rrbracket \rightarrow \mathbb{Z}$ is prescribed by setting  
 		\begin{flalign}
 			\label{lkcifused} 
 			\mathds{L}_k^{(c)} (i) = \ell_{[c,n]}  -\mathfrak{m}_{\le k - 1}^{\ge c} \big( \mu (i) \big), \qquad \text{for each  $i \in \llbracket 0, M+N \rrbracket$}.
 		\end{flalign}
 		
 		\noindent The fact that this defines a colored line ensemble follows from \Cref{lmu1fused} below. We moreover set the differences $\bm{\Lambda}^{(c)} = \big( \Lambda_1^{(c)}, \Lambda_2^{(c)}, \ldots \big)$ of $\mathds{L}$ by $\Lambda_k^{(c)} (i) = \mathds{L}_k^{(c)} (i)  - \mathds{L}_k^{(c+1)} (i)$, for each $(c, k, i) \in \llbracket 1, n \rrbracket \times \mathbb{Z}_{>0} \times \llbracket 0, M + N \rrbracket$, where $\mathds{L}_k^{(n+1)} : \llbracket 0, M + N \rrbracket \rightarrow \mathbb{Z}$ is defined by setting $\mathds{L}_k^{(n+1)} (i) = 0$ for each $(k, i) \in \mathbb{Z}_{>0} \times  \llbracket 0, M + N \rrbracket$. 
 		
 	\end{definition}

 	\begin{figure} 
 		
 		\begin{center}
 			\begin{tikzpicture}[
 				>=stealth, 
 				scale = .65]{
 					
 					\draw[dotted] (0, 2.5) -- (7, 2.5) node[right, scale = .75]{$0$};
 					\draw[dotted] (0, 3.5) -- (7, 3.5) node[right, scale = .75]{$1$};
 					\draw[dotted] (0, 4.5) -- (7, 4.5) node[right, scale = .75]{$2$};
 					\draw[dotted] (0, 5.5) -- (7, 5.5) node[right, scale = .75]{$3$};
 					
 					\draw[dotted] (0, 2.5) node[below = 2, scale = .75]{$0$} -- (0, 5.5);
 					\draw[dotted] (1, 2.5) node[below = 2, scale = .75]{$1$} -- (1, 5.5);
 					\draw[dotted] (2, 2.5) node[below = 2, scale = .75]{$2$} -- (2, 5.5);
 					\draw[dotted] (3, 2.5) node[below = 2, scale = .75]{$3$} -- (3, 5.5);
 					\draw[dotted] (4, 2.5) node[below = 2, scale = .75]{$4$} -- (4, 5.5);
 					\draw[dotted] (5, 2.5) node[below = 2, scale = .75]{$5$} -- (5, 5.5);
 					\draw[dotted] (6, 2.5) node[below = 2, scale = .75]{$6$} -- (6, 5.5);
 					\draw[dotted] (7, 2.5) node[below = 2, scale = .75]{$7$} -- (7, 5.5);
 					
					\draw[thick, blue] (0, 5.65) -- (2, 5.65) -- (3, 4.65) -- (4, 3.65) -- (6, 3.65) -- (7, 2.65);
					\draw[thick, blue] (0, 5.575) -- (2, 5.575) -- (3, 3.575) -- (5, 3.575) -- (6, 2.575) -- (7, 2.575);
					\draw[thick, blue] (0, 5.5) -- (2, 5.5) -- (3, 3.5) -- (5, 3.5) -- (6, 2.5) -- (7, 2.5);
					\draw[thick, blue] (0, 5.425) -- (1, 5.425) -- (2, 4.425) -- (3, 3.424) -- (4, 2.425) -- (7, 2.425);
					\draw[thick, blue] (0, 5.35) -- (1, 5.35) -- (2, 4.35) -- (3, 2.35) -- (7, 2.35);
					
					\draw[] (3.25, 4.5) circle [radius = 0] node[right, scale = .75]{$\mathds{L}_1^{(2)}$};
					\draw[] (2.5, 3.5) circle [radius = 0] node[left = 1, scale = .75]{$\mathds{L}_5^{(2)}$};

 					\draw[dotted] (12, 0) node[below = 2, scale = .75]{$0$} -- (12, 8);
 					\draw[dotted] (13, 0) node[below = 2, scale = .75]{$1$} -- (13, 8);
 					\draw[dotted] (14, 0) node[below = 2, scale = .75]{$2$} -- (14, 8);
 					\draw[dotted] (15, 0) node[below = 2, scale = .75]{$3$} -- (15, 8);
 					\draw[dotted] (16, 0) node[below = 2, scale = .75]{$4$} -- (16, 8);
 					\draw[dotted] (17, 0) node[below = 2, scale = .75]{$5$} -- (17, 8);
 					\draw[dotted] (18, 0) node[below = 2, scale = .75]{$6$} -- (18, 8);
 					\draw[dotted] (19, 0) node[below = 2, scale = .75]{$7$} -- (19, 8);
 					
 					\draw[thick, violet] (12, 8.15) -- (15, 5.15) -- (16, 3.15) -- (17, 3.15) -- (18, 2.15) -- (19, .15);
 					\draw[thick, violet] (12, 8.075) -- (14, 6.075) -- (15, 4.075) -- (16, 3.075) -- (17, 2.075) -- (18, .075) -- (19, .075);
 					\draw[thick, violet] (12, 8) -- (13, 5) -- (14, 5) -- (15, 3) -- (16, 3) -- (17, 1) -- (18, 0) -- (19, 0);
 					\draw[thick, violet] (12, 7.925) -- (13, 4.925) -- (14, 3.925) -- (15, 2.925) -- (16, .925) -- (17, -.075) -- (19, -.075);
 					\draw[thick, violet] (12, 7.85) -- (13, 4.85) -- (14, 3.85) -- (15, 1.85) -- (16, .85) -- (17, -.15) -- (19, -.15);
 					
 					\draw[dotted] (12, 0) -- (19, 0) node[right, scale = .75]{$0$};
 					\draw[dotted] (12, 1) -- (19, 1) node[right, scale = .75]{$1$};
 					\draw[dotted] (12, 2) -- (19, 2) node[right, scale = .75]{$2$};
 					\draw[dotted] (12, 3) -- (19, 3) node[right, scale = .75]{$3$};
 					\draw[dotted] (12, 4) -- (19, 4) node[right, scale = .75]{$4$};
 					\draw[dotted] (12, 5) -- (19, 5) node[right, scale = .75]{$5$};
 					\draw[dotted] (12, 6) -- (19, 6) node[right, scale = .75]{$6$};
 					\draw[dotted] (12, 7) -- (19, 7) node[right, scale = .75]{$7$};
 					\draw[dotted] (12, 8) -- (19, 8) node[right, scale = .75]{$8$};
 					
 					\draw[] (14, 6.35) circle [radius = 0] node[right, scale = .75]{$\mathds{L}_1^{(1)}$}; 
 					\draw[] (12.9, 5.35) circle [radius = 0] node[left, scale = .75]{$\mathds{L}_5^{(1)}$};
 				}
 			\end{tikzpicture}
 		\end{center}
 		
 		\caption{\label{00lmufused2} Depicted above is the colored line ensemble associated with the fused path ensemble in \Cref{00lmufused}.} 
 		
 	\end{figure}
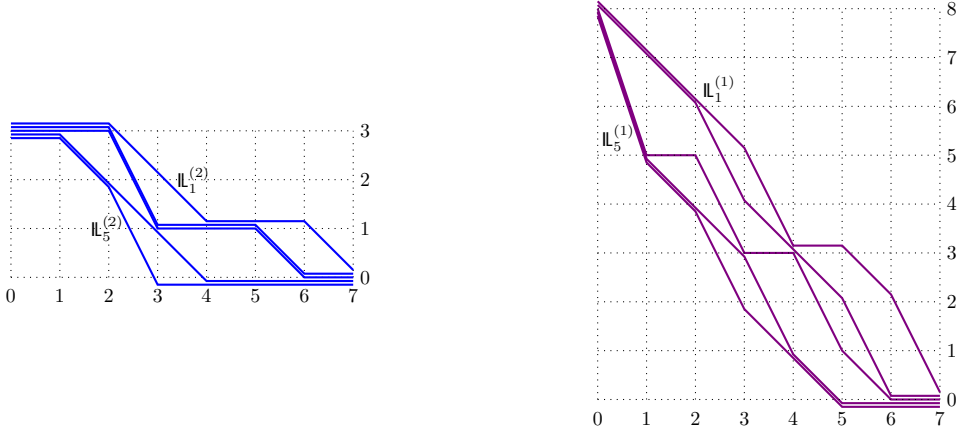

 	\begin{lem} 
 		
 		\label{lmu1fused} 
 		
 		Adopting the notation and assumptions of \Cref{lmufused}, $\mathds{L}_{\bm{\mu}}$ is a colored line ensemble, which satisfies the following three properties for any $c \in \llbracket 1, n \rrbracket$, $k \in \mathbb{Z}_{>0}$, and $i \in \llbracket 0, M + N \rrbracket$.   
 		\begin{enumerate}
 			\item We have $\mathds{L}_1^{(c)} (i) \ge \mathds{L}_2^{(c)} (i) \ge \cdots $ and $\mathds{L}_k^{(1)} (i) \ge \mathds{L}_k^{(2)} (i) \ge \cdots$.  
 			\item We have $\Lambda_k^{(c)} (i) - \Lambda_{k+1}^{(c)} (i) = \mathfrak{m}_k \big( \mu^{(c)} (i) \big)$. 
 			\item If $i \ge 1$, we have $\Lambda_k^{(c)} (i-1) - \Lambda_k^{(c)} (i) = \mathfrak{Q}_c^{\bm{\mu}} (k, i)$. 
 		\end{enumerate} 
 		
 	\end{lem} 
 
 	\begin{proof}
 		
 		The proofs of the first two parts of the lemma are very similar to those of \Cref{lmu1} and are therefore omitted. The third follows from the fact that 
 		\begin{flalign*}
 			\Lambda_k^{(c)} (i-1) - \Lambda_k^{(c)} (i) & = \big( \mathds{L}_k^{(c)} (i-1) - \mathds{L}_k^{(c)} (i) \big) - \big( \mathds{L}_k^{(c+1)} (i-1) - \mathds{L}_k^{(c+1)} (i) \big) \\ 
 			& = \mathfrak{m}_{\le k-1}^{\ge c} \big( \mu(i-1) \big) - \mathfrak{m}_{\le k-1}^{\ge c+1} \big( \mu(i) \big) \\
 			& = \mathfrak{Q}_{[c,n]}^{\bm{\mu}} (k, i) - \mathfrak{Q}_{[c+1, n]}^{\bm{\mu}} (k, i) = \mathfrak{Q}_c^{\bm{\mu}} (k, i) \ge 0,
 		\end{flalign*}
 	
 		\noindent where in the first statement we used the definition of $\bm{\Lambda}^{(c)}$ from \Cref{lmufused}; in the second we used \eqref{lkcifused}; and in the third we used \eqref{qmucfused}. The second and third statements of the lemma together imply that each $\bm{\Lambda}^{(c)}$ is a line ensemble, and thus $\mathds{L}_{\bm{\mu}}$ is a colored line ensemble by \Cref{cl}. 
  	\end{proof}

 	\begin{rem} 
 		
 		As in \Cref{mnmurow2fused}, we may interpret $\bm{\mu}$ as associated with a colored fused path ensemble $\mathcal{E}_{\bm{\mu}} \in \mathfrak{P}_{\mathbbm{f}\mathbb{G}} (M; \sigma; \bm{\mathrm{R}})$ on $\mathbb{Z}_{\le 0} \times \llbracket 1, M + N \rrbracket$. Then $\mathds{L}_k^{(c)} (i) = \mathfrak{h}_{\ge c}^{\leftarrow} (-k, i)$, where the height function $\mathfrak{h}_{\ge c}^{\leftarrow}$ is with respect to $\mathcal{E}_{\bm{\mu}}$; stated alternatively, $\mathds{L}_k^{(c)} (i)$ denotes the number of arrows with color at least $c$ that horizontally exit the column $\{ x = -k \}$ strictly above the vertex $(-k, i)$. See \Cref{00lmufused2} for a depiction of the colored line ensemble in the example at the end of \Cref{mnmurow2fused} (and shown in \Cref{00lmufused}). 
 		
 	\end{rem}

 	By \Cref{lmu1fused}, \Cref{lmufused} associates a colored line ensemble to a given $(M; \sigma; \bm{\mathrm{R}})$-ascending sequence of $n$-compositions. Since the former are in bijection with colored higher spin path ensembles in $\mathfrak{P}_{\mathbbm{f}\mathbb{G}} (M; \sigma; \bm{\mathrm{R}})$ by \Cref{mnmurow2fused}, this associates a colored line ensemble with any element of $\mathfrak{P}_{\mathbbm{f}\mathbb{G}} (M; \sigma; \bm{\mathrm{R}})$. The following definition is towards the reverse direction; it associates a colored fused path ensemble with a colored line ensemble $\mathds{L}$.

 	\begin{definition} 
 		
 		\label{le2fused}
 		
 		Fix a colored line ensemble $\mathds{L} = \big(\mathds{L}^{(1)}, \mathds{L}^{(2)}, \ldots , \mathds{L}^{(n)} \big)$, and for each $c \in \llbracket 1, n \rrbracket$ denote $\mathds{L}^{(c)} = \big( \mathds{L}_1^{(c)}, \mathds{L}_2^{(c)}, \ldots \big)$. For any $v = (-k, i) \in \mathbb{Z}_{\le 0} \times \llbracket 1, M + N \rrbracket$, define the arrow configuration $\big( \bm{A}^{\mathds{L}} (v), \bm{B}^{\mathds{L}} (v); \bm{C}^{\mathds{L}} (v), \bm{D}^{\mathds{L}} (v) \big)$ as follows. For each $c \in \llbracket 1, n \rrbracket$ and set 
 		\begin{flalign*}
 			& A_c^{\mathds{L}} (v) = \Lambda_k^{(c)} (i-1) - \Lambda_{k+1}^{(c)} (i-1) = \mathds{L}_k^{(c)} (i-1) - \mathds{L}_{k+1}^{(c)} (i-1) - \big( \mathds{L}_k^{(c+1)} (i-1) - \mathds{L}_{k+1}^{(c+1)} (i-1) \big); \\ 
 			& B_c^{\mathds{L}} (v) = \Lambda_{k+1}^{(c)} (i-1) - \Lambda_{k+1}^{(c)} (i) = \mathds{L}_{k+1}^{(c)} (i-1) - \mathds{L}_{k+1}^{(c)} (i) - \big( \mathds{L}_{k+1}^{(c+1)} (i-1) - \mathds{L}_{k+1}^{(c+1)} (i) \big); \\
 			& C_c^{\mathds{L}} (v) = \Lambda_k^{(c)} (i) - \Lambda_{k+1}^{(c)} (i) = \mathds{L}_k^{(c)} (i) - \mathds{L}_{k+1}^{(c)} (i) - \big( \mathds{L}_k^{(c+1)} (i) - \mathds{L}_{k+1}^{(c+1)} (i) \big); \\
 			& D_c^{\mathds{L}} (v) = \Lambda_k^{(c)} (i-1) - \Lambda_k^{(c)} (i) = \mathds{L}_k^{(c)} (i-1) - \mathds{L}_k^{(c)} (i) - \big( \mathds{L}_k^{(c+1)} (i-1) - \mathds{L}_k^{(c+1)} (i) \big),
 		\end{flalign*}
 		
 		\noindent where we observe that all four quantities are nonnegative since $\bm{\Lambda}^{(c)}$ is a line ensemble. This assignment of arrow configurations is consistent and therefore defines a colored fused path ensemble $\mathcal{E}^{\mathds{L}}$ associated with the colored line ensemble $\mathds{L}$. 
 		
 	\end{definition} 
 	
 	The following lemma indicates that the assocations from \Cref{lmufused} and \Cref{le2fused} are compatible; we omit its proof, which is a quick verification using the second and third properties of \Cref{lmu1fused}.

 	\begin{lem} 
 		
 		\label{llmufused}

 		If $\mathcal{E}^{\mathds{L}} = \mathcal{E}_{\bm{\mu}}$ for some $(M; \sigma; \bm{\mathrm{R}})$-ascending sequence $\bm{\mu}$ of $n$-compositions, then $\mathds{L}$ is associated with $\bm{\mu}$ in the sense of \Cref{lmufused}. 
 		
 	\end{lem}

 	\subsection{Properties of Random Fused Colored Line Ensembles} 
 	
 	\label{ConditionWFused}
 	
 	In this section we discuss some properties of colored line ensembles $\mathds{L}_{\bm{\mu}}$ associated (recall \Cref{lmufused}) with an $(M; \sigma; \bm{\mathrm{R}})$-ascending sequence $\bm{\mu}$ of $n$-compositions sampled from the measure $\mathbb{P}_{\mathbbm{f}\mathbb{G}}^{\sigma}$ (recall \Cref{measurefgfused}). Let us first give notation to this law on colored line ensembles.
 	
 	\begin{definition} 
 		
 		\label{llmu0fused} 
 		
 		Let $\mathbb{P}_{\cL}^{\sigma} = \mathbb{P}_{\cL; n; s; \bm{x}; \bm{y}; \bm{r}; \bm{t}}^{\sigma}$ denote the law of a colored line ensemble $\mathds{L}_{\bm{\mu}}$ associated with a random $(M; \sigma; \bm{\mathrm{R}})$-ascending sequence $\bm{\mu}$ of $n$-compositions (as in \Cref{lmufused}) sampled from the measure $\mathbb{P}_{\mathbbm{f}\mathbb{G}; n; s; \bm{x}; \bm{y}; \bm{r}; \bm{t}}^{\sigma}$. 
 		
 	\end{definition}

 	The following result, which is a quick consequence of \Cref{hfg2fused}, provides under this setup a matching in law between the top curves of $\mathds{L}$ (under $\mathbb{P}_{\cL}^{\sigma}$) and the height functions for a colored stochastic fused vertex model (recall \Cref{probabilityvertexfused}). 
 	
 	\begin{thm} 
 		
 		\label{lmu2fused} 
 		
 		Sample a colored line ensemble $\mathds{L}$ on $\llbracket 0, M+N \rrbracket$ from the measure $\mathbb{P}_{\cL; n; s; \bm{x}; \bm{y}; \bm{r}; \bm{t}}^{\sigma}$, and sample a random colored fused path ensemble $\mathcal{E}$ under $\mathbb{P}_{\FV; \bm{x}; \bm{y}; \bm{r}; \bm{t}}^{\sigma}$.  For each $c \in \llbracket 1, n \rrbracket$, define the function $H_c : \llbracket 0, M + N \rrbracket \rightarrow \mathbb{Z}$ by setting  
 		\begin{flalign*} 
 			H_c (k) = \mathfrak{h}_{\ge c}^{\leftarrow} (M, k), \quad \text{if $k \in \llbracket 0, N \rrbracket$}; \qquad H_c (k) = \mathfrak{h}_{\ge c}^{\leftarrow} (M+N-k,  N), \quad \text{if $k \in \llbracket N, M+N \rrbracket$},
 		\end{flalign*} 
 		
 		\noindent where $\mathfrak{h}_{\ge c}^{\leftarrow}$ is the height function associated with $\mathcal{E}$. Then, the joint law of $\big( \mathds{L}_1^{(1)}, \mathds{L}_1^{(2)}, \ldots , \mathds{L}_1^{(n)} \big)$ is the same as that of $(H_1, H_2, \ldots,  H_n)$.
 		
 	\end{thm} 
 	
 	\begin{proof}
 		
 		Since $\mathfrak{h}_{\ge c}^{\leftarrow} (i, j) = \ell_{[c, n]} - \mathfrak{h}_{\ge c}^{\rightarrow} (i, j)$ holds for any integer $c \in \llbracket 1, n \rrbracket$ and vertex $(i, j) \in \big\{ (M, 0), (M, 1), \ldots , (M,N), (M, N-1), \ldots , (0, N) \big\}$ along the northeast boundary of $\llbracket 0, M \rrbracket \times \llbracket 0, N \rrbracket$, this theorem follows from \Cref{lmufused} and \Cref{hfg2fused}. 
 	\end{proof}

 	The next theorem explains the effect of conditioning on some of the curves in $\mathds{L}$ (the Gibbs property), if $\bm{\mu}$ is sampled under the $\mathbb{P}_{\cL}^{\sigma}$ measure, which is given by the following theorem; its proof is omitted, as it is very similar to that of \Cref{conditionl}. Below, we recall the vertex weights $W_{x; r, s}$ from \Cref{wabcdw}; the association of a colored line ensemble with an ascending sequence of $n$-compositions from \Cref{lmufused}; and the notation from \Cref{le2fused}.

 	\begin{thm} 
 		
 		\label{conditionlfused}
 		
 		Sample $\mathds{L} = \mathds{L}_{\bm{\mu}}$ under $\mathbb{P}_{\cL; n; s; \bm{x}; \bm{y}; \bm{r}; \bm{t}}^{\sigma}$. Fix integers $j > i \ge 0$ and $u, v \in \llbracket 0, M + N \rrbracket$ with $u < v$; set $i_0 = \max \{ i, 1 \}$; and condition on the curves $\mathds{L}_k^{(c)} (m)$ for all $c \in \llbracket 1, n \rrbracket$ and $(k, m) \in \big( \mathbb{Z}_{> 0} \times \llbracket 0, M+N \rrbracket \big) \setminus \big( \llbracket i+1, j \rrbracket \times \llbracket u, v-1 \rrbracket \big)$. For any colored line ensemble $\bm{\bm{\mathsf{l}}}$ that is $\llbracket i+1, j \rrbracket \times \llbracket u, v-1 \rrbracket$-compatible with $\mathds{L}$, we have
 		\begin{flalign}
 			\label{llprobabilityfused}
 			\begin{aligned}
 				\mathbb{P} [\mathds{L} = \bm{\bm{\mathsf{l}}}] & = \mathcal{Z}^{-1} \cdot \displaystyle\prod_{k=i_0}^j \displaystyle\prod_{\substack{m \in \llbracket u, v \rrbracket \\ m \le N}} W_{x_m; r_m, s} \big( \bm{A}^{\bm{\bm{\mathsf{l}}}} (-k, m), \bm{B}^{\bm{\bm{\mathsf{l}}}} (-k, m); \bm{C}^{\bm{\bm{\mathsf{l}}}} (-k, m), \bm{D}^{\bm{\bm{\mathsf{l}}}} (-k, m) \big) \\
 				& \qquad \times \displaystyle\prod_{k=i_0}^j \displaystyle\prod_{\substack{m \in \llbracket u, v \rrbracket \\ m > N}} W_{y_{m-N}; t_{m-N}, s} \big( \bm{A}^{\bm{\bm{\mathsf{l}}}} (-k, m), \bm{B}^{\bm{\bm{\mathsf{l}}}} (-k, m); \bm{C}^{\bm{\bm{\mathsf{l}}}} (-k, m), \bm{D}^{\bm{\bm{\mathsf{l}}}} (-k, m) \big).
 			\end{aligned}
 		\end{flalign}
 		
 		\noindent Here, the probability on the left side of \eqref{llprobabilityfused} is with respect to the conditional law of $\mathds{L}$. Moreover, $\mathcal{Z}$ is a normalizing constant defined so that the sum of the right side of \eqref{llprobabilityfused}, over all colored line ensembles $\bm{\bm{\mathsf{l}}}$ that are $\llbracket i+1, j \rrbracket \times \llbracket u, v-1 \rrbracket$-compatible with $\mathds{L}_{\bm{\mu}}$, is equal to $1$.
 		
 	\end{thm}

 	Let us also describe color merging properties for line ensembles sampled according to $\mathbb{P}_{\cL}^{\sigma}$, which will be parallel to those discussed in \Cref{ColorL}. To that end, we have the following proposition generalizing \Cref{lmerge}, where below we recall the definition of $\vartheta$ from \eqref{mu121} (as a function on $\Comp_n$), \eqref{1210} (as a functional on the set of functions $\varsigma : \llbracket 1, N \rrbracket \rightarrow \llbracket 1, n \rrbracket$), and \eqref{12mu0} (as a function on sequences of $n$-compositions); we also recall that $\breve{\sigma} = \vartheta (\sigma)$ from \eqref{1210}.
	
	\begin{prop}
		
		\label{lmergefused}
		
		Sample $\mathds{L} = \big( \mathds{L}^{(1)}, \mathds{L}^{(2)}, \ldots , \mathds{L}^{(n)} \big)$ under $\mathbb{P}_{\cL; n; s; \bm{x}; \bm{y}; \bm{r}; \bm{t}}^{\sigma}$. Then the joint law of the colored line ensemble $\breve{\mathds{L}} = \big( \mathds{L}^{(1)}, \mathds{L}^{(3)}, \mathds{L}^{(4)}, \ldots , \mathds{L}^{(n)} \big)$ (with $n-1$ colors) is given by $\mathbb{P}_{\cL; n-1; s; \bm{x}; \bm{y}; \bm{r}; \bm{t}}^{\breve{\sigma}}$. 
		
	\end{prop} 	
 	
 	To establish \Cref{lmergefused}, we require the following lemma that is parallel to \Cref{mergemunu}. 
 	
 	\begin{lem} 
 		
 		\label{mergemunufused}
 		
 		Fix an integer $k \ge 1$; $(n-1)$-compositions $\breve{\mu}, \breve{\nu}, \breve{\kappa} \in \Comp_{n-1}$; and $n$-compositions $\kappa, \mu \in \Comp_n$, such that $\vartheta (\kappa) = \breve{\kappa}$ and $\vartheta (\mu) = \breve{\mu}$. We have 
 		\begin{flalign}
 			\label{fgmerge} 
 			\displaystyle\sum_{\substack{\varsigma : \llbracket 1, N \rrbracket \rightarrow \llbracket 1, n \rrbracket \\ \vartheta (\varsigma) = \breve{\sigma}}} \displaystyle\sum_{\substack{\nu \in \Comp_n \\ \vartheta(\nu) = \breve{\nu}}} \mathbbm{f}_{\mu/\nu; s}^{\varsigma} (\bm{x}; \bm{r}) = \mathbbm{f}_{\breve{\mu} / \breve{\nu}; s}^{\breve{\sigma}} (\bm{x}; \bm{r}); \qquad \displaystyle\sum_{\substack{\nu \in \Comp_n \\ \vartheta(\nu) = \breve{\nu}}} \mathbb{G}_{\nu / \kappa; s} (\bm{y}; \bm{t}) = \mathbb{G}_{\breve{\nu} / \breve{\kappa}; s} (\bm{y}; \bm{t}).
 		\end{flalign}
 		
 	\end{lem} 
 	
 	\begin{proof}[Proof] 
 		
 		Combining \Cref{mergemunu} with \Cref{fgfg} yields the first statement of \eqref{fgmerge}, as well as the second one if $t_i^2 \in q^{\mathbb{Z}_{< 0}}$ for each $i \in \llbracket 1, M \rrbracket$. The fact that the second statement in \eqref{fgmerge} holds for arbitrary $\bm{t}$ then follows from uniqueness of analytic continuation, as $\mathbb{G}$ is rational in $\bm{t}$ (since the $W$-weights of \Cref{wabcdw} are). 
 	\end{proof} 
 	
 	The next lemma is parallel to \Cref{nn1mu}. Its proof given \Cref{mergemunufused} is entirely analogous to that of \Cref{nn1mu} given \Cref{mergemunu} and is therefore omitted. 

 	\begin{lem}
 		
 		\label{nn1mufused} 
 		
 		If an $(M; \sigma; \bm{\mathrm{R}})$-ascending sequence of $n$-compositions $\bm{\mu}$ is sampled from $\mathbb{P}_{\mathbbm{f}\mathbb{G}; n; s; \bm{x}; \bm{y}}^{\sigma}$, then the $(M; \breve{\sigma})$-ascending sequence $\vartheta(\bm{\mu})$ of $(n-1)$-compositions has law $\mathbb{P}_{\mathbbm{f}\mathbb{G}; n-1; s; \bm{x}; \bm{y}}^{\breve{\sigma}}$. 
 		
 	\end{lem} 

 	Now we can establish \Cref{lmergefused}.
 	
 	\begin{proof}[Proof of \Cref{lmergefused}]
 		
 		Sample $\bm{\mu}$ under $\mathbb{P}_{\mathbbm{f}\mathbb{G}; n; s; \bm{x}; \bm{y}; \bm{r}; \bm{t}}^{\sigma}$ (recall \Cref{measurefgfused}). By \Cref{llmu0fused}, we may identify $\mathds{L}$ as the colored line ensemble $\mathds{L}_{\bm{\mu}}$ associated with $\bm{\mu}$. By \Cref{lmufused}, the $n-1$ line ensembles $\breve{\mathds{L}} = \big( \mathds{L}^{(1)}, \mathds{L}^{(3)}, \ldots , \mathds{L}^{(n)} \big)$ in $\mathds{L} = \mathds{L}_{\bm{\mu}}$ constitute the colored line ensemble associated with $\vartheta (\bm{\mu})$, which has law $\mathbb{P}_{\mathbbm{f}\mathbb{G}; n-1; s; \bm{x}; \bm{y}; \bm{r}; \bm{t}}^{\breve{\sigma}}$ by \Cref{nn1mufused}. Hence, again by \Cref{llmu0fused}, $\breve{\mathds{L}}$ has law  $\mathbb{P}_{\cL; n-1; s; \bm{x}; \bm{y}; \bm{r}; \bm{t}}^{\breve{\sigma}}$, thereby establishing the proposition.
 	\end{proof}
 	
 	\begin{rem}
 		
 		The above discussion describes the merging of colors $1$ and $2$. As in \Cref{12interval}, it is more generally possible to merge several (disjoint) intervals of colors, which would correspond in \Cref{lmergefused} to omitting line ensembles $\mathds{L}^{(i)}$, for $i \in \llbracket 2, n \rrbracket$ arbitrary (depending on the corresponding merged color intervals), in $\mathds{L}$. 
 		
 	\end{rem}

 	\section{Line Ensembles for Discrete Time $q$-Boson Models}
 	
 	\label{L2} 
 	
 	In this section we specialize the results of \Cref{L0Fused} to the discrete time $q$-boson model, which involves setting each $\mathrm{R}_j$ there equal to $1$. Throughout this section, we adopt the notation of \Cref{L0Fused} (where here we do not necessarily assume that \eqref{rtxy} holds), and set $\mathrm{R}_j = 1$ for all $j \in \llbracket 1, n \rrbracket$, so that each $r_j = q^{-1/2}$.

 	\subsection{Colored Stochastic Higher Spin Vertex and $q$-Boson Models}
 	
 	\label{LModelDiscrete} 
 	
 	In this section we describe the $\mathrm{R} = 1$ case of the stochastic fused vertex models from \Cref{FusedPath}. Recall that we have set each $\mathrm{R}_j = 1$, and thus each $r_j = q^{-1/2}$. 
 	
 	Under this specialization, it is quickly verified that the $U_{z; q^{-1/2}, s}$-weights (recall \Cref{wzabcd}) permit at most one arrow along any horizontal edge, that is, $U_{z; q^{-1/2}, s} (\bm{A}, \bm{B}; \bm{C}, \bm{D}) = 0$ unless there exist indices $b, d \in \llbracket 0, n \rrbracket$ such that $\bm{B} = \bm{e}_b$ and $\bm{D} = \bm{e}_d$. Therefore, the associated colored stochastic fused vertex model (recall \Cref{FusedPath}) is called the \emph{colored stochastic higher spin vertex model}; it originally appeared in \cite[Appendix A]{SRM}. For any $\bm{A}, \bm{C} \in \mathbb{Z}_{\ge 0}^n$ and $b, d \in \llbracket 0, n \rrbracket$, we denote these specialized stochastic weights by 
 	\begin{flalign*} 
 		U_{z; s}^{\hs} (\bm{A}, b; \bm{C}, d) = U_{z; q^{-1/2}, s} (\bm{A}, \bm{e}_b; \bm{C}, \bm{e}_d).
 	\end{flalign*} 

	\noindent These weights are depicted in the second to last row of \Cref{qweight}. We also denote the associated probability measure on path ensembles (recall \Cref{probabilityvertexfused}) by $\mathbb{P}_{\hs; \bm{x}; \bm{y}; \bm{t}}^{\sigma} = \mathbb{P}_{\hs; \bm{x}; \bm{y}; q^{-1/2}; \bm{t}}^{\sigma}$.

 	\begin{figure} 
 		\begin{center}
 			
 			\begin{tikzpicture}[
 				>=stealth,
 				scale = .9
 				]	
 				\draw[-, black] (-7.5, 3.1) -- (7.5, 3.1);
 				\draw[-, black] (-7.5, -3.1) -- (7.5, -3.1);
 				\draw[-, black] (-7.5, -1.1) -- (7.5, -1.1);
 				\draw[-, black] (-7.5, -.4) -- (7.5, -.4);
 				\draw[-, black] (-7.5, 2.4) -- (7.5, 2.4);
 				\draw[-, black] (-7.5, -2.1) -- (7.5, -2.1); 
 				\draw[-, black] (-7.5, -3.1) -- (-7.5, 3.1);
 				\draw[-, black] (7.5, -3.1) -- (7.5, 3.1);
 				\draw[-, black] (-5, -3.1) -- (-5, 2.4);
 				\draw[-, black] (5, -3.1) -- (5, 3.1);
 				\draw[-, black] (-2.5, -3.1) -- (-2.5, 2.4);
 				\draw[-, black] (2.5, -3.1) -- (2.5, 2.4);
 				\draw[-, black] (0, -3.1) -- (0, 3.1);
 				
 				\draw[->, thick, blue] (-6.3, .1) -- (-6.3, 1.9); 
 				\draw[->, thick, green] (-6.2, .1) -- (-6.2, 1.9); 
 				\draw[->, thick, blue] (-3.8, .1) -- (-3.8, 1) -- (-2.85, 1);
 				\draw[->, thick, green] (-3.7, .1) -- (-3.7, 1.9);
 				\draw[->, thick, blue] (-1.35, .1) -- (-1.35, 1.9);
 				\draw[->, thick, green] (-1.25, .1) -- (-1.25, 1.9);
 				\draw[->, thick,  orange] (-2.15, 1.1) -- (-1.15, 1.1) -- (-1.15, 1.9);
 				
 				\draw[->, thick, red] (.35, 1) -- (1.15, 1) -- (1.15, 1.9);
 				\draw[->, thick, blue] (1.25, .1) -- (1.25, 1.9);
 				\draw[->, thick, green] (1.35, .1) -- (1.35, 1.1) -- (2.15, 1.1);
 				\draw[->, thick, blue] (3.65, .1) -- (3.65, 1) -- (4.65, 1);
 				\draw[->, thick, green] (3.75, .1) -- (3.75, 1.9);
 				\draw[->, thick, orange] (2.85, 1.1) -- (3.85, 1.1) -- (3.85, 1.9); 
 				\draw[->, thick, red] (5.35, 1) -- (7.15, 1); 
 				\draw[->, thick, blue] (6.2, .1) -- (6.2, 1.9);
 				\draw[->, thick, green] (6.3, .1) -- (6.3, 1.9); 
 				\filldraw[fill=gray!50!white, draw=black] (-2.85, 1) circle [radius=0] node [black, right = -1, scale = .7] {$i$};
 				\filldraw[fill=gray!50!white, draw=black] (2.15, 1) circle [radius=0] node [black, right = -1, scale = .7] {$j$};
 				\filldraw[fill=gray!50!white, draw=black] (4.65, 1) circle [radius=0] node [black, right = -1, scale = .7] {$i$};
 				\filldraw[fill=gray!50!white, draw=black] (7.15, 1) circle [radius=0] node [black, right = -1, scale = .7] {$i$};
 				\filldraw[fill=gray!50!white, draw=black] (5.35, 1) circle [radius=0] node [black, left = -1, scale = .7] {$i$};
 				\filldraw[fill=gray!50!white, draw=black] (2.85, 1) circle [radius=0] node [black, left = -1, scale = .7] {$j$};
 				\filldraw[fill=gray!50!white, draw=black] (.35, 1) circle [radius=0] node [black, left = -1, scale = .7] {$i$};
 				\filldraw[fill=gray!50!white, draw=black] (-2.15, 1) circle [radius=0] node [black, left = -1, scale = .7] {$i$};
 				\filldraw[fill=gray!50!white, draw=black] (-6.25, 1.9) circle [radius=0] node [black, above = -1, scale = .7] {$\bm{A}$};
 				\filldraw[fill=gray!50!white, draw=black] (-3.75, 1.9) circle [radius=0] node [black, above = -1, scale = .7] {$\bm{A}_i^-$};
 				\filldraw[fill=gray!50!white, draw=black] (-1.25, 1.9) circle [radius=0] node [black, above = -1, scale = .7] {$\bm{A}_i^+$};
 				\filldraw[fill=gray!50!white, draw=black] (1.25, 1.9) circle [radius=0] node [black, above = -1, scale = .65] {$\bm{A}_{ij}^{+-}$};
 				\filldraw[fill=gray!50!white, draw=black] (3.75, 1.9) circle [radius=0] node [black, above = -1, scale = .65] {$\bm{A}_{ji}^{+-}$};
 				\filldraw[fill=gray!50!white, draw=black] (6.25, 1.9) circle [radius=0] node [black, above = -1, scale = .7] {$\bm{A}$};
 				\filldraw[fill=gray!50!white, draw=black] (-6.25, .1) circle [radius=0] node [black, below = -1, scale = .7] {$\bm{A}$};
 				\filldraw[fill=gray!50!white, draw=black] (-3.75, .1) circle [radius=0] node [black, below = -1, scale = .7] {$\bm{A}$};
 				\filldraw[fill=gray!50!white, draw=black] (-1.25, .1) circle [radius=0] node [black, below = -1, scale = .7] {$\bm{A}$};
 				\filldraw[fill=gray!50!white, draw=black] (1.25, .1) circle [radius=0] node [black, below = -1, scale = .7] {$\bm{A}$};
 				\filldraw[fill=gray!50!white, draw=black] (3.75, .1) circle [radius=0] node [black, below = -1, scale = .7] {$\bm{A}$};
 				\filldraw[fill=gray!50!white, draw=black] (6.25, .1) circle [radius=0] node [black, below = -1, scale = .7] {$\bm{A}$};
 				\filldraw[fill=gray!50!white, draw=black] (-3.75, 2.75) circle [radius=0] node [black] {$1 \le i \le n$};
 				\filldraw[fill=gray!50!white, draw=black] (2.5, 2.75) circle [radius=0] node [black] {$1 \le i < j \le n$}; 
 				\filldraw[fill=gray!50!white, draw=black] (6.25, 2.75) circle [radius=0] node [black] {};
 				\filldraw[fill=gray!50!white, draw=black] (-6.25, -.75) circle [radius=0] node [black, scale = .8] {$(\bm{A}, 0; \bm{A}, 0)$};
 				\filldraw[fill=gray!50!white, draw=black] (-3.75, -.75) circle [radius=0] node [black, scale = .8] {$\big( \bm{A}, 0; \bm{A}_i^-, i \big)$};
 				\filldraw[fill=gray!50!white, draw=black] (-1.25, -.75) circle [radius=0] node [black, scale = .8] {$\big( \bm{A}, i; \bm{A}_i^+, 0 \big)$};
 				\filldraw[fill=gray!50!white, draw=black] (1.25, -.75) circle [radius=0] node [black, scale = .8] {$\big( \bm{A}, i; \bm{A}_{ij}^{+-}, j \big)$};
 				\filldraw[fill=gray!50!white, draw=black] (3.75, -.75) circle [radius=0] node [black, scale = .8] {$\big( \bm{A}, j; \bm{A}_{ji}^{+-}, i \big)$};
 				\filldraw[fill=gray!50!white, draw=black] (6.25, -.75) circle [radius=0] node [black, scale = .8] {$(\bm{A}, i; \bm{A}, i)$};
 				
 				\filldraw[fill=gray!50!white, draw=black] (-6.25, -2.6) circle [radius=0] node [black, scale = .8] {$\displaystyle\frac{1 + \nu q^{A_{[1, n]}}}{1 + \nu}$};
 				\filldraw[fill=gray!50!white, draw=black] (-3.75, -2.6) circle [radius=0] node [black, scale = .75] {$\displaystyle\frac{\nu (1 - q^{A_i}) q^{A_{[i+1, n]}}}{1+\nu}$};
 				\filldraw[fill=gray!50!white, draw=black] (-1.25, -2.6) circle [radius=0] node [black, scale = .8] {$\displaystyle\frac{1 + \nu q^{A_{[1, n]}}}{1+ \nu}$};
 				\filldraw[fill=gray!50!white, draw=black] (1.25, -2.6) circle [radius=0] node [black, scale = .75] {$\displaystyle\frac{\nu (1 - q^{A_j}) q^{A_{[j+1,n]}}}{1+ \nu}$};
 				\filldraw[fill=gray!50!white, draw=black] (3.75, -2.6) circle [radius=0] node [black, scale = .75] {$\displaystyle\frac{\nu (1 - q^{A_i}) q^{A_{[i+1,n]}}}{1+\nu}$};
 				\filldraw[fill=gray!50!white, draw=black] (6.25, -2.6) circle [radius=0] node [black, scale = .75] {$\displaystyle\frac{\nu (1- q^{A_i})q^{A_{[i+1,n]}}}{1+\nu}$};
 				
 				\filldraw[fill=gray!50!white, draw=black] (-6.25, -1.6) circle [radius=0] node [black, scale = .8] {$\displaystyle\frac{1 - s^2 x q^{A_{[1, n]}}}{1 - s^2 x}$};
 				\filldraw[fill=gray!50!white, draw=black] (-3.75, -1.6) circle [radius=0] node [black, scale = .7] {$\displaystyle\frac{s^2 x (q^{A_i}-1) q^{A_{[i+1, n]}}}{1 - s^2 x}$};
 				\filldraw[fill=gray!50!white, draw=black] (-1.25, -1.6) circle [radius=0] node [black, scale = .8] {$\displaystyle\frac{1 - s^2 q^{A_{[1, n]}}}{1 - s^2 x}$};
 				\filldraw[fill=gray!50!white, draw=black] (1.25, -1.6) circle [radius=0] node [black, scale = .7] {$\displaystyle\frac{s^2 x (q^{A_j} - 1) q^{A_{[j+1,n]}}}{1 - s^2 x}$};
 				\filldraw[fill=gray!50!white, draw=black] (3.75, -1.6) circle [radius=0] node [black, scale = .75] {$\displaystyle\frac{s^2 (q^{A_i} - 1) q^{A_{[i+1,n]}}}{1 - s^2 x}$};
 				\filldraw[fill=gray!50!white, draw=black] (6.25, -1.6) circle [radius=0] node [black, scale = .75] {$\displaystyle\frac{s^2 (q^{A_i} - x)q^{A_{[i+1,n]}}}{1 - s^2 x}$};
 			\end{tikzpicture}
 		\end{center}	
 		\caption{\label{qweight} Depicted in the second to last row are the weights for the colored stochastic higher spin vertex model, and depicted in the last row are those for the discrete time $q$-boson model.} 
 	\end{figure}

 		Given an $M$-tuple of positive parameters $\bm{\nu} = (\nu_1, \nu_2, \ldots , \nu_M)$, the (Bernoulli) \emph{discrete time $q$-boson model} is the special case of the colored stochastic higher spin vertex model when 
 		\begin{flalign}
 			\label{zij1} 
 			 x_j y_i^{-1} = 1, \quad \text{and} \quad \nu_i = -t_i^2, \qquad \text{for each $(i, j) \in \llbracket 1, M \rrbracket \times \llbracket 1, N \rrbracket$}.
 		\end{flalign} 
 		
 		\noindent We denote the associated stochastic weights 
 		\begin{flalign*} 
 			U_{\nu}^{\dqb} (\bm{A}, b; \bm{C}, d) = U_{1; \sqrt{-\nu}}^{\hs} (\bm{A}, b; \bm{C}, d),
 		\end{flalign*} 
 		
 		\noindent which satisfy
 		\begin{flalign*}
 			U_{\nu}^{\dqb} (\bm{A}, i; \bm{A}_i^+, 0) = \displaystyle\frac{1 + \nu q^{A_{[1,n]}}}{1 + \nu}, \quad \text{and} \quad U_{\nu}^{\dqb} (\bm{A}, i; \bm{A}_{ij}^{+-}, j) = \displaystyle\frac{\nu}{1+\nu} \cdot (1 - q^{A_i}) q^{A_{[i+1,n]}},
 		\end{flalign*}	
 		
 		\noindent for any $i, j \in \llbracket 0, n \rrbracket$ and $\bm{A} \in \mathbb{Z}_{\ge 0}^n$; see the last row of \Cref{qweight}. The colorless ($n=1$) case of this model was introduced in \cite[Definition 1.6]{DT}. We denote the associated probability measure by $\mathbb{P}_{\dqb; \bm{\nu}}^{\sigma}$ (which is the specialization of $\mathbb{P}_{\hs; \bm{x}; \bm{y}; \bm{t}}^{\sigma}$ under \eqref{zij1}). 
 		
 		\begin{rem} 
 			
 			Fix parameters $\bm{\alpha} = (\alpha_1, \alpha_2, \ldots )$. Suppose we consider the discrete time $q$-boson model, set $\nu_j = \varepsilon \alpha_j$ for each $j \in \llbracket 1, M \rrbracket$, scale $N$ by $\varepsilon^{-1}$, let $M$ tend to $\infty$, and let $\varepsilon$ tend to $0$. This gives rise to a continuous-time Markov process on $\mathbb{Z}_{> 0}$, in which a particle of color $c \in \llbracket 1, n \rrbracket$ at site $k \in \mathbb{Z}_{>0}$ jumps to the right according to an exponential clock of rate $\alpha_k q^{A_{[i+1, n]} (k)} (1 - q^{A_i (k)})$, where $A_j (k)$ denotes the number of particles of color $j$ at site $k$; see \cite[Proposition 12.4.1]{CSVMST}. This model is called the \emph{colored $q$-boson model} or the \emph{colored $q$-deformed totally asymmetric zero range process} (TAZRP). The colorless ($n=1$) case of this model was introduced in \cite[Equation (2.6)]{ERDM}, which under a change of variables is equivalent to the $q$-TASEP introduced in \cite[Definition 3.3.7]{P}.
 			
 		\end{rem}

 		\subsection{$q$-Hahn Weights} 
 		
 		\label{Weightss} 
 		
 		We eventually seek to degenerate \Cref{conditionlfused} to the ($\mathrm{R} = 1$) colored stochastic higher spin vertex model case. Recall that this result involves the $W$-weights from \eqref{wweight0} with an arbitrary choice of the parameter $s$, which are in general a bit intricate. So, in this section we discuss a specific choice for $s$ that simplifies these weights considerably.\footnote{This is a colored generalization of what was done in \cite[Section 7.3]{RFSVM}, where a similar idea was implemented to match the height function of the uncolored stochastic higher spin vertex model with the length of a partition sampled under the spin Hall-Littlewood / spin Whittaker measure.} This corresponds to setting $s = x$, a special case that has appeared numerous times in the prior literature \cite{SRF,SP,SRM} and is sometimes known as the $q$-Hann specialization; it is given as the below lemma. 
 		
		\begin{lem}
			
			\label{sw}
			
			Adopting the notation of \Cref{wabcdw}, we have for any $s, t \in \mathbb{C}$ that  
			\begin{flalign*}
				W_{s; t, s} (\bm{A}, \bm{B}; \bm{C}, \bm{D}) & = (-st^{-2})^d  q^{\varphi (\bm{D}, \bm{A} - \bm{D})} \cdot \displaystyle\frac{(t^{-2} s^2; q)_{a-d} (t^2; q)_d}{(s^2; q)_a} \cdot \displaystyle\prod_{i=1}^n \displaystyle\frac{(q; q)_{A_i}}{(q; q)_{A_i - D_i} (q; q)_{D_i}} \\
				& \qquad \times  \mathbbm{1}_{\bm{A} + \bm{B} = \bm{C} + \bm{D}} \cdot \mathbbm{1}_{\bm{A} \ge \bm{D}}.
			\end{flalign*}
		\end{lem}
	
		\begin{proof}
			
			We assume throughout this proof that $\bm{A} + \bm{B} = \bm{C} + \bm{D}$, as otherwise $W_{s; t, s} (\bm{A}, \bm{B}; \bm{C}, \bm{D}) = 0$ by \Cref{wabcdw}. Due to the factor of $(s^{-1} x; q)_{b-p}$ in \eqref{wweight0}, any nonzero summand on the right side of \eqref{wweight0} must at $x = s$ satisfy $p = b$. Together with the fact that $\bm{P} \le \bm{B}$ in this sum, it follows that the only nonzero summand is given by $\bm{P} = \bm{B}$; in particular (since $\bm{P} \le \bm{C}$), we must have $\bm{B} \le \bm{C}$ for $W_{s; t, s} (\bm{A}, \bm{B}; \bm{C}, \bm{D})$ to be nonzero. We assume this in what follows, meaning (as $\bm{A} + \bm{B} = \bm{C} + \bm{D}$) that $\bm{A} \ge \bm{D}$. Then, \eqref{wweight0} yields
			\begin{flalign*}
				W_{s; t, s} (\bm{A}, \bm{B}; \bm{C}, \bm{D}) & = (-1)^d s^d t^{2(c-a)} q^{\varphi (\bm{D}, \bm{C})} \displaystyle\frac{(t^2; q)_d}{(t^2; q)_b} \cdot t^{-2b} \displaystyle\frac{(t^{-2} s^2; q)_{c-b} (t^2; q)_b}{(s^2; q)_{c+d-b}} \\
				& \qquad \times q^{-\varphi(\bm{D}, \bm{B})} \displaystyle\prod_{i=1}^n \displaystyle\frac{(q; q)_{C_i + D_i - B_i}}{(q; q)_{D_i} (q; q)_{C_i - B_i}}.
			\end{flalign*} 
		
			\noindent This, together with the fact that $\bm{A} + \bm{B} = \bm{C} + \bm{D}$ (and hence $a + b = c + d$) yields the lemma.
		\end{proof}
	
		\begin{rem}
			
			\label{fsx} 
			
			Observe for any $\mu, \kappa \in \Comp_n$ that, if $y_i = s$ for each $\in \llbracket 1, n \rrbracket$, then $\mathbb{G}_{\mu/\kappa; s} (\bm{y}; \bm{t})$ (from \Cref{fgfused}) is nonzero only if $\max \mu - \max \kappa \le k$. Indeed, if $\max \mu \ge \max \kappa + k + 1$, then any fused path ensemble $\mathcal{E} \in \mathfrak{P}_{\mathbb{G}} (\mu/\kappa)$ must have at least one arrow configuration $(\bm{A}, \bm{B}; \bm{C}, \bm{D})$ that does not satisfy $\bm{B} \le \bm{C}$ (equivalently, that does not satisfy $\bm{D} \ge \bm{A}$). Hence, \Cref{sw} implies that $W_{\bm{y}; \bm{t}, s} (\mathcal{E}) = 0$ for each $\mathcal{E} \in \mathfrak{P}_{\mathbb{G}} (\mu/\kappa)$, meaning by \eqref{fgmunufused} that $\mathbb{G}_{\mu/\kappa; s} (\bm{y}; \bm{t}) = 0$. 
		
		\end{rem} 
		
		\begin{rem} 
		
		\label{fgsx} 
		
		By \Cref{fsx}, if $y_i = s$ for each $i \in \llbracket 1, M \rrbracket$, then the probability measure $P_{\hs; \bm{x}; \bm{y}; \bm{t}}^{\sigma}$ from \Cref{LModelDiscrete} (see also \eqref{fgprobabilitymufused}) is supported on $(M; \sigma; \bm{\mathrm{R}})$-ascending sequences of $n$-compositions $\bm{\mu} = \big( \mu(0), \mu(1), \ldots , \mu(M+N) \big)$ satisfying $\max \mu(N) \le M$. It follows that this measure is supported on only finitely many such sequences $\bm{\mu}$, and so by analytic continuation \eqref{fgprobabilitymufused} defines a probability measure for any\footnote{It can be verified that the quantity $\prod_{j=1}^N \mathbbm{f}_{\mu(j)/\mu(j-1); s}^{\sigma(j)} (x_j; q^{-1/2}) \cdot \prod_{i=1}^M \prod_{j=1}^N (1 - x_j s^{-1})$ appearing in the $\mathrm{R}_j = 1$ case of \eqref{fgprobabilitymufused} is a polynomial in $\bm{x}$ whenever $\max \mu(N) \le M$ (using the fact that $(1 - x s^{-1}) \cdot \widehat{W}_{x; q^{-1/2}, s}$ is a polynomial in $x$), and so this measure has no singularities in $\bm{x}$.} choice of the parameters $\bm{x}$ and $\bm{t}$, even if \eqref{rtxy} is not satisfied. In particular, the results from \Cref{EqualityCDFused} and \Cref{L0Fused} continue to hold when all $y_i = s$, without assuming \eqref{rtxy}.
		
		\end{rem} 
	
		\subsection{Colored Line Ensembles}
		
		\label{Ensembleq} 
		
		In this section we establish the following proposition describing a colored line ensemble for the discrete time $q$-Boson model (from \Cref{LModelDiscrete}).\footnote{Proposition \ref{nuensemble} below examines this line ensemble on the interval $\llbracket N, M + N \rrbracket$, which corresponds to the behavior of the discrete time $q$-boson model along the north boundary of the rectangle $\llbracket 1, M \rrbracket \times \llbracket 1, N \rrbracket$. One can also formulate more general statements about this line ensemble on its full domain $\llbracket 0, M + N \rrbracket$, but we will not do so here.} In the below, we recall the notation from \Cref{le2fused}, and we define the vertex weights (see \Cref{sw})
		\begin{flalign}
			\label{wnu} 
			\mathcal{W}_{\nu} (\bm{A}, \bm{B}; \bm{C}, \bm{D}) =  q^{\varphi (\bm{D}, \bm{A} - \bm{D})} \nu^{-d} (-\nu; q)_d \cdot  \displaystyle\prod_{i=1}^n \displaystyle\frac{(q;q)_{A_i}}{(q; q)_{A_i  -D_i} (q; q)_{D_i}} \cdot  \mathbbm{1}_{\bm{A} + \bm{B} = \bm{C} + \bm{D}} \cdot \mathbbm{1}_{\bm{A} \ge \bm{D}},
		\end{flalign}
	
		\noindent for any $\nu \in \mathbb{C}$ and $\bm{A}, \bm{B}, \bm{C}, \bm{D} \in \mathbb{Z}_{\ge 0}^n$, where we have set $d = |\bm{D}|$.
		
		\begin{prop} 
		
			\label{nuensemble} 
			
			Fix $\bm{\nu} = (\nu_1, \nu_2, \ldots , \nu_M) \in \mathbb{R}_{> 0}^M$; sample a colored path ensemble $\mathcal{E}$ on $\llbracket 1, M \rrbracket \times \llbracket 1, N \rrbracket$ according to the discrete time $q$-boson model $\mathbb{P}_{\dqb; \bm{\nu}}^{\sigma}$; and for each $c \in \llbracket 1, n \rrbracket$ define $H_c : \llbracket N, M + N \rrbracket \rightarrow \mathbb{Z}$ by setting $H_c (k) = \mathfrak{h}_{\ge c}^{\leftarrow} (M + N - k, N)$ for each $k \in \llbracket N, M + N \rrbracket$, where $\mathfrak{h}_{\ge c}^{\leftarrow}$ is the height function with respect to $\mathcal{E}$.  There exists a random colored line ensemble $\mathds{L} = \big( \mathds{L}^{(1)}, \mathds{L}^{(2)}, \ldots , \mathds{L}^{(n)} \big)$ on $\llbracket 0, M + N \rrbracket$ satisfying the following properties.  
			
			\begin{enumerate}
				\item The joint law of the functions $\big( \mathds{L}_1^{(1)} |_{\llbracket N, M+N \rrbracket}, \mathds{L}_1^{(2)} |_{\llbracket N, M + N \rrbracket}, \ldots , \mathds{L}_1^{(n)} |_{\llbracket N, M + N \rrbracket} \big)$ is the same as that of $(H_1, H_2, \ldots , H_n)$.
				\item We almost surely have $\bm{A}^{\mathds{L}} (-k, i) \ge \bm{D}^{\mathds{L}} (-k, i)$ for each $(-k, i) \in \mathbb{Z}_{\le 0} \times \llbracket N, M + N \rrbracket$. 
				\item Fix integers $j > i \ge 0$ and $u, v \in \llbracket N + 1, M + N \rrbracket$ with $u < v$; set $i_0 = \max \{ i, 1 \}$; and condition on the curves $\mathds{L}_k^{(c)} (m)$ for all $c \in \llbracket 1, n \rrbracket$ and $(k, m) \notin \llbracket i+1, j \rrbracket \times \llbracket u, v-1 \rrbracket$. For any colored line ensemble $\bm{\mathsf{l}}$ that is $\llbracket i+1, j \rrbracket \times \llbracket u, v-1 \rrbracket$-compatible with $\mathds{L}$, we have 
				\begin{flalign}
					\label{llq} 
					\mathbb{P} [\mathds{L} = \bm{\mathsf{l}}] = \mathcal{Z}^{-1} \cdot \displaystyle\prod_{k=i_0}^j \displaystyle\prod_{m = u}^v \mathcal{W}_{\nu_{m-N}} \big( \bm{A}^{\bm{\mathsf{l}}} (-k, m), \bm{B}^{\bm{\mathsf{l}}} (-k, m); \bm{C}^{\bm{\mathsf{l}}} (-k, m), \bm{D}^{\bm{\mathsf{l}}} (-k, m) \big).
				\end{flalign} 
			
				\noindent Here, the probability on the left side of \eqref{llq} is with respect to the conditional law of $\mathds{L}$. Moreover, $\mathcal{Z}$ is a normalizing constant defined so that the sum of the right side of \eqref{llq}, over all colored line ensembles $\bm{\mathsf{l}}$ that are $\llbracket i+1, j \rrbracket \times \llbracket u, v-1 \rrbracket$-compatible with $\mathds{L}$, is equal to $1$.
			\end{enumerate}	
		\end{prop} 
	
		\begin{proof}
			
			Let $s > 0$ be a small real number; set $x_j = s = y_i$ for each $(i, j) \in \llbracket 1, M \rrbracket \times \llbracket 1, N \rrbracket$ (recalling that $r_j = q^{-1/2}$ for each $j \in \llbracket 1, N \rrbracket$); and set $t_i^2 = -\nu_i$ for each $i \in \llbracket 1, M \rrbracket$. Observe from \eqref{zij1} that this specialization sends the stochastic fused vertex model to the discrete time $q$-boson model. Sample a colored line ensemble $\mathds{L} = \big( \mathds{L}^{(1)}, \mathds{L}^{(2)}, \ldots , \mathds{L}^{(n)} \big)$ from the measure $\mathbb{P}_{\cL; n; s; \bm{x}; \bm{y}; \bm{r}; \bm{t}}^{\sigma}$ (recall \Cref{llmu0fused}); we will show that the proposition holds for the limit\footnote{This limit exists since that of $\mathbb{P}_{\mathbbm{f}\mathbb{G}; n; s; \bm{x}; \bm{y}; \bm{r}; \bm{t}}^{\sigma}$ does as $s$ tends to $0$, as it can be quickly verified that the products $s^{|\mu(N)|} \cdot \prod_{j=1}^N \mathbbm{f}_{\mu(j)/\mu(j-1);s}^{\sigma(j)} (x_j; q^{-1/2})$ and $s^{-|\mu(N)|} \cdot \prod_{i=1}^N \mathbb{G}_{\mu(i-1)/\mu(i);s} (s; t_{i-N})$ appearing in \eqref{fgprobabilitymufused} remain nonsingular as $s$ tends to $0$ (since the weights $s^{\mathbbm{1}_{j>0}} \cdot \widehat{W}_{x;q^{-1/2},s} (\bm{A}, i; \bm{A}_{ij}^{+-}, j)$ and $s^{|\bm{D}|} \cdot W_{s; t, s} (\bm{A}, \bm{B}; \bm{C}, \bm{D})$ do).} of $\mathds{L}$, as $s$ tends to $0$. 
			
			The fact that $\mathds{L}$ (for any $s \ge 0$) satisfies the first statement of the proposition follows from \Cref{lmu2fused} (and \Cref{fgsx}, so that we need not assume that \eqref{rtxy} holds). That it satisfies the second (also for any $s \ge 0$) follows from \eqref{llprobabilityfused} (and \Cref{fgsx}), together with the fact by \Cref{sw} that for $m > N$ and $y_{m-N} = s$ we have 
			\begin{flalign*} 
				W_{y_{m-N}; t_{m-N}, s} \big( \bm{A}^{\bm{\mathsf{l}}} (-k, m), \bm{B}^{\bm{\mathsf{l}}} (-k, m); \bm{C}^{\bm{\mathsf{l}}} (-k, m), \bm{D}^{\bm{\mathsf{l}}} (-k, m) \big) = 0, 
			\end{flalign*} 
		
			\noindent unless $\bm{A}^{\bm{\mathsf{l}}} (-k, m) \ge \bm{D}^{\bm{\mathsf{l}}} (-k, m)$. 
			
			To verify the third, observe by \Cref{conditionlfused} (and \Cref{fgsx}) that, for any $s \ge 0$,
			\begin{flalign}
				\label{lls2} 
				\begin{aligned} 
				\mathbb{P} [\mathds{L} = \bm{\mathsf{l}}] & = \mathcal{Z}_0^{-1} \cdot \displaystyle\prod_{k=i_0}^j \displaystyle\prod_{m=u}^v s^{|\bm{D}^{\bm{\mathsf{l}}} (-k, m)|} \\
				& \quad \times \displaystyle\prod_{k=i_0}^j \displaystyle\prod_{m=u}^v s^{-|\bm{D}^{\bm{\mathsf{l}}} (-k, m)|} \cdot W_{s; t_{m-N}; s} \big( \bm{A}^{\bm{\mathsf{l}}} (-k, m), \bm{B}^{\bm{\mathsf{l}}} (-k, m); \bm{C}^{\bm{\mathsf{l}}} (-k, m), \bm{D}^{\bm{\mathsf{l}}} (-k, m) \big),
				\end{aligned}
			\end{flalign}
		
			\noindent for some normalization constant $\mathcal{Z}_0 > 0$. Next, we have  
			\begin{flalign*}
				\displaystyle\sum_{m=u}^v \big| \bm{D}^{\bm{\mathsf{l}}} (-k, m) \big| = \displaystyle\sum_{m=u}^v \big( \bm{\mathsf{l}}_k^{(1)} (m-1) - \bm{\mathsf{l}}_k^{(1)} (m) \big) = \mathsf{l}_k^{(1)} (u-1) - \mathsf{l}_k^{(1)} (v) = \mathds{L}_k^{(1)} (u-1) - \mathds{L}_k^{(1)} (v),
			\end{flalign*}
		
			\noindent where the first equality holds by \Cref{le2fused}, the second by performing the sum, and the third by the fact that $\bm{\mathsf{l}}$ is $\llbracket u, v-1 \rrbracket \times \llbracket i + 1, j \rrbracket$-compatible with $\mathds{L}$. This quantity is fixed by the conditioning, so the first product on the right side of \eqref{lls2} can be incorporated into the normalization constant $\mathcal{Z}_0$. This gives
				\begin{flalign*}
				\mathbb{P} [\mathds{L} = \bm{\mathsf{l}}] = \mathcal{Z}_1^{-1} \cdot \displaystyle\prod_{k=i_0}^j \displaystyle\prod_{m=u}^v s^{-|\bm{D}^{\bm{\mathsf{l}}} (-k, m)|} \cdot W_{s; t_{m-N}; s} \big( \bm{A}^{\bm{\mathsf{l}}} (-k, m), \bm{B}^{\bm{\mathsf{l}}} (-k, m); \bm{C}^{\bm{\mathsf{l}}} (-k, m), \bm{D}^{\bm{\mathsf{l}}} (-k, m) \big),
		\end{flalign*}				
			   
			\noindent for some normalization constant $\mathcal{Z}_1 > 0$. Letting $s$ tend to $0$, together with the fact by \Cref{sw} and \eqref{wnu} (and as $t_i^2 = -\nu_i$) that 
			\begin{flalign*} 
				\lim_{s \rightarrow 0} s^{-|\bm{D}|} \cdot W_{s; t_i, s} (\bm{A}, \bm{B}; \bm{C}, \bm{D}) = \mathcal{W}_{\nu_i} (\bm{A}, \bm{B}; \bm{C}, \bm{D}),
			\end{flalign*} 
		
			\noindent yields the third part of the proposition.
		\end{proof}

	\appendix 
	
	\section{Degeneration to the Log-Gamma Polymer} 
	
	\label{VertexPolymer}
	
	In this section we explain how to recover the log-gamma polymer (introduced in \cite{SDPBC}) from the colored stochastic fused vertex model (recall \Cref{FusedPath}); this involves a complementation procedure (of the type alluded to in \Cref{a0c0u}), an analytic continuation, and limit degeneration. We begin by describing the complementation procedure and evaluating the associated complemented weights in Section \ref{WeightsComplement}. We then analyze (limits of) anaytic continuations of these weights in Section \ref{LimitU} and Section \ref{WeightD}, and analyze the behavior of them as $q$ tends to $1$ in Section \ref{Weightsq1}, Section \ref{ProofSum2}, and Section \ref{WeightAq1}. We conclude in Section \ref{Polymerq1} by explaining convergence of the vertex model with these degenerated weights to the log-gamma polymer (and also describing the reason for choosing this specialization in \Cref{q0}). In what follows, given a real number $\theta > 0$, we say that $X$ is a $\mathrm{Gamma} (\theta)$ random variable if $\mathbb{P} [x \le X \le x + dx] = \Gamma(\theta)^{-1} \cdot x^{\theta - 1} e^{-x} dx$ for any $x \in \mathbb{R}$. Throughout this section, $n \ge 1$ is an integer, and we assume that $q \in (0, 1)$.  
	
	\subsection{Complemented Fused Weights}
	
	\label{WeightsComplement} 
	
	In this section we implement (a case of) the complementation procedure mentioned in \Cref{a0c0u}, which will be necessary to degenerate the fused stochastic vertex model of \Cref{FusedPath} to the log-gamma polymer. To that end, we require some additional notation. Let $\mathrm{L} \in \mathbb{Z}_{\ge 0}$ be an integer, and let
	\begin{flalign}
		\label{x1x} 
		\overline{\bm{X}} = (X_1, X_2, \ldots , X_{n-1}) \in \mathbb{Z}_{\ge 0}^{n-1}; \qquad \text{for any $\bm{X} \in \mathbb{Z}_{\ge 0}^n$, and set} \quad \overline{x} = |\overline{\bm{X}}|. 
	\end{flalign} 
	
	\noindent For any $\bm{B}, \bm{D} \in \mathbb{Z}_{\ge 0}^n$ such that $|\bm{B}| \le \mathrm{L}$ and $\bm{D} \le \mathrm{L}$, we set 
	\begin{flalign}
		\label{bnlbn}
		\mathsf{B}_n = \mathrm{L} - B_n; \qquad \mathsf{b} = \mathrm{L} - |\bm{B}|; \qquad \mathsf{D}_n = \mathrm{L} - D_n; \qquad \mathsf{d} = \mathrm{L} - |\bm{D}|.
	\end{flalign}
	
	\noindent In this way, we have $\bm{B} = (\overline{\bm{B}}, \mathrm{L} - \mathsf{B}_n)$ and $\bm{D} = (\overline{\bm{D}}, \mathrm{L} - \mathsf{D}_n)$, so that 
	\begin{flalign}
		\label{bbdd}
		\mathsf{b} = \mathsf{B}_n - \overline{b}; \qquad \mathsf{d} = \mathsf{D}_n - \overline{d}. 
	\end{flalign}
	
	We will evaluate the weight $U_{z; r, s} (\bm{A}, \bm{B}; \bm{C}, \bm{D})$ when $r^2 = q^{-\mathrm{L}}$ (which is the reason we imposed the conditions $|\bm{B}|, |\bm{D}| \le \mathrm{L}$ above), $s^2 = q^{-\mathfrak{M}}$, and $z = q^{\mathfrak{N} - \mathrm{L} + 1}$ (for some complex numbers $\mathfrak{M}, \mathfrak{N} \in \mathbb{C}$). This weight will happen to be rational in $q^{\mathrm{L}}$, so let us define it when the integer $\mathrm{L}$ is replaced by an arbitrary complex number $\mathfrak{L}$. In what follows, for any integers $m, k \ge 0$, and complex numbers $z \in \mathbb{C}$, $a_1, a_2, \ldots , a_k \in \mathbb{C}$, and $b_1, b_2, \ldots , b_k \in \mathbb{C}$, we denote the terminating basic hypergeometric series by
	\begin{flalign}
		\label{43abcd}
		{_{k+1} \varphi_k} \bigg( \begin{array}{ccc} q^{-m}; a_1, a_2, \ldots , a_k \\ b_1, b_2, \ldots , b_k  \end{array} \bigg| q, z \bigg) = \displaystyle\sum_{i=0}^m z^i \displaystyle\frac{(q^{-m}; q)_i}{(q;q)_i} \cdot \displaystyle\prod_{j=1}^k \displaystyle\frac{(a_j; q)_i}{(b_j; q)_i}.
	\end{flalign}
	
	\begin{definition} 
		
		\label{2u}
		
		Let $\mathfrak{L}, \mathfrak{M}, \mathfrak{N} \in \mathbb{C}$ be complex numbers; $\mathsf{B}_n, \mathsf{D}_n \in \mathbb{Z}_{\ge 0}$ be integers; $\overline{\bm{B}}, \overline{\bm{D}} \in \mathbb{Z}_{\ge 0}^n$ be $(n-1)$-tuples, such that $\overline{b} = |\bm{\overline{B}}| \le \mathsf{B}_n$ and $\overline{d} = |\overline{\bm{D}}| \le \mathsf{D}_n$; and $\bm{A}, \bm{C} \in \mathbb{Z}_{\ge 0}^n$ be $n$-tuples. Setting $a = |\bm{A}|$ and $c = |\bm{C}|$; letting $\mathsf{b}, \mathsf{d} \in \mathbb{Z}_{\ge 0}$ be as in \eqref{bbdd}, and recalling \eqref{x1x}, define the weight
		\begin{flalign*}
			& \mathcal{U}_{q^{\mathfrak{L}}; q^{\mathfrak{M}}; q^{\mathfrak{N}}} \big( \bm{A}, (\overline{\bm{B}}, \mathsf{B}_n); \bm{C}, (\overline{\bm{D}}, \mathsf{D}_n) \big) \\
			& \quad= (-1)^{C_n} q^{\varphi (\overline{\bm{D}}, \overline{\bm{C}}) + \binom{C_n+1}{2} + \overline{c} \mathsf{d} + \mathfrak{L}C_n} \cdot \mathbbm{1}_{\overline{\bm{A}} + \overline{\bm{B}} = \overline{\bm{C}} + \overline{\bm{D}}} \cdot \mathbbm{1}_{\mathsf{b} - a = \mathsf{d} - c}\\
			& \qquad \times \displaystyle\frac{(q; q)_{\mathsf{b}}}{(q; q)_{\mathsf{d}}} \displaystyle\frac{(q^{\mathfrak{M}-\mathfrak{N}-c}; q)_{\mathsf{d}}}{(q; q)_{C_n}} \displaystyle\frac{(q^{\mathsf{D}_n - C_n -\mathfrak{L}}; q)_{C_n}}{(q^{-\mathfrak{N}}; q)_{\mathsf{b}}} \displaystyle\frac{(q^{-\mathfrak{N}}; q)_{\infty}}{(q^{\mathfrak{L}-\mathfrak{N}}; q)_{\infty}} \displaystyle\frac{(q^{\mathfrak{L}+\mathfrak{M}-\mathfrak{N}}; q)_{\infty}}{(q^{\mathfrak{M}-\mathfrak{N}}; q)_{\infty}} \cdot \displaystyle\prod_{j=1}^{n-1} \displaystyle\frac{(q; q)_{B_j}}{(q; q)_{D_j}} \\
			& \qquad \times \displaystyle\sum_{\overline{\bm{P}} \le \overline{\bm{B}}, \overline{\bm{C}}} (-1)^{\overline{p}} q^{\varphi (\overline{\bm{B}} - \overline{\bm{D}} - \overline{\bm{P}}, \overline{\bm{P}})} \displaystyle\frac{(q^{-\mathfrak{N}-1}; q)_{\overline{p}} (q^{\mathfrak{M}-\mathfrak{N}+\mathsf{d}-c}; q)_{\overline{p}}}{(q^{\mathfrak{M}-\mathfrak{N}-c}; q)_{\overline{p}} (q^{\mathsf{b}-\mathfrak{N}}; q)_{\overline{p}}} \cdot \displaystyle\prod_{j=1}^{n-1} \displaystyle\frac{(q; q)_{C_j + D_j - P_j}}{(q; q)_{C_j - P_j} (q; q)_{P_j} (q; q)_{B_j - P_j}} \\
			& \qquad \qquad \quad \times q^{\overline{p} (\mathsf{b} - \mathsf{d} +1) + \binom{\overline{p}}{2}} \cdot {_4 \varphi_3} \bigg( \begin{array}{cc} q^{-C_n}; q^{\mathsf{B}_n - \mathfrak{L}}, q^{\overline{p}-\mathfrak{N}+1}, q^{\mathfrak{M}-\mathfrak{N}+\mathsf{d}-c+\overline{p}} \\ q^{\mathfrak{M}-\mathfrak{N}-c+\overline{p}}, q^{\mathsf{b}-\mathfrak{N}+\overline{p}}, q^{\mathsf{D}_n - C_n - \mathfrak{L}} \end{array} \bigg| q, q \bigg).
		\end{flalign*}
		
	\end{definition} 
	
	\begin{rem} 
		
		\label{43fused}
		
	The appearance of the ${_4 \varphi_3}$ basic hypergeometric series in the $\mathcal{U}$-weights defined above is a common phenomenon for fused vertex weights in the colorless $n=1$ case; see, for example, \cite[Equation (1.1)]{ESVM} and \cite[Theorem 3.15]{SHSVML}.
	
	\end{rem} 

	The following lemma provides an expression for the $U_z$-weight at $z = q^{\mathfrak{N} - \mathrm{L} + 1}$ in terms of $\mathcal{U}$.

	\begin{lem} 
		
		\label{rqnl1} 
		
		Fix an integer $\mathrm{L} \ge 1$ and complex number $\mathfrak{M}, \mathfrak{N} \in \mathbb{C}$. Let $\bm{A}, \bm{B}, \bm{C}, \bm{D} \in \mathbb{Z}_{\ge 0}^n$, let $a = |\bm{A}|$ and $c = |\bm{C}|$, and adopt the notation in \eqref{x1x} and \eqref{bnlbn}. For any $\mathfrak{N} \in \mathbb{C}$, we have
		\begin{flalign*} 
			U_{q^{\mathfrak{N}-\mathrm{L}+1}; q^{-\mathrm{L}/2}, q^{-\mathfrak{M}/2}} (\bm{A}, \bm{B}; \bm{C}, \bm{D}) = \mathcal{U}_{q^{\mathrm{L}}; q^{\mathfrak{M}}; q^{\mathfrak{N}}} \big(\bm{A}, (\overline{\bm{B}}, \mathsf{B}_n); \bm{C}, (\overline{\bm{D}}, \mathsf{D}_n) \big). 
		\end{flalign*} 
		
	\end{lem} 
	
	\begin{proof}
		
		Let us assume throughout this proof that $\bm{A} + \bm{B} = \bm{C} + \bm{D}$, or equivalently that $\overline{\bm{A}} + \overline{\bm{B}} = \overline{\bm{C}} + \overline{\bm{D}}$ and $\mathsf{b} - a = \mathrm{L} - |\bm{A}| - |\bm{B}| = \mathrm{L} - |\bm{C}| - |\bm{D}| = \mathsf{d} - c$; otherwise, \eqref{weightu} implies that $U_{q^{\mathfrak{N}-\mathrm{L}+1}; q^{-\mathrm{L}/2}, q^{-\mathfrak{M}/2}} (\bm{A}, \bm{B}; \bm{C}, \bm{D}) = 0$, which matches with $\mathcal{U}_{q^{\mathfrak{L}}; q^{\mathfrak{M}}; q^{\mathfrak{N}}} (\bm{A}, \bm{B}; \bm{C}, \bm{D})$, by \Cref{2u}. Next, inserting \eqref{bnlbn} into \eqref{weightu} (and using the facts that $a-c = |\bm{A}| - |\bm{C}| = |\bm{D}| - |\bm{B}| = \mathsf{b} - \mathsf{d}$ and that $\varphi (\bm{B}, \bm{X}) = \varphi (\overline{\bm{B}}, \bm{X})$ and $\varphi (\bm{D}, \bm{X}) = \varphi (\overline{\bm{D}}, \bm{X})$ for any $\bm{X} \in \mathbb{Z}_{\ge 0}^n$), we obtain
		\begin{flalign}
			\label{rqn0} 
			\begin{aligned} 
				& U_{q^{\mathfrak{N}-\mathrm{L}+1}; q^{-\mathrm{L}/2}, q^{-\mathfrak{M}/2}} (\bm{A}, \bm{B}; \bm{C}, \bm{D}) \\
				& \quad  = q^{(\mathfrak{N}+1)(\mathsf{b}-\mathsf{d}) - (\mathrm{L}-\mathsf{d})\mathfrak{M} + \varphi (\overline{\bm{D}}, \bm{C})} \displaystyle\frac{(q^{-\mathrm{L}}; q)_{\mathrm{L}-\mathsf{d}}}{(q^{-\mathrm{L}}; q)_{\mathrm{L}-\mathsf{b}}} \displaystyle\frac{(q; q)_{\mathrm{L}-\mathsf{B}_n}}{(q; q)_{\mathrm{L}-\mathsf{D}_n}} \displaystyle\prod_{j=1}^{n-1} \displaystyle\frac{(q; q)_{B_j}}{(q; q)_{D_j}} \\
				& \qquad \times \displaystyle\sum_{p=0}^{\min \{ \mathrm{L}-\mathsf{b}, c \}} q^{(\mathfrak{N}+1)p} \displaystyle\frac{(q^{\mathfrak{N}-\mathfrak{M}+1}; q)_{c-p}}{(q^{\mathfrak{N}-\mathrm{L}-\mathfrak{M}+1}; q)_{\mathrm{L} + c - \mathsf{d}-p}} (q^{-\mathfrak{N}-1}; q)_p (q^{\mathfrak{N}-\mathrm{L}+1}; q)_{\mathrm{L} - \mathsf{b} -p} \\
				& \qquad \times \displaystyle\sum_{\substack{\bm{P} \le \bm{B}, \bm{C} \\ |\bm{P}| = p}} q^{\varphi (\overline{\bm{B}} - \overline{\bm{D}} - \overline{\bm{P}}, \bm{P})} \displaystyle\frac{(q; q)_{\mathrm{L}+C_n - \mathsf{D}_n - P_n}}{(q; q)_{C_n - P_n} (q; q)_{P_n} (q; q)_{\mathrm{L} - \mathsf{B}_n - P_n}} \displaystyle\prod_{j=1}^{n-1} \displaystyle\frac{(q; q)_{C_j + D_j - P_j}}{(q; q)_{C_j - P_j} (q; q)_{P_j} (q; q)_{B_j - P_j}}.
			\end{aligned}
		\end{flalign} 
		
		\noindent We will next separate the parameter $\mathrm{L}$ (as well as $p$ and $P_n$) from the other ones in the subscripts of the $q$-Pochhammer symbols. To that end, observe that 
		\begin{flalign}
			\label{qlbn0} 
			\displaystyle\frac{(q;q)_{\mathrm{L}-\mathsf{B}_n}}{(q; q)_{\mathrm{L}-\mathsf{D}_n}} \displaystyle\frac{(q^{-\mathrm{L}}; q)_{\mathrm{L}-\mathsf{d}}}{(q^{-\mathrm{L}}; q)_{\mathrm{L}-\mathsf{b}}} = (-1)^{\mathsf{b} + \mathsf{B}_n - \mathsf{d} - \mathsf{D}_n} q^{\binom{\mathsf{B}_n}{2} - \binom{\mathsf{D}_n}{2} + \mathrm{L} (\mathsf{D}_n - \mathsf{B}_n) + \binom{\mathsf{d}+1}{2} - \binom{\mathsf{b}+1}{2}} \displaystyle\frac{(q^{-\mathrm{L}}; q)_{\mathsf{D}_n}}{(q^{-\mathrm{L}}; q)_{\mathsf{B}_n}} \displaystyle\frac{(q; q)_{\mathsf{b}}}{(q; q)_{\mathsf{d}}},
		\end{flalign} 
		
		\noindent due to the identities  
		\begin{flalign*}
			& \displaystyle\frac{(q; q)_{\mathrm{L}-\mathsf{B}_n}}{(q; q)_{\mathrm{L}-\mathsf{D}_n}} = \displaystyle\frac{(q^\mathrm{L}; q^{-1})_{\mathsf{D}_n}}{(q^\mathrm{L}; q^{-1})_{\mathsf{B}_n}} = (-1)^{\mathsf{B}_n - \mathsf{D}_n} q^{\binom{\mathsf{B}_n}{2} - \binom{\mathsf{D}_n}{2} + \mathrm{L}(\mathsf{D}_n - \mathsf{B}_n)} \displaystyle\frac{(q^{-\mathrm{L}}; q)_{\mathsf{D}_n}}{(q^{-\mathrm{L}}; q)_{\mathsf{B}_n}}; \\
			& \displaystyle\frac{(q^{-\mathrm{L}}; q)_{\mathrm{L}-\mathsf{d}}}{(q^{-\mathrm{L}}; q)_{\mathrm{L}-\mathsf{b}}} = \displaystyle\frac{(q^{-\mathsf{b}}; q)_{\mathsf{b}}}{(q^{-\mathsf{d}}; q)_{\mathsf{d}}} = (-1)^{\mathsf{b}-\mathsf{d}} q^{\binom{\mathsf{d}+1}{2} - \binom{\mathsf{b}+1}{2}} \displaystyle\frac{(q;q)_{\mathsf{b}}}{(q;q)_{\mathsf{d}}},
		\end{flalign*} 
		
		\noindent where in the first we used the fact that $(u; q^{-1})_k = (-u)^k q^{-\binom{k}{2}} (u^{-1}; q)_k$ for any $k \in \mathbb{Z}_{\ge 0}$ and $u \in \mathbb{C}$. We further have that 
		\begin{flalign}
			\label{qnm1} 
			\begin{aligned} 
				\displaystyle\frac{(q^{\mathfrak{N}-\mathfrak{M}+1}; q)_{c-p}}{(q^{\mathfrak{N}-\mathrm{L}-\mathfrak{M}+1};q)_{\mathrm{L}+c-\mathsf{d}-p}}  (q^{\mathfrak{N}-\mathrm{L}+1}; q)_{\mathrm{L}-\mathsf{b}-p} & = \displaystyle\frac{(q^{\mathfrak{N}-\mathfrak{M}+1}; q)_c}{(q^{\mathfrak{N}-\mathfrak{M}+c}; q^{-1})_p} \displaystyle\frac{(q^{\mathfrak{N}-\mathrm{L}+1}; q)_\mathrm{L}}{(q^{\mathfrak{N}}; q^{-1})_{\mathsf{b}+p}} \displaystyle\frac{(q^{\mathfrak{N} - \mathfrak{M}}; q^{-1})_{\mathsf{d} - c - p}}{(q^{\mathfrak{N}-\mathfrak{M}-\mathrm{L}+1}; q)_\mathrm{L}}; \\
				\displaystyle\frac{(q;q)_{\mathrm{L}+C_n-\mathsf{D}_n-P_n}}{(q;q)_{C_n-P_n} (q; q)_{\mathrm{L}-\mathsf{B}_n - P_n}} & = \displaystyle\frac{(q^\mathrm{L}; q^{-1})_{\mathsf{B}_n+P_n} (q^{C_n}; q^{-1})_{P_n}}{(q;q)_{C_n} (q^\mathrm{L}; q^{-1})_{\mathsf{D}_n - C_n + P_n}},
			\end{aligned} 	
		\end{flalign}
		
		\noindent due the identities
		\begin{flalign*}
			& (q^{\mathfrak{N}-\mathrm{L}+1}; q)_{\mathrm{L}-\mathsf{b}-p} = \displaystyle\frac{(q^{\mathfrak{N}-\mathrm{L}+1}; q)_\mathrm{L}}{(q^\mathfrak{N}; q^{-1})_{\mathsf{b}+p}}; \qquad (q^{\mathfrak{N}-\mathfrak{M}+1}; q)_{c-p} = \displaystyle\frac{(q^{\mathfrak{N}-\mathfrak{M}+1}; q)_c}{(q^{\mathfrak{N}-\mathfrak{M}+c}; q^{-1})_p}; \\
			& (q^{\mathfrak{N}-\mathfrak{M}-\mathrm{L}+1}; q)_{\mathrm{L}+c-\mathsf{d}-p} = \displaystyle\frac{(q^{\mathfrak{N}-\mathfrak{M}-\mathrm{L}+1}; q)_\mathrm{L}}{(q^{\mathfrak{N}-\mathfrak{M}}; q^{-1})_{\mathsf{d}-c+p}}; \qquad (q; q)_{\mathrm{L}-\mathsf{B}_n-P_n} = \displaystyle\frac{(q; q)_\mathrm{L}}{(q^\mathrm{L}; q^{-1})_{\mathsf{B}_n + P_n}}; \\
			& (q; q)_{\mathrm{L}+C_n-\mathsf{D}_n-P_n} = \displaystyle\frac{(q; q)_\mathrm{L}}{(q^\mathrm{L}; q^{-1})_{\mathsf{D}_n - C_n + P_n}}; \qquad (q; q)_{C_n - P_n} = \displaystyle\frac{(q; q)_{C_n}}{(q^{C_n}; q^{-1})_{P_n}}.
		\end{flalign*}
		
		\noindent Inserting \eqref{qlbn0} and \eqref{qnm1} into \eqref{rqn0}, we obtain 	
		\begin{flalign}
			\label{rqnl5}
			\begin{aligned}
				& U_{q^{\mathfrak{N}-\mathrm{L}+1}; q^{-\mathrm{L}/2}, q^{-\mathfrak{M}/2}} (\bm{A}, \bm{B}; \bm{C}, \bm{D}) \\
				& \quad  = (-1)^{\mathsf{b} + \mathsf{B}_n -\mathsf{d} - \mathsf{D}_n} q^{(\mathfrak{N}+1)(b-\mathsf{d}) - (\mathrm{L}-\mathsf{d})\mathfrak{M} + \varphi (\overline{\bm{D}}, \bm{C})} q^{\binom{\mathsf{B}_n}{2} - \binom{\mathsf{D}_n}{2} + \mathrm{L} (\mathsf{D}_n - \mathsf{B}_n)} q^{\binom{\mathsf{d}+1}{2} - \binom{\mathsf{b}+1}{2}} \\
				& \qquad \times \displaystyle\frac{(q; q)_{\mathsf{b}}}{(q; q)_{\mathsf{d}}} \displaystyle\frac{(q^{\mathfrak{N}-\mathfrak{M}+1}; q)_c}{(q; q)_{C_n}} \displaystyle\frac{(q^{-\mathrm{L}}; q)_{\mathsf{D}_n}}{(q^{-\mathrm{L}}; q)_{\mathsf{B}_n}} \displaystyle\frac{(q^{\mathfrak{N}-\mathrm{L}+1}; q)_\mathrm{L}}{(q^{\mathfrak{N}-\mathfrak{M}-\mathrm{L}+1}; q)_\mathrm{L}} \displaystyle\prod_{j=1}^{n-1} \displaystyle\frac{(q; q)_{B_j}}{(q; q)_{D_j}} \\
				& \qquad \times \displaystyle\sum_{p=0}^{\min \{ b, c \}} q^{(\mathfrak{N}+1)p} \displaystyle\frac{(q^{-\mathfrak{N}-1}; q)_p (q^{\mathfrak{N}-\mathfrak{M}}; q^{-1})_{\mathsf{d}-c+p}}{(q^{\mathfrak{N}-\mathfrak{M}+c}; q^{-1})_p (q^\mathfrak{N}; q^{-1})_{\mathsf{b}+p}} \\
				& \qquad \quad \times \displaystyle\sum_{\substack{\bm{P} \le \bm{B}, \bm{C} \\ |\bm{P}| = p}} q^{\varphi (\overline{\bm{B}} - \overline{\bm{D}} - \overline{\bm{P}}, \bm{P})} \displaystyle\frac{(q^\mathrm{L}; q^{-1})_{\mathsf{B}_n + P_n} (q^{C_n}; q^{-1})_{P_n}}{(q; q)_{P_n} (q^\mathrm{L}; q^{-1})_{\mathsf{D}_n - C_n + P_n}} \displaystyle\prod_{j=1}^{n-1} \displaystyle\frac{(q; q)_{C_j + D_j - P_j}}{(q; q)_{C_j - P_j} (q; q)_{P_j} (q; q)_{B_j - P_j}}.
			\end{aligned} 
		\end{flalign}

		Next observe, as $(u; q)_k = (-u)^k q^{\binom{k}{2}} (q^{1-k} u^{-1}; q)_k$ for any $k \in \mathbb{Z}_{\ge 0}$ and $u \in \mathbb{C}$, that 
		\begin{flalign}
			\label{qnl2}
			& \displaystyle\frac{(q^{\mathfrak{N}-\mathrm{L}+1}; q)_\mathrm{L}}{(q^{\mathfrak{N}-\mathrm{L}-\mathfrak{M}+1}; q)_\mathrm{L}} = q^{\mathrm{L}\mathfrak{M}} \displaystyle\frac{(q^{-\mathfrak{N}}; q)_\mathrm{L}}{(q^{\mathfrak{M}-\mathfrak{N}}; q)_\mathrm{L}}; \qquad (q^{\mathfrak{N}-\mathfrak{M}+1}; q)_c = (-1)^c q^{(\mathfrak{N}-\mathfrak{M})c + \binom{c+1}{2}} (q^{\mathfrak{M}-\mathfrak{N}-c}; q)_c.
		\end{flalign}
		
		\noindent We further convert the quantities of the form $(u; q^{-1})_k$ on the right side of \eqref{rqnl5} into ones of the form $(u'; q)_k$. To do so, we repeatedly use the identity $(u; q^{-1})_k = (-u)^k q^{-\binom{k}{2}} (u^{-1}; q)_k$ for any $k \in \mathbb{Z}_{\ge 0}$ and $u \in \mathbb{C}$, which gives   
		\begin{flalign}
			\label{qnl3} 
			\begin{aligned} 
				\displaystyle\frac{(q^{\mathfrak{N}-\mathfrak{M}}; q^{-1})_{\mathsf{d}-c+p}}{(q^\mathrm{L}; q^{-1})_{\mathsf{D}_n - C_n + P_n}} & = (-1)^{\mathsf{d}-c + p - \mathsf{D}_n + C_n - P_n} q^{(\mathfrak{N}-\mathfrak{M})(\mathsf{d}-c+p) - \mathrm{L} (\mathsf{D}_n - C_n + P_n)} q^{\binom{\mathsf{D}_n - C_n + P_n}{2} - \binom{\mathsf{d}-c+p}{2}} \\
				& \qquad \times  \displaystyle\frac{(q^{\mathfrak{M}-\mathfrak{N}}; q)_{\mathsf{d}-c+p}}{(q^{-\mathrm{L}}; q)_{\mathsf{D}_n - C_n + P_n}};
			\end{aligned} 
		\end{flalign}
		
		\noindent and 
		\begin{flalign}
			\label{qnl4} 
			\begin{aligned}
				& \displaystyle\frac{(q^\mathrm{L}; q^{-1})_{\mathsf{B}_n + P_n}}{(q^\mathfrak{N}; q^{-1})_{\mathsf{b}+p}} = (-1)^{\mathsf{B}_n + P_n - b - p} q^{\mathrm{L} (\mathsf{B}_n + P_n) - \mathfrak{N}(\mathsf{b}+p)} q^{\binom{\mathsf{b}+p}{2} - \binom{\mathsf{B}_n + P_n}{2}} \displaystyle\frac{(q^{-\mathrm{L}}; q)_{\mathsf{B}_n + P_n}}{(q^{-\mathfrak{N}}; q)_{\mathsf{b}+p}}; \\
				& \displaystyle\frac{(q^{C_n}; q^{-1})_{P_n}}{(q^{\mathfrak{N}-\mathfrak{M}+c}; q^{-1})_p} = (-1)^{p+P_n} q^{P_n C_n + (\mathfrak{M}-\mathfrak{N}-c)p} q^{\binom{p}{2} - \binom{P_n}{2}} \displaystyle\frac{(q^{-C_n}; q)_{P_n}}{(q^{\mathfrak{M}-\mathfrak{N}-c}; q)_p}.
			\end{aligned}
		\end{flalign}
		
		\noindent Inserting \eqref{qnl2}, \eqref{qnl3}, and \eqref{qnl4} into \eqref{rqnl5} gives 
		\begin{flalign*}
			& U_{q^{\mathfrak{N}-\mathrm{L}+1}; q^{-\mathrm{L}/2}, q^{-\mathfrak{M}/2}} (\bm{A}, \bm{B}; \bm{C}, \bm{D}) \\
			& \quad  = (-1)^{\mathsf{b} + \mathsf{B}_n -\mathsf{d} - \mathsf{D}_n + c} q^{(\mathfrak{N}+1)(\mathsf{b}-\mathsf{d}) +\mathsf{d}\mathfrak{M} + \varphi (\overline{\bm{D}}, \bm{C})} q^{\binom{\mathsf{B}_n}{2} - \binom{\mathsf{D}_n}{2} + \mathrm{L} (\mathsf{D}_n - \mathsf{B}_n)} q^{\binom{\mathsf{d}+1}{2} - \binom{\mathsf{b}+1}{2} + c(\mathfrak{N}-\mathfrak{M}) + \binom{c+1}{2}} \\
			& \qquad \times \displaystyle\frac{(q; q)_{\mathsf{b}}}{(q; q)_{\mathsf{d}}} \displaystyle\frac{(q^{\mathfrak{M}-\mathfrak{N}-c}; q)_c}{(q; q)_{C_n}} \displaystyle\frac{(q^{-\mathrm{L}}; q)_{\mathsf{D}_n}}{(q^{-\mathrm{L}}; q)_{\mathsf{B}_n}} \displaystyle\frac{(q^{-\mathfrak{N}}; q)_\mathrm{L}}{(q^{\mathfrak{M}-\mathfrak{N}}; q)_\mathrm{L}} \displaystyle\prod_{j=1}^{n-1} \displaystyle\frac{(q; q)_{B_j}}{(q; q)_{D_j}} \\
			& \qquad \times \displaystyle\sum_{p=0}^{\min \{ \mathrm{L}-\mathsf{b}, c \}} \displaystyle\sum_{\bm{P} \le \bm{B}, \bm{C}} (-1)^{\mathsf{d}-c+p-\mathsf{D}_n-C_n+P_n} q^{(\mathfrak{N}+1)p} q^{(\mathfrak{N}-\mathfrak{M})(\mathsf{d}-c+p) - \mathrm{L} (\mathsf{D}_n - C_n + P_n)} \\
			& \qquad \qquad \times q^{\binom{\mathsf{D}_n - C_n + P_n}{2} - \binom{\mathsf{d}-c+p}{2}} (-1)^{\mathsf{B}_n + P_n - \mathsf{b} -p} q^{\mathrm{L}(\mathsf{B}_n+P_n) - \mathfrak{N}(\mathsf{b}+p)} q^{\binom{\mathsf{b}+p}{2} - \binom{\mathsf{B}_n+P_n}{2}} \\
			& \qquad \qquad \times \displaystyle\frac{(q^{-\mathfrak{N}-1}; q)_p (q^{\mathfrak{M}-\mathfrak{N}}; q)_{\mathsf{d}-c+p}}{(q^{\mathfrak{M}-\mathfrak{N}-c}; q)_p (q^{-\mathfrak{N}}; q)_{\mathsf{b}+p}} (-1)^{p + P_n} q^{P_n C_n + p(\mathfrak{M}-\mathfrak{N}-c)} q^{\binom{p}{2} - \binom{P_n}{2}} q^{\varphi (\overline{\bm{B}} - \overline{\bm{D}} - \overline{\bm{P}}, \bm{P})} \\ 
			& \qquad \qquad \times \displaystyle\frac{(q^{-\mathrm{L}}; q)_{\mathsf{B}_n + P_n} (q^{-C_n}; q)_{P_n}}{(q; q)_{P_n} (q^{-\mathrm{L}}; q)_{\mathsf{D}_n - C_n + P_n}} \displaystyle\prod_{j=1}^{n-1} \displaystyle\frac{(q; q)_{C_j + D_j - P_j}}{(q; q)_{C_j - P_j} (q; q)_{P_j} (q; q)_{B_j - P_j}}.
		\end{flalign*} 
		
		\noindent Before proceeding, we next simplify the powers of $-1$ and $q$ appearing on the right side. Doing so directly yields 
		\begin{flalign*}
			& U_{q^{\mathfrak{N}-\mathrm{L}+1}; q^{-\mathrm{L}/2}, q^{-\mathfrak{M}/2}} (\bm{A}, \bm{B}; \bm{C}, \bm{D}) \\
			& \quad  =  (-1)^{C_n} q^{\varphi (\overline{\bm{D}}, \bm{C})} q^{\mathrm{L}C_n + \binom{\mathsf{B}_n}{2} - \binom{\mathsf{D}_n}{2}} q^{\binom{\mathsf{d}+1}{2} - \binom{\mathsf{b}+1}{2} + \mathsf{b} - \mathsf{d} + \binom{c+1}{2}} \\
			& \qquad \times \displaystyle\frac{(q; q)_{\mathsf{b}}}{(q; q)_{\mathsf{d}}} \displaystyle\frac{(q^{\mathfrak{M}-\mathfrak{N}-c}; q)_c}{(q; q)_{C_n}} \displaystyle\frac{(q^{-\mathrm{L}}; q)_{\mathsf{D}_n}}{(q^{-\mathrm{L}}; q)_{\mathsf{B}_n}} \displaystyle\frac{(q^{-\mathfrak{N}}; q)_\mathrm{L}}{(q^{\mathfrak{M}-\mathfrak{N}}; q)_\mathrm{L}} \displaystyle\prod_{j=1}^{n-1} \displaystyle\frac{(q; q)_{B_j}}{(q; q)_{D_j}} \\
			& \qquad \times \displaystyle\sum_{p=0}^{\min \{ \mathrm{L} - \mathsf{b}, c \}} \displaystyle\sum_{\bm{P} \le \bm{B}, \bm{C}} (-1)^{p + P_n} q^{\binom{\mathsf{D}_n - C_n + P_n}{2} + P_n C_n - \binom{P_n}{2} - \binom{\mathsf{d}-c+p}{2} + \binom{p}{2} + p - pc} q^{\binom{\mathsf{b}+p}{2} - \binom{\mathsf{B}_n+P_n}{2}} \\
			& \qquad \qquad \qquad \qquad \quad  \times q^{\varphi (\overline{\bm{B}} - \overline{\bm{D}} - \overline{\bm{P}}, \bm{P})} \displaystyle\frac{(q^{-\mathfrak{N}-1}; q)_p (q^{\mathfrak{M}-\mathfrak{N}}; q)_{\mathsf{d}-c+p}}{(q^{\mathfrak{M}-\mathfrak{N}-c}; q)_p (q^{-\mathfrak{N}}; q)_{\mathsf{b}+p}} \displaystyle\frac{(q^{-\mathrm{L}}; q)_{\mathsf{B}_n + P_n} (q^{-C_n}; q)_{P_n}}{(q; q)_{P_n} (q^{-\mathrm{L}}; q)_{\mathsf{D}_n - C_n + P_n}} \\
			& \qquad \qquad \qquad \qquad \quad  \times \displaystyle\prod_{j=1}^{n-1} \displaystyle\frac{(q; q)_{C_j + D_j - P_j}}{(q; q)_{C_j - P_j} (q; q)_{P_j} (q; q)_{B_j - P_j}}.
		\end{flalign*} 
		
		\noindent Let us continue to simplify the power of $q$ above. Since 
		\begin{flalign*}
			& \binom{\mathsf{D}_n - C_n + P_n}{2} + P_n C_n - \binom{P_n}{2} = \binom{\mathsf{D}_n - C_n}{2} + P_n \mathsf{D}_n; \\ 
			& \binom{\mathsf{B}_n + P_n}{2} = \binom{\mathsf{B}_n}{2} + \binom{P_n}{2} + \mathsf{B}_n P_n; \\
			&  \binom{\mathsf{d}-c+p}{2} - \binom{p}{2} + pc = \binom{\mathsf{d}-c}{2} + p \mathsf{d}; \qquad \binom{\mathsf{b}+p}{2} = \binom{\mathsf{b}+1}{2} - \mathsf{b} + \binom{p}{2} + \mathsf{b} p, 
		\end{flalign*}
		
		\noindent it follows that
		\begin{flalign}
			\label{qrn5}
			\begin{aligned}
				& U_{q^{\mathfrak{N}-\mathrm{L}+1}; q^{-\mathrm{L}/2}, q^{-\mathfrak{M}/2}} (\bm{A}, \bm{B}; \bm{C}, \bm{D}) \\
				& \quad  =  (-1)^{C_n} q^{\varphi (\overline{\bm{D}}, \bm{C})} q^{\binom{\mathsf{D}_n - C_n}{2} - \binom{\mathsf{d}-c}{2} + \mathrm{L}C_n - \binom{\mathsf{D}_n}{2}} q^{\binom{\mathsf{d}+1}{2} - \mathsf{d} + \binom{c+1}{2}} \\
				& \qquad \times \displaystyle\frac{(q; q)_{\mathsf{b}}}{(q; q)_{\mathsf{d}}} \displaystyle\frac{(q^{\mathfrak{M}-\mathfrak{N}-c}; q)_c}{(q; q)_{C_n}} \displaystyle\frac{(q^{-\mathrm{L}}; q)_{\mathsf{D}_n}}{(q^{-\mathrm{L}}; q)_{\mathsf{B}_n}} \displaystyle\frac{(q^{-\mathfrak{N}}; q)_\mathrm{L}}{(q^{\mathfrak{M}-\mathfrak{N}}; q)_\mathrm{L}} \displaystyle\prod_{j=1}^{n-1} \displaystyle\frac{(q; q)_{B_j}}{(q; q)_{D_j}} \\
				& \qquad \times \displaystyle\sum_{p=0}^{\min \{ \mathrm{L} - \mathsf{b}, c \}} \displaystyle\sum_{\bm{P} \le \bm{B}, \bm{C}} (-1)^{p + P_n} q^{P_n (\mathsf{D}_n - \mathsf{B}_n) + p (\mathsf{b} - \mathsf{d} + 1) + \binom{p}{2} - \binom{P_n}{2}} \\
				& \qquad \qquad \qquad \qquad \quad  \times q^{\varphi (\overline{\bm{B}} - \overline{\bm{D}} - \overline{\bm{P}}, \bm{P})} \displaystyle\frac{(q^{-\mathfrak{N}-1}; q)_p (q^{\mathfrak{M}-\mathfrak{N}}; q)_{\mathsf{d}-c+p}}{(q^{\mathfrak{M}-\mathfrak{N}-c}; q)_p (q^{-\mathfrak{N}}; q)_{\mathsf{b}+p}} \displaystyle\frac{(q^{-\mathrm{L}}; q)_{\mathsf{B}_n + P_n} (q^{-C_n}; q)_{P_n}}{(q; q)_{P_n} (q^{-\mathrm{L}}; q)_{\mathsf{D}_n - C_n + P_n}} \\
				& \qquad \qquad \qquad \qquad \quad  \times \displaystyle\prod_{j=1}^{n-1} \displaystyle\frac{(q; q)_{C_j + D_j - P_j}}{(q; q)_{C_j - P_j} (q; q)_{P_j} (q; q)_{B_j - P_j}}.
			\end{aligned} 
		\end{flalign} 
		
		\noindent We additionally have by \eqref{bbdd} and the equality $p = \overline{p} + P_n$ that 
		\begin{flalign*} 
			& \binom{\mathsf{d}+1}{2} - \mathsf{d} = \binom{\mathsf{d}}{2}; \qquad \varphi (\overline{\bm{B}} - \overline{\bm{D}} - \overline{\bm{P}}, \bm{P}) = \varphi (\overline{\bm{B}} + \overline{\bm{D}} - \overline{\bm{P}}, \overline{\bm{P}}) + (\overline{b} - \overline{d} - \overline{p}) P_n; \\ 
			& \binom{p}{2} - \binom{P_n}{2} - \overline{p} P_n = \binom{\overline{p}}{2}; \qquad p (\mathsf{b}-\mathsf{d}+1) = \overline{p} (\mathsf{b} - \mathsf{d} + 1) + P_n (\mathsf{b} - \mathsf{d}) + P_n; \\
			& P_n (\mathsf{b} - \mathsf{d} + \mathsf{D}_n - \mathsf{B}_n + \overline{b} - \overline{d}) = 0; \qquad \binom{\mathsf{D}_n-C_n}{2} - \binom{\mathsf{D}_n}{2} = \binom{C_n+1}{2} - C_n \mathsf{D}_n; \\
			& \varphi (\overline{\bm{D}}, \bm{C}) = \varphi (\overline{\bm{D}}, \overline{\bm{C}}) + C_n \overline{d}; \qquad \binom{c+1}{2} + \binom{\mathsf{d}}{2} - \binom{\mathsf{d}-c}{2} = c \mathsf{d} = \overline{c} \mathsf{d} +C_n (\mathsf{D}_n-\overline{d}).
		\end{flalign*} 
		
		\noindent Together with \eqref{bbdd} and \eqref{qrn5}, these give
		\begin{flalign*}
			U_{q^{\mathfrak{N}-\mathrm{L}+1}; q^{-\mathrm{L}/2}, q^{-\mathfrak{M}/2}} (\bm{A}, \bm{B}; \bm{C}, \bm{D}) 
			& = (-1)^{C_n} q^{\varphi (\overline{\bm{D}}, \overline{\bm{C}}) + \binom{C_n+1}{2} + \overline{c} \mathsf{d} + \mathrm{L}C_n} \\
			& \qquad \times \displaystyle\frac{(q; q)_{\mathsf{b}}}{(q; q)_{\mathsf{d}}} \displaystyle\frac{(q^{\mathfrak{M}-\mathfrak{N}-c}; q)_c}{(q; q)_{C_n}} \displaystyle\frac{(q^{-\mathrm{L}}; q)_{\mathsf{D}_n}}{(q^{-\mathrm{L}}; q)_{\mathsf{B}_n}} \displaystyle\frac{(q^{-\mathfrak{N}}; q)_\mathrm{L}}{(q^{\mathfrak{M}-\mathfrak{N}}; q)_\mathrm{L}} \displaystyle\prod_{j=1}^{n-1} \displaystyle\frac{(q; q)_{B_j}}{(q; q)_{D_j}} \\
			& \qquad \times \displaystyle\sum_{p=0}^{\min \{ \mathrm{L} - \mathsf{b}, c \}} \displaystyle\sum_{\bm{P} \le \bm{B}, \bm{C}}(-1)^{\overline{p}} q^{\overline{p} (\mathsf{b}-\mathsf{d}+1) + P_n + \binom{\overline{p}}{2}} q^{\varphi (\overline{\bm{B}} - \overline{\bm{D}} - \overline{\bm{P}}, \overline{\bm{P}})}\\
			& \qquad \qquad  \times \displaystyle\frac{(q^{\mathfrak{M}-\mathfrak{N}}; q)_{\mathsf{d}-c+p}}{(q^{\mathfrak{M}-\mathfrak{N}-c}; q)_p} \displaystyle\frac{(q^{-\mathrm{L}}; q)_{\mathsf{B}_n + P_n} (q^{-C_n}; q)_{P_n}}{(q; q)_{P_n} (q^{-\mathrm{L}}; q)_{\mathsf{D}_n - C_n + P_n}} \\
			& \qquad \qquad  \times  \displaystyle\frac{(q^{-\mathfrak{N}-1}; q)_p}{(q^{-\mathfrak{N}}; q)_{\mathsf{b}+p}} \displaystyle\prod_{j=1}^{n-1} \displaystyle\frac{(q; q)_{C_j + D_j - P_j}}{(q; q)_{C_j - P_j} (q; q)_{P_j} (q; q)_{B_j - P_j}}.
		\end{flalign*} 
		
		\noindent Next, observe by \eqref{bbdd} and the equality $p = \overline{p} + P_n$ that  
		\begin{flalign*}
			& (q^{\mathfrak{M}-\mathfrak{N}-c}; q)_c (q^{\mathfrak{M}-\mathfrak{N}}; q)_{\mathsf{d}-c+p}  = (q^{\mathfrak{M}-\mathfrak{N}-c}; q)_{\mathsf{d}} (q^{\mathfrak{M} - \mathfrak{N} + \mathsf{d} - c}; q)_{\overline{p}} (q^{\mathfrak{M} - \mathfrak{N} + \mathsf{d} - c + \overline{p}}; q)_{P_n}; \\
			& (q^{\mathfrak{M}-\mathfrak{N}-c}; q)_p = (q^{\mathfrak{M}-\mathfrak{N}-c}; q)_{\overline{p}} (q^{\mathfrak{M}-\mathfrak{N}-c+\overline{p}}; q)_{P_n}; \quad \displaystyle\frac{(q^{-\mathrm{L}}; q)_{\mathsf{D}_n}}{(q^{-\mathrm{L}}; q)_{\mathsf{D}_n - C_n + P_n}} = \displaystyle\frac{(q^{\mathsf{D}_n - C_n - \mathrm{L}}; q)_{C_n}}{(q^{\mathsf{D}_n - C_n - \mathrm{L}}; q)_{P_n}}; \\
			& \displaystyle\frac{(q^{-\mathrm{L}}; q)_{\mathsf{B}_n + P_n}}{(q^{-\mathrm{L}}; q)_{\mathsf{B}_n}} = (q^{\mathsf{B}_n - \mathrm{L}}; q)_{P_n}; \quad \displaystyle\frac{(q^{-\mathfrak{N}-1}; q)_p}{(q^{-\mathfrak{N}}; q)_{\mathsf{b}+p}} = \displaystyle\frac{(q^{-\mathfrak{N}-1}; q)_{\overline{p}} (q^{\overline{p}-\mathfrak{N}-1}; q)_{P_n}}{(q^{-\mathfrak{N}}; q)_{\mathsf{b}} (q^{\mathsf{b}-\mathfrak{N}}; q)_{\overline{p}} (q^{\mathsf{b} - \mathfrak{N} + \overline{p}}; q)_{P_n}},
		\end{flalign*}
		
		\noindent from which it follows that 
		\begin{flalign*}
			& U_{q^{\mathfrak{N}-\mathrm{L}+1}; q^{-\mathrm{L}/2}, q^{-\mathfrak{M}/2}} (\bm{A}, \bm{B}; \bm{C}, \bm{D}) \\
			& \quad= (-1)^{C_n} q^{\varphi (\overline{\bm{D}}, \overline{\bm{C}}) + \binom{C_n+1}{2} + \overline{c} \mathsf{d} + \mathrm{L}C_n} \displaystyle\frac{(q; q)_{\mathsf{b}}}{(q; q)_{\mathsf{d}}} \displaystyle\frac{(q^{\mathfrak{M}-\mathfrak{N}-c}; q)_{\mathsf{d}}}{(q; q)_{C_n}} \displaystyle\frac{(q^{\mathsf{D}_n - C_n -\mathrm{L}}; q)_{C_n}}{(q^{-\mathfrak{N}}; q)_{\mathsf{b}}} \displaystyle\frac{(q^{-\mathfrak{N}}; q)_\mathrm{L}}{(q^{\mathfrak{M}-\mathfrak{N}}; q)_\mathrm{L}} \displaystyle\prod_{j=1}^{n-1} \displaystyle\frac{(q; q)_{B_j}}{(q; q)_{D_j}} \\
			& \qquad \times \displaystyle\sum_{\overline{\bm{P}} \le \overline{\bm{B}}, \overline{\bm{C}}} (-1)^{\overline{p}} q^{\varphi (\overline{\bm{B}} - \overline{\bm{D}} - \overline{\bm{P}}, \overline{\bm{P}})} \displaystyle\frac{(q^{-\mathfrak{N}-1}; q)_{\overline{p}} (q^{\mathfrak{M}-\mathfrak{N}+\mathsf{d}-c}; q)_{\overline{p}}}{(q^{\mathfrak{M}-\mathfrak{N}-c}; q)_{\overline{p}} (q^{\mathsf{b}-\mathfrak{N}}; q)_{\overline{p}}} \displaystyle\prod_{j=1}^{n-1} \displaystyle\frac{(q; q)_{C_j + D_j - P_j}}{(q; q)_{C_j - P_j} (q; q)_{P_j} (q; q)_{B_j - P_j}} \\
			& \qquad \quad \times q^{\overline{p} (\mathsf{b}-\mathsf{d}+1) + \binom{\overline{p}}{2}} \displaystyle\sum_{P_n = 0}^{C_n} q^{P_n} \displaystyle\frac{(q^{\overline{p} - \mathfrak{N} - 1}; q)_{P_n}}{(q^{\mathfrak{M}-\mathfrak{N}-c+\overline{p}}; q)_{P_n}} \displaystyle\frac{(q^{\mathfrak{M}-\mathfrak{N}+\mathsf{d}-c+\overline{p}}; q)_{P_n}}{(q^{\mathsf{b}-\mathfrak{N}+\overline{p}}; q)_{P_n}}	\displaystyle\frac{(q^{\mathsf{B}_n-\mathrm{L}}; q)_{P_n} (q^{-C_n}; q)_{P_n}}{(q; q)_{P_n} (q^{\mathsf{D}_n - C_n - \mathrm{L}}; q)_{P_n}},
		\end{flalign*} 
		
		\noindent where we used the fact that any summands on the right side with $P_n > B_n = \mathrm{L} - \mathsf{B}_n$ are equal to $0$, due to the factor of $(q^{\mathsf{B}_n - \mathrm{L}}; q)_{P_n}$. This, with $(q^{-\mathfrak{N}}; q)_\mathrm{L} = (q^{-\mathfrak{N}}; q)_{\infty} (q^{\mathrm{L} - \mathfrak{N}}; q)_{\infty}^{-1}$, $(q^{\mathfrak{M}-\mathfrak{N}}; q)_\mathrm{L}^{-1} = (q^{\mathrm{L} + \mathfrak{M} - \mathfrak{N}}; q)_{\infty} (q^{\mathfrak{M} - \mathfrak{N}}; q)_{\infty}^{-1}$, and \eqref{43abcd}, yields the lemma.
	\end{proof}

	\subsection{Degenerations of the Fused Complemented Weights} 
	
	\label{LimitU}  
	
	We next proceed to take limits of the weights from \Cref{2u} that will eventually lead us to the log-gamma polymer. The first is to let $q^{\mathfrak{M}}$ and $q^{\mathfrak{N}}$ tend to infinity, in such a way that ($\mathfrak{L}$ and) $q^{\mathfrak{M} - \mathfrak{N}} = q^{-\mathfrak{L}} \gamma$ (for some constant $\gamma \in \mathbb{C}$) remains fixed.

	\begin{lem} 
		
		\label{ufusedqmqn} 
		
		Adopting the notation of \Cref{rqnl1}, we have for any complex number $\gamma \in \mathbb{C}$ that 
		\begin{flalign}
			\label{qmnlimit0} 
			\begin{aligned}
				& \displaystyle\lim_{q^{\mathfrak{M}} \rightarrow \infty} \mathcal{U}_{q^{\mathfrak{L}}; q^{\mathfrak{M}}; q^{\mathfrak{M}+\mathfrak{L}} / \gamma} \big(\bm{A}, (\overline{\bm{B}}, \mathsf{B}_n); \bm{C}, (\overline{\bm{D}}, \mathsf{D}_n) \big) \\
				& \quad= (-1)^{C_n} q^{\varphi (\overline{\bm{D}}, \overline{\bm{C}}) + \binom{C_n+1}{2} + \overline{c} \mathsf{d} + \mathfrak{L}C_n} \cdot \displaystyle\prod_{j=1}^{n-1} \displaystyle\frac{(q;q)_{B_j}}{(q;q)_{D_j}} \cdot \mathbbm{1}_{\overline{\bm{A}} + \overline{\bm{B}} = \overline{\bm{C}} + \overline{\bm{D}}} \cdot \mathbbm{1}_{\mathsf{b} - a = \mathsf{d} - c}\\
				& \qquad \times \displaystyle\frac{(q; q)_{\mathsf{b}}}{(q; q)_{\mathsf{d} - C_n}} \displaystyle\frac{(q^{-c-\mathfrak{L}} \gamma; q)_{\mathsf{d}}}{(q; q)_{C_n}} \displaystyle\frac{(q^{\mathsf{B}_n -\mathfrak{L}}; q)_{C_n} (\gamma; q)_{\infty}}{(q^{\mathfrak{L} + c - C_n - \overline{p} + 1} \gamma^{-1}; q)_{C_n} (q^{-\mathfrak{L}} \gamma; q)_{\infty}} \\
				& \qquad \times \displaystyle\sum_{\overline{\bm{P}} \le \overline{\bm{B}}, \overline{\bm{C}}} (-1)^{\overline{p}} q^{\varphi (\overline{\bm{B}} - \overline{\bm{D}} - \overline{\bm{P}}, \overline{\bm{P}})} \displaystyle\frac{(q^{\mathsf{d}-c-\mathfrak{L}} \gamma; q)_{\overline{p}}}{(q^{-\mathfrak{L}-c} \gamma; q)_{\overline{p}}} \displaystyle\prod_{j=1}^{n-1} \displaystyle\frac{(q; q)_{C_j + D_j - P_j}}{(q; q)_{C_j - P_j} (q; q)_{P_j} (q; q)_{B_j - P_j}} \\
				& \qquad \qquad \quad \times q^{\overline{p} (\mathsf{b}-\mathsf{d}+1) + \binom{\overline{p}}{2}} \cdot  {_3 \varphi_2} \bigg( \begin{array}{cc} q^{-C_n}; q^{-A_n}, q^{\mathfrak{L} - C_n + c - \overline{p} + 1} \gamma^{-1} \\ q^{\mathsf{d}-C_n+1}, q^{\mathfrak{L} - \mathsf{B}_n - C_n + 1} \end{array} \bigg| q, q \bigg).
			\end{aligned} 	
		\end{flalign} 
		
	\end{lem} 
	
	\begin{proof} 
		
		We assume throughout this proof that $\overline{\bm{A}} + \overline{\bm{B}} = \overline{\bm{C}} + \overline{\bm{D}}$ and $\mathsf{b} - a = \mathsf{d} - c$, for otherwise both sides of \eqref{qmnlimit0} are equal to $0$ by \Cref{rqnl1}. Then letting $q^\mathfrak{N}$ and $q^\mathfrak{M}$ tend to $\infty$, while keeping $q^{\mathfrak{M}-\mathfrak{N}} = q^{-\mathfrak{L}} \gamma$ fixed, in \Cref{rqnl1} yields 
		\begin{flalign} 
			\label{qrn6} 
			\begin{aligned}
				& \displaystyle\lim_{q^{\mathfrak{M}} \rightarrow \infty} \mathcal{U}_{q^{\mathfrak{L}}; q^{\mathfrak{M}}; q^{\mathfrak{M} + \mathfrak{L}} / \gamma} \big(\bm{A}, (\overline{\bm{B}}, \mathsf{B}_n); \bm{C}, (\overline{\bm{D}}, \mathsf{D}_n) \big) \\
				& \quad= (-1)^{C_n} q^{\varphi (\overline{\bm{D}}, \overline{\bm{C}}) + \binom{C_n+1}{2} + \overline{c} \mathsf{d} + \mathfrak{L}C_n} \displaystyle\frac{(q; q)_{\mathsf{b}}}{(q; q)_{\mathsf{d}}} \displaystyle\frac{(q^{-c-\mathfrak{L}} \gamma; q)_{\mathsf{d}}}{(q; q)_{C_n}} \displaystyle\frac{(q^{\mathsf{D}_n - C_n -\mathfrak{L}}; q)_{C_n} (\gamma; q)_{\infty}}{(q^{-\mathfrak{L}} \gamma; q)_{\infty}} \displaystyle\prod_{j=1}^{n-1} \displaystyle\frac{(q; q)_{B_j}}{(q; q)_{D_j}} \\
				& \qquad \times \displaystyle\sum_{\overline{\bm{P}} \le \overline{\bm{B}}, \overline{\bm{C}}} (-1)^{\overline{p}} q^{\varphi (\overline{\bm{B}} - \overline{\bm{D}} - \overline{\bm{P}}, \overline{\bm{P}})} \displaystyle\frac{(q^{\mathsf{d}-c-\mathfrak{L}} \gamma; q)_{\overline{p}}}{(q^{-\mathfrak{L}-c} \gamma; q)_{\overline{p}}} \displaystyle\prod_{j=1}^{n-1} \displaystyle\frac{(q; q)_{C_j + D_j - P_j}}{(q; q)_{C_j - P_j} (q; q)_{P_j} (q; q)_{B_j - P_j}} \\
				& \qquad \qquad \quad \times q^{\overline{p} (\mathsf{b}-\mathsf{d}+1) + \binom{\overline{p}}{2}} \cdot {_3 \varphi_2} \bigg( \begin{array}{cc} q^{-C_n}; q^{\mathsf{B}_n - \mathfrak{L}}, q^{\mathsf{d}-c+\overline{p}-\mathfrak{L}} \gamma \\ q^{\overline{p} - \mathfrak{L} - c} \gamma, q^{\mathsf{D}_n - C_n - \mathfrak{L}} \end{array} \bigg| q, q \bigg).
			\end{aligned}
		\end{flalign} 
		
		Next recall from the Sears identity  \cite[Equation (3.2.9)]{BHS} that, for any integer $m \ge 0$ and complex numbers $A, B, C, D, E, F \in \mathbb{C}$ with $q^{1-m} ABC = DEF$, we have 
		\begin{flalign*}
			{_4 \varphi_3} \bigg( \begin{array}{cccc} q^{-m}; A, B, C \\ D, E, F \end{array} \bigg| q, q \bigg) = \displaystyle\frac{(B; q)_m \big( \frac{DE}{AB}; q \big)_m \big( \frac{DE}{BC}; q \big)_m}{(D; q)_m (E; q)_m \big( \frac{DE}{ABC}; q \big)_m} \cdot {_4 \varphi_3} \bigg( \begin{array}{cccc} q^{-m}; \frac{D}{B}, \frac{E}{B}, \frac{DE}{ABC} \\ \frac{DE}{AB}, \frac{DE}{BC}, q^{1-m} B^{-1} \end{array} \bigg| q, q \bigg).
		\end{flalign*}
		
		\noindent Letting $A$ and $E$ tend to $0$ in such a way that $\frac{A}{E} = q^{m-1} \frac{DF}{BC}$ remains fixed, we deduce (under no restrictions on $B, C, D, F \in \mathbb{C}$) that
		\begin{flalign*}
			{_3 \varphi_2} \bigg( \begin{array}{cccc} q^{-m}; B, C \\ D, F \end{array} \bigg| q, q \bigg) = \displaystyle\frac{(B; q)_m (q^{1-m} C F^{-1}; q)_m}{(D; q)_m (q^{1-m} F^{-1}; q)_m} \cdot {_3 \varphi_2} \bigg( \begin{array}{cccc} q^{-m}; \frac{D}{B}, q^{1-m} F^{-1} \\ q^{1-m} C F^{-1}, q^{1-m} B^{-1} \end{array} \bigg| q, q \bigg).
		\end{flalign*}
		
		\noindent Taking $m = C_n$ and $(B, C; D, F) = (q^{\mathsf{B}_n - \mathfrak{L}}, q^{\mathsf{d} - c + \overline{p} - \mathfrak{L}} \gamma; q^{\mathsf{D}_n - C_n - \mathfrak{L}}, q^{\overline{p} - \mathfrak{L} - c} \gamma)$ (and using the fact that $\mathsf{B}_n - A_n = \mathsf{D}_n - C_n$, by \eqref{bbdd} and the equalities $\mathsf{b} - a = \mathsf{d} - c$ and $\overline{\bm{A}} + \overline{\bm{B}} = \overline{\bm{C}} + \overline{\bm{D}}$), we obtain   
		\begin{flalign*}
			{_3 \varphi_2} \bigg( \begin{array}{cc} q^{-C_n}; q^{\mathsf{B}_n - \mathfrak{L}}, q^{\mathsf{d}-c+\overline{p}-\mathfrak{L}} \gamma \\ q^{\overline{p}-\mathfrak{L}-c} \gamma, q^{\mathsf{D}_n - C_n - \mathfrak{L}} \end{array} \bigg| q, q \bigg) & = \displaystyle\frac{(q^{\mathsf{B}_n - \mathfrak{L}}; q)_{C_n} (q^{\mathsf{d} - C_n + 1}; q)_{C_n}}{(q^{\mathsf{D}_n - C_n - \mathfrak{L}}; q)_{C_n} (q^{\mathfrak{L}+c-C_n - \overline{p}+1} \gamma^{-1}; q)_{C_n}} \\
			& \qquad \times  {_3 \varphi_2} \bigg( \begin{array}{cc} q^{-C_n}; q^{-A_n}, q^{\mathfrak{L}-C_n + c - \overline{p} + 1} \gamma^{-1} \\ q^{\mathsf{d}-C_n + 1}, q^{\mathfrak{L}-C_n - \mathsf{B}_n + 1}  \end{array} \bigg| q, q \bigg).
		\end{flalign*} 
		
		\noindent Inserting this into \eqref{qrn6}, and using the fact that $(q^{\mathsf{d} - C_n + 1}; q)_{C_n} (q; q)_{\mathsf{d}}^{-1} = (q; q)_{\mathsf{d} - C_n}^{-1}$, gives the lemma.
	\end{proof}
	
	We next take the further limit in \Cref{ufusedqmqn} as $q^{\mathfrak{L}}$ tends to $\infty$. The below definition provides this limit; see the lemma that follows.
	
	\begin{definition} 
		
		\label{ugamma}
		
		Adopting the notation of \Cref{2u}, define for any $\gamma \in \mathbb{C}$ the \emph{$q$-discrete polymer weight} $\mathcal{U}_{\gamma}^{\qdp; n} \big( \bm{A}, (\overline{\bm{B}}, \mathsf{B}_n); \bm{C}, (\overline{\bm{D}}, \mathsf{D}_n) \big)= \mathcal{U}_{\gamma}^{\qdp} \big( \bm{A}, (\overline{\bm{B}}, \mathsf{B}_n); \bm{C}, (\overline{\bm{D}}, \mathsf{D}_n) \big)$ by
		\begin{flalign}
			\label{uqdp}
			\begin{aligned} 
				\mathcal{U}_{\gamma}^{\qdp} & \big(\bm{A}, (\overline{\bm{B}}, \mathsf{B}_n); \bm{C}, (\overline{\bm{D}}, \mathsf{D}_n) \big) \\
				& = q^{\varphi (\overline{\bm{D}}, \overline{\bm{C}}) + \overline{c} ( \mathsf{B}_n - A_n - \overline{d})} (\gamma; q)_{\infty} \displaystyle\frac{(q;q)_{\mathsf{b}}}{(q;q)_{\mathsf{B}_n - A_n - \overline{d}}} \cdot \mathbbm{1}_{\overline{\bm{A}} + \overline{\bm{B}} = \overline{\bm{C}} + \overline{\bm{D}}} \cdot \mathbbm{1}_{\mathsf{b} - a = \mathsf{d} - c} \cdot \displaystyle\prod_{j=1}^{n-1} \displaystyle\frac{(q;q)_{C_j+D_j}}{(q;q)_{C_j} (q;q)_{D_j}} \\
				& \qquad \times \displaystyle\sum_{k=0}^{C_n} \displaystyle\frac{\gamma^{C_n - k}}{(q; q)_{C_n-k}} \displaystyle\frac{(q^{A_n-k+1}; q)_k}{(q; q)_k (q^{\mathsf{B}_n - A_n - \overline{d}+1}; q)_k} q^{k(\mathsf{B}_n - A_n + \overline{c} +k)}\\
				& \qquad \quad \times \displaystyle\sum_{\overline{\bm{P}} \le \overline{\bm{B}}, \overline{\bm{C}}} q^{\overline{p} (A_n - k + 1) + \varphi (\overline{\bm{P}}, \overline{\bm{D}} - \overline{\bm{B}})} \displaystyle\prod_{j=1}^{n-1} \displaystyle\frac{(q^{-B_j}; q)_{P_j} (q^{-C_j}; q)_{P_j}}{(q;q)_{P_j} (q^{-C_j - D_j}; q)_{P_j}}.
			\end{aligned}
		\end{flalign}
		
	\end{definition}

	\begin{rem}
		
		\label{n1qdp}
		
		If $n = 1$, then the $\mathcal{U}^{\qdp}$ weights from \Cref{ugamma} coincide with those of the geometric $q$-PushTASEP introduced in \cite[Section 6.3]{RCRP}, which degenerates to the log-gamma polymer \cite[Theorem 8.7]{RCRP}. This is the reason behind the term, ``$q$-discrete polymer weight'' in \Cref{ugamma}. 
		
	\end{rem}

	\begin{lem}
		
		\label{ufusedql} 
		
		Under the notation of \Cref{ufusedqmqn}, we have
		\begin{flalign*}
			\displaystyle\lim_{q^{\mathfrak{L}} \rightarrow \infty} \bigg( \displaystyle\lim_{q^{\mathfrak{M}} \rightarrow \infty} \mathcal{U}_{q^{\mathfrak{L}}; q^{\mathfrak{M}}; q^{\mathfrak{M} + \mathfrak{L}} / \gamma} \big( \bm{A}, (\overline{\bm{B}}, \mathsf{B}_n); \bm{C}, (\overline{\bm{D}}, \mathsf{D}_n) \big) \bigg) = \mathcal{U}_{\gamma}^{\qdp} \big( \bm{A}, (\overline{\bm{B}}, \mathsf{B}_n); \bm{C}, (\overline{\bm{D}}, \mathsf{D}_n) \big). 
		\end{flalign*} 
		
	\end{lem} 
	
	\begin{proof} 
		
		Throughout this proof, we assume that $\overline{\bm{A}} + \overline{\bm{B}} = \overline{\bm{C}} + \overline{\bm{D}}$ and $\mathsf{b} - a = \mathsf{d} - c$, for otherwise the lemma holds by \Cref{ufusedqmqn} and \Cref{ugamma}. Since $\lim_{a \rightarrow \infty} a^{-k} (ab; q)_k = (-b)^k q^{\binom{k}{2}}$ for any $b \in \mathbb{C}$ and $k \in \mathbb{Z}_{\ge 0}$, we have  
		\begin{flalign*}
			& \displaystyle\lim_{q^{\mathfrak{L}} \rightarrow \infty} (-1)^{C_n} q^{\mathfrak{L} C_n + \binom{C_n+1}{2}} (q^{\mathfrak{L} + c - C_n - \overline{p} + 1} \gamma^{-1}; q)_{C_n}^{-1} = q^{C_n (C_n - c + \overline{p})} \gamma^{C_n}; \\
			& \displaystyle\lim_{q^{\mathfrak{L}} \rightarrow \infty} {_3 \varphi_2} \bigg( \begin{array}{cc} q^{-C_n}; q^{-A_n}, q^{\mathfrak{L} - C_n + c - \overline{p} + 1} \gamma^{-1} \\ q^{\mathsf{d} - C_n + 1}, q^{\mathfrak{L} - \mathsf{B}_n - C_n + 1} \end{array} \bigg| q, q \bigg) = {_2 \varphi_1} \bigg( \begin{array}{cc} q^{-C_n}; q^{-A_n} \\ q^{\mathsf{d} - C_n + 1} \end{array} \bigg| q, q^{\mathsf{B}_n + c - \overline{p} + 1 } \gamma^{-1} \bigg),
		\end{flalign*}
		
		\noindent where in the second statement we used \eqref{43abcd}. Inserting these (with the fact that $C_n - c = -\overline{c}$), into the limit as $q^{\mathfrak{L}}$ tends to $\infty$ of \eqref{ufusedqmqn}, we find	
		\begin{flalign}
			\label{qrn7} 
			\begin{aligned}
				\displaystyle\lim_{q^{\mathfrak{L}} \rightarrow \infty} & \bigg( \displaystyle\lim_{q^{\mathfrak{M}} \rightarrow \infty} \mathcal{U}_{q^{\mathfrak{L}}; q^{\mathfrak{M}}; q^{\mathfrak{M} + \mathfrak{L}} / \gamma} \big( \bm{A}, (\overline{\bm{B}}, \mathsf{B}_n); \bm{C}, (\overline{\bm{D}}, \mathsf{D}_n) \big) \bigg) \\
				& \quad= q^{\varphi (\overline{\bm{D}}, \overline{\bm{C}}) + \overline{c} (\mathsf{d} - C_n)} \gamma^{C_n} \displaystyle\frac{(q; q)_{\mathsf{b}}}{(q; q)_{\mathsf{d}-C_n} (q; q)_{C_n}} (\gamma; q)_{\infty} \cdot \displaystyle\prod_{j=1}^{n-1} \displaystyle\frac{(q; q)_{B_j}}{(q; q)_{D_j}} \\
				& \qquad \times \displaystyle\sum_{\overline{\bm{P}} \le \overline{\bm{B}}, \overline{\bm{C}}} (-1)^{\overline{p}} q^{\varphi (\overline{\bm{B}} - \overline{\bm{D}} - \overline{\bm{P}}, \overline{\bm{P}})} \cdot \displaystyle\prod_{j=1}^{n-1} \displaystyle\frac{(q; q)_{C_j + D_j - P_j}}{(q; q)_{C_j - P_j} (q; q)_{P_j} (q; q)_{B_j - P_j}} \\
				& \qquad \qquad \quad \times q^{\overline{p} (\mathsf{b}-\mathsf{d}+1) + \binom{\overline{p}}{2}} q^{\overline{p} C_n} \cdot {_2 \varphi_1} \bigg( \begin{array}{cc} q^{-C_n}, q^{-A_n} \\ q^{\mathsf{d}-C_n+1} \end{array} \bigg| q, q^{\mathsf{B}_n + c - \overline{p} + 1} \gamma^{-1} \bigg).
			\end{aligned} 
		\end{flalign} 
		
		\noindent Next we separate the $P_j$ subscripts in the $q$-Pochhammer symbols on the right side of \eqref{qrn7} from the other ones. To that end, observe since $(q; q)_{m-k} = (-1)^k q^{\binom{k}{2} - mk} (q; q)_m (q^{-m}; q)_k^{-1}$ that we have
		\begin{flalign}
			\label{qcdq} 
			& \displaystyle\frac{(q; q)_{C_j + D_j - P_j}}{(q; q)_{C_j - P_j} (q; q)_{B_j - P_j}} = (-1)^{P_j} q^{P_j (B_j - D_j) - \binom{P_j}{2}} \displaystyle\frac{(q; q)_{C_j + D_j}}{(q; q)_{C_j} (q; q)_{B_j}} \displaystyle\frac{(q^{-C_j}; q)_{P_j} (q^{-B_j}; q)_{P_j}}{(q^{-C_j - D_j}; q)_{P_j}}.
		\end{flalign}
		
		\noindent We also have 
		\begin{flalign}
			\label{sumq0} 
			\begin{aligned} 
				\varphi (\overline{\bm{B}} - \overline{\bm{D}} - \overline{\bm{P}}, \overline{\bm{P}}) + \overline{p} (\mathsf{b} - \mathsf{d} + C_n + 1) + \binom{\overline{p}}{2} + \displaystyle\sum_{j=1}^{n-1} \Bigg( & P_j (B_j - D_j) - \binom{P_j}{2} \Bigg) \\
				& = \overline{p} (A_n + 1) + \varphi (\overline{\bm{P}}, \overline{\bm{D}} - \overline{\bm{B}}),
			\end{aligned}
		\end{flalign}
		
		\noindent since 
		\begin{flalign*}
			& \binom{\overline{p}}{2} = \varphi (\overline{\bm{P}}, \overline{\bm{P}}) + \displaystyle\sum_{j=1}^{n-1} \binom{P_j}{2}; \qquad \varphi (\overline{\bm{B}} - \overline{\bm{D}}, \overline{\bm{P}}) + \overline{p} (\overline{d} - \overline{b}) + \displaystyle\sum_{j=1}^{n-1} P_j (B_j - D_j) = \varphi (\bm{\overline{P}}, \bm{\overline{D}} - \bm{\overline{B}}); \\
			& \overline{p} (\overline{b} - \overline{d}) + \overline{p} (\mathsf{b}-\mathsf{d}) + \overline{p} C_n = \overline{p} (\mathsf{B}_n - \mathsf{D}_n) + \overline{p} C_n = \overline{p} A_n.
		\end{flalign*}
		
		\noindent Inserting \eqref{qcdq} and \eqref{sumq0} into \eqref{qrn7} yields
		\begin{flalign*} 
			\displaystyle\lim_{q^{\mathfrak{L}} \rightarrow \infty} & \bigg( \displaystyle\lim_{q^{\mathfrak{M}} \rightarrow \infty} \mathcal{U}_{q^{\mathfrak{L}}; q^{\mathfrak{M}}; q^{\mathfrak{M} + \mathfrak{L}} / \gamma} \big( \bm{A}, (\overline{\bm{B}}, \mathsf{B}_n); \bm{C}, (\overline{\bm{D}}, \mathsf{D}_n) \big) \bigg) \\
			& = q^{\varphi (\overline{\bm{D}}, \overline{\bm{C}}) + \overline{c} (\mathsf{d}-C_n)} \gamma^{C_n} \displaystyle\frac{(q; q)_{\mathsf{b}}}{(q; q)_{\mathsf{d}-C_n} (q; q)_{C_n}} (\gamma; q)_{\infty} \cdot \displaystyle\prod_{j=1}^{n-1} \displaystyle\frac{(q; q)_{C_j + D_j}}{(q; q)_{C_j} (q; q)_{D_j}} \\
			& \qquad \times \displaystyle\sum_{\overline{\bm{P}} \le \overline{\bm{B}}, \overline{\bm{C}}} q^{\overline{p} (A_n + 1) + \varphi (\overline{\bm{P}}, \overline{\bm{D}} - \overline{\bm{B}})} \cdot \displaystyle\prod_{j=1}^{n-1} \displaystyle\frac{(q^{-C_j}; q)_{P_j} (q^{-B_j}; q)_{P_j}}{(q^{-C_j - D_j}; q)_{P_j} (q; q)_{P_j}} \\
			& \qquad \qquad \quad \times {_2 \varphi_1} \bigg( \begin{array}{cc} q^{-C_n}, q^{-A_n} \\ q^{\mathsf{d}-C_n+1} \end{array} \bigg| q, q^{\mathsf{B}_n + c - \overline{p} + 1} \gamma^{-1} \bigg).
		\end{flalign*}
		
		\noindent Thus, using \eqref{43abcd} to write the ${_2 \varphi_1}$ basic hypergeometric series as a sum, we obtain
		\begin{flalign*}
			\displaystyle\lim_{q^{\mathfrak{L}} \rightarrow \infty} & \bigg( \displaystyle\lim_{q^{\mathfrak{M}} \rightarrow \infty} \mathcal{U}_{q^{\mathfrak{L}}; q^{\mathfrak{M}}; q^{\mathfrak{M} + \mathfrak{L}} / \gamma} \big( \bm{A}, (\overline{\bm{B}}, \mathsf{B}_n); \bm{C}, (\overline{\bm{D}}, \mathsf{D}_n) \big) \bigg) \\
			& = q^{\varphi (\overline{\bm{D}}, \overline{\bm{C}}) + \overline{c} (\mathsf{d}-C_n)} (\gamma; q)_{\infty} \displaystyle\frac{(q;q)_{\mathsf{b}}}{(q;q)_{\mathsf{d}-C_n}} \displaystyle\prod_{j=1}^{n-1} \displaystyle\frac{(q;q)_{C_j+D_j}}{(q;q)_{C_j} (q;q)_{D_j}} \\
			& \qquad \times \displaystyle\sum_{k=0}^{C_n} \displaystyle\frac{\gamma^{C_n - k}}{(q; q)_{C_n}} \displaystyle\frac{(q^{-C_n}; q)_k (q^{-A_n}; q)_k}{(q; q)_k (q^{\mathsf{d}-C_n+1}; q)_k} q^{k(\mathsf{B}_n + c +1)} \\
			& \qquad \qquad \times \displaystyle\sum_{\overline{\bm{P}} \le \overline{\bm{B}}, \overline{\bm{C}}} q^{\overline{p} (A_n - k + 1) + \varphi (\overline{\bm{P}}, \overline{\bm{D}} - \overline{\bm{B}})} \displaystyle\prod_{j=1}^{n-1} \displaystyle\frac{(q^{-B_j}; q)_{P_j} (q^{-C_j}; q)_{P_j}}{(q;q)_{P_j} (q^{-C_j - D_j}; q)_{P_j}}.
		\end{flalign*}
		
		\noindent Together with the facts that 
		\begin{flalign*} 
			(q^{-C_n}; q)_k (q; q)_{C_n}^{-1} = (-1)^k q^{\binom{k}{2} - C_n k} (q; q)_{C_n - k}^{-1}; \qquad (q^{-A_n}; q)_k = (-1)^k q^{\binom{k}{2} - A_n k} (q^{A_n - k + 1}; q)_k,
		\end{flalign*} 
		
		\noindent that $c - C_n = \overline{c}$, and that $\mathsf{d} - C_n =\mathsf{D}_n - \overline{d} - C_n = \mathsf{B}_n - A_n - \overline{d}$, this yields the lemma.
	\end{proof}

	\begin{rem} 
		
		\label{stochasticqdp}
		
		By \Cref{stochasticu}, \Cref{rqnl1}, \Cref{ufusedqmqn}, \Cref{ufusedql}, and uniqueness of analytic continuation (with the fact that the $\mathcal{U}^{\qdp}$ weights are rational in $q^{\mathfrak{L}}$, to extend from $\mathfrak{L} \in \mathbb{Z}_{\ge 0}$ to $\mathfrak{L} \in \mathbb{C}$) the $\mathcal{U}^{\qdp}$ weights are stochastic, in the sense that for any $\bm{A} \in \mathbb{Z}_{\ge 0}^n$ and $(\overline{\bm{B}}, \mathsf{B}_n) \in \mathbb{Z}_{\ge 0}^{n-1} \times \mathbb{Z}_{\ge 0}$, we have $\sum \mathcal{U}_{\gamma} \big( \bm{A}, (\overline{\bm{B}}, \mathsf{B}_n); \bm{C}, (\overline{\bm{D}}, \mathsf{D}_n) \big) = 1$, where we sum over all $\bm{C} \in \mathbb{Z}_{\ge 0}^n$ and $(\overline{\bm{D}}, \mathsf{D}_n) \in \mathbb{Z}_{\ge 0}^{n-1} \times \mathbb{Z}_{\ge 0}$.

	\end{rem}

	\subsection{An Additional Degenerated Weight} 
	
	\label{WeightD} 
	
	In this section we provide an additional degeneration of the complemented $U_{z; r, s}$-weight, which will eventually serve as boundary weights for the polymer model. We begin with the following lemma from \cite{SVMP} providing a specialization of the $U_{z; r, s}$-weights from \Cref{wzabcd}.
	
	\begin{lem}[{\cite[Lemma 6.8]{SVMP}}]
		
		\label{qd1} 
		
		For any integer $\mathrm{L} \ge 1$ and $n$-tuples $\bm{A}, \bm{C}, \bm{D} \in \mathbb{Z}_{\ge 0}^n$, we have
		\begin{flalign*}
			\displaystyle\lim_{s \rightarrow 0} U_{z/s^2; q^{-\mathrm{L}/2}, s} ( \bm{A}, \mathrm{L} \bm{e}_n; \bm{C}, \bm{D}) = \displaystyle\frac{(zq^{\mathrm{L}})^{|\bm{D}|}}{(z; q)_{\mathrm{L}}} \displaystyle\frac{(q^{-\mathrm{L}}; q)_{|\bm{D}|}}{(q; q)_{|\bm{D}|}} \cdot \mathbbm{1}_{\bm{A} + \mathrm{L} \cdot \bm{e}_n = \bm{C} + \bm{D}} \cdot \mathbbm{1}_{\bm{D} = |\bm{D}| \cdot \bm{e}_n}.
		\end{flalign*}
		
	\end{lem} 
	
	\noindent We will next implement the complementation procedure explained in Section \ref{WeightsComplement} on the weights specialized as in \Cref{qd1}. The below definition provides the eventual form they will take, which is shown in the corollary following it.
	
	\begin{definition}
		
		\label{u3}
		
		For any complex numbers $z, \mathfrak{L} \in \mathbb{C}$ and integer $k \ge 0$, define the \emph{$q$-discrete boundary weight} 
		\begin{flalign*}
			\mathcal{U}_{z; q^{\mathfrak{L}}}^{\bq; n} (k) = \mathcal{U}_{z; q^{\mathfrak{L}}}^{\bq} (k) = z^{-k} \displaystyle\frac{(z^{-1}; q)_{\infty}}{(q^{-\mathfrak{L}} z^{-1}; q)_{\infty}} \displaystyle\frac{(q^{-\mathfrak{L}}; q)_{k}}{(q; q)_{k}}. 
		\end{flalign*}
		
		\noindent We further set 
		\begin{flalign*} 
			\mathcal{U}_z^{\bq; n} (k) = \mathcal{U}_z^{\bq} (k) = \mathcal{U}_{z; \infty}^{\bq} (k) = z^{-k} \displaystyle\frac{(z^{-1}; q)_k}{(q; q)_k}.
		\end{flalign*} 
	
		\noindent Observe that both the $\mathcal{U}_{z; q^{\mathfrak{L}}}^{\bq}$-weights and the $\mathcal{U}_z^{\bq}$-weights are stochastic, by the $q$-binomial theorem.
	\end{definition}
	
	\begin{cor}
		
		\label{qd2}
		
		Adopt the notation of \Cref{qd1}, and set $\mathsf{D}_n = \mathrm{L} - |\bm{D}|$. We have 
		\begin{flalign}
			\label{qd4}
			\displaystyle\lim_{s \rightarrow 0} U_{z/s^2; q^{-\mathrm{L}/2}, s} (\bm{A}, \mathrm{L} \bm{e}_n; \bm{C}, \bm{D}) = \mathcal{U}_{z/q; q^{\mathrm{L}}}^{\bq} (\mathsf{D}_n) \cdot \mathbbm{1}_{\bm{A} + \mathsf{D}_n \cdot \bm{e}_n = \bm{C}} \cdot \mathbbm{1}_{\bm{D} = (\mathrm{L} - \mathsf{D}_n) \cdot \bm{e}_n}.
		\end{flalign}
	\end{cor}
	
	\begin{proof}
		
		We may assume throughout this proof that $\bm{A} + \mathrm{L} \bm{e}_n = \bm{C} + \bm{D}$ and that $\bm{D} = |\bm{D}| \bm{e}_n = (\mathrm{L} - \mathsf{D}_n) \bm{e}_n$, for otherwise both sides of \eqref{qd4} are equal to $0$ (by \Cref{qd1}). Then, combining the equalities  
		\begin{flalign*} 
			& (zq^{\mathrm{L}})^{|\bm{D}|} = (zq^{\mathrm{L}})^{\mathrm{L} - \mathsf{D}_n}; \qquad (z; q)_{\mathrm{L}} = (-z)^{\mathrm{L}} q^{\binom{\mathrm{L}}{2}} (q^{1-\mathrm{L}} z^{-1}; q)_{\mathrm{L}}; \\ 
			& (q^{-\mathrm{L}}; q)_{\mathrm{L} - \mathsf{D}_n} = (-1)^{\mathrm{L} - \mathsf{D}_n} q^{\binom{\mathsf{D}_n+1}{2} - \binom{\mathrm{L} + 1}{2}} \displaystyle\frac{(q; q)_{\mathrm{L}}}{(q; q)_{\mathsf{D}_n}}; \qquad (q; q)_{\mathrm{L} - \mathsf{D}_n} = (-1)^{\mathsf{D}_n} q^{\binom{\mathsf{D}_n}{2} -\mathrm{L} \mathsf{D}_n} \displaystyle\frac{(q; q)_{\mathrm{L}}}{(q^{-\mathrm{L}}; q)_{\mathsf{D}_n}},
		\end{flalign*} 
		
		\noindent with \Cref{qd2}, we deduce 
		\begin{flalign*}
			\displaystyle\lim_{s \rightarrow 0} U_{z/s^2; q^{-\mathrm{L}/2}, s} (\bm{A}, \mathrm{L} \bm{e}_n; \bm{C}, \bm{D}) = \displaystyle\frac{(zq^{\mathrm{L}})^{\mathrm{L} - \mathsf{D}_n}}{(z; q)_{\mathrm{L}}} \displaystyle\frac{(q^{-\mathrm{L}}; q)_{\mathrm{L} - \mathsf{D}_n}}{(q; q)_{\mathrm{L} - \mathsf{D}_n}} = \displaystyle\frac{(qz^{-1})^{\mathsf{D}_n}}{(q^{1 - \mathrm{L}} z^{-1}; q)_{\mathrm{L}}} \displaystyle\frac{(q^{-\mathrm{L}}; q)_{\mathsf{D}_n}}{(q; q)_{\mathsf{D}_n}}.
		\end{flalign*}
		
		\noindent This, with \Cref{u3}, yields the corollary.
	\end{proof}

	\subsection{The $q \rightarrow 1$ Limit }
	
	\label{Weightsq1} 
	
	In this section we analyze the limit as $q$ tends to $1$ of the $q$-discrete polymer weights $\mathcal{U}^{\qdp}$ from \Cref{ugamma} (and their boundary counterparts $\mathcal{U}^{\bq}$ from \Cref{u3}). Throughout this section, we fix a real number $\theta > 0$ and $n$-tuples of nonnegative real numbers $(\mathfrak{a}_1, \mathfrak{a}_2, \ldots , \mathfrak{a}_n)$ and $(\mathfrak{b}_1, \mathfrak{b}_2, \ldots , \mathfrak{b}_n)$. Moreover, $\varepsilon \in (0, 1)$ will be a parameter that we view as tending to $0$. If for some functions $f$ and $g$ of $\varepsilon$ we have $\lim_{\varepsilon \rightarrow 0} \big| f(\varepsilon) - g(\varepsilon) \big| = 0$, then we write $f \sim g$. We further let $\mathsf{B}_n = \mathsf{B}_n^{\varepsilon} \in \mathbb{Z}_{\ge 0}$ be an integer, $\bm{\overline{B}} = \bm{\overline{B}}^{\varepsilon} \in \mathbb{Z}_{\ge 0}^{n-1}$ be an $(n-1)$-tuple, and $\bm{A} = \bm{A}^{\varepsilon} \in \mathbb{Z}_{\ge 0}^n$ be an $n$-tuple (all dependent on $\varepsilon$), with 
	\begin{flalign}
		\label{anaj} 
		\varepsilon A_n -  \log \varepsilon^{-1} \sim  \mathfrak{a}_n; \qquad \varepsilon A_j \sim  \mathfrak{a}_j; \qquad \varepsilon \mathsf{B}_n - \log \varepsilon^{-1} = \mathfrak{b}_n; \qquad \varepsilon B_j \sim \mathfrak{b}_j, 
	\end{flalign}
	
	\noindent for each $j \in \llbracket 1, n-1 \rrbracket$. Also define $q = q_{\varepsilon} \in (0, 1)$ and $\gamma = \gamma_{\varepsilon, \theta} \in (0, 1)$ by setting
	\begin{flalign}
		\label{qgammaq}
		q_{\varepsilon} = e^{-\varepsilon}; \qquad \gamma = q^{\theta}.
	\end{flalign}
	
	Recalling from \Cref{stochasticqdp} that the $\mathcal{U}^{\qdp}$ weights are stochastic, the following proposition describes the limiting law of the $n$-tuple $\bm{C}$ sampled\footnote{We do not verify here that the $\mathcal{U}_{\gamma}^{\qdp}$-weights are all nonnegative for fixed $q \in (0, 1)$. However, it can quickly be deduced from the proof of \Cref{sumpdb} below that the sum of the absolute values of all negative $\mathcal{U}_{\gamma}^{\qdp}$-weights tends to $0$, as $\varepsilon$ tends to $0$. Together with the stochasticity of the $\mathcal{U}^{\qdp}$ weights, this implies that one can view $\mathcal{U}^{\qdp}$ as prescribing a probability measure, as $\varepsilon$ tends to $0$.} according to $\mathcal{U}_{\gamma}^{\qdp} (\bm{A}, \bm{B}; \bm{C}, \bm{D})$ as $\varepsilon$ tends to $0$; it indicates that this law can be expressed through a single $\mathrm{Gamma}(\theta)$ random variable $\mathfrak{Y}$.	
	
	\begin{prop} 
		
		\label{abcgamma} 
		
		Let $q = q_{\varepsilon} \in (0, 1)$ and $\gamma = \gamma_{\varepsilon, \theta} \in (0, 1)$ be as in \eqref{qgammaq}; also let $\mathsf{B}_n \in \mathbb{Z}_{\ge 0}$, $\overline{\bm{B}} \in \mathbb{Z}_{\ge 0}^{n-1}$, and $\bm{A} \in \mathbb{Z}_{\ge 0}^n$ be as in \eqref{anaj}. Sample $\big( \bm{C}, (\overline{\bm{D}}, \mathsf{D}_n) \big) \in \mathbb{Z}_{\ge 0}^n \times \mathbb{Z}_{\ge 0}^{n-1} \times \mathbb{Z}_{\ge 0}$ with probability $\mathcal{U}_{\gamma}^{\qdp} \big(\bm{A}, (\overline{\bm{B}}, \mathsf{B}_n); \bm{C}, (\overline{\bm{D}}, \mathsf{D}_n) \big)$, and denote 
		\begin{flalign*}
			\mathfrak{c}_n = \varepsilon C_n - \log \varepsilon^{-1}, \quad \text{and} \quad \mathfrak{c}_j = \varepsilon C_j, \qquad \text{for each $j \in \llbracket 1, n-1 \rrbracket$}. 
		\end{flalign*} 
		
		\noindent The joint law of $(\mathfrak{c}_{[j, n]})_{j \in \llbracket 1, n \rrbracket}$ converges to that of $\big( \log (e^{\mathfrak{a}_{[j,n]} - \mathfrak{b}_n + \mathfrak{b}_{[j,n-1]}} + 1) - \log \mathfrak{Y} \big)_{j \in \llbracket 1, n \rrbracket}$ as $\varepsilon$ tends to $0$, where $\mathfrak{Y} \in \mathbb{R}_{> 0}$ is a $\mathrm{Gamma}(\theta)$ random variable. 
		
	\end{prop} 

	\begin{rem} 
		
		\label{dabcgamma} 
		
		 Denoting $(\mathfrak{d}_1, \mathfrak{d}_2, \ldots , \mathfrak{d}_n) = (\varepsilon D_1, \varepsilon D_2, \ldots , \varepsilon D_{n-1}, \varepsilon \mathsf{D}_n - \log \varepsilon^{-1})$ in \Cref{abcgamma}, the joint limiting law of $(\mathfrak{d}_{[j,n]})_{j \in \llbracket 1, n \rrbracket}$ is determined by that of $(\mathfrak{c}_{[j,n]})_{j \in \llbracket 1, n \rrbracket}$ by arrow conservation. 
		
	\end{rem} 
	
	To establish \Cref{abcgamma}, we will express the right side of \eqref{uqdp} as a convolution of the following two functions, which we will then analyze separately. 
	
	\begin{definition} 
		
		\label{psi} 
		
		Define the function $\Phi_{\varepsilon}: \mathbb{Z}_{\ge 0} \rightarrow \mathbb{R}$ by setting, for any integer $\ell \ge 0$, 
		\begin{flalign}
			\label{gammaq} 
			\Phi_{\varepsilon} (\ell) = (\gamma; q)_{\infty} \cdot \displaystyle\frac{\gamma^{\ell}}{(q;q)_{\ell}}.
		\end{flalign}
		
		\noindent Also define the function $\Psi_{\varepsilon}: \mathbb{Z}_{\ge 0} \times \mathbb{Z}_{\ge 0}^{n-1}$ by setting, for any integer $k \ge 0$ and $(n-1)$-tuple $\overline{\bm{C}} \in \mathbb{Z}_{\ge 0}^{n-1}$,
		\begin{flalign}
			\label{psik}
			\begin{aligned}
				\Psi_{\varepsilon} ( k; \overline{\bm{C}})  & = q^{\varphi (\overline{\bm{D}}, \overline{\bm{C}}) + \overline{c} (\mathsf{B}_n - A_n - \overline{d})} \displaystyle\frac{(q;q)_{\mathsf{b}}}{(q;q)_{\mathsf{B}_n - A_n -\overline{d}}} \cdot  \displaystyle\prod_{j=1}^{n-1} \displaystyle\frac{(q;q)_{C_j+D_j}}{(q;q)_{C_j} (q;q)_{D_j}} \cdot q^{k(\mathsf{B}_n - A_n + \overline{c} + k)} \\
				& \quad \times \displaystyle\frac{(q^{A_n-k+1}; q)_k}{(q; q)_k (q^{\mathsf{B}_n - A_n - \overline{d}+1}; q)_k} \displaystyle\sum_{\overline{\bm{P}} \le \overline{\bm{B}}, \overline{\bm{C}}} q^{\overline{p} (A_n - k + 1) + \varphi (\overline{\bm{P}}, \overline{\bm{D}} - \overline{\bm{B}})} \displaystyle\prod_{j=1}^{n-1} \displaystyle\frac{(q^{-B_j}; q)_{P_j} (q^{-C_j}; q)_{P_j}}{(q;q)_{P_j} (q^{-C_j - D_j}; q)_{P_j}},
			\end{aligned}
		\end{flalign}
		
		\noindent where $\overline{\bm{D}} = \overline{\bm{A}} + \overline{\bm{B}} - \overline{\bm{C}}$.
		
	\end{definition} 
	
	\begin{rem}
		
		\label{usumpsi} 
		
		By \Cref{ugamma}, we have 
		\begin{flalign}
			\label{upsi} 
			\mathcal{U}_{\gamma}^{\qdp} \big( \bm{A}, (\overline{\bm{B}}, \mathsf{B}_n); \bm{C}, (\overline{\bm{D}}, \mathsf{D}_n) \big) = \displaystyle\sum_{k=0}^{C_n} \Phi_{\varepsilon} (C_n - k) \cdot \Psi_{\varepsilon} (k; \overline{\bm{C}}).
		\end{flalign}
		
		\noindent Hence, to sample $\bm{C}$ as in \Cref{abcgamma}, we may first sample $\ell$ and $(k; \overline{\bm{C}})$ under the density functions $\Phi_{\varepsilon}$ and $\Psi_{\varepsilon}$ (the fact that these give stochastic weights is due to \Cref{sum1} below), respectively, and then set $\bm{C} = (\overline{\bm{C}}, k + \ell)$. 
		
	\end{rem}

	\begin{lem} 
		
		\label{sum1} 
		
		We have  
		\begin{flalign}
			\label{psisum1}
			\displaystyle\sum_{\ell=0}^{\infty} \Phi_{\varepsilon} (\ell) = 1; \qquad \displaystyle\sum_{\substack{k \ge 0 \\ \overline{\bm{C}} \in \mathbb{Z}_{\ge 0}^{n-1}}} \Psi_{\varepsilon} (k; \overline{\bm{C}}) = 1. 
		\end{flalign}
		
	\end{lem} 
	
	\begin{proof} 
		
		The first statement in \eqref{psisum1} holds by the $q$-binomial theorem. The second holds since  
		\begin{flalign*}
			\displaystyle\sum_{\substack{k \ge 0 \\ \overline{\bm{C}} \in \mathbb{Z}_{\ge 0}^{n-1}}} \Psi_{\varepsilon} (k; \overline{\bm{C}}) = \displaystyle\sum_{\substack{k, \ell \ge 0 \\ \overline{\bm{C}} \in \mathbb{Z}_{\ge 0}^{n-1}}} \Phi_{\varepsilon} (\ell) \cdot \Psi_{\varepsilon} (k; \overline{\bm{C}}) = \displaystyle\sum_{\substack{C_n \ge 0 \\ \overline{\bm{C}} \in \mathbb{Z}_{\ge 0}^{n-1} \\ \bm{D} \in \mathbb{Z}_{\ge 0}^n}} \mathcal{U}_{\gamma}^{\qdp} \big( \bm{A}, (\mathsf{B}_n, \overline{\bm{B}}); (\overline{\bm{C}}, C_n), (\mathsf{D}_n, \overline{\bm{D}}) \big) = 1, 
		\end{flalign*}
		
		\noindent where in the first equality we used the first statement of \eqref{psisum1}; in the second we used \eqref{upsi}; and in the third we used the stochasticity of $\mathcal{U}^{\qdp}$ (from \Cref{stochasticqdp}).
	\end{proof}
	
	We next have the following two lemmas that describe the probability measures arising as the limit when $\varepsilon$ tends to $0$ of those with distribution functions $\Phi_{\varepsilon}$ and $\Psi_{\varepsilon}$. The first shows that the former (as well as the $\mathcal{U}_z^{\bq}$ weights from \Cref{u3}) gives rise to the $\mathrm{Gamma}(\theta)$ random variable from \Cref{abcgamma}; it is due to \cite{SWP}. The second shows that the latter is given by the delta distribution at a specific pair of real number and $(n-1)$-tuple; we outline its proof (which is similar to previous results, such as \cite[Lemma 8.17]{RCRP} and \cite[Proposition 6.19]{SVMP}) in Section \ref{ProofSum2} below.

	\begin{lem}[{\cite[Lemma 2.1]{SWP}}]
		
		\label{gamma0} 
		
		The following two statements hold.
		
		\begin{enumerate} 
			\item Sample $\ell_{\varepsilon} \in \mathbb{Z}_{\ge 0}$ according to the distribution $\mathbb{P} [\ell_{\varepsilon} = \ell] = \Phi_{\varepsilon} (\ell)$. Then, $\varepsilon^{-1} \cdot e^{-\varepsilon \ell_{\varepsilon}}$ converges in distribution to a $\mathrm{Gamma} (\theta)$ random variable, as $\varepsilon$ tends to $0$. 
			
			\item Sample $\ell_{\varepsilon} \in \mathbb{Z}_{\ge 0}$ according to the distribution $\mathbb{P} [\ell_{\varepsilon} = \ell] = \mathcal{U}_{1/\gamma}^{\bq} (\ell)$. Then, $\varepsilon^{-1} \cdot e^{-\varepsilon \ell_{\varepsilon}}$ converges in distribution to a $\mathrm{Gamma} (\theta)$ random variable, as $\varepsilon$ tends to $0$. 
		\end{enumerate} 
	\end{lem}

	\begin{lem}
		
		\label{sumpdb} 
		
		Sample an $(n-1)$-tuple and integer $(\overline{\bm{C}}; k) \in \mathbb{Z}_{\ge 0}^{n-1} \times \mathbb{Z}_{\ge 0}$ with probability $\Psi_{\varepsilon} (\overline{\bm{C}}; k)$; denote $\bm{\mathfrak{c}} = (\mathfrak{c}_1, \mathfrak{c}_2, \ldots , \mathfrak{c}_{n-1}) = \varepsilon \cdot \overline{\bm{C}} \in \mathbb{R}_{\ge 0}^{n-1}$ and $\mathfrak{k} = \varepsilon k \in \mathbb{R}$. As $\varepsilon$ tends to $0$, we have that
		\begin{flalign}
			\label{kcj0} 
			e^{\mathfrak{k} + \mathfrak{c}_{[j,n-1]}} \quad \text{converges in probability to the constant} \quad e^{\mathfrak{a}_{[j,n]} - \mathfrak{b}_n + \mathfrak{b}_{[j,n-1]}} + 1, \quad \text{for each $j \in \llbracket 1, n \rrbracket$}.
		\end{flalign}
		
	\end{lem}
	
	Now we can quickly establish \Cref{abcgamma}. 
	
	\begin{proof}[Proof of \Cref{abcgamma}]
		
		This follows from \Cref{usumpsi}, \Cref{gamma0}, and \Cref{sumpdb}.	
	\end{proof} 
	
	\subsection{Proof Outline of \Cref{sumpdb}}
	
	\label{ProofSum2}
	
	In this section we outline the proof of \Cref{sumpdb}.

	\begin{proof}[Proof of \Cref{sumpdb} (Outline)]
		
		First, it is quickly verified that the contribution to the second sum in \eqref{psisum1} coming from $k \ge (2 \varepsilon^{-1}) \log \varepsilon^{-1}$ becomes negligible as $\varepsilon$ tends to $0$, due to the factor of $q^{k (k + \mathsf{B}_n - A_n + \overline{c})}$ on the right side of \eqref{psik} (namely, the fact that the exponent of $q$ is quadratic in $k$). We may therefore restrict our attention to when $k \le (2 \varepsilon)^{-1} \log \varepsilon^{-1}$, in which case it can be confirmed that
		\begin{flalign}
			\label{sumplimit} 
			\displaystyle\lim_{\varepsilon \rightarrow 0} \displaystyle\sum_{\overline{\bm{P}} \le \overline{\bm{B}}, \overline{\bm{C}}} q^{\overline{p} (A_n - k + 1) + \varphi (\overline{\bm{P}}, \overline{\bm{D}} - \overline{\bm{B}})} \displaystyle\prod_{j=1}^{n-1} \displaystyle\frac{(q^{-B_j}; q)_{P_j} (q^{-C_j}; q)_{P_j}}{(q;q)_{P_j} (q^{-C_j - D_j}; q)_{P_j}} = 1.
		\end{flalign}
		
		\noindent Indeed, for sufficiently small $\varepsilon$, we have $q^{\overline{p} (A_n - k + 1)} \le \varepsilon^{\overline{p}/3}$ (recalling \eqref{anaj}, \eqref{qgammaq}, and the fact that $k \le (2 \varepsilon)^{-1} \log \varepsilon^{-1}$), while the remaining factors in \eqref{sumplimit} grow at most exponentially in $\overline{p}$ (with a uniformly bounded base). Therefore the sum on the left side of \eqref{sumplimit} is asymptotically supported on the term $\bm{P} = \bm{e}_0$, for which $p=0$, which gives \eqref{sumplimit}. 
		
		The remaining factors on the right side of \eqref{psik} are nonnegative, since $q \in (0, 1)$. Let us analyze how their $k$-dependent parts 
		\begin{flalign}
			\label{kcf} 
			F_{\varepsilon} (k; \bm{\overline{C}}) = q^{k (k + \mathsf{B}_n - A_n + \overline{c})} \displaystyle\frac{(q^{A_n-k+1}; q)_k}{(q;q)_k (q^{\mathsf{B}_n - A_n - \overline{d} + 1}; q)_k},
		\end{flalign} 
		
		\noindent behave as $\varepsilon$ tends to $0$. Observe for any real numbers $y, z \in \mathbb{R}$ (with $y \ge 0$) that
		\begin{flalign}
			\label{qqa} 
			\varepsilon \cdot \log (q^{z/\varepsilon}; q)_{\lfloor y/\varepsilon \rfloor} = \varepsilon \displaystyle\sum_{j=0}^{\lfloor y/\varepsilon \rfloor} \log (1 - q^{z/\varepsilon + j}) = \varepsilon \displaystyle\sum_{x \in [0, y] \cap \varepsilon \cdot \mathbb{Z}} \log (1 - e^{-x-z}) \sim \displaystyle\int_0^y \log (e^{-x-z}) dx,
		\end{flalign}
		
		\noindent as $\varepsilon$ tends to $0$. Applying this in \eqref{kcf}, and setting $\mathfrak{k} = \varepsilon k$ and $\bm{\overline{\mathfrak{c}}} = (\mathfrak{c}_1, \mathfrak{c}_2, \ldots , \mathfrak{c}_{n-1}) = \varepsilon \cdot \overline{\bm{C}}$, we obtain 
		\begin{flalign}
			\label{0f}
			\varepsilon \cdot \log F_{\varepsilon} (k; \overline{\bm{C}}) & \sim \mathfrak{k} (\mathfrak{a}_n - \mathfrak{b}_n - \mathfrak{k} - \mathfrak{c}_{[1,n-1]}) - \displaystyle\int_0^{\mathfrak{k}} \log (1 - e^{-x}) dx \\
			& \qquad - \displaystyle\int_0^{\mathfrak{k}} \log (1 - e^{\mathfrak{a}_{[1,n]} + \mathfrak{b}_{[1,n-1]} - \mathfrak{b}_n -\mathfrak{c}_{[1,n-1]} - x}) dx,
		\end{flalign}
		
		\noindent where we used the facts that $\overline{d} = \overline{a} + \overline{b} - \overline{c}$ and that $\varepsilon \cdot \log (q^{A_n-k+1}; q)_k \sim 0$ (as $\varepsilon A_n$ grows faster than any constant as $\varepsilon$ tends to $0$, by \eqref{anaj}). Denoting the right side of \eqref{0f} by $G(\mathfrak{k})$, we find  
		\begin{flalign*}
			& G' (\mathfrak{k}) = \mathfrak{a}_n - \mathfrak{b}_n -\mathfrak{c}_{[1,n-1]} - 2 \mathfrak{k} - \log (1 - e^{-\mathfrak{k}}) - \log (1 - e^{\mathfrak{a}_{[1,n]} + \mathfrak{b}_{[1,n-1]} - \mathfrak{b}_n - \mathfrak{c}_{[1,n-1]} - \mathfrak{k}}); \\
			& G'' (\mathfrak{k}) = -2 - \displaystyle\frac{e^{-\mathfrak{k}}}{1 - e^{-\mathfrak{k}}} - \displaystyle\frac{e^{\mathfrak{a}_{[1,n]} + \mathfrak{b}_{[1,n-1]} - \mathfrak{b}_n - \mathfrak{c}_{[1,n-1]} - \mathfrak{k}}}{1 - e^{\mathfrak{a}_{[1,n]} + \mathfrak{b}_{[1,n-1]} - \mathfrak{b}_n - \mathfrak{c}_{[1,n-1]} - \mathfrak{k}}} \le -2 < 0.
		\end{flalign*}
		
		\noindent Due to the negativity of $G''$, we deduce that $G$ is maximized at the solution $\mathfrak{k}$ to the equation $G' (\mathfrak{k}) = 0$, or equivalently to the unique nonnegative solution of  
		\begin{flalign}
			\label{k0abc}
			(e^{\mathfrak{k}}-1) (e^{\mathfrak{k}} - e^{\mathfrak{a}_{[1,n]} - \mathfrak{b}_n + \mathfrak{b}_{[1,n-1]} - \mathfrak{c}_{[1,n-1]}}) = e^{\mathfrak{a}_n - \mathfrak{b}_n - \mathfrak{c}_{[1,n-1]}}.
		\end{flalign}
		
		\noindent Hence, as $\varepsilon$ tends to $0$, the function $F_{\varepsilon} (k; \overline{\bm{C}})$, and thus $\Psi_{\varepsilon} (k; \overline{\bm{C}})$ is maximized when $\varepsilon k = \mathfrak{k}$. By the strict concavity of $G$, these functions decay exponentially in $\varepsilon^{-1/2} \cdot |\varepsilon k - \mathfrak{k}|$. Hence, sampling $(k; \overline{\bm{C}})$ under $\Psi$, we find that \eqref{k0abc} must hold as $\varepsilon$ tends to $0$.
		
		We next analyze the terms in $\Psi_{\varepsilon} (k; \bm{\overline{C}})$ that asymptotically depend on a given $C_i$. Denoting
		\begin{flalign*}
			R = \mathsf{B}_n - A_n - \overline{a} - \overline{b} + k,
		\end{flalign*} 
		
		\noindent this quantity is given by (using \eqref{sumplimit} with the fact that $\overline{\bm{D}} = \overline{\bm{A}} + \overline{\bm{B}} - \overline{\bm{C}}$)  
		\begin{flalign*}
			F_{\varepsilon; i} (k; \overline{\bm{C}}) = q^{\varphi (\overline{\bm{A}} + \overline{\bm{B}} - \overline{\bm{C}}, \overline{\bm{C}}) + \overline{c} (R + \overline{c})} \displaystyle\frac{1}{(q; q)_{R + \overline{c}}} \displaystyle\frac{1}{(q; q)_{C_i} (q; q)_{A_j + B_j - C_i}}.
		\end{flalign*}
		
		\noindent Using \eqref{qqa}, and setting 
		\begin{flalign}
			\label{rk} 
			\mathfrak{r} =  \mathfrak{b}_n - \mathfrak{a}_{[1,n]} - \mathfrak{b}_{[1,n-1]} + \mathfrak{k} \sim \varepsilon R; \qquad \mathfrak{p}_j = \mathfrak{a}_j + \mathfrak{b}_j,
		\end{flalign} 
		
		\noindent for each $j \in \llbracket 1, n-1 \rrbracket$, it follows that 
		\begin{flalign}
			\label{0fi} 
			\begin{aligned} 
				\varepsilon \cdot \log F_{\varepsilon; i} (k; \overline{\bm{C}}) & \sim \displaystyle\sum_{1 \le j < h \le n-1} \mathfrak{c}_j \mathfrak{c}_h - \displaystyle\sum_{j=1}^{n-1} \mathfrak{p}_{[1,j-1]} \mathfrak{c}_j - \mathfrak{c}_{[1,n-1]} \mathfrak{r} - \mathfrak{c}_{[1,n-1]}^2 - \displaystyle\int_0^{\mathfrak{r} + \mathfrak{c}_{[1,n-1]}} \log (1 - e^{-x}) dx \\
				& \qquad - \displaystyle\int_0^{\mathfrak{c}_i} \log (1 - e^{-x}) dx - \displaystyle\int_0^{\mathfrak{p}_i - \mathfrak{c}_i} \log (1 - e^{-x}) dx.
			\end{aligned} 	
		\end{flalign}
		
		\noindent Denoting the right side of \eqref{0fi} (as a function in $\mathfrak{c}_i$) by $G_i (\mathfrak{c}_i)$, we find that 
		\begin{flalign*}
			G_i' (\mathfrak{c}_i) & = \displaystyle\sum_{j \ne i} \mathfrak{c}_i + \mathfrak{p}_{[1,i-1]} - \mathfrak{r} - 2 \mathfrak{c}_{[1,n-1]} - \log (1 - e^{-\mathfrak{c}_i}) - \log (1 - e^{-\mathfrak{r} - \mathfrak{c}_{[1,n-1]}}) + \log (1 - e^{\mathfrak{c}_i - \mathfrak{p}_i}) \\
			& = \log (1 - e^{\mathfrak{c}_i - \mathfrak{p}_i}) - \mathfrak{p}_{[1, i-1]} - \log (e^{\mathfrak{c}_i}-1) - \log (e^{\mathfrak{r} + \mathfrak{c}_{[1,n-1]}}-1),
		\end{flalign*}
		
		\noindent and hence
		\begin{flalign*}
			G_i'' (\mathfrak{c}_i) = - \displaystyle\frac{e^{\mathfrak{c}_i-\mathfrak{p}_i}}{1-e^{\mathfrak{c}_i - \mathfrak{p}_i}} - \displaystyle\frac{e^{\mathfrak{c_i}}}{e^{\mathfrak{c}_i}-1} - \displaystyle\frac{e^{\mathfrak{r} + \mathfrak{c}_{[1,n-1]}}}{e^{\mathfrak{r} + \mathfrak{c}_{[1,n-1]}} - 1} < -2.
		\end{flalign*}
		
		Due to this negativity of $G_i''$, we deduce that $G_i$ is maximized at the solution $\mathfrak{c}_i$ to the equation $G_i' (\mathfrak{c}_i) = 0$, or equivalently the solution of 
		\begin{flalign*}
			\displaystyle\frac{\mathfrak{e}^{\mathfrak{p}_i}- e^{\mathfrak{c_i}}}{e^{\mathfrak{c}_i} - 1} = e^{\mathfrak{p}_{[1,i]}} \cdot (e^{\mathfrak{r} + \mathfrak{c}_{[1,n-1]}} - 1).
		\end{flalign*} 
		
		\noindent Setting
		\begin{flalign}
			\label{etaeta} 
			\eta = \displaystyle\frac{(e^{\mathfrak{r} + \mathfrak{c}_{[1,n-1]}} - 1) e^{\mathfrak{p}_{[1,n-1]}}}{(e^{\mathfrak{r} + \mathfrak{c}_{[1,n-1]}} - 1) e^{\mathfrak{p}_{[1,n-1]}} + 1}, \qquad \text{so that} \qquad \displaystyle\frac{1-\eta}{\eta} = e^{-\mathfrak{p}_{[1,n-1]}} (e^{\mathfrak{r} + \mathfrak{c}_{[1,n-1]}} - 1)^{-1},
		\end{flalign}
		
		\noindent it follows that $\mathfrak{c}_i$ solves 
		\begin{flalign}
			\label{psiciequation}
			\displaystyle\frac{e^{\mathfrak{c}_i} - 1}{e^{\mathfrak{p}_i} - e^{\mathfrak{c}_i}} = e^{\mathfrak{p}_{[i+1, n-1]}} \cdot \displaystyle\frac{1-\eta}{\eta}, \qquad \text{or equivalently} \qquad e^{\mathfrak{c}_i} = \displaystyle\frac{\eta + e^{\mathfrak{p}_{[i,n-1]}} (1 - \eta)}{\eta + e^{\mathfrak{p}_{[i+1,n-1]}} (1-\eta)}.
		\end{flalign}
		
		\noindent Hence, as $\varepsilon$ tends to $0$, the function $F_{\varepsilon; i} (k; \overline{\bm{C}})$, and thus $\Psi_{\varepsilon} (k; \overline{\bm{C}})$, is maximized when $\varepsilon C_i = \mathfrak{c}_i$ solves the equation \eqref{psiciequation}. By the strict concavity of $G_i$, these functions decay exponentially in $\varepsilon^{-1/2} \cdot |\varepsilon C_i - \mathfrak{c}_i|$. Hence, sampling $(k; \overline{\bm{C}})$ under $\Psi$, we find that \eqref{psiciequation} must hold for all $i \in \llbracket 1, n - 1 \rrbracket$, as $\varepsilon$ tends to $0$; recall that \eqref{k0abc} also must hold. 
		
		Let us now solve these equations. Fixing some $j \in \llbracket 1, n - 1 \rrbracket$, and taking the product in \eqref{psiciequation} over all $i \in \llbracket j,n-1 \rrbracket$, we find that 
		\begin{flalign}
			\label{j1nc} 
			e^{\mathfrak{c}_{[j,n-1]}} =  \eta + e^{\mathfrak{p}_{[j,n-1]}} (1 - \eta), \qquad \text{for all $j \in \llbracket 1, n-1 \rrbracket$}.
		\end{flalign}
		
		\noindent At $j=1$, \eqref{j1nc} gives (using \eqref{etaeta} and \eqref{rk})
		\begin{flalign}
			\label{ke} 
			e^{\mathfrak{c}_{[1,n-1]}} = \displaystyle\frac{e^{\mathfrak{r} + \mathfrak{c}_{[1,n-1]} + \mathfrak{p}_{[1,n-1]}}}{(e^{\mathfrak{r} + \mathfrak{c}_{[1,n-1]}} - 1) e^{\mathfrak{p}_{[1,n-1]}} + 1}, \qquad \text{so} \qquad e^{\mathfrak{b}_n - \mathfrak{a}_n + \mathfrak{k}} = e^{\mathfrak{r} + \mathfrak{p}_{[1,n-1]}} = \displaystyle\frac{e^{\mathfrak{p}_{[1,n-1]}} - 1}{e^{\mathfrak{c}_{[1,n-1]}} - 1}.
		\end{flalign}
		
		\noindent Inserting this into \eqref{k0abc} yields
		\begin{flalign*}
			\bigg( \displaystyle\frac{e^{\mathfrak{a}_n - \mathfrak{b}_n} (e^{\mathfrak{p}_{[1,n-1]}} - 1)}{e^{\mathfrak{c}_{[1,n-1]}} - 1} - 1 \bigg)\cdot e^{\mathfrak{a}_n - \mathfrak{b}_n} \bigg( \displaystyle\frac{e^{\mathfrak{p}_{[1,n-1]}} - 1}{e^{\mathfrak{c}_{[1,n-1]}} - 1} - e^{\mathfrak{p}_{[1,n-1]} - \mathfrak{c}_{[1,n-1]}} \bigg) = e^{\mathfrak{a}_n - \mathfrak{b}_n - \mathfrak{c}_{[1,n-1]}},
		\end{flalign*}
		
		\noindent and thus 
		\begin{flalign*}
			e^{\mathfrak{a}_n - \mathfrak{b}_n} (e^{\mathfrak{p}_{[1,n-1]}} - 1) (e^{\mathfrak{p}_{[1,n-1]} - \mathfrak{c}_{[1,n-1]}} - 1) = (1 - e^{-\mathfrak{c}_{[1,n-1]}}) (e^{\mathfrak{p}_{[1,n-1]}} - 1).
		\end{flalign*}
		
		\noindent Hence, 
		\begin{flalign}
			\label{cn1k}
			e^{\mathfrak{c}_{[1,n-1]}} = \displaystyle\frac{e^{\mathfrak{a}_n - \mathfrak{b}_n + \mathfrak{p}_{[1,n-1]}} + 1}{e^{\mathfrak{a}_n-\mathfrak{b}_n} + 1}, \qquad \text{so} \qquad q^{\mathfrak{k}} = e^{\mathfrak{a}_n - \mathfrak{b}_n} + 1.
		\end{flalign}
		
		\noindent where in the last equality we applied \eqref{ke}. Together with \eqref{rk}, it follows that 
		\begin{flalign*} 
			e^{\mathfrak{r} + \mathfrak{c}_{[1,n-1]}} = e^{\mathfrak{b}_n - \mathfrak{a}_n - \mathfrak{p}_{[1,n-1]} + \mathfrak{k} + \mathfrak{c}_{[1,n-1]}} = e^{\mathfrak{b}_n - \mathfrak{a}_n - \mathfrak{p}_{[1,n-1]}} + 1,
		\end{flalign*} 
		
		\noindent and hence \eqref{etaeta} implies  
		\begin{flalign*}
			\eta = \displaystyle\frac{e^{\mathfrak{b}_n - \mathfrak{a}_n}}{e^{\mathfrak{b}_n - \mathfrak{a}_n} + 1}, \qquad \text{so} \qquad e^{\mathfrak{c}_{[j,n-1]}} = \displaystyle\frac{e^{\mathfrak{a}_n - \mathfrak{b}_n + \mathfrak{p}_{[j,n-1]}} + 1}{e^{\mathfrak{a}_n - \mathfrak{b}_n} + 1}, \quad \text{for any $j \in \llbracket 1, n-1 \rrbracket$},
		\end{flalign*}
		
		\noindent where in the last equality we used \eqref{j1nc}. This, together with \eqref{cn1k} yields \eqref{kcj0}.
	\end{proof}

	\subsection{The $q \rightarrow 1$ Limiting Weights in Transitionary Rows} 
	
	\label{WeightAq1}
	
	Recall that in Section \ref{Weightsq1} we analyzed the limits of the $\mathcal{U}_{\gamma}^{\qdp}$ weight (from \Cref{ugamma}), under the conditions \eqref{anaj} of its arguments. While \eqref{anaj} may hold in several rows of our eventual fused vertex model converging to the log-gamma polymer, it will not hold in all of them. In particular, there will be some rows in which $A_n = 0$, such as the lowest one in which an arrow of color $n$ enters the system. Such rows may be viewed as ``transitionary'' regions where the ``dominant color'' changes from $n-1$ to $n$, in the sense that vertices output about $\varepsilon^{-1} \log \varepsilon^{-1}$ vertical arrows of color $n-1$ in the previous row, but they output about $\varepsilon^{-1} \log \varepsilon^{-1}$ vertical arrows of color $n$ in this one. 
	
	In this section we analyze the limit as $q$ tends to $1$ of the $\mathcal{U}_{\gamma}^{\qdp}$ weight, under this regime $A_n = 0$. Throughout this section, we fix a real number $\theta > 0$ and nonnegative real numbers $(\mathfrak{a}_1, \mathfrak{a}_2, \ldots , \mathfrak{a}_{n-1})$ and $(\mathfrak{b}_1, \mathfrak{b}_2, \ldots , \mathfrak{b}_n)$. Moreover, let $\varepsilon \in (0, 1)$ be a real number; set $q = q_{\varepsilon}$ and $\gamma = \gamma_{\varepsilon, \theta}$ as in \eqref{qgammaq}; let $\mathsf{B}_n = \mathsf{B}_n^{\varepsilon} \in \mathbb{Z}_{\ge 0}$ be an integer; and let $\overline{\bm{A}} = \overline{\bm{A}}^{\varepsilon}$ and $\overline{\bm{B}} = \overline{\bm{B}}^{\varepsilon} \in \mathbb{Z}_{\ge 0}^{n-1}$ be $(n-1)$-tuples of integers such that for some (uniformly bounded) integer $\beta \ge 1$ we have 
	\begin{flalign}
		\label{anj}
		\begin{aligned} 
			& \qquad \qquad \qquad \qquad \quad \varepsilon A_j \sim \mathfrak{a}_j, \quad \text{and} \quad \varepsilon B_j \sim \mathfrak{b}_j, \quad \text{for each $j \in \llbracket 1, n - 2 \rrbracket$}; \\		
			& \varepsilon A_{n-1} - \log \varepsilon^{-1} \sim \mathfrak{a}_{n-1}; \quad A_n = 0; \quad  \varepsilon B_{n-1} - (\beta-1) \log \varepsilon^{-1} \sim \mathfrak{b}_{n-1}; \quad \varepsilon \mathsf{B}_n - \beta \log \varepsilon^{-1} \sim \mathfrak{b}_n.
		\end{aligned} 
	\end{flalign}
	
	The following lemma describes the limiting law of the $n$-tuple $\bm{C}$ sampled according to the stochastic weight $\mathcal{U}_{\gamma}^{\qdp}$ (from \Cref{ugamma}), as $\varepsilon$ tends to $0$. 
	
	\begin{lem} 
		
		\label{ulimitgamma2} 
		
		Let $q = q_{\varepsilon} \in (0, 1)$ and $\gamma = \gamma_{\varepsilon, \theta} \in (0, 1)$ be as in \eqref{qgammaq}; also let $\mathsf{B}_n \in \mathbb{Z}_{\ge 0}$, $\overline{\bm{B}} \in \mathbb{Z}_{\ge 0}^{n-1}$, and $\bm{A} \in \mathbb{Z}_{\ge 0}^n$ be as in \eqref{anj}. Sample $\big( \bm{C}, (\overline{\bm{D}}, \mathsf{D}_n) \big) \in \mathbb{Z}_{\ge 0}^n \times \mathbb{Z}_{\ge 0}^{n-1} \times \mathbb{Z}_{\ge 0}$ with probability $\mathcal{U}_{\gamma}^{\qdp} \big( \bm{A}, (\overline{\bm{B}}, \mathsf{B}_n); \bm{C}, (\overline{\bm{D}}, \mathsf{D}_n) \big)$, and denote 
		\begin{flalign*}
			\mathfrak{c}_n = \varepsilon C_n - \log \varepsilon^{-1}, \quad \text{and} \quad \mathfrak{c}_j = \varepsilon C_j, \quad \text{for each $j \in \llbracket 1, n-1 \rrbracket$}.
		\end{flalign*}
		
		\noindent The the joint law of $(\mathfrak{c}_{[j,n]})_{j \in \llbracket 1, n \rrbracket}$ converges to that of $\big( \log (e^{\mathfrak{a}_{[j,n-1]} - \mathfrak{b}_n + \mathfrak{b}_{[j,n-1]}} + 1) - \log \mathfrak{Y} \big)_{j \in \llbracket 1, n-1 \rrbracket}$ as $\varepsilon$ tends to $0$, where $\mathfrak{Y} \in \mathbb{R}_{> 0}$ is a $\mathrm{Gamma} (\theta)$ random variable.
	\end{lem} 

	\begin{rem} 
		
		\label{dlimitugamma2} 
		
		Set $(\mathfrak{d}_1, \mathfrak{d}_2, \ldots , \mathfrak{d}_n) = \big(\varepsilon D_1, \varepsilon D_2, \ldots , \varepsilon D_{n-2}, \varepsilon D_{n-1} - \beta \log \varepsilon^{-1}, \varepsilon \mathsf{D}_n - (\beta + 1) \log \varepsilon^{-1} \big)$ in \Cref{ulimitgamma2}. Then, the joint limiting law of $(\mathfrak{d}_{[j,n]})_{j \in \llbracket 1, n \rrbracket}$ is determined by that of $(\mathfrak{c}_{[j,n]})_{j \in \llbracket 1, n \rrbracket}$ by arrow conservation.
		
	\end{rem} 

	\begin{rem} 
		
		\label{0beta} 
		
		Observe in \Cref{ulimitgamma2} that the limiting law of $(\mathfrak{c}_{[j,n]})_{j \in \llbracket 1, n \rrbracket}$ is independent of $\beta$, which will be useful in the proof of \Cref{convergegamma} below.
		 
	\end{rem} 
	
	\begin{proof}[Proof of \Cref{ulimitgamma2} (Outline)]
		
		The proof of this lemma is similar to that of \Cref{abcgamma}, and so we only briefly outline it. Recalling the notation from \Cref{psi}, \Cref{usumpsi} indicates that, to sample $\bm{C}$, we may first sample $\ell$ and $(k; \overline{\bm{C}})$ under the density functions $\Phi_{\varepsilon}$ and $\Psi_{\varepsilon}$, respectively, and then set $\bm{C} = (\overline{\bm{C}}; k + \ell)$. By \Cref{gamma0}, the law of $\varepsilon^{-1} \cdot e^{-\varepsilon \ell}$ converges to a $\mathrm{Gamma}(\theta)$ random variable $\mathfrak{Y}$, as $\varepsilon$ tends to $0$. Hence, it suffices to show that $\varepsilon C_{[j,n-1]}$ converges to $\log (e^{\mathfrak{a}_{[j,n-1]} - \mathfrak{b}_n + \mathfrak{b}_{[1,n-1]}} + 1)$ as $\varepsilon$ tends to $0$, if $\overline{\bm{C}} = (C_1, C_2, \ldots , C_{n-1}) \in \mathbb{Z}_{\ge 0}^{n-1}$ is sampled according to $\Psi_{\varepsilon}$. 
		
		By the factor of $(q^{A_n - k + 1}; q)_k$ on the right side of \eqref{psik}, and the fact that $A_n = 0$, we must have $k = 0$ almost surely under $\Psi_{\varepsilon}$. Then, recalling \eqref{bbdd} and the fact that $\overline{\bm{D}} = \overline{\bm{A}} + \overline{\bm{B}} - \overline{\bm{C}}$ in \eqref{psik}, it follows that 
		\begin{flalign}
			\label{psisump0}
			\begin{aligned} 
				\Psi_{\varepsilon} (0; \overline{\bm{C}}) & = q^{\varphi (\overline{\bm{A}} + \overline{\bm{B}} - \overline{\bm{C}}, \overline{\bm{C}}) + \overline{c} (\mathsf{B}_n - \overline{a} - \overline{b} + \overline{c})} \displaystyle\frac{(q; q)_{\mathsf{B}_n - \overline{b}}}{(q; q)_{\mathsf{B}_n - \overline{a} - \overline{b} + \overline{c}}} \cdot \displaystyle\prod_{j=1}^{n-1} \displaystyle\frac{(q; q)_{A_j + B_j}}{(q; q)_{C_j} (q; q)_{A_j + B_j - C_j}} \\
				& \qquad \times \displaystyle\sum_{\overline{\bm{P}} \le \overline{\bm{B}}, \overline{\bm{C}}} q^{\overline{p} + \varphi (\overline{\bm{P}}, \overline{\bm{D}} - \overline{\bm{B}})} \displaystyle\prod_{j=1}^{n-1} \displaystyle\frac{(q^{-B_j}; q)_{P_j} (q^{-C_j}; q)_{P_j}}{(q; q)_{P_j} (q^{-A_j - B_j}; q)_{P_j}}.
			\end{aligned} 
		\end{flalign}
		
		We first restrict to the case when $\overline{c} = \mathcal{O} (\varepsilon^{-1})$, by showing that otherwise $\big| \Psi_{\varepsilon} (0; \overline{\bm{C}}) \big|$ would decay as faster than $e^{-\varepsilon \overline{c}^2 / 3}$ and therefore become negligible. Indeed, since $C_j \le A_j + B_j = \mathcal{O} (\varepsilon^{-1})$ for $j \in \llbracket 1, n-2 \rrbracket$ by \eqref{anj}, it suffices to verify that $\big| \Psi_{\varepsilon} (0; \overline{\bm{C}}) \big| \le e^{\mathcal{O} (C_{n-1}) -\varepsilon C_{n-1}^2 / 2}$ if $\varepsilon C_{n-1}$ is sufficiently large. To that end, first observe since $\varepsilon (\mathsf{B}_n - \overline{a} - \overline{b}) = \mathcal{O} (1)$ by \eqref{anj}, we have $q^{\overline{c} (\mathsf{B}_n - \overline{a} - \overline{b} + \overline{c})} \le q^{\overline{c}^2 - \mathcal{O} (C_{n-1} / \varepsilon)} \le e^{\mathcal{O}(C_{n-1}) - \varepsilon C_{n-1}^2}$. Moreover (since $\big| (q^{-B_j}; q)_{P_j} \big| < \big| (q^{-A_j - B_j}; q)_{P_j} \big|$), the remaining terms on the right side of \eqref{psisump0} decay at most exponentially in $\varepsilon^{-1}$, except for the term $(q^{-C_{n-1}}; q)_{P_{n-1}} (q; q)_{P_{n-1}}^{-1}$, which is at most $e^{\mathcal{O} (C_{n-1}) -\varepsilon C_{n-1}^2/2}$. Therefore, $\big| \Psi_{\varepsilon} (0; \overline{\bm{C}}) \big| \le e^{\mathcal{O} (C_{n-1}) -\varepsilon C_{n-1}^2 / 2}$, which verifies the above-mentioned decay.
		
		Thus we may assume $\overline{c} = \mathcal{O} (\varepsilon^{-1})$, so $\varepsilon (D_{n-1} - B_{n-1}) = \varepsilon (A_{n-1} - C_{n-1}) = \log \varepsilon^{-1} - \mathcal{O} (1)$. Due to factor of $q^{\varphi (\overline{\bm{P}}, \overline{\bm{D}} - \overline{\bm{B}})}$ in the sum on the right side of \eqref{psisump0} (and the fact that the remaining terms in the sum grow at most exponentially in $|\overline{\bm{P}}|$), one can verify that this sum is asymptotically supported on the $\bm{P} \in \mathbb{Z}_{\ge 0}^{n-1}$ satisfying $P_j = 0$ for $j \in \llbracket 1, n-2 \rrbracket$. Moreover, due to the fact that $\big| (q^{-B_{n-1}}; q)_{P_{n-1}} (q^{-C_{n-1}}; q)_{P_{n-1}} (q; q)_{P_{n-1}}^{-1} (q^{-A_{n-1} - B_{n-1}}; q)_{P_{n-1}}^{-1} \big|$ decays as $q^{P_{n-1} (A_{n-1} - C_{n-1})}$, the asymptotic support of this sum also requires $P_{n-1} = 0$, and hence $\bm{P} = \bm{e}_0$. Thus, to understand the behavior of $(k; \overline{\bm{C}}) = (0; \overline{\bm{C}})$ sampled from the density function $\Psi_{\varepsilon}$, it suffices to understand the limiting behavior as $\varepsilon$ tends to $0$ of 
		\begin{flalign*}
			\widetilde{\Psi}_{\varepsilon} (0; \overline{\bm{C}}) & = q^{\varphi (\overline{\bm{A}} + \overline{\bm{B}} - \overline{\bm{C}}, \overline{\bm{C}}) + \overline{c} (\mathsf{B}_n - \overline{a} - \overline{b} + \overline{c})} \displaystyle\frac{(q; q)_{\mathsf{B}_n - \overline{b}}}{(q; q)_{\mathsf{B}_n - \overline{a} - \overline{b} + \overline{c}}} \cdot \displaystyle\prod_{j=1}^{n-1} \displaystyle\frac{(q; q)_{A_j + B_j}}{(q; q)_{C_j} (q; q)_{A_j + B_j - C_j}}.
		\end{flalign*}
		
		As in the proof of \Cref{sumpdb}, the $\overline{\bm{C}}$ that contribute to $\widetilde{\Psi}_{\varepsilon}$ will asymptotically be supported on a single value of $\overline{\bm{\mathfrak{c}}} = (\mathfrak{c}_1, \mathfrak{c}_2, \ldots , \mathfrak{c}_{n-1}) = \varepsilon \cdot \overline{\bm{C}}$. To understand which value, we take the logarithm of the part of $\widetilde{\Psi}_{\varepsilon}$ that depends on $\overline{\bm{C}}$, and multiply by $\varepsilon$ to obtain (using \eqref{qqa} and the facts that $\varepsilon (A_{n-1} + B_{n-1} - C_{n-1}) \ge \log \varepsilon^{-1} - \mathcal{O}(1)$) the function
		\begin{flalign*}
			G_{\varepsilon} (\overline{\bm{\mathfrak{c}}}) &  = \mathfrak{c}_{[1,n-1]} \cdot (\mathfrak{p}_{[1,n-1]} - \mathfrak{b}_n - \mathfrak{c}_{[1,n-1]}) - \displaystyle\sum_{j=1}^{n-1} \mathfrak{p}_{[1,j-1]} \mathfrak{c}_j - \displaystyle\sum_{j=1}^{n-1} \displaystyle\int_0^{\mathfrak{c}_j} \log (1 - e^{-x}) dx\\
			& \qquad + \displaystyle\sum_{1 \le j < h \le n-1} \mathfrak{c}_j \mathfrak{c}_h - \displaystyle\int_0^{\mathfrak{b}_n - \mathfrak{p}_{[1,n-1]} + \mathfrak{c}_{[1,n-1]}} \log (1 - e^{-x}) dx - \displaystyle\sum_{j=1}^{n-2} \displaystyle\int_0^{\mathfrak{p}_j - \mathfrak{c}_j} \log (1 - e^{-x}) dx,
		\end{flalign*}
		
		\noindent where we have denoted $\mathfrak{p}_j = \mathfrak{a}_j + \mathfrak{b}_j$ for each $j \in \llbracket 1, n-1 \rrbracket$. The $\overline{\bm{C}}$ that contributes to $\widetilde{\Psi}_{\varepsilon}$ is then obtained at the maximum of $G_{\varepsilon}$, which is the solution of $\partial_{\mathfrak{c}_i} G_{\varepsilon} (\overline{\bm{\mathfrak{c}}}) = 0$ for each $i \in \llbracket 1, n - 1 \rrbracket$. These equations are given by 
		\begin{flalign*}
			& \log (e^{\mathfrak{p}_i} - e^{\mathfrak{c}_i}) - \log (e^{\mathfrak{c}_i} - 1) - \mathfrak{p}_{[1,i]} - \log (e^{\mathfrak{b}_n + \mathfrak{c}_{[1,n-1]} - \mathfrak{p}_{[1,n-1]}} - 1) = 0, \qquad \text{for $i \in \llbracket 1, n - 2 \rrbracket$}; \\
			& - \log (e^{\mathfrak{c}_{n-1}} - 1) - \log (e^{\mathfrak{c}_{[1,n-1]} + \mathfrak{b}_n - \mathfrak{p}_{[1.n-1]}} - 1) - \mathfrak{p}_{[1,n-2]} = 0.
		\end{flalign*}
		
		\noindent Denoting 
		\begin{flalign}
			\label{eta2}
			\eta = \displaystyle\frac{e^{\mathfrak{b}_n + \mathfrak{c}_{[1,n-1]} - \mathfrak{p}_{[1,n-1]}}}{e^{\mathfrak{b}_n + \mathfrak{c}_{[1,n-1]} - \mathfrak{p}_{[1,n-1]}} - 1}, \qquad \text{so that} \qquad \displaystyle\frac{1}{\eta-1} = e^{\mathfrak{b}_n + \mathfrak{c}_{[1,n-1]} - \mathfrak{p}_{[1,n-1]}} - 1,
		\end{flalign} 
		
		\noindent it follows that 
		\begin{flalign*}
			e^{\mathfrak{c}_j} = e^{\mathfrak{p}_j} \cdot \displaystyle\frac{e^{\mathfrak{p}_{[1,j-1]}} + \eta - 1}{e^{\mathfrak{p}_{[1,j]}} + \eta - 1}, \quad \text{for each $j \in \llbracket 1, n - 2 \rrbracket$}; \qquad e^{\mathfrak{c}_{n-1}} = e^{-\mathfrak{p}_{[1,n-2]}} (\eta - 1 + e^{\mathfrak{p}_{[1,n-2]}}).
		\end{flalign*} 
		
		\noindent Thus, for each $j \in \llbracket 1, n-1 \rrbracket$, we have 
		\begin{flalign*}
			e^{\mathfrak{c}_{[j, n-1]}} = e^{-\mathfrak{p}_{[1,j-1]}} \cdot (e^{\mathfrak{p}_{[1,j-1]}} + \eta - 1),
		\end{flalign*}
		
		\noindent which taken at $j=1$ implies that $\eta = e^{\mathfrak{c}_{[1,n-1]}}$. Together with \eqref{eta2}, this yields $\eta = e^{\mathfrak{c}_{[1,n-1]}} = e^{\mathfrak{p}_{[1,n-1]} - \mathfrak{b}_n} + 1$, and therefore   
		\begin{flalign}
			\label{2cjn1}
			e^{\mathfrak{c}_{[j,n-1]}} = e^{\mathfrak{p}_{[j,n-1]} - \mathfrak{b}_n} + 1.
		\end{flalign}
		
		\noindent Hence, if we sample $\overline{\bm{C}}$ according to $\Psi_{\varepsilon}$ and set $\overline{\mathfrak{c}} = (\mathfrak{c}_1, \mathfrak{c}_2, \ldots , \mathfrak{c}_{n-1}) = \varepsilon \cdot \overline{\bm{C}}$, then \eqref{2cjn1} holds. As mentioned previously, this (with \Cref{usumpsi} and \Cref{gamma0}) implies the lemma. 
	\end{proof}

	\subsection{Convergence to the Log-Gamma Polymer} 
	
	\label{Polymerq1} 
	
	In this section we show convergence of a stochastic vertex model with weights $\mathcal{U}_{\gamma}^{\qdp}$ and $\mathcal{U}_{1/\gamma}^{\bq}$ (from \Cref{ugamma} and \Cref{u3}) to the log-gamma polymer. We begin by defining the latter; throughout this section, we fix positive real parameters $\bm{\theta} = (\theta_1, \theta_2, \ldots )$ and $\bm{\theta}' = (\theta_1', \theta_2', \ldots)$.  
	
	For each vertex $(i, j) \in \mathbb{Z}_{> 0}^2$, let $\mathfrak{Y}_{ij}$ denote a $\mathrm{Gamma}(\theta_i + \theta_j')$ random variable, with all $(\mathfrak{Y}_{ij})$ mutually independent over $(i, j) \in \mathbb{Z}_{> 0}^2$. For any vertices $u, v \in \mathbb{Z}_{> 0}^2$ such that $v - u \in \mathbb{Z}_{\ge 0}^2$, a directed path $\mathcal{P} = (w_0, w_1, \ldots , w_k) \in \mathbb{Z}_{> 0}$ from $u$ to $v$ is a sequence of vertices starting at $w_0 = u$, ending at $w_k = v$, and satisfying $w_i - w_{i-1} \in \big\{ (1, 0), (0, 1) \big\}$ for each $i \in \llbracket 1, k \rrbracket$. The \emph{polymer weight} of such a directed path is defined to be 
	\begin{flalign}
		\label{pz1} 
		\mathfrak{Z} (\mathcal{P}) = \displaystyle\prod_{(i, j) \in \mathcal{P}} \mathfrak{Y}_{ij}^{-1},
	\end{flalign}
	
	\noindent and the \emph{point-to-point partition function} from $u$ to $v$ is defined to be
	\begin{flalign}
		\label{zuv1} 
		\mathfrak{Z} (u \rightarrow v) = \displaystyle\sum_{\mathcal{P}} \mathfrak{Z} (\mathcal{P}),
	\end{flalign}
	
	\noindent where the sum is over all directed paths $\mathcal{P} \subset \mathbb{Z}_{> 0}^2$ from $u$ to $v$. 
	
	Now let us describe the stochastic vertex model that converges to the log-gamma polymer. To that end, let $\varepsilon \in (0, 1)$ denote a real number, and set
	\begin{flalign*}
		q = q_{\varepsilon} = e^{-\varepsilon}, \quad \text{and} \quad \gamma_{ij} = q^{\theta_i + \theta_j'}, \quad \text{for each $(i, j) \in \mathbb{Z}_{> 0}^2$}. 
	\end{flalign*}
	
	\noindent Fix (possibly infinite) integers $\mathrm{K}_1, \mathrm{K}_2, \ldots \ge 1$. At any vertex $(i, j) \in \mathbb{Z}_{> 0}^2$, we will sample a random colored (complemented) arrow configuration $\big( \bm{A} (i, j), (\overline{\bm{B}} (i,j), \mathsf{B}_m (i,j)); \bm{C} (i,j), (\overline{\bm{D}} (i,j), \mathsf{D}_m (i, j)) \big)$, where $\mathsf{B}_m (i, j), \mathsf{D}_m (i,j) \in \mathbb{Z}_{\ge 0}$; $\overline{\bm{B}} (i,j), \overline{\bm{D}} (i,j) \in \mathbb{Z}_{\ge 0}^{m-1}$; and $\bm{A} (i,j), \bm{C} (i,j) \in \mathbb{Z}_{\ge 0}^m$. Here, $m = m(j) \ge 1$ is the (unique) positive integer such that $\mathrm{K}_{[1,m-1]} + 1 \le j \le \mathrm{K}_{[1,m]}$. These arrow configurations will be consistent, in the sense that $\bm{A} (i, j + 1) = \bm{C} (i, j)$ and $\big( \overline{\bm{B}} (i+1, j), \mathsf{B}_m (i + 1, j) \big) = \big( \overline{\bm{D}} (i, j), \mathsf{D}_m (i, j) \big)$ for each $(i, j) \in \mathbb{Z}_{> 0}^2$ (where if $j = \mathrm{K}_{[1,m]}$, the $m$-tuple $\bm{C} (i, j)$ is interpreted as an $(m+1)$-tuple by setting its $(m+1)$-th entry to $0$). 
	
	We will sample these arrow configurations recursively on triangles of the form $\mathbb{T}_N = \big\{ (x, y) \in \mathbb{Z}_{> 0}^2 : x + y \le N \big\}$. Given some integer $N \ge 1$, suppose that arrow configurations have been assigned to all vertices in $\mathbb{T}_{N-1}$; we will explain how to sample them on $\mathbb{T}_N$. Fix a vertex on the diagonal $(i, j) \in \mathbb{D}_N = \mathbb{T}_N \setminus \mathbb{T}_{N-1}$. First, set $\bm{A} (N-1, 1) = \bm{e}_0$, $\overline{\bm{B}} (1, N-1) = \bm{e}_0$, and $\mathsf{B}_{m(N-1)} (1,N-1) = 0$. Then, if $i = 1$ set $\overline{\bm{D}} (1, j) = \bm{e}_0$, and inductively define $\bm{C} (1, j) = \bm{A} (1, j) + \mathsf{D}_{m(j)} (1, j) \cdot \bm{e}_{m(j)}$, where $\mathsf{D}_m (1, j)$ is sampled according to the probability (recalling \Cref{u3})
	\begin{flalign}
		\label{ijmd} 
		\mathbb{P} \big[\mathsf{D}_m (1,j) = k \big] = \mathcal{U}_{1/\gamma_{1,j}}^{\bq; m} (k)
	\end{flalign}
	
	If $i > 1$, then the inputs $\bm{A} (i, j) = \bm{C} (i, j-1)$ and $\big( \overline{\bm{B}} (i,j), \mathsf{B}_m (i, j) \big) = \big( \overline{\bm{D}} (i-1, j), \mathsf{D}_m (i-1, j) \big)$ at $(i, j)$ have already been assigned (as they arise from arrow configurations at vertices in $\mathbb{T}_{N-1}$). Sample its outputs $\big( \bm{C} (i, j), (\overline{\bm{D}} (i,j), \mathsf{D}_m (i, j)) \big)$ according to the probability (recalling \Cref{ugamma}) 
	\begin{flalign*} 
		\mathbb{P} \bigg[ \Big(  \bm{C} (i, j), \big(\overline{\bm{D}} (& i,j), \mathsf{D}_m (i, j) \big) \Big) \bigg| \Big( \bm{A} (i, j), \big( \overline{\bm{B}} (i, j), \mathsf{B}_m (i, j) \big) \Big) \bigg] \\
		& = \mathcal{U}_{\gamma_{ij}}^{\bq; m} \Big( \bm{A} (i, j), \big(\overline{\bm{B}} (i,j), \mathsf{B}_m (i, j) \big); \bm{C} (i, j), (\overline{\bm{D}} (i,j), \mathsf{D}_m (i, j)) \Big).
	\end{flalign*} 
	
	\noindent This assigns a random arrow configuration to each vertex in $\mathbb{T}_N$; letting $N$ tend to $\infty$ yields a random ensemble of arrow configurations on the quadrant $\mathbb{Z}_{> 0}^2$.
	
	Observe in this way that, by \eqref{ijmd}, arrows of color $n$ begin to enter the system through its $(\mathrm{K}_{[1,n-1]} + 1)$-st row. As in Section \ref{WeightAq1}, we will sometimes refer to such rows as \emph{transitionary} (as they will be where the most frequent color transitions from $n-1$ to $n$). Moreover, due to the forms (from \Cref{ugamma}) of the weights $\mathcal{U}^{\qdp}$, arrow conservation in the ensemble sampled above appears slightly different from that in the colored fused path ensembles introduced in \Cref{FusedPath}. Here, it depends on the vertex $(i, j) \in \mathbb{Z}_{> 0}^2$, and in particular on the index\footnote{This is related to the fact that we are implicitly ``complementing'' the arrows of color $m(j)$ on the $j$-th row.} $m(j)$ associated with its row; it is given by 
	\begin{flalign}
		\label{abcdij} 
		\overline{\bm{A}} (i, j) + \overline{\bm{B}} (i, j) = \overline{\bm{C}} (i, j) + \overline{\bm{D}} (i, j); \qquad \mathsf{B}_{m(j)} (i, j) - A_{m(j)} (i, j) = \mathsf{D}_{m(j)} (i, j) - C_{m(i,j)} (i, j).
	\end{flalign}
	
	\noindent Still, we will explain how the above ensemble arises from the colored stochastic fused vertex model (from \Cref{FusedPath}) in \Cref{abcdpath} below.
	
	To state convergence of this model to the log-gamma polymer, we need to prescribe an associated height function. Given an integer $c \ge 1$, define the color (at least) $c$ height functions $\mathfrak{h}_c^{\rightarrow}, \mathfrak{h}_{\ge c}^{\rightarrow}: \mathbb{Z}_{\ge 0}^2 \rightarrow \mathbb{Z}$ by for any $(i, j) \in \mathbb{Z}_{\ge 0}^2$ setting (where we let $X_c = 0$ for any $\bm{X} \in \mathbb{Z}^n$ and $c > n$)
	\begin{flalign}
		\label{hc2ij}
		\mathfrak{h}_c^{\rightarrow} (i, j) = -\displaystyle\sum_{k = 1}^j \mathbbm{1}_{c = m(k)} \cdot \mathsf{D}_{m(k)} (1, k) - \displaystyle\sum_{k=2}^i C_c (k, j); \qquad \mathfrak{h}_{\ge c}^{\rightarrow} (i, j) = \displaystyle\sum_{c' = c}^{\infty} \mathfrak{h}_{c'} (i, j). 
	\end{flalign} 
	
	\begin{rem} 
		
		\label{hpath2} 
		
		The height functions are defined above by summing (negative) entries of arrow configurations along the specific up-right path from $(1, 1)$ to $(i, j)$ that first proceeds north to $(1, j)$ and then east to $(i, j)$. Using arrow conservation \eqref{abcdij}, one obtains the same result by replacing this by any up-right path. For instance, by instead considering the path that first proceeds east to $(i, 1)$ and then north to $(i, j)$, we also have
		\begin{flalign*}
			\mathfrak{h}_c^{\rightarrow} (i, j) = \displaystyle\sum_{k=1}^j \mathbbm{1}_{c < m(k)} \cdot D_c (i, k) - \displaystyle\sum_{k = 1}^j \mathbbm{1}_{c  = m(k)} \cdot \mathsf{D}_{m(k)} (i, k) - \displaystyle\sum_{k=2}^i \mathbbm{1}_{c \le m(i)} \cdot C_c (k, 1).
		\end{flalign*}
		
	\end{rem}  
	
	Given this notation, we have the following theorem. It indicates the convergence, as $\varepsilon$ tends to $0$, of the height functions for the above vertex model to the point-to-point polymer partition functions of the log-gamma polymer. Observe in this result that the colors are lost in the polymer degeneration, and they instead track one endpoint of the polymer.
	
	\begin{thm} 
		
		\label{convergegamma}
		
		Under the above setup, the random variables $\mathfrak{X}_c^{\varepsilon} (i, j) = \varepsilon^{i+j} \cdot e^{-\mathfrak{h}_{\ge c}^{\rightarrow} (i, j)}$ converge to the log-gamma polymer partition functions $\mathfrak{Z} \big( (1, \mathrm{K}_{[1,c-1]+1}) \rightarrow (i, j) \big)$, as $\varepsilon$ tends to $0$, jointly over all $(c, i, j)$ in compact subsets of $\mathbb{Z}_{> 0} \times \mathbb{Z}_{> 0} \times \mathbb{Z}_{> 0}$. 
		
	\end{thm} 
	
	\begin{proof}  
		
		For each $(i, j) \in \mathbb{Z}_{> 0}^2$, let $\mathfrak{Y}_{ij} = \mathfrak{Y}_{i,j}$ denote a $\mathrm{Gamma} (\theta_i + \theta_j')$ random variable, with all $(\mathfrak{Y}_{ij})$ mutually independent over $(i, j) \in \mathbb{Z}_{> 0}^2$. Further set $\mathfrak{Z}_c (i, j) = \mathfrak{Z} \big( (1, \mathrm{K}_{[1,c-1]} + 1) \rightarrow (i, j) \big)$ for each $c \in \mathbb{Z}_{> 0}$ and $(i, j) \in \mathbb{Z}_{> 0}^2$. Then, it follows from \eqref{pz1} and \eqref{zuv1} that for $c > 0$ the $\mathfrak{Z}_c (i, j)$ are determined by the recursive relations
		\begin{flalign}
			\label{zijij} 
			\begin{aligned} 
				& \mathfrak{Z}_c (i, j) = \mathfrak{Y}_{ij}^{-1} \cdot \big( \mathfrak{Z}_c (i-1, j) + \mathfrak{Z}_c (i, j-1) \big), \quad \text{for $(i, j) \in \mathbb{Z}_{> 1} \times \mathbb{Z}_{> \mathrm{K}_{[1,c-1]+1}}$}; \\
				& \mathfrak{Z}_c (1, j) = \mathfrak{Y}_{1j}^{-1} \cdot \mathfrak{Z}_c (1,j-1), \quad \text{for $j \ge \mathrm{K}_{[1,c-1]} + 1$}; \\
				&  \mathfrak{Z}_c (i, \mathrm{K}_{[1,c-1]} + 1) = \mathfrak{Y}_{i, \mathrm{K}_{[1,c-1]} + 1}^{-1} \cdot \mathfrak{Z}_c (i-1, \mathrm{K}_{[1,c-1]} + 1), \quad \text{for $i \ge 2$}; \\
				& \mathfrak{Z}_c (1, \mathrm{K}_{[1,c-1]} + 1) = \mathfrak{Y}_{1, \mathrm{K}_{[1,c-1]} + 1}^{-1}.
			\end{aligned} 
		\end{flalign}
		
		\noindent It therefore suffices to show that $\mathfrak{X}_c^{\varepsilon} (i, j)$ satisfies the same recursion, as $\varepsilon$ tends to $0$. 
		
		To that end let us analyze the behavior, as $\varepsilon$ tends to $0$, of the dynamics for the arrow configurations $\big( \bm{A} (i, j), (\overline{\bm{B}} (i, j), \mathsf{B}_{m(j)} (i, j)); \bm{C} (i, j), (\overline{\bm{D}} (i, j), \mathsf{D}_{m(j)} (i, j)) \big)$. First observe from \eqref{ijmd} and the second part of \Cref{gamma0} that for each $j \ge 1$ we have $B_c (1, j) = D_c (1, j) = 0$ for $c \ne m(j)$, and $\mathsf{B}_{m(j)} (1, j) = 0$ and $\mathsf{D}_{m(j)} (1, j) = \varepsilon^{-1} \log \varepsilon^{-1} + \mathcal{O} (\varepsilon^{-1})$; more specifically, we have  
		\begin{flalign}
			\label{dconverge}
			\varepsilon \cdot e^{\varepsilon \mathsf{D}_{m(j)} (1, j)} \qquad \text{converges in distribution to} \qquad \mathfrak{Y}_{1j}^{-1}, \qquad \text{as $\varepsilon$ tends to $0$}.
		\end{flalign}
		
		We will next understand the behavior of these arrow configurations for $(i, j) \in \mathbb{Z}_{> 1} \times \mathbb{Z}_{> 0}$. We will see that their behavior depends on whether $j$ is a transitionary row, that is, if $j = \mathrm{K}_{[1,m(j)-1]} + 1$. In particular, we will show the  following by induction on the lexicographic pair $(j, i)$; here, we abbreviate $m = m(j)$, and $\big(\mathfrak{z}_1 (i, j), \mathfrak{z}_2 (i, j), \ldots , \mathfrak{z}_m (i, j) \big)$ are some real numbers that are bounded by a random number independent of $\varepsilon$, for each index $\mathfrak{z} \in \{ \mathfrak{a}, \mathfrak{b}, \mathfrak{c}, \mathfrak{d} \}$. If $j = \mathrm{K}_{[1,m-1]} + 1$ indexes a transitionary row, then we will show
		\begin{flalign}
			\label{abcdmm1kj} 
			\begin{aligned} 
				& A_{m-1} (i, j) = \varepsilon^{-1} \log \varepsilon^{-1} + \varepsilon^{-1} \mathfrak{a}_{m-1} (i, j); \qquad B_{m-1} (i, j) = (i-2) \cdot \varepsilon^{-1} \log \varepsilon^{-1} + \varepsilon^{-1} \mathfrak{b}_{m-1} (i, j); \\
				& C_{m-1} (i, j) = \varepsilon^{-1} \mathfrak{c}_{m-1} (i, j) ; \qquad \qquad \qquad \qquad D_{m-1} (i, j) = (i-1) \cdot \varepsilon^{-1} \log \varepsilon^{-1} + \varepsilon^{-1} \mathfrak{d}_{m-1} (i, j); \\
				&  A_m (i, j) = 0; \qquad \qquad \qquad \qquad \qquad \qquad \qquad \mathsf{B}_m (i, j) = (i-1) \cdot \varepsilon^{-1} \log \varepsilon^{-1} + \varepsilon^{-1} \mathfrak{b}_m (i, j); \\
				&  C_m = \varepsilon^{-1} \log \varepsilon^{-1} + \varepsilon^{-1} \mathfrak{c}_m (i, j); \qquad \qquad \qquad \mathsf{D}_m (i, j) = i \cdot \varepsilon^{-1} \log \varepsilon^{-1} + \varepsilon^{-1} \mathfrak{d}_m (i, j).
			\end{aligned} 	
		\end{flalign}
		
		\noindent In this way, the transitionary row ``absorbs'' most color $m-1$ arrows from the previous row, and also ``emits'' arrows of color $m$; due to the arrow conservation \eqref{abcdij}, this leads to an accumulation of arrows of both colors $m-1$ and $m$ in this row. If instead $j > \mathrm{K}_{[1,m-1]} + 1$ does not index a transitionary row, then we will show
		\begin{flalign}
			\label{xm1ym1ij} 
			\begin{aligned} 
				& Z_{m-1} (i, j) = \varepsilon^{-1}  \mathfrak{z}_{m-1} (i, j), \qquad \qquad \qquad \text{if $Z \in \{ A, B, C, D \}$}; \\
				& Z_m (i, j) = \varepsilon^{-1} \log \varepsilon^{-1} + \varepsilon^{-1} \mathfrak{z}_m (i, j), \qquad \text{if $Z \in \{A, C \}$}; \\
				& \mathsf{Z}_m (i, j) = \varepsilon^{-1} \log \varepsilon^{-1} + \varepsilon^{-1} \mathfrak{z}_m (i, j), \qquad \text{if $\mathsf{Z} \in \{ \mathsf{B}, \mathsf{D} \}$}.
			\end{aligned} 
		\end{flalign}
		
		\noindent In both cases for $j$, we will also show for any index $Z \in \{ A, B, C, D  \}$ that
		\begin{flalign}
			\label{xycij} 
			Z_c (i, j) = \varepsilon^{-1} \mathfrak{z}_c (i, j), \qquad  \text{if $c \le m - 2$}.
		\end{flalign}
		
		Let us first verify that \eqref{abcdmm1kj} and \eqref{xycij} both hold if $j = \mathrm{K}_{[1,m-1]} + 1$ indexes a transitionary row. To that end, observe by the inductive hypothesis (and the previous discussion on the first column, if $i = 2$) that these statements for the entrance data $A_k (i,j)$, $B_k (i, j)$, and $\mathsf{B}_k (i, j)$ hold. The fact that it also holds for the exit data $C_k (i, j)$, $D_k (i, j)$, and $\mathsf{D}_m (i, j)$ then follows from (the $\beta = i-1$ case of) \Cref{ulimitgamma2}, together with \Cref{0beta} and arrow conservation \eqref{abcdij} (recall \Cref{dlimitugamma2}). We next verify that \eqref{xm1ym1ij} and \eqref{xycij} hold if $j > \mathrm{K}_{[1,m-1]} + 1$ does not index a transitionary row. In this case, again the inductive hypothesis (together with the previous discussion on the first column, if $i=2$, and \eqref{abcdmm1kj} if $j = \mathrm{K}_{[1,m-1]} + 2$ indexes a row directly above a transitionary one) verifies these statements for the entrance data $A_k (i, j)$ and $B_k (i, j)$. The fact that it also holds for the exit data $C_k (i, j)$, $D_k (i, j)$, and $\mathsf{D}_m (i, j)$ then follows from \Cref{abcgamma}, together with arrow conservation \eqref{abcdij}. This verifies \eqref{abcdmm1kj}, \eqref{xm1ym1ij}, and \eqref{xycij}; by \Cref{abcgamma} and \Cref{ulimitgamma2}, it also shows that 
		\begin{flalign}
			\label{convergeckm} 
			e^{\mathfrak{c}_{[k, m]} (i, j)} \qquad \text{converges in distribution to} \qquad \mathfrak{Y}_{ij}^{-1} \cdot (e^{\mathfrak{a}_{[k,m]} (i,j) - \mathfrak{b}_m (i, j) + \mathfrak{b}_{[k,m]} (i,j)} + 1),
		\end{flalign}
		
		\noindent  as $\varepsilon$ tends to $0$ (where we have set $\mathfrak{a}_n (i,j) = 0$ if $j = \mathrm{K}_{[1,m-1]} + 1$ indexes a transitionary row, due to the fact that $A_m = 0$ in \eqref{abcdmm1kj}).
		
		Now let us analyze the height function for this model. First observe for any $(i, j) \in \mathbb{Z}_{> 1} \times \mathbb{Z}_{> 0}$ and $c \in \llbracket 1, m-1 \rrbracket$, where $m = m(j)$, that the definition \eqref{hc2ij} of the height function, consistency of the ensemble, arrow conservation \eqref{abcdij}, and \Cref{hpath2} together imply that
		\begin{flalign}
			\label{hcijabc}
			\begin{aligned} 
				\mathfrak{h}_c^{\rightarrow} (i, j - 1) = \mathfrak{h}_c^{\rightarrow} (i-1 & , j-1) - A_c (i, j); \qquad \mathfrak{h}_c^{\rightarrow} (i-1, j) = \mathfrak{h}_c^{\rightarrow} (i-1, j-1) + B_c (i, j); \\
				& \mathfrak{h}_c^{\rightarrow} (i, j) = \mathfrak{h}_c^{\rightarrow} (i - 1, j) - C_c (i, j),
			\end{aligned} 
		\end{flalign}
		
		\noindent and that 
		\begin{flalign}
			\label{hm2} 
			\begin{aligned} 
				\mathfrak{h}_m^{\rightarrow} (i, j-1) = \mathfrak{h}_m^{\rightarrow} (i-1 & , j-1) - A_m (i, j); \qquad \mathfrak{h}_m^{\rightarrow} (i-1, j) = \mathfrak{h}_m (i-1, j-1) - \mathsf{B}_m (i, j); \\ 
				& \mathfrak{h}_m^{\rightarrow} (i, j) = \mathfrak{h}_m^{\rightarrow} (i-1, j) - C_m (i, j). 
			\end{aligned} 
		\end{flalign}
		
		\noindent Hence, for any integer $c \ge 1$, we have
		\begin{flalign}
			\label{0abhcij}
			\mathfrak{h}_{\ge c}^{\rightarrow} (i-1, j) = \mathfrak{h}_{\ge c}^{\rightarrow} (i, j-1) + A_{[c,m]} (i, j) + B_{[c,m-1]} (i, j) - \mathsf{B}_m (i, j).
		\end{flalign}
		
		Thus, setting $\mathfrak{X}_c^0 (i, j) = \lim_{\varepsilon \rightarrow 0} \mathfrak{X}_c^{\varepsilon} (i, j)$, we have if $i > 1$ and $j > \mathrm{K}_{[1,c-1]} + 1$ that 
		\begin{flalign}
			\label{xc0ij} 
			\begin{aligned} 
				\mathfrak{X}_c^0 (i, j) & = \displaystyle\lim_{\varepsilon \rightarrow 0} \varepsilon^{i+j} e^{-\varepsilon \mathfrak{h}_{\ge c}^{\rightarrow} (i, j)} \\
				& = \displaystyle\lim_{\varepsilon \rightarrow 0} \varepsilon^{i+j} e^{-\varepsilon \mathfrak{h}_{\ge c}^{\rightarrow} (i-1, j)} \cdot e^{ \varepsilon C_{[c,m]} (i, j)} \\
				& = \displaystyle\lim_{\varepsilon \rightarrow 0} \varepsilon^{i+j-1} e^{-\varepsilon \mathfrak{h}_{\ge c}^{\rightarrow} (i-1,j)} \cdot e^{\mathfrak{c}_{[c,m]} (i, j)} \\
				& = \displaystyle\lim_{\varepsilon \rightarrow 0} \varepsilon^{i+j-1} e^{-\varepsilon \mathfrak{h}_{[c,m]}^{\rightarrow} (i-1, j)} \cdot (e^{\mathfrak{a}_{[j,m]} (i, j) + \mathfrak{b}_{[j,m-1]} (i, j) - \mathfrak{b}_m (i, j)} + 1) \cdot \mathfrak{Y}_{ij}^{-1} \\
				& = \displaystyle\lim_{\varepsilon \rightarrow 0} \varepsilon^{i+j-1} e^{-\varepsilon \mathfrak{h}_{\ge c}^{\rightarrow} (i-1, j)} \cdot (e^{\varepsilon A_{[c,m]} (i, j) + \varepsilon B_{[c,m-1]} (i, j) - \varepsilon \mathsf{B}_m (i, j)} + 1) \cdot \mathfrak{Y}_{ij}^{-1} \\
				& = \displaystyle\lim_{\varepsilon \rightarrow 0} \varepsilon^{i + j-1}(e^{-\varepsilon \mathfrak{h}_{\ge c}^{\rightarrow} (i, j-1)} + e^{-\varepsilon \mathfrak{h}_{\ge c}^{\rightarrow} (i-1, j)}) \cdot \mathfrak{Y}_{ij}^{-1} = \mathfrak{Y}_{ij}^{-1} \cdot ( \mathfrak{X}_c^0 (i, j - 1) + \mathfrak{X}_c^0 (i - 1, j) \big),
			\end{aligned} 
		\end{flalign}
		
		\noindent where the first equality follows from the definition of $\mathfrak{X}_c^{\varepsilon} (i, j) = \varepsilon^{i+j} \cdot e^{-\mathfrak{h}_{\ge c} (i, j)}$; the second from the third statements in \eqref{hcijabc} and \eqref{hm2}; the third from the definition of $\mathfrak{c}_k$ in terms of $C_k$ from \eqref{abcdmm1kj}, \eqref{xm1ym1ij}, and \eqref{xycij}; the fourth from \eqref{convergeckm}; the fifth from the definitions of $\mathfrak{a}_k$ and $\mathfrak{b}_k$ in terms of $A_k$ and $B_k$, respectively, from \eqref{abcdmm1kj}, \eqref{xm1ym1ij}, and \eqref{xycij}; the sixth from \eqref{0abhcij}; and the seventh again from the definition of $\mathfrak{X}_c^{\varepsilon} (i, j)$. If instead $i > 1$ and $j = \mathrm{K}_{[1,c-1]} + 1$ (meaning that $c = m$ and $j$ indexes a transitionary row), then following \eqref{xc0ij} we obtain
		\begin{flalign*}
			\mathfrak{X}_c^0 (i, j) & = \displaystyle\lim_{\varepsilon \rightarrow 0} \varepsilon^{i+j-1} e^{-\varepsilon \mathfrak{h}_{[c,m]}^{\rightarrow} (i-1, j)} \cdot (e^{\mathfrak{a}_{[j,m]} (i, j) + \mathfrak{b}_{[j,m-1]} (i, j) - \mathfrak{b}_m (i, j)} + 1) \cdot \mathfrak{Y}_{ij}^{-1} \\
			& = \displaystyle\lim_{\varepsilon \rightarrow 0} \varepsilon^{i+j-1} e^{-\varepsilon \mathfrak{h}_{\ge c}^{\rightarrow} (i-1, j)} \cdot (\varepsilon e^{\varepsilon A_{[c,m]} (i, j) + \varepsilon B_{[c,m-1]} (i, j) - \varepsilon \mathsf{B}_m (i, j)} + 1) \cdot \mathfrak{Y}_{ij}^{-1} \\
			& = \displaystyle\lim_{\varepsilon \rightarrow 0} \varepsilon^{i + j-1} e^{-\varepsilon \mathfrak{h}_{\ge c}^{\rightarrow} (i-1, j)} \cdot \mathfrak{Y}_{ij}^{-1} = \mathfrak{Y}_{ij}^{-1} \cdot \mathfrak{X}_c^0 (i - 1, j),
		\end{flalign*} 
		
		\noindent where the second statement follows from \eqref{abcdmm1kj}; similarly, it is quickly verified that $\mathfrak{X}_c^0 (1, j) = \mathfrak{Y}_{1j}^{-1} \cdot \mathfrak{X}_c^0 (1, j-1)$ for $j > \mathrm{K}_{[1,c-1]} + 1$ and $\mathfrak{X}_c^0 (1, \mathrm{K}_{[1,c-1]} + 1) = \mathfrak{Y}_{1, \mathrm{K}_{[1,c-1]}}^{-1}$, using \eqref{dconverge}. These, with \eqref{xc0ij} confirm that $\mathfrak{X}_c^0$ satisfies the same recursion \eqref{zijij} determining the $\mathfrak{Z}_c (i, j)$. It follows that $\mathfrak{X}_c^0 (i, j) = \mathfrak{Z}_c (i, j)$, which establishes the theorem.
	\end{proof}

	\begin{rem} 
		
		\label{abcdpath} 
		
		Let us explain how the random ensemble described in this section arises as a specialization of the colored stochastic fused vertex model introduced in \Cref{FusedPath}. To that end, let $\mathfrak{M} \in \mathbb{R}$ be a real number and $\mathrm{L} \ge 1$ be an integer (that we will analytically continue in). The parameters $(x_j, r_j)$ associated with the $j$-th row of the model, and those $(y_i, s_i)$ associated with the $i$-th column, will be given by 
		\begin{flalign*}
			(x_j, r_j) = (q^{1 -\theta_j'}, q^{-\mathrm{L}/2}), \quad \text{for each $j \ge 1$}; \qquad (y_i, s_i) = (q^{\theta_i - \mathfrak{M}}, q^{-\mathfrak{M}/2}), \quad \text{for each $i \ge 1$}.
		\end{flalign*}
		
		Consider a colored stochastic fused vertex model on $\mathbb{Z}_{> 0}^2$ with these parameters, and with $\mathrm{L}$ arrows of color $m(j)$ entering through row\footnote{Since $m = m(j)$ is defined so that $j \in \llbracket \mathrm{K}_{[1,m-1]} + 1, \mathrm{K}_{[1,m]} \rrbracket$, this means that there are $\mathrm{K}_1, \mathrm{K}_2, \ldots $ rows inputting arrows of colors $1, 2, \ldots $ into the model, respectively.} $j$, for each $j \ge 1$. Let $\big( \bm{A} (i,j), \bm{B} (i,j); \bm{C} (i, j), \bm{D} (i, j) \big)$ denote the colored fused arrow configuration in this model at any vertex $(i, j) \in \mathbb{Z}_{> 0}^2$. We next ``complement'' arrows of color $m(j)$ in the $j$-th row, by setting $\bm{B} (i,j) = \big( \overline{\bm{B}} (i,j), \mathrm{L} - \mathsf{B}_{m(j)} (i, j) \big)$ and $\bm{D} (i, j) = \big( \overline{\bm{D}} (i, j), \mathrm{L} - \mathsf{D}_{m(j)} (i, j) \big)$ for each $i \ge 1$. We then track the complemented arrow configuration $\big( \bm{A} (i, j), (\overline{\bm{B}} (i, j), \mathsf{B}_{m(j)} (i, j)); \bm{C} (i, j), (\overline{\bm{D}} (i, j), \mathsf{D}_{m(j)} (i, j)) \big)$ over $(i, j) \in \mathbb{Z}_{> 0}^2$.
		
		Now analytically continue in $\mathrm{L}$, replacing it with a real number $\mathfrak{L} \ge 0$. Let $\mathfrak{M}$ tend to $\infty$, and then let $\mathfrak{L}$ tend to $\infty$. Since $x_j y_i^{-1} = q^{1 - \theta_i - \theta_j' + \mathfrak{M}} = q^{\mathfrak{M} + 1} \gamma_{ij}^{-1}$, \Cref{rqnl1}, \Cref{ufusedqmqn}, and \Cref{ufusedql} (with the $q^{\mathfrak{N}}$ in the first two equal to $q^{\mathfrak{M} + \mathrm{L}} \gamma_{ij}^{-1}$ here, so that $x_j y_i^{-1} = q^{\mathfrak{N} - \mathrm{L} + 1}$) implies that this procedure yields at $(i, j) \in \mathbb{Z}_{> 0}^2$ the $\mathcal{U}_{\gamma}^{\qdp; m(j)}$-weights from \Cref{ufusedql}. Since $x_j (y_i s_i^2)^{-1} = q \gamma_{ij}^{-1}$, it also at any vertex $(1, j)$ (in the first column) yields the $\mathcal{U}_{1/\gamma_{1j}}^{b;m(j)}$-weights from \Cref{u3}. Hence, this gives rise to the vertex model described above \Cref{convergegamma}.    
		
	\end{rem} 

	\begin{rem}
		
		\label{q0} 
		
		Let us briefly (and informally) explain the reason behind our parameter choices when degenerating the colored fused stochastic vertex model to the log-gamma polymer in \Cref{abcdpath}. In the uncolored case $n=1$, \cite[Proposition 7.26]{RFSVM} describes how to specialize the $U_q (\widehat{\mathfrak{sl}}_2)$ stochastic fused vertex model to the $q$-Hahn PushTASEP introduced in \cite{TPT}. However, the arrows in that vertex model are reversed relative to here; they are directed up-left instead of up-right \cite[Figure 19]{RFSVM}. To remedy this, one must complement those arrow configurations by tracking how many arrows its horizontal edges are from being saturated (as in \eqref{bnlbn}); this directs its paths up-right (and also imposes a change in the spectral parameters involved in the vertex weights). The vertex weights obtained in this way precisely coincide with the $n=1$ case of the complemented ones provided in \Cref{2u} (and \Cref{rqnl1}). The subsequent limits taken in \Cref{LimitU} correspond to degenerating the $q$-Hahn PushTASEP to the $q$-PushTASEP; those taken in \Cref{Weightsq1} correspond to degenerating the latter to the log-gamma polymer (as described in \cite{RCRP}).
		
		In the colored case $n > 1$, the setup is similar, though it admits a few differences. The main one is that we only complement arrows of the largest color when defining the complemented weights in Secion \ref{WeightsComplement} (a choice that is essentially forced by the step type boundary conditions we consider). This leads to a seemingly new presence of ``transitionary rows,'' where the largest color in the model changes, and the limiting behavior of the weights in these rows must be addressed in \Cref{Weightsq1}.  
		
	\end{rem}

	\section{Effective Convergence of the Six-Vertex Model to ASEP}
	
	\label{Converge}
	
	In \Cref{L0} we described colored line ensembles for the stochastic six-vertex model. It was shown in \cite{SSVM,CSSVMP} that under a certain limit degeneration that the latter, with weights $(b_1, b_2)$ as depicted in \Cref{rz2}, converges to the asymmetric simple exclusion process (ASEP), with left jump rate $L$ and right jump rate $R$. This degeneration takes $(b_1, b_2) = (\varepsilon L, \varepsilon R)$, scales the vertical coordinate by $\varepsilon^{-1}$, lets $\varepsilon$ tend to $0$, and observes the vertically exiting arrows along the main diagonal of $\mathbb{Z}^2$; this in particular makes time (the vertical coordinate) continuous and space (the horizontal coordinate) infinite. A similar limit can be taken on the associated colored line ensemble, and most of its relevant properties would be preserved, including its height function match \Cref{lmu2} (in this case, to the colored ASEP) and its Gibbs property \Cref{conditionl}. However, the domain of this colored line ensemble would be infinite, and therefore its boundary conditions would be lost. 
	
	When analyzing line ensembles, it is at times (including in the forthcoming work \cite{PP}) useful to keep these boundary conditions in tact. Thus, to understand the colored ASEP, it can be helpful not to directly study its line ensemble by letting $\varepsilon$ tend to $0$ in the colored stochastic six-vertex one, but instead to analyze the latter at $\varepsilon > 0$ (where its boundary conditions are present) and then let $\varepsilon$ tend to $0$ afterwards. This can require effective convergence rates of the stochastic six-vertex model to the ASEP, which were not proven in \cite{CSSVMP}. 
	
	In this section we provide such a convergence result. We state it in Section \ref{VertexColor} after introducing the colored ASEP in Section \ref{ParticleColor}; its proof is then given in Section \ref{Exponential} and Section \ref{ProofCouple0}. Throughout this section, we fix real numbers\footnote{The ASEP usually imposes the asymmetry condition $R \ne L$, but we will not require that here.} $R, L \ge 0$. The constants below might implicitly depend on $R$ and $L$, even when not stated explicitly.
	
	\subsection{Properties of the Colored ASEP}
	
	\label{ParticleColor} 
	
	The \emph{colored ASEP} is a continuous time Markov process that can be described as follows. Particles are initially, at time $0$, placed on $\mathbb{Z}$ in such a way that exactly one particle occupies any site. Assigned to each particle is a color, which is a nonnegative integer label that informally measures the ``priority'' of the particle (those of a larger color are viewed as having higher priority than those of a smaller one). Denote the color of the particle at site $x \in \mathbb{Z}$ and time $t \in \mathbb{R}_{\ge 0}$ by $\eta_t (x) \in \mathbb{Z}_{\ge 0}$; further denote the full state of the process at time $t$ by $\bm{\eta}_t = \big( \eta_t (x) \big)_{x \in \mathbb{Z}}$. 
	
	Associated with each site $x \in \mathbb{Z}$ are two independent exponential clocks, a ``left'' one of rate $L$ and a ``right'' one of rate $R$. If the left clock of $x$ rings at some time $s \in \mathbb{R}_{> 0}$, then the particle at site $x$ switches places with the one at site $x-1$ if $\eta_{s^-} (x) > \eta_{s^-} (x-1)$ (and does nothing otherwise). Similarly, if the right clock of $x$ rings at time $s$, then the particle at site $x$ switches places with the one at site $x+1$ if $\eta_{s^-} (x) > \eta_{s^-} (x+1)$ (and does nothing otherwise).
	
	We next recall from \cite{NIPML,ASVPGM} a graphical representation for the colored ASEP. For any $x \in \mathbb{Z}$, let $\bm{\mathsf{S}} (x) = \big( \mathsf{S}_1 (x), \mathsf{S}_2 (x), \ldots \big)$ and $\bm{\mathsf{T}} (x) = \big( \mathsf{T}_1 (x), \mathsf{T}_2 (x), \ldots \big)$ denote the ringing times (in increasing order) for the left and right clocks associated with site $x$, respectively. For each integer $i \ge 1$, draw a directed arrow on $\mathbb{Z} \times \mathbb{R}_{\ge 0}$ from $\big( x, \mathsf{S}_i (x) \big)$ to $\big( x-1, \mathsf{S}_i (x) \big)$ and from $\big( x, \mathsf{T}_i (x) \big)$ to $\big( x+1, \mathsf{T}_i (x) \big)$. The union of these arrows from a directed graph $\mathcal{G}$, which we call the ASEP time graph; its horizontal and vertical directions index space and time, respectively. Given $\mathcal{G}$, the dynamics of the colored ASEP are defined by having each particle remain at its site $x$ until it reaches a time $t$ at which there is an arrow in $\mathcal{G}$ connecting $(x, t)$ to some $(y, t)$ (for $y \in \{ x - 1, x + 1 \}$). At this time $t$, the particle at site $x$ switches locations with the one at site $y$ if either $\eta_{t^-} (x) > \eta_{t^-} (y)$ and the edge is directed from $(x, t)$ to $(y, t)$, or if $\eta_{t^-} (x) < \eta_{t^-} (y)$ and the edge is directed from $(y, t)$ to $(x, t)$; otherwise, the particle at $x$ stays in place. 
	
	Before proceeding, let us record the following lemma, which is sometimes known as a finite speed of discrepancy bound. It states that, if two colored ASEPs initially agree on a (sufficiently long) interval, then with high probability their dynamics can be coupled so as to agree on a shorter interval, up until a given time.

	\begin{lem} 
		
		\label{xietaexclusion} 
		
		There exists a constant $C > 1$ such that the following holds. Let $T \ge 0$ and $K > 1$ be real numbers; $U \le V$ be integers; and $\bm{\xi} = \big( \xi_t (x) \big)$ and $\bm{\eta} = \big( \eta_t (x) \big)$ denote two colored ASEPs whose initial data satisfy $\xi_0 (x) = \eta_0 (x)$ for each $x \in \llbracket U - CKT, V + CKT \rrbracket$. Then it is possible to couple $\bm{\xi}$ and $\bm{\eta}$ such that, with probability at least $1 - C e^{-K(T+1)}$, we have $\xi_t (x) = \eta_t (x)$ for each $(x, t) \in \llbracket U, V \rrbracket \times [0, T]$. 
		
	\end{lem} 
	
	\begin{proof}
		
		Let $C > 1$ be a constant to be fixed later. Couple $\bm{\xi}$ and $\bm{\eta}$ under the same time graph $\mathcal{G}$. Then $\xi_{t_0} (x_0) \ne \eta_{t_0} (x_0)$ holds for some $(x_0, t_0) \in \llbracket U, V \rrbracket \times [0, T]$ only if there exists some $(x_1, t_1) \in  \big(\mathbb{Z} \setminus \llbracket U - CKT, V + CKT \rrbracket \big) \times [0, T]$, such that a particle at site $x_1$ at time $t_1$ could (for some trajectories of the remaining particles) enter the interval $\llbracket U, V \rrbracket$ sometime during $[0, T]$ under $\mathcal{G}$; this is contained in the union of two events. The first is that there exists a sequence of times $0 \le r_1 \le r_2 \le \ldots \le r_{\lfloor CKT \rfloor} \le T$ for which there exists an arrow in $\mathcal{G}$ connecting $\big( U - \lfloor CKT \rfloor + i - 1, r_i \big)$ and $\big( U - \lfloor CKT \rfloor +i, r_i \big)$ for each $i \in \llbracket 1, CKT \rrbracket$; the second is that there exists a sequence of times $0 \le r_1'\le r_2'\le \cdots \le r_{\lfloor CKT \rfloor}' \le T$ for which there exists an arrow in $\mathcal{G}$ connecting $\big( V + \lfloor CKT \rfloor - i +1, r_i')$ and $\big(V + \lfloor CKT \rfloor -i, r_i' \big)$ for each $i \in \llbracket 1, CKT \rrbracket$. 
		
		Now recall that the set of times at which arrows in $\mathcal{G}$ enter or exit a given column $\{ x \} \times \mathbb{R}_{\ge 0}$ has the same law as that of the ringing times of an exponential clock of rate $R + L$. Hence, the probabilty of each of the above events is bounded by the probability that of a sum of $T$ exponential random variables with parameter $R+L$ travels a distance of at least $CKT$. A Chernoff bound implies that the latter is at most $C e^{-K(T+1)}$ if $C$ is sufficiently large, establishing the lemma.
	\end{proof}

	\subsection{Properties of the Colored Stochastic Six-Vertex Model}
	
	\label{VertexColor}
	
	In this section we state the effective convergence of the stochastic six-vertex model to the colored ASEP, and also provide several properties of the former. Let $b_1, b_2 \in [0, 1]$ denote real numbers; set $q = b_1 b_2^{-1}$; consider the colored stochastic six-vertex model (as defined in \Cref{ModelVertex}) on the quadrant $\mathbb{Z}_{> 0}^2$, with spectral parameter $z_{i,j} = (1 - b_2) (1 - b_1)^{-1}$ at each $(i, j) \in \mathbb{Z}_{> 0}^2$; and assume that the largest color in this system is at most some integer $n \ge 0$. See \Cref{rz2} for the stochastic weights of this model. For any integer $(x, y) \in \mathbb{Z}_{> 0}^2$, let $\eta_y (x) \in \llbracket 0, n \rrbracket$ denote the color of the arrow vertically exiting $(x, y)$ in this model; also let $\bm{\eta}_y = \big( \eta_y (x) \big)_{x > 0}$ denote the full state of the process at vertical coordinate $y$.

	\begin{figure} 
		\begin{center}
			\begin{tikzpicture}[
				>=stealth,
				scale = .75
				]
				
				\draw[-, black] (-5, 3.1) -- (7.5, 3.1);
				\draw[-, black] (-5, -1.9) -- (7.5, -1.9);
				\draw[-, black] (-5, -1.1) -- (7.5, -1.1);
				\draw[-, black] (-5, -.4) -- (7.5, -.4);
				\draw[-, black] (-5, 2.4) -- (7.5, 2.4);
				\draw[-, black] (7.5, -1.9) -- (7.5, 3.1);
				\draw[-, black] (-5, -1.9) -- (-5, 3.1);
				\draw[-, black] (5, -1.9) -- (5, 2.4);
				\draw[-, black] (-2.5, -1.9) -- (-2.5, 3.1);
				\draw[-, black] (2.5, -1.9) -- (2.5, 2.4);
				\draw[-, black] (0, -1.9) -- (0, 2.4);			
				
				\draw[->, thick, blue] (3.75, .1) -- (3.75, 1) -- (4.65, 1);
				\draw[->, thick, red] (2.85, 1) -- (3.75, 1) -- (3.75, 1.9); 
				
				\draw[->, thick, blue] (-1.25, .1) -- (-1.25, 1.9);
				\draw[->, thick,  red] (-2.15, 1) -- (-.35, 1);
				
				\draw[->, thick, blue] (.35, 1) -- (2.15, 1);
				\draw[->, thick, red] (1.25, .1) -- (1.25, 1.9);
				
				\draw[->, thick, blue] (5.35, 1) -- (6.25, 1) -- (6.25, 1.9);
				\draw[->, thick, red] (6.25, .1) -- (6.25, 1) -- (7.15, 1); 
				
				\draw[->, thick, red] (-3.75, .1) -- (-3.75, 1.9);
				\draw[->, thick, red] (-4.65, 1) -- (-2.85, 1);		
				
				\filldraw[fill=gray!50!white, draw=black] (-2.85, 1) circle [radius=0] node [black, right = -1, scale = .75] {$i$};
				\filldraw[fill=gray!50!white, draw=black] (-.35, 1) circle [radius=0] node [black, right = -1, scale = .75] {$i$};
				\filldraw[fill=gray!50!white, draw=black] (2.15, 1) circle [radius=0] node [black, right = -1, scale = .75] {$j$};
				\filldraw[fill=gray!50!white, draw=black] (4.65, 1) circle [radius=0] node [black, right = -1, scale = .75] {$j$};
				\filldraw[fill=gray!50!white, draw=black] (7.15, 1) circle [radius=0] node [black, right = -1, scale = .75] {$i$};
				
				\filldraw[fill=gray!50!white, draw=black] (5.35, 1) circle [radius=0] node [black, left = -1, scale = .75] {$j$};
				\filldraw[fill=gray!50!white, draw=black] (2.85, 1) circle [radius=0] node [black, left = -1, scale = .75] {$i$};
				\filldraw[fill=gray!50!white, draw=black] (.35, 1) circle [radius=0] node [black, left = -1, scale = .75] {$j$};
				\filldraw[fill=gray!50!white, draw=black] (-2.15, 1) circle [radius=0] node [black, left = -1, scale = .75] {$i$};
				\filldraw[fill=gray!50!white, draw=black] (-4.65, 1) circle [radius=0] node [black, left = -1, scale = .75] {$i$};
				
				\filldraw[fill=gray!50!white, draw=black] (-3.75, 1.9) circle [radius=0] node [black, above = -1, scale = .75] {$i$};
				\filldraw[fill=gray!50!white, draw=black] (-1.25, 1.9) circle [radius=0] node [black, above = -1, scale = .75] {$j$};
				\filldraw[fill=gray!50!white, draw=black] (1.25, 1.9) circle [radius=0] node [black, above = -1, scale = .75] {$i$};
				\filldraw[fill=gray!50!white, draw=black] (3.75, 1.9) circle [radius=0] node [black, above = -1, scale = .75] {$i$};
				\filldraw[fill=gray!50!white, draw=black] (6.25, 1.9) circle [radius=0] node [black, above = -1, scale = .75] {$j$};
				
				\filldraw[fill=gray!50!white, draw=black] (-3.75, .1) circle [radius=0] node [black, below = -1, scale = .75] {$i$};
				\filldraw[fill=gray!50!white, draw=black] (-1.25, .1) circle [radius=0] node [black, below = -1, scale = .75] {$j$};
				\filldraw[fill=gray!50!white, draw=black] (1.25, .1) circle [radius=0] node [black, below = -1, scale = .75] {$i$};
				\filldraw[fill=gray!50!white, draw=black] (3.75, .1) circle [radius=0] node [black, below = -1, scale = .75] {$j$};
				\filldraw[fill=gray!50!white, draw=black] (6.25, .1) circle [radius=0] node [black, below = -1, scale = .75] {$i$};	
				
				\filldraw[fill=gray!50!white, draw=black] (-3.75, .1) circle [radius=0] node [black, below = -1, scale = .75] {$i$};
				\filldraw[fill=gray!50!white, draw=black] (-1.25, .1) circle [radius=0] node [black, below = -1, scale = .75] {$j$};
				\filldraw[fill=gray!50!white, draw=black] (1.25, .1) circle [radius=0] node [black, below = -1, scale = .75] {$i$};
				\filldraw[fill=gray!50!white, draw=black] (3.75, .1) circle [radius=0] node [black, below = -1, scale = .75] {$j$};
				\filldraw[fill=gray!50!white, draw=black] (6.25, .1) circle [radius=0] node [black, below = -1, scale = .75] {$i$};
				
				\filldraw[fill=gray!50!white, draw=black] (-3.75, 2.75) circle [radius=0] node [black] {$i \in \llbracket 0, n \rrbracket$};
				\filldraw[fill=gray!50!white, draw=black] (2.5, 2.75) circle [radius=0] node [black] {$0 \le i < j \le n$}; 
				
				\filldraw[fill=gray!50!white, draw=black] (-3.75, -.75) circle [radius=0] node [black] {$(i, i; i, i)$};
				\filldraw[fill=gray!50!white, draw=black] (-1.25, -.75) circle [radius=0] node [black] {$(j, i; j, i)$};
				\filldraw[fill=gray!50!white, draw=black] (1.25, -.75) circle [radius=0] node [black] {$(i, j; i, j)$};
				\filldraw[fill=gray!50!white, draw=black] (3.75, -.75) circle [radius=0] node [black] {$(j, i; i, j)$};
				\filldraw[fill=gray!50!white, draw=black] (6.25, -.75) circle [radius=0] node [black] {$(i, j; j, i)$};
				
				\filldraw[fill=gray!50!white, draw=black] (-3.75, -1.5) circle [radius=0] node [black] {$1$};
				\filldraw[fill=gray!50!white, draw=black] (-1.25, -1.5) circle [radius=0] node [black] {$b_1$};
				\filldraw[fill=gray!50!white, draw=black] (1.25, -1.5) circle [radius=0] node [black] {$b_2$};
				\filldraw[fill=gray!50!white, draw=black] (3.75, -1.5) circle [radius=0] node [black] {$1 - b_1$};
				\filldraw[fill=gray!50!white, draw=black] (6.25, -1.5) circle [radius=0] node [black] {$1 - b_2$};
			\end{tikzpicture}
		\end{center}
		
		\caption{\label{rz2} The weights for the colored stochastic six-vertex model in Section \ref{VertexColor} are depicted above.}
	\end{figure} 
	
	Next we state the effective\footnote{We made no effort to optimize the exponent $\varepsilon^{1/8}$ in the probability in \Cref{convergexieta}.} convergence result; it will be proven in Section \ref{Exponential} below. In what follows, given a function $\varpi : \mathbb{Z} \rightarrow \llbracket 0, n \rrbracket$, we say that the colored stochastic six-vertex model $\bm{\eta}$ on the quadrant $\mathbb{Z}_{> 0 }^2$ has boundary data $\varpi$ if the below holds. For each integer $x \le 0$, an arrow of color $\varpi (x)$ horizontally enters the quadrant through $(0, 1 - x)$ and, for each integer $x > 0$, an arrow of color $\varpi (x)$ vertically enters the quadrant through $(x, 0)$. We say that the boundary data for this model matches the initial condition for a colored ASEP $\bm{\xi} = \big( \xi_t (x) \big)$ if $\varpi = \xi_0$. 
	
	\begin{prop} 
		
		\label{convergexieta} 
		
		There exists a constant $C > 1$ such that the following holds. Let $\varepsilon \in (0 ,1)$ and $T \ge 0$ be real numbers; let $U \le V$ be integers; let $\mathcal{T} \subset [0, T]$ denote a finite set of real numbers; and let $\bm{\xi} = \big( \xi_t (x) \big)$ denote a colored ASEP with at most $n$ colors. Further let $\bm{\eta} = \big( \eta_y (x) \big)$ denote a colored stochastic six-vertex model on $\mathbb{Z}_{> 0}^2$ with parameters $(b_1, b_2) = (\varepsilon L, \varepsilon R)$, whose boundary data matches the initial data of $\bm{\xi}$. It is possible to couple $\bm{\xi}$ and $\bm{\eta}$ so that, with probability at least $1 - C (V-U+1) (T+1)^2 |\mathcal{T}| \varepsilon^{1/8}$, we have
		\begin{flalign}
			\label{xieta0}
			\xi_t (x) = \eta_{\lfloor t / \varepsilon \rfloor} \big(x + \lfloor t \varepsilon^{-1} \rfloor \big), \qquad \text{for each $(x, t) \in \llbracket U, V \rrbracket \times \mathcal{T}$ with $x + \lfloor t \varepsilon^{-1} \rfloor > 0$}.
		\end{flalign}

	\end{prop} 
	
	To prove \Cref{convergexieta}, it will be useful to introduce notation for the randomness defining the colored stochastic six-vertex model (this discussion also appears in a slightly different, though equivalent, form in \cite[Section 2.3]{LSLSSSVM} and \cite[Section 2.2]{CSDSVM}). To that end, for any vertex $(x, y) \in \mathbb{Z}_{\ge 0}$, we associate Bernoulli random variables $\chi_{x,y}, \psi_{x, y} \in \{ 0, 1 \}$ with $\mathbb{P} [\chi_{x, y} = 1] = b_1$ and $\mathbb{P}[\psi_{x, y} = 1] = b_2$. Given all of these random variables, the dynamics of the colored stochastic six-vertex model are defined as follows. Fix a vertex $(x, y) \in \mathbb{Z}_{> 0}^2$, and suppose that colored six-vertex arrow configurations have been assigned to each $(x', y') \in \mathbb{Z}_{> 0}^2$ with $x' + y' < x + y$; this fixes the colors $a = a(x, y)$ and $b = b(x, y)$ of the arrows vertically and horizontally entering $(x, y)$, respectively. Then define the colors $c = c(x, y)$ and $d = d(x,y)$ of the arrows vertically and horizontally exiting $(x,y)$, respectively, according to the below procedure. 
	\begin{enumerate} 
		\item If $a = b$, then set $c = d = a = b$.
		\item If $a > b$, then set $(c, d) = (a, b)$ if $\chi_{x, y} = 1$ and $(c, d) = (b, a)$ if $\chi_{x, y} = 0$. 
		\item If $a < b$, then set $(c, d) = (a, b)$ if $\psi_{x, y} = 1$ and $(c, d) = (b, a)$ if $\psi_{x, y} = 0$. 
	\end{enumerate}
	
	This provides a way of sampling the colored stochastic six-vertex model defined above. In what follows, for any $x \in \mathbb{Z}$, define the increasing integer sequences $\bm{\mathfrak{s}} (x) = \big( \mathfrak{s}_1 (x), \mathfrak{s}_2 (x), \ldots )$ and $\bm{\mathfrak{t}} (x) = \big( t_1 (x), t_2 (x), \ldots \big)$, such that $\mathfrak{s} \in \bm{\mathfrak{s}} (x)$ if and only if $\chi_{x+\mathfrak{s}-1, \mathfrak{s}} > 0$, and $\mathfrak{t} \in \bm{\mathfrak{t}} (x)$ if and only if $\psi_{x+\mathfrak{t}, \mathfrak{t}} > 0$ (here, the offsets $x + \mathfrak{s}$ and $x + \mathfrak{t}$ are introduced to match with \eqref{xieta0}). Observe that the $\big( \bm{\mathfrak{s}} (x) \big)$ and $\big( \bm{\mathfrak{t}} (x) \big)$ (over all $x \in \mathbb{Z}$) together determine all $(\chi_{x, y}, \psi_{x, y})_{x,y>0}$.
	
	Before proceeding, we record the following lemma that provides a finite speed of discrepancy bound for the colored stochastic six-vertex model (and is analogous to \Cref{xietaexclusion}). 
	
	\begin{lem} 
		
		\label{xietaprocess}
		
		There exists a constant $C > 1$ such that the following holds. Let $T \ge 0$, $K > 1$, and $\varepsilon \in (0, 1)$ be real numbers; set $(b_1, b_2) = (\varepsilon L, \varepsilon R)$; and let $U \le V$ be integers. Further let $\bm{\xi}_y = \big( \xi_y (x) \big)_{x > 0}$ and $\bm{\eta}_y = \big( \eta_y (x) \big)_{x > 0}$ denote two colored stochastic six-vertex models with parameters $(b_1, b_2)$ and boundary data $\varpi^{\bm{\xi}} : \mathbb{Z} \rightarrow \llbracket 0, n \rrbracket$ and $\varpi^{\bm{\eta}} : \mathbb{Z} \rightarrow \llbracket 0, n \rrbracket$, respectively. Assume that $\varpi^{\bm{\xi}} (x) = \varpi^{\bm{\eta}} (x)$ for each $x \in \llbracket U - CKT, V + CKT \rrbracket$. Then it is possible to couple $\bm{\xi}$ and $\bm{\eta}$ such that, with probability at least $1 - C e^{-K(T+1)}$, we have $\xi_y (x) = \eta_y (x)$ for each $(x, y) \in \mathbb{Z}_{> 0} \times \llbracket 1, T \varepsilon^{-1} \rrbracket$ satisfying $U  \le x-y \le V$.
		
	\end{lem} 
	
	\begin{proof}
		
		Couple $\bm{\xi}$ and $\bm{\eta}$ under the same processes $\big( \bm{\mathfrak{s}} (x) \big)$ and $\big( \bm{\mathfrak{t}} (x) \big)$, and define the strip 
		\begin{flalign*}
			\mathcal{R} = \big\{ (x, y) \in \mathbb{Z}_{> 0} \times \llbracket 1, T \varepsilon^{-1} \rrbracket : U \le x - y \le V \big\}.
		\end{flalign*} 
		
		\noindent Then $\xi_{y_0} (x_0) \ne \eta_{y_0} (x_0)$ holds for some $(x_0, y_0) \in \mathcal{R}$ only if there exists some $(x_1, y_1)$ on the $x$-axis or $y$-axis with $U - CKT \le x_1 - y_1 \le V + CKT$, such that a colored path entering the quadrant through $(x_1, y_1)$ could (for some trajectories of the remaining colored paths) enter $\mathcal{R}$ under the above randomness $\big( \bm{\mathfrak{s}} (x) \big)$ and $\big( \bm{\mathfrak{t}} (x) \big)$. This is contained in the event on which there exists an integer $k \ge 0$; a sequence of integers $w_0, w_1, \ldots , w_k$ with $|w_i - w_{i-1}| = 1$ for each $i \in \llbracket 1, k \rrbracket$; and a sequence of positive integers $0 \le \mathfrak{r}_0 \le \mathfrak{r}_1 \le \ldots \le \mathfrak{r}_k \le T \varepsilon^{-1}$ such that $\mathfrak{r}_i \in \bm{\mathfrak{s}} (w_i) \cup \bm{\mathfrak{t}} (w_i)$ for each $i \in \llbracket 1, k \rrbracket$, and one of the following two possibilities holds. The first is that $w_0 \le U - CKT \le U  \le w_k$; the second is that $w_0 \ge V + CKT \ge V \ge w_k$. 
		
		Now, observe that the entry differences of any $\bm{\mathfrak{s}} (x) \cup \bm{\mathfrak{t}} (x)$ are distributed as independent geometric random variables with parameter $1 - (1-\varepsilon L)(1 - \varepsilon R) \le \varepsilon (L + R)$. Hence, the probabilty of each of the above possibilities is bounded by the probability that of a sum of $\lfloor T \varepsilon^{-1} \rfloor$ geometric random variables with parameter $\varepsilon (R+L)$ travels a distance of at least $CKT$. A Chernoff bound implies that the latter is at most $C e^{-K(T+1)}$ if $C$ is sufficiently large, establishing the lemma.
	\end{proof}

	\subsection{Proof of \Cref{convergexieta}} 
	
	\label{Exponential} 
	
	In this section we establish \Cref{convergexieta}. Its proof will use the below lemma, which is shown in Section \ref{ProofCouple0} below. It indicates that the processes $\bm{\mathsf{S}}(x)$ and $\bm{\mathsf{T}}(x)$ from Section \ref{ParticleColor} can be coupled to nearly coincide (after scaling) with the $\bm{\mathfrak{s}} (x)$ and $\bm{\mathfrak{t}} (x)$ from Section \ref{VertexColor}, on a long interval, with high probability. In what follows, we recall the notation from those sections, assocating the parameters $(L; R)$ with the $(\bm{\mathsf{S}}; \bm{\mathsf{T}})$ processes and $(\varepsilon L; \varepsilon R)$ with the $(\bm{\mathfrak{s}}; \bm{\mathfrak{t}})$ ones. 
	
	\begin{lem}
		
		\label{stcouple} 
		
		There exists a coupling between $\big( \bm{\mathsf{S}} (x); \bm{\mathsf{T}} (x) \big)_{x \in \mathbb{Z}}$ and $\big( \bm{\mathfrak{s}} (x); \bm{\mathfrak{t}} (x) \big)_{x \in \mathbb{Z}}$, and a constant $C > 1$ such that, for any real numbers $U_0 \le V_0$, $\varepsilon \in (0, 1)$, and $T \ge 0$, the following statements all hold with probability at least $1 - C (V_0 - U_0 + 1)(T+1) \varepsilon^{1/4}$.
		
		\begin{enumerate} 
			\item For any index pair $(\mathsf{R}, \mathfrak{r}) \in \big\{ (\mathsf{S}, \mathfrak{s}), (\mathsf{T}, \mathfrak{t}) \big\}$ and each pair $(i, x) \in \mathbb{Z}_{\ge 1} \times \llbracket U_0, V_0 \rrbracket$ such that $\mathsf{R}_i (x) \le T$ or $\mathfrak{r}_i (x) \le T \varepsilon^{-1}$, we have $\varepsilon \mathfrak{r}_i (x) \le \mathsf{R}_i (x) \le \varepsilon \mathfrak{r}_i (x) + \varepsilon^{1/2} < T$. 
			\item For distinct triples $(\mathsf{R}, i, x), (\mathsf{R}', i', x') \in \{ \mathsf{S}, \mathsf{T} \} \times \mathbb{Z}_{\ge 1} \times \llbracket U_0, V_0 \rrbracket$ such that $\mathsf{R}_i (x), \mathsf{R}_{i'}' (x') < T$, we have $\big| \mathsf{R}_i (x) - \mathsf{R}_{i'}' (x') \big| > 2 \varepsilon^{1/2}$.
		\end{enumerate}

	\end{lem}

	\begin{proof}[Proof of \Cref{convergexieta}]
		
		We will assume for notational convenience that the set of times $\mathcal{T} = \{ T \}$ (from which the proof of the proposition for general $\mathcal{T}$ quickly follows by a union bound). Set $K = \lceil \varepsilon^{-1/8} \rceil$, and let $C_0 > 1$ denote the maximum of the constants $C$ from \Cref{xietaexclusion}, \Cref{xietaprocess}, and \Cref{stcouple}. 
		
		We first use the finite speed of discrepancy bounds (\Cref{xietaexclusion} and \Cref{xietaprocess}) to ``cut off'' the initial data for $\bm{\xi}$ and $\bm{\eta}$. To that end, let $\bm{\xi}' = \big( \xi_t' (x) \big)$ denote a colored ASEP with initial data $\xi_0' (x) = \xi_0 (x) \cdot \mathbbm{1}_{x \in [U - C_0 KT, V + C_0 K T]}$, and let $\bm{\eta}' = \big( \eta_y'(x) \big)$ denote a colored stochastic six-vertex model with parameters $(b_1, b_2) = (\varepsilon L, \varepsilon R)$, whose boundary data matches the initial data of $\bm{\xi}'$. By \Cref{xietaexclusion} and \Cref{xietaprocess}, we may couple $\bm{\xi}$ and $\bm{\xi}'$, and also $\bm{\eta}$ and $\bm{\eta}'$, so that with probability at least $1 - 2 C_0 e^{-K} \ge 1 - 16 C_0 \varepsilon$ we have $\xi_t (x) = \xi_t' (x)$ for each $(x, t) \in \llbracket U, V \rrbracket \times [0, T]$, and $\eta_y (x) = \eta_y' (x)$ for each $(x, y) \in \mathbb{Z}_{> 0} \times \llbracket 1, T \varepsilon^{-1} \rrbracket$ with $U \le x - y \le V$. Hence, we may replace $\bm{\xi}$ and $\bm{\eta}$ by $\bm{\xi}'$ and $\bm{\eta}'$, respectively, so we will assume in what follows that $\bm{\xi} = \bm{\xi}'$ and $\bm{\eta} = \bm{\eta}'$. 
		
		Next define the process $\bm{\zeta} = \big( \zeta_y (x) \big)$ from $\bm{\eta}$ by setting $\zeta_y (x) = \eta_y (x+y)$ (see \eqref{xieta0}) for each $(x, y) \in \mathbb{Z}_{> 0}^2$. Recall from Section \ref{ParticleColor} and Section \ref{VertexColor} that $\bm{\xi}$ and $\bm{\zeta}$ are determined from the processes $\big( \bm{\mathsf{S}} (x); \bm{\mathsf{T}} (x) \big)_{x \in \mathbb{Z}}$ and $\big( \bm{\mathfrak{s}} (x); \bm{\mathfrak{t}} (x) \big)_{x \in \mathbb{Z}}$, respectively; it will be useful to restrict these processes to a bounded subset of $x \in \mathbb{Z}$. To that end, recall for any $x \in \mathbb{Z}$ that $\bm{\mathsf{S}} (x) \cup \bm{\mathsf{T}} (x)$ is given by the ringing times of an exponential clock of rate $L+R$. Hence, the trajectory of the rightmost particle of nonzero color in $\bm{\xi}$ is stochastically dominated by a random walk starting at $V + C_0 KT$, that jumps one space to the right whenever an exponential clock of rate $L+R$ rings. Therefore, a Chernoff bound implies (after increasing $C_0$ if necessary) that this particle remains left of $V + 2C_0 KT$ with probability at least $1 - C_0 e^{-K} \ge 1 - 8 C_0 \varepsilon$. Similarly, the leftmost particle of nonzero color in $\bm{\xi}$ remains right of $U - 2C_0 KT$ with probability at least $1 - 8 C_0 \varepsilon$. 
		
		We may apply analogous reasoning to $\bm{\zeta}$. In particular, for any $x \in \mathbb{Z}$, the differences of the entries in $\bm{\mathfrak{s}} (x) \cup \bm{\mathfrak{t}} (x)$ are given by mutually independent geometric random variables of parameter $1 - (1 - \varepsilon L) (1 - \varepsilon R) \le \varepsilon (L + R)$. Hence, recalling the diagonal shifts in the definitions of $\bm{\mathfrak{s}}$ and $\bm{\mathfrak{t}}$ from Section \ref{VertexColor}, the trajectory of the rightmost path (of nonzero color) in $\bm{\zeta}$ is stochastically dominated by a random walk starting at $V + C_0 KT$ that jumps to the right according to a geometric random variable of parameter $\varepsilon (L + R)$. Therefore, a Chernoff bound implies (after increasing $C_0$ if necessary) that this path remains to the left of $V + 2C_0 KT$ with probability at least $1 - C_0 e^{-K} \ge 1 - 8 C_0 \varepsilon$. Similarly, the leftmost path of nonzero color in $\bm{\zeta}$ remains right of $U - 2C_0 KT$ with probability at least $1 - 8 C_0 \varepsilon$. Together, these facts yield $\mathbb{P}[\mathscr{E}] \ge 1 - 32 C_0 \varepsilon$, where $\mathscr{E} = \mathscr{E}_1 \cap \mathscr{E}_2$ and we have denoted the events 
		\begin{flalign*} 
			& \mathscr{E}_1 = \big\{ \xi_t (x) = 0, \quad \text{for all $x \notin \llbracket U - 2C_0 KT, V + 2C_0 KT \rrbracket$ and $t \in [0, T]$} \big\}; \\ 
			& \mathscr{E}_2 = \big\{ \zeta_y (x) = 0, \quad \text{for all $(x, y) \in \llbracket U - 2C_0 KT, V + 2C_0 KT \rrbracket \times \llbracket 0, T \varepsilon^{-1} \rrbracket$ with $x + y > 0$} \big\}. 
		\end{flalign*} 
		
		Now, \Cref{stcouple} gives a coupling between $\big( \bm{\mathsf{S}} (x); \bm{\mathsf{T}} (x) \big)$ and $\big( \bm{\mathfrak{s}} (x); \bm{\mathfrak{t}} (x) \big)$ such that the following holds on an event $\mathscr{F}$ with $\mathbb{P} [\mathscr{F}] \ge 1 - C_0 (V - U + 4C_0 KT + 1)(T+1) \varepsilon^{1/4}$. First, for any index pair $(\mathsf{R}, \mathfrak{r}) \in \big\{ (\mathsf{S}, \mathfrak{s}), (\mathsf{T}, \mathfrak{t}) \big\}$ and each pair $(i, x) \in \mathbb{Z}_{\ge 1} \times \llbracket U - 2C_0 KT, V + 2C_0 KT \rrbracket$ with $\mathsf{R}_i (x) \le T$, we have 
		\begin{flalign}
			\label{rr0} 
			\varepsilon \mathfrak{r}_i (x) \le \mathsf{R}_i (x) \le \varepsilon \mathfrak{r}_i (x) + \varepsilon^{1/2} < T.
		\end{flalign} 
		
		\noindent Second, for any distinct triples $(\mathsf{R}, i, x), (\mathsf{R}', i', x') \in \{ \mathsf{S}, \mathsf{T} \} \times \mathbb{Z}_{\ge 1} \times \llbracket U - 2C_0 KT, V + 2C_0 KT \rrbracket$ with $\mathsf{R}_i (x), \mathsf{R}_{i'}' (x') \le T$, we have 
		\begin{flalign}
			\label{riri1} 
			\big| \mathsf{R}_i (x) - \mathsf{R}_{i'}' (x') \big| > 2 \varepsilon^{1/2}.
		\end{flalign}
		
		We claim on $\mathscr{E} \cap \mathscr{F}$ that the ASEP and stochastic six-vertex configurations coincide, namely,
		\begin{flalign}
			\label{xizeta0}
			\xi_T (x) = \zeta_{\lfloor T/\varepsilon \rfloor} (x), \qquad \text{for each $x \in \mathbb{Z}$ with $x + \lfloor T \varepsilon^{-1} \rfloor > 0$}.
		\end{flalign}
		
		\noindent To verify this, first observe on $\mathscr{E}$ that $(\bm{\zeta}_t)_{t \le T}$ and $(\bm{\xi}_y)_{y \le T/\varepsilon}$ only depend on $\big( \bm{\mathsf{S}} (x), \bm{T} (x) \big)$ and $\big( \bm{\mathfrak{s}} (x), \bm{\mathfrak{t}} (x) \big)$ for $x \in \llbracket U - 2C_0 KT, V + 2C_0 KT \rrbracket$, respectively. Let $\mathsf{R}_1 < \mathsf{R}_2 < \cdots < \mathsf{R}_m$  and $\mathfrak{r}_1 < \mathfrak{r}_2 < \cdots < \mathfrak{r}_{m'}$ be such that 
		\begin{flalign*} 
			& [0, T] \cap \bigcup_{x= \lceil U - 2C_0 KT \rceil}^{\lfloor V - 2C_0 KT \rfloor} \big( \bm{\mathsf{S}} (x) \cup \bm{T} (x) \big) = \{ \mathsf{R}_1, \mathsf{R}_2, \ldots , \mathsf{R}_m \}; \\
			& \llbracket 1, T \varepsilon^{-1} \rrbracket \cap \bigcup_{x = \lceil U - 2C_0 KT \rceil}^{\lfloor V - 2C_0 KT \rfloor} \big( \bm{\mathfrak{s}} (x) \cup \bm{\mathfrak{t}} (x) \big) = \{ \mathfrak{r}_1, \mathfrak{r}_2, \ldots , \mathfrak{r}_{m'} \}.
		\end{flalign*} 
		
		\noindent For each $i \in \llbracket 1, m \rrbracket$, let $w_i \in \llbracket U - 2C_0 KT, V + 2C_0 KT \rrbracket$ be such that $\mathsf{R}_i \in \bm{\mathsf{S}} (w_i) \cup \bm{\mathsf{T}} (w_i)$. Then, by \eqref{rr0} and \eqref{riri1}, we have that $m' = m$ and for each $i \in \llbracket 1, m \rrbracket$ that $\mathfrak{r}_i \in \bm{\mathfrak{s}} (w_i) \cup \bm{\mathfrak{t}} (w_i)$ and 
		\begin{flalign}
			\label{rii0} 
			\mathsf{R}_1 - \varepsilon^{1/2} \le \mathfrak{r}_1 \le \mathsf{R}_1 < \mathsf{R}_2 - \varepsilon^{1/2} \le \mathfrak{r}_2 \le \mathsf{R}_2 < \cdots \le \mathsf{R}_m - \varepsilon^{1/2} \le \mathfrak{r}_m \le \mathsf{R}_m < T - \varepsilon^{1/2}.
		\end{flalign}      
		
		\noindent Now, observe if $y \notin \{ \mathfrak{r}_1, \mathfrak{r}_2, \ldots , \mathfrak{r}_m \}$ then $\eta_{y+1} (x+1) = \eta_y (x)$ holds for all $x \in \mathbb{Z}$, as all paths in $\bm{\eta}$ proceed one step horizontally and one step vertically; thus, $\bm{\zeta}_{y+1} = \bm{\zeta}_y$. Similarly, if $t \notin \{ \mathsf{R}_1, \mathsf{R}_2, \ldots , \mathsf{R}_m \}$ then $\bm{\xi}_t = \bm{\xi}_{t^-}$. Hence, to show \eqref{xizeta0}, it suffices to show for each $(i, x) \in \llbracket 1, m \rrbracket \times \mathbb{Z}$ that $\xi_{\mathsf{R}_i} (x) = \zeta_{\lfloor \mathsf{R}_i / \varepsilon \rfloor} (x)$. 
		
		We do this by induction on $i \in \llbracket 0, m \rrbracket$, where we set $\mathsf{R}_0 = 0$. It holds it $i = 0$, since the boundary data of $\bm{\eta}$ matches that of $\bm{\xi}$. Hence, let us assume it holds for all $i \le m_0 - 1$ for some integer $m_0 \in \llbracket 1, m \rrbracket$ and verify it holds for $i = m_0$. By \eqref{rii0} and the inductive hypothesis (with the above discussion that $\bm{\zeta}_{y+1} = \bm{\zeta}_y$ if $y \notin \{ \mathfrak{r}_1, \mathfrak{r}_2, \ldots , \mathfrak{r}_m \}$ and $\bm{\xi}_t = \bm{\xi}_{t^-}$ if $t \notin \{ \mathsf{R}_1, \mathsf{R}_2, \ldots , \mathsf{R}_m \}$), we have $\xi_{\mathsf{R}_{m_0}^-} (x) = \xi_{\mathsf{R}_{m_0-1}} (x) = \zeta_{\lfloor \mathsf{R}_{m_0-1} / \varepsilon  \rfloor} (x) = \bm{\zeta}_{\lfloor \mathsf{R}_{m_0} / \varepsilon - 1 \rfloor} (x)$ for each $x \in \mathbb{Z}$. Next, if $\mathfrak{r}_{m_0} \in \bm{\mathfrak{s}} (w_{m_0})$, then it follows from the discussion in Section \ref{VertexColor} that $\eta_{\mathfrak{r}_{m_0}} (w_{m_0} + \mathfrak{r}_{m_0} - 1) = \eta_{\mathfrak{r}_{m_0} - 1} (w_{m_0} + \mathfrak{r}_{m_0} - 1)$, as then the path of color $\eta_{\mathfrak{r}_{m_0}-1} (w_{m_0} + \mathfrak{r}_{m_0} - 1)$ entering $(w_{m_0} + \mathfrak{r}_{m_0} - 1, \mathfrak{r}_{m_0})$ proceeds one step vertically and no steps horizontally; thus, 
		\begin{flalign}
			\label{zetarm0} 
			\zeta_{\lfloor \mathsf{R}_{m_0 - 1} / \varepsilon \rfloor} (w_{m_0}) = \zeta_{\mathfrak{r}_{m_0}-1} (w_{m_0}) = \zeta_{\mathfrak{r}_{m_0}} (w_{m_0} - 1) = \zeta_{\lfloor \mathsf{R}_{m_0} / \varepsilon \rfloor} (w_{m_0} - 1),
		\end{flalign}
		
		\noindent where in the first and last equality we again used \eqref{rii0}. By \eqref{rii0} and \eqref{riri1}, we then also have $\mathsf{R}_{m_0} \in \bm{\mathsf{S}} (w_{m_0})$, and so $\zeta_{\lfloor \mathsf{R}_{m_0-1} / \varepsilon \rfloor} (w_{m_0}) = \xi_{\mathsf{R}_{m_0-1}} (w_{m_0}) = \xi_{\mathsf{R}_{m_0}^-} (w_{m_0}) = \xi_{\mathsf{R}_{m_0}} (w_{m_0} - 1)$. This coincides with \eqref{zetarm0}; the proof that $\zeta_{\lfloor \mathsf{R}_{m_0} / \varepsilon \rfloor} (x) = \xi_{\mathsf{R}_{m_0}} (x)$ for all other $x \in \mathbb{Z}$ is entirely analogous. This verifies the statement at $i = m_0$ in this case. 
		
		If instead $\mathfrak{r}_{m_0} \in \bm{\mathfrak{t}} (w_{m_0})$, then $\eta_{\mathfrak{r}_{m_0}} (w_{m_0}+\mathfrak{r}_{m_0} + 1) = \eta_{\mathfrak{r}_{m_0} - 1} (w_{m_0} + \mathfrak{r}_{m_0} - 1)$, as then the path of color $\eta_{\mathfrak{r}_{m_0}-1} (w_{m_0} + \mathfrak{r}_{m_0} - 1)$ at $(w_{m_0} + \mathfrak{r}_{m_0} - 1, \mathfrak{r}_{m_0} - 1)$ proceeds two steps horizontally and one step vertically. So, 
		\begin{flalign} 
			\label{zetarm01} 
			\zeta_{\lfloor \mathsf{R}_{m_0-1} / \varepsilon \rfloor} (w_{m_0}) = \zeta_{\mathfrak{r}_{m_0} - 1} (w_{m_0}) = \zeta_{\mathfrak{r}_{m_0}} (w_{m_0}) = \zeta_{\lfloor \mathsf{R}_{m_0} / \varepsilon \rfloor} (w_{m_0} + 1).
		\end{flalign} 
		
		\noindent By \eqref{rii0} and \eqref{riri1}, we then also have $\mathsf{R}_{m_0} \in \bm{\mathsf{T}} (w_{m_0})$. Hence, $\zeta_{\lfloor \mathsf{R}_{m_0 - 1} / \varepsilon \rfloor} (w_{m_0}) = \xi_{\mathsf{R}_{m_0-1}} (w_{m_0}) = \xi_{\mathsf{R}_{m_0}^-} (w_{m_0}) = \xi_{\mathsf{R}_{m_0}} (w_{m_0} + 1)$. This again coincides with \eqref{zetarm01}, and the proof that $\zeta_{\lfloor \mathsf{R}_{m_0} / \varepsilon \rfloor} (x) = \xi_{\mathsf{R}_{m_0}} (x)$ for all other $x \in \mathbb{Z}$ is entirely analogous. This confirims \eqref{xizeta0}. 
		
		Thus, \eqref{xieta0} holds on $\mathscr{E} \cap \mathscr{F}$. Together with the fact that 
		\begin{flalign*}
			\mathbb{P} [\mathscr{E} \cap \mathscr{F}] & \ge 1 - 32 C_0 \varepsilon - C_0 (V - U + 4C_0 KT + 1) (T + 1) \varepsilon^{1/4} \\
			& \ge 1 - 40 C_0^2 (V - U + 1) (T + 1)^2 K \varepsilon^{1/4} \ge 1 - 80 C_0^2 (V - U + 1) (T + 1)^2 \varepsilon^{1/8},
		\end{flalign*}
		
		\noindent this establishes the proposition.
	\end{proof} 
	
	\subsection{Proof of \Cref{stcouple}} 
	
	\label{ProofCouple0} 
	
	In this section we establish \Cref{stcouple} as a quick consequence of the two lemmas below.
	
	\begin{lem}
		
		\label{rtime} 
		
		For any real number $A > 0$, there exists a constant $C = C(A) > 1$ such that, if $\bm{R} = (r_1, r_2, \ldots)$ denotes the ringing times in increasing order for an exponential clock with parameter $A$, then the following two statements hold.
		
		\begin{enumerate}
			\item For any real numbers $T \ge 0$ and $K \ge 1$, we have $\mathbb{P} [r_{\lfloor CK(T+1) \rfloor} \ge T] \ge 1 - Ce^{-K(T+1)}$. 
			\item For any real numbers $0 < \delta \le 1 \le B$, we have $\mathbb{P} \big[ \min_{1 \le i \le B} (r_{i+1} - r_i) \ge \delta \big] \ge 1 - AB \delta$.
		\end{enumerate}
	\end{lem} 
	
	\begin{proof}
		
		First observe that the event on which $r_{\lfloor C(T+1) \rceil} \ge T$ is that on which the sum of $\big\lfloor CK(T+1) \big\rfloor$ independent exponential random variables is at least $T$. By a Chernoff bound, this is at least $1 - C e^{-K(T+1)}$ if $C = C(A) > 1$ is sufficiently large, which verifies the first statement of the lemma. Further observe for any integer $i \ge 1$ that, since $r_{i+1} - r_i$ is an exponential random variable of parameter $A$, we have $\mathbb{P}[r_{i+1} - r_i \le \delta] = 1 - e^{- A\delta} \le A\delta$. This, together with a union bound over $i \in \llbracket 1, B \rrbracket$, implies the second statement of the lemma.
	\end{proof}

	\begin{lem} 
		
		\label{gcouplee} 
		
		Let $A > 0$ and $\delta \in (0, 1)$ be real numbers; let $\mathfrak{g} \in \mathbb{Z}_{\ge 0}$ denote a geometric random variable with $\mathbb{P}[\mathfrak{g}=k] = A \delta (1 - A \delta)^k$ for each $k \in \mathbb{Z}_{\ge 0}$; and let $\mathfrak{e} \in \mathbb{R}_{\ge 0}$ denote an exponential random variable with $\mathbb{P}[\mathfrak{e} > x] = e^{-Ax}$ for each $x \in \mathbb{R}_{\ge 0}$. It is possible to couple $\mathfrak{e}$ and $\mathfrak{g}$ such that
		\begin{flalign*}
			\mathbb{P} \big[ \mathfrak{g} = \lfloor \delta^{-1} \mathfrak{e} \rfloor\big] \ge 1 - 12 A \delta.
		\end{flalign*}
		
	\end{lem} 
	
	\begin{proof}
		
		We may assume that $\delta < (2A)^{-1}$, for otherwise $1 - 12 A \delta < 0$. It suffices to show that
		\begin{flalign}
			\label{esdelta}
			e^{-A^2 S \delta^2} \cdot \mathbb{P} \big[ S \delta \le \mathfrak{e} < (S+1) \delta \big] \le  \mathbb{P} [\mathfrak{g} = S] \le (1 + 2 A \delta) \cdot \mathbb{P} \big[ S \delta \le \mathfrak{e} < (S+1) \delta \big],
		\end{flalign}
		
		\noindent for any integer $S \ge 0$, or equivalently that
		\begin{flalign}
			\label{esdelta2} 
			 e^{-A^2 S \delta^2} \le \displaystyle\frac{A \delta  (1 - A \delta)^S}{e^{-A S \delta} (1 - e^{-A \delta})} \le 1 + 2 A \delta.
		\end{flalign}
		
		\noindent Indeed, given \eqref{esdelta}, it would follow that 
		\begin{flalign*}
			\displaystyle\sum_{S = 0}^{\infty} \Big| \mathbb{P} \big[ S \delta \le \mathfrak{e} < (S+1) \delta \big] - \mathbb{P} [\mathfrak{g} = S] \Big| & \le 2 \displaystyle\sum_{S = 0}^{\infty} ( 1 - e^{-A^2 S \delta^2} + A\delta) \cdot e^{-A S \delta} (1 - e^{-A \delta}) \\
			& \le 2A^2 \delta^2 \displaystyle\sum_{S = 0}^{\infty} (A S \delta + 1) e^{-AS \delta} \\ 
			& \le \displaystyle\frac{2A^2 \delta^2}{1 - e^{-A \delta}} + \displaystyle\frac{2A^3 \delta^3}{(1 - e^{-A \delta})^2} \le 12 A \delta,
		\end{flalign*}
		
		\noindent where in the first equality we used the explicit probability distribution for $\mathfrak{e}$; in the second we used the facts that $1 - e^{-A^2 S \delta^2} \le A^2 S \delta^2$ and that $1 - e^{-A \delta} \le A \delta$; in the third we bounded the sums $\sum_{i = 0}^{\infty} r^i = (1 - r)^{-1}$ and $\sum_{i = 0}^{\infty} i r^i = r (1 - r)^{-2} \le (1 - r)^{-2}$ at $r = e^{-A \delta} \in (0, 1)$; and in the fourth we bounded $1 - e^{-A \delta} \ge A \delta / 2$ (as $\delta < (2A)^{-1}$). This implies that it is possible to couple $\mathfrak{e}$ and $\mathfrak{g}$ so that $\mathfrak{g} \le \delta^{-1} \mathfrak{e} < \mathfrak{g} + 1$, or equivalently that $\mathfrak{g} = \lfloor \delta^{-1} \mathfrak{e} \rfloor$, with probability at least $1 - 12 A \delta$.
		
		It therefore remains to confirm \eqref{esdelta2}. To that end, since $1 - e^{-A \delta} \le A \delta$ and $1 - e^{-A \delta} \ge A \delta (1 - A\delta) \ge A \delta (1 + 2 A \delta)^{-1}$ (the latter as $\delta < (2A)^{-1}$), observe that 
		\begin{flalign}
			\label{adelta0} 
			1 \le \displaystyle\frac{A\delta}{1 - e^{-A \delta}} \le 1 + 2 A \delta.
		\end{flalign}
		
		\noindent We also have since $\log (1-x) \le -x$ for each $x \ge 0$, since $\log (1 - x) \ge -x (1 + x)$ for $x \in ( 0, 1 / 2 )$, and since $\delta < (2A)^{-1}$ that
		\begin{flalign}
			\label{asdelta0}
			e^{-A^2 S \delta^2}  \le e^{AS \delta} (1 - A \delta)^S \le 1.
		\end{flalign}
		
		\noindent Combining \eqref{adelta0} and \eqref{asdelta0} yields \eqref{esdelta2} and thus the lemma.	
	\end{proof}

	\begin{proof}[Proof of \Cref{stcouple}]
		
		Let $C_0$ denote the constant $C$ from \Cref{rtime}; we will assume throughout this proof that $\varepsilon < C_0^{-4} (T + 1)^{-4}$, for otherwise $1 - C (V_0 - U_0 + 1) (T + 1) \varepsilon^{1/4} \le 0$ for any $C > C_0$. Let us first bound the cardinalities of $\bm{\mathsf{S}} (x) \cap [0, T]$ and $\bm{T} (x) \cap [0, T]$, with high probability. To that end, observe by the first part of \Cref{rtime} (with $A \in \{ L, R \}$ and $K = C_0^{-1} (T + 1)^{-1} \varepsilon^{-1/4} \ge 1$) that there exist constants $c > 0$ and $C_1 > 1$ such that $\mathbb{P} \big[ \mathsf{S}_{\lfloor \varepsilon^{-1/4} \rfloor} (x) \ge T \big] \ge 1 - c^{-1} e^{- c \varepsilon^{-1/4}} \ge 1 - C_1 \varepsilon$, and similarly $\mathbb{P} \big[ \mathsf{T}_{\lfloor \varepsilon^{-1/4} \rfloor} (x) \ge T \big] \ge 1 - C_1 \varepsilon$, both hold for any $x \in \mathbb{Z}$. Together with a union bound, it follows that $\mathbb{P} [\mathscr{A}_1] \ge 1 - 2C_1 \varepsilon (V_0 - U_0 + 1)$, where 
		\begin{flalign}
			\label{a1event}
			\mathscr{A}_1 = \bigcap_{x = \lceil U_0 \rceil}^{\lfloor V_0 \rfloor} \Big\{ \# \big\{  \bm{\mathsf{S}} (x) \cap [0, T] \big\} \le \varepsilon^{-1/4} \Big\} \cap \Big\{ \# \big\{ \bm{\mathsf{T}} (x) \cap [0, T] \big\} \le \varepsilon^{-1/4}  \Big\}.
		\end{flalign}
		
		\noindent We may therefore restrict to the event $\mathscr{A}_1$ in what follows.
		
		On $\mathscr{A}_1$, there are at most $\varepsilon^{-1/4}$ entries in any $\bm{\mathsf{S}} (x) \cap [0, T]$ or $\bm{\mathsf{T}} (x) \cap [0, T]$; their differences are exponential random variables with parameters $L$ and $R$, respectively. Moreover, the differences between consecutive entries of any $\bm{\mathfrak{s}} (x)$ and $\bm{\mathfrak{t}} (x)$ are geometric random variables of parameters $b_1 = 1 - \varepsilon L$ and $b_2 = 1 - \varepsilon R$, respectively. Hence, fixing an integer $x \in \llbracket U, V \rrbracket$ and an index $(\mathsf{R}, \mathfrak{r}) \in \big\{ (\mathsf{S}, \mathfrak{s}), (\mathsf{T}, \mathfrak{t}) \big\}$, and applying \Cref{gcouplee} to the at most $\varepsilon^{-1/4}$ differences in $\bm{\mathsf{R}} (x)$, it follows that we may couple $\bm{\mathsf{R}} (x) \cap [0, T] = \big( \mathsf{R}_1 (x), \mathsf{R}_2 (x), \ldots , \mathsf{\mathsf{R}}_m (x) \big)$ with $\bm{\mathfrak{r}} (x) \cap \llbracket 1, T \varepsilon^{-1} \rrbracket = \big( \mathfrak{r}_1 (x), \mathfrak{r}_2 (x), \ldots , \mathfrak{r}_{m'} (x) \big)$ so that with probability $1 - 12 (R+L) \varepsilon^{3/4}$ we have $m \le m'$ and $\varepsilon \mathfrak{r}_i (x) \le \mathsf{R}_i (x) \le \varepsilon \big( \mathfrak{r}_i (x) + \varepsilon^{-1/4} \big) \le \varepsilon \mathfrak{r}_i (x) + \varepsilon^{1/2}$, for each $i \in \llbracket 1, m \rrbracket$. After increasing $C_0$ if necessary, we further have that $ \varepsilon \mathfrak{r}_{m'} (x) < T - \varepsilon^{1/2}$ with probability at least $1 - C_0 \varepsilon^{1/2}$ (since the probability that $\mathfrak{r}_{m'} (x)$ lies in any fixed interval at size $\varepsilon^{-1/2} + 1$ is at most $(R+L) \varepsilon \cdot (\varepsilon^{-1/2} + 1)  = 2 (R+L) \varepsilon^{1/2}$), and so $m = m'$ and $\varepsilon \mathfrak{r}_i (x) \le \mathsf{R}_i (x) \le \varepsilon \mathfrak{r}_i (x) + \varepsilon^{1/2} < T$ for all $i \in \llbracket 1, m \rrbracket$. Applying a union bound over $x \in \llbracket U_0, V_0 \rrbracket$ gives the first statement of the lemma. 
		
		To establish the second, observe that the law of $\bigcup_{x = \lceil U_0 \rceil}^{\lfloor V_0 \rfloor} \big( \bm{\mathsf{S}} (x) \cup \bm{\mathsf{T}} (x) \big)$ coincides with that of the ringing times of an exponential clock of rate $(R + L) \cdot \big( \lfloor V_0 \rfloor - \lceil U_0 \rceil + 1 \big) \le (V_0 - U_0 + 1) (R + L)$. Hence the second part of \Cref{rtime} implies, for any real number $B \ge 1$, that with probability at least $1 - 2 (V_0 - U_0 + 1) (R + L) B \varepsilon^{1/2}$ we have $\big| \mathsf{R}_i (x) - \mathsf{R}_{i'}' (x') \big| > 2 \varepsilon^{1/2}$, for any distinct triples $(\mathsf{R}, i, x), (\mathsf{R}', i', x') \in \{ \mathsf{S}, \mathsf{T} \} \times \llbracket 1, B \rrbracket \times \llbracket U_0, V_0 \rrbracket$. Applying this at $B = \varepsilon^{-1/4}$, using our restriction to $\mathscr{A}_1$, and applying a union bound then yields the second part of the lemma.
	\end{proof}


\begin{thebibliography}{10}
		
		\bibitem{CSSVMP}
		A.~Aggarwal.
		\newblock Convergence of the stochastic six-vertex model to the {ASEP}:
		stochastic six-vertex model and {ASEP}.
		\newblock {\em Math. Phys. Anal. Geom.}, 20(2):Paper No. 3, 20, 2017.
		
		\bibitem{CFSSVM}
		A.~Aggarwal.
		\newblock Current fluctuations of the stationary {ASEP} and six-vertex model.
		\newblock {\em Duke Math. J.}, 167(2):269--384, 2018.
		
		\bibitem{LSLSSSVM}
		A.~Aggarwal.
		\newblock Limit shapes and local statistics for the stochastic six-vertex
		model.
		\newblock {\em Comm. Math. Phys.}, 376(1):681--746, 2020.
		
		\bibitem{SSE}
		A.~Aggarwal, A.~Borodin, and A.~Bufetov.
		\newblock Stochasticization of solutions to the {Y}ang-{B}axter equation.
		\newblock {\em Ann. Henri Poincar\'{e}}, 20(8):2495--2554, 2019.
		
		\bibitem{CFVMSF}
		A.~Aggarwal, A.~Borodin, and M.~Wheeler.
		\newblock Colored fermionic vertex models and symmetric functions.
		\newblock {\em Comm. Amer. Math. Soc.}, 3:400--630, 2023.
		
		\bibitem{PP}
		A.~Aggarwal, I.~Corwin, and M.~Hegde.
		\newblock In preparation.
		
		\bibitem{SCLE}
		A.~Aggarwal and J.~Huang.
		\newblock Strong characterization for the {A}iry line ensemble.
		\newblock Preprint, ar{X}iv:2308.11908.
		
		\bibitem{CIPSR}
		A.~Aggarwal, M.~Nicoletti, and L.~Petrov.
		\newblock Colored interacting particle systems on the ring: {S}tationary
		measures from {Y}ang--{B}axter equation.
		\newblock Preprint, ar{X}iv:2309.11865.
		
		\bibitem{SVMHH}
		G.~Barraquand, A.~Borodin, I.~Corwin, and M.~Wheeler.
		\newblock Stochastic six-vertex model in a half-quadrant and half-line open
		asymmetric simple exclusion process.
		\newblock {\em Duke Math. J.}, 167(13):2457--2529, 2018.
		
		\bibitem{EHSP}
		G.~Barraquand, I.~Corwin, and S.~Das.
		\newblock {KPZ} exponents for the half-space log-gamma polymer.
		\newblock Preprint, ar{X}iv:2310.10019.
		
		\bibitem{STEL}
		G.~Barraquand, I.~Corwin, and E.~Dimitrov.
		\newblock Spatial tightness at the edge of {G}ibbsian line ensembles.
		\newblock {\em Comm. Math. Phys.}, 397(3):1309--1386, 2023.
		
		\bibitem{GQ}
		Y.~Baryshnikov.
		\newblock G{UE}s and queues.
		\newblock {\em Probab. Theory Related Fields}, 119(2):256--274, 2001.
		
		\bibitem{TSTECA}
		V.~V. Bazhanov.
		\newblock Trigonometric solutions of triangle equations and classical {L}ie
		algebras.
		\newblock {\em Phys. Lett. B}, 159(4-6):321--324, 1985.
		
		\bibitem{SRF}
		A.~Borodin.
		\newblock On a family of symmetric rational functions.
		\newblock {\em Adv. Math.}, 306:973--1018, 2017.
		
		\bibitem{SSVMP}
		A.~Borodin, A.~Bufetov, and M.~Wheeler.
		\newblock Between the stochastic six vertex model and {H}all-{L}ittlewood
		processes.
		\newblock Preprint, ar{X}iv:1611.09486.
		
		\bibitem{P}
		A.~Borodin and I.~Corwin.
		\newblock Macdonald processes.
		\newblock {\em Probab. Theory Related Fields}, 158(1-2):225--400, 2014.
		
		\bibitem{DT}
		A.~Borodin and I.~Corwin.
		\newblock Discrete time {$q$}-{TASEP}s.
		\newblock {\em Int. Math. Res. Not. IMRN}, (2):499--537, 2015.
		
		\bibitem{SSVM}
		A.~Borodin, I.~Corwin, and V.~Gorin.
		\newblock Stochastic six-vertex model.
		\newblock {\em Duke Math. J.}, 165(3):563--624, 2016.
		
		\bibitem{AGRS}
		A.~Borodin and P.~L. Ferrari.
		\newblock Anisotropic growth of random surfaces in {$2+1$} dimensions.
		\newblock {\em Comm. Math. Phys.}, 325(2):603--684, 2014.
		
		\bibitem{SVMP}
		A.~Borodin, V.~Gorin, and M.~Wheeler.
		\newblock Shift-invariance for vertex models and polymers.
		\newblock {\em Proc. Lond. Math. Soc. (3)}, 124(2):182--299, 2022.
		
		\bibitem{NNDP}
		A.~Borodin and L.~Petrov.
		\newblock Nearest neighbor {M}arkov dynamics on {M}acdonald processes.
		\newblock {\em Adv. Math.}, 300:71--155, 2016.
		
		\bibitem{HSVMSRF}
		A.~Borodin and L.~Petrov.
		\newblock Higher spin six vertex model and symmetric rational functions.
		\newblock {\em Selecta Math. (N.S.)}, 24(2):751--874, 2018.
		
		\bibitem{CSVMPL}
		A.~Borodin and M.~Wheeler.
		\newblock Observables of coloured stochastic vertex models and their polymer
		limits.
		\newblock {\em Probab. Math. Phys.}, 1(1):205--265, 2020.
		
		\bibitem{SP}
		A.~Borodin and M.~Wheeler.
		\newblock Spin {$q$}-{W}hittaker polynomials.
		\newblock {\em Adv. Math.}, 376:Paper No. 107449, 50, 2021.
		
		\bibitem{CSVMST}
		A.~Borodin and M.~Wheeler.
		\newblock Colored stochastic vertex models and their spectral theory.
		\newblock {\em Ast\'{e}risque}, (437):ix+225, 2022.
		
		\bibitem{STRR}
		G.~Bosnjak and V.~V. Mangazeev.
		\newblock Construction of {$R$}-matrices for symmetric tensor representations
		related to {$U_q(\widehat{sl_n})$}.
		\newblock {\em J. Phys. A}, 49(49):495204, 19, 2016.
		
		\bibitem{RFSVM}
		A.~Bufetov, M.~Mucciconi, and L.~Petrov.
		\newblock Yang-{B}axter random fields and stochastic vertex models.
		\newblock {\em Adv. Math.}, 388:Paper No. 107865, 94, 2021.
		
		\bibitem{FSSF}
		A.~Bufetov and L.~Petrov.
		\newblock Yang-{B}axter field for spin {H}all-{L}ittlewood symmetric functions.
		\newblock {\em Forum Math. Sigma}, 7:Paper No. e39, 70, 2019.
		
		\bibitem{SLM}
		O.~Busani, T.~Sepp\"{a}l\"{a}inen, and E.~Sorensen.
		\newblock Scaling limit of multi-type invariant measures via the directed
		landscape.
		\newblock Preprint, ar{X}iv:2310.09824.
		
		\bibitem{SHG}
		O.~Busani, T.~Sepp\"{a}l\"{a}inen, and E.~Sorensen.
		\newblock The stationary horizon and semi-infinite geodesics in the directed
		landscape.
		\newblock Preprint, ar{X}iv:2203.13242.
		
		\bibitem{SDHR}
		L.~Cantini.
		\newblock Algebraic {B}ethe ansatz for the two species {ASEP} with different
		hopping rates.
		\newblock {\em J. Phys. A}, 41(9):095001, 16, 2008.
		
		\bibitem{LCDP}
		Z.~Chen, J.~de~Gier, I.~Hiki, T.~Sasamoto, and M.~Usui.
		\newblock Limiting current distribution for a two species asymmetric exclusion
		process.
		\newblock {\em Comm. Math. Phys.}, 395(1):59--142, 2022.
		
		\bibitem{TFSVM}
		I.~Corwin and E.~Dimitrov.
		\newblock Transversal fluctuations of the {ASEP}, stochastic six vertex model,
		and {H}all-{L}ittlewood {G}ibbsian line ensembles.
		\newblock {\em Comm. Math. Phys.}, 363(2):435--501, 2018.
		
		\bibitem{ECT}
		I.~Corwin, P.~Ghosal, and A.~Hammond.
		\newblock K{PZ} equation correlations in time.
		\newblock {\em Ann. Probab.}, 49(2):832--876, 2021.
		
		\bibitem{PLE}
		I.~Corwin and A.~Hammond.
		\newblock Brownian {G}ibbs property for {A}iry line ensembles.
		\newblock {\em Invent. Math.}, 195(2):441--508, 2014.
		
		\bibitem{LE}
		I.~Corwin and A.~Hammond.
		\newblock K{PZ} line ensemble.
		\newblock {\em Probab. Theory Related Fields}, 166(1-2):67--185, 2016.
		
		\bibitem{TPT}
		I.~Corwin, K.~Matveev, and L.~Petrov.
		\newblock The {$q$}-{H}ahn {P}ush{TASEP}.
		\newblock {\em Int. Math. Res. Not. IMRN}, (3):2210--2249, 2021.
		
		\bibitem{TCF}
		I.~Corwin, N.~O'Connell, T.~Sepp\"{a}l\"{a}inen, and N.~Zygouras.
		\newblock Tropical combinatorics and {W}hittaker functions.
		\newblock {\em Duke Math. J.}, 163(3):513--563, 2014.
		
		\bibitem{SHSVML}
		I.~Corwin and L.~Petrov.
		\newblock Stochastic higher spin vertex models on the line.
		\newblock {\em Comm. Math. Phys.}, 343(2):651--700, 2016.
		
		\bibitem{SWP}
		I.~Corwin, T.~Sepp\"{a}l\"{a}inen, and H.~Shen.
		\newblock The strict-weak lattice polymer.
		\newblock {\em J. Stat. Phys.}, 160(4):1027--1053, 2015.
		
		\bibitem{TL}
		D.~Dauvergne, J.~Ortmann, and B.~Vir\'{a}g.
		\newblock The directed landscape.
		\newblock {\em Acta Math.}, 229(2):201--285, 2022.
		
		\bibitem{AMT}
		N.~Elkies, G.~Kuperberg, M.~Larsen, and J.~Propp.
		\newblock Alternating-sign matrices and domino tilings. {II}.
		\newblock {\em J. Algebraic Combin.}, 1(3):219--234, 1992.
		
		\bibitem{BLPM}
		S.~Ganguly, M.~Hedge, and L.~Zhang.
		\newblock Brownian bridge limit of path measures in the upper tail of {KPZ}
		models.
		\newblock Preprint, ar{X}iv:2311.12009.
		
		\bibitem{UTEL}
		S.~Ganguly and M.~Hegde.
		\newblock Sharp upper tail estimates and limit shapes for the {KPZ} equation
		via the tangent method.
		\newblock Preprint, ar{X}iv:2208.08922.
		
		\bibitem{BHS}
		G.~Gasper and M.~Rahman.
		\newblock {\em Basic hypergeometric series}, volume~96 of {\em Encyclopedia of
			Mathematics and its Applications}.
		\newblock Cambridge University Press, Cambridge, second edition, 2004.
		\newblock With a foreword by Richard Askey.
		
		\bibitem{SVMRS}
		L.-H. Gwa and H.~Spohn.
		\newblock Six-vertex model, roughened surfaces, and an asymmetric spin
		{H}amiltonian.
		\newblock {\em Phys. Rev. Lett.}, 68(6):725--728, 1992.
		
		\bibitem{PNA}
		J.~Haglund, S.~Mason, and J.~Remmel.
		\newblock Properties of the nonsymmetric {R}obinson-{S}chensted-{K}nuth
		algorithm.
		\newblock {\em J. Algebraic Combin.}, 38(2):285--327, 2013.
		
		\bibitem{MCPWP}
		A.~Hammond.
		\newblock Modulus of continuity of polymer weight profiles in {B}rownian last
		passage percolation.
		\newblock {\em Ann. Probab.}, 47(6):3911--3962, 2019.
		
		\bibitem{PRPW}
		A.~Hammond.
		\newblock A patchwork quilt sewn from {B}rownian fabric: regularity of polymer
		weight profiles in {B}rownian last passage percolation.
		\newblock {\em Forum Math. Pi}, 7:e2, 69, 2019.
		
		\bibitem{ERDP}
		A.~Hammond.
		\newblock Exponents governing the rarity of disjoint polymers in {B}rownian
		last passage percolation.
		\newblock {\em Proc. Lond. Math. Soc. (3)}, 120(3):370--433, 2020.
		
		\bibitem{RLE}
		A.~Hammond.
		\newblock Brownian regularity for the {A}iry line ensemble, and multi-polymer
		watermelons in {B}rownian last passage percolation.
		\newblock {\em Mem. Amer. Math. Soc.}, 277(1363):v+133, 2022.
		
		\bibitem{NIPML}
		T.~E. Harris.
		\newblock Nearest-neighbor {M}arkov interaction processes on multidimensional
		lattices.
		\newblock {\em Adv. Math.}, 9:66--89, 1972.
		
		\bibitem{ASVPGM}
		T.~E. Harris.
		\newblock Additive set-valued {M}arkov processes and graphical methods.
		\newblock {\em Ann. Probability}, 6(3):355--378, 1978.
		
		\bibitem{BCFHS}
		J.~He.
		\newblock Boundary current fluctuations for the half space {ASEP} and six
		vertex model.
		\newblock Preprint, ar{X}iv:2303.16335.
		
		\bibitem{SIHSIM}
		J.~He.
		\newblock Shift invariance of half space integrable models.
		\newblock Preprint, ar{X}iv:2205.13029.
		
		\bibitem{QMGS}
		M.~Jimbo.
		\newblock Quantum {$R$} matrix for the generalized {T}oda system.
		\newblock {\em Comm. Math. Phys.}, 102(4):537--547, 1986.
		
		\bibitem{RTT}
		W.~Jockusch, J.~Propp, and P.~Shor.
		\newblock Random domino tilings and the arctic circle theorem.
		\newblock Preprint, ar{X}iv:9801068, 1998.
		
		\bibitem{DPGDP}
		K.~Johansson.
		\newblock Discrete polynuclear growth and determinantal processes.
		\newblock {\em Comm. Math. Phys.}, 242(1-2):277--329, 2003.
		
		\bibitem{SLNP}
		S.~G.~G. Johnston and N.~O'Connell.
		\newblock Scaling limits for non-intersecting polymers and {W}hittaker
		measures.
		\newblock {\em J. Stat. Phys.}, 179(2):354--407, 2020.
		
		\bibitem{DSGI}
		M.~Kardar, G.~Parisi, and Y.-C. Zhang.
		\newblock Dynamic scaling of growing interfaces.
		\newblock {\em Phys. Rev. Lett.}, 56(9):889, 1986.
		
		\bibitem{TC}
		A.~N. Kirillov.
		\newblock Introduction to tropical combinatorics.
		\newblock In {\em Physics and combinatorics, 2000 ({N}agoya)}, pages 82--150.
		World Sci. Publ., River Edge, NJ, 2001.
		
		\bibitem{ERT}
		P.~P. Kulish, N.~Y. Reshetikhin, and E.~K. Sklyanin.
		\newblock Yang-{B}axter equations and representation theory. {I}.
		\newblock {\em Lett. Math. Phys.}, 5(5):393--403, 1981.
		
		\bibitem{SRM}
		A.~Kuniba, V.~V. Mangazeev, S.~Maruyama, and M.~Okado.
		\newblock Stochastic {$R$} matrix for {$U_q(A_n^{(1)})$}.
		\newblock {\em Nuclear Phys. B}, 913:248--277, 2016.
		
		\bibitem{CSDSVM}
		Y.~Lin.
		\newblock Classification of stationary distributions for the stochastic vertex
		models.
		\newblock {\em Electron. J. Probab.}, 28:Paper No. 1, 2023.
		
		\bibitem{ESVM}
		V.~V. Mangazeev.
		\newblock On the {Y}ang-{B}axter equation for the six-vertex model.
		\newblock {\em Nuclear Phys. B}, 882:70--96, 2014.
		
		\bibitem{DFA}
		S.~Mason.
		\newblock A decomposition of {S}chur functions and an analogue of the
		{R}obinson-{S}chensted-{K}nuth algorithm.
		\newblock {\em S\'{e}m. Lothar. Combin.}, 57:Art. B57e, 24, 2006/08.
		
		\bibitem{TFP}
		K.~Matetski, J.~Quastel, and D.~Remenik.
		\newblock The {KPZ} fixed point.
		\newblock {\em Acta Math.}, 227(1):115--203, 2021.
		
		\bibitem{RCRP}
		K.~Matveev and L.~Petrov.
		\newblock {$q$}-randomized {R}obinson-{S}chensted-{K}nuth correspondences and
		random polymers.
		\newblock {\em Ann. Inst. Henri Poincar\'{e} D}, 4(1):1--123, 2017.
		
		\bibitem{SPDQ}
		M.~Mucciconi and L.~Petrov.
		\newblock Spin {$q$}-{W}hittaker polynomials and deformed quantum {T}oda.
		\newblock {\em Comm. Math. Phys.}, 389(3):1331--1416, 2022.
		
		\bibitem{IDL}
		M.~Nica.
		\newblock Intermediate disorder limits for multi-layer semi-discrete directed
		polymers.
		\newblock {\em Electron. J. Probab.}, 26:Paper No. 62, 50, 2021.
		
		\bibitem{TCBA}
		M.~Noumi and Y.~Yamada.
		\newblock Tropical {R}obinson-{S}chensted-{K}nuth correspondence and birational
		{W}eyl group actions.
		\newblock In {\em Representation theory of algebraic groups and quantum
			groups}, volume~40 of {\em Adv. Stud. Pure Math.}, pages 371--442. Math. Soc.
		Japan, Tokyo, 2004.
		
		\bibitem{PTRW}
		N.~O'Connell.
		\newblock A path-transformation for random walks and the {R}obinson-{S}chensted
		correspondence.
		\newblock {\em Trans. Amer. Math. Soc.}, 355(9):3669--3697, 2003.
		
		\bibitem{DPQL}
		N.~O'Connell.
		\newblock Directed polymers and the quantum {T}oda lattice.
		\newblock {\em Ann. Probab.}, 40(2):437--458, 2012.
		
		\bibitem{MSE}
		N.~O'Connell and J.~Warren.
		\newblock A multi-layer extension of the stochastic heat equation.
		\newblock {\em Comm. Math. Phys.}, 341(1):1--33, 2016.
		
		\bibitem{AT}
		N.~O'Connell and M.~Yor.
		\newblock Brownian analogues of {B}urke's theorem.
		\newblock {\em Stochastic Process. Appl.}, 96(2):285--304, 2001.
		
		\bibitem{CFPALG}
		A.~Okounkov and N.~Reshetikhin.
		\newblock Correlation function of {S}chur process with application to local
		geometry of a random 3-dimensional {Y}oung diagram.
		\newblock {\em J. Amer. Math. Soc.}, 16(3):581--603, 2003.
		
		\bibitem{SIDP}
		M.~Pr\"{a}hofer and H.~Spohn.
		\newblock Scale invariance of the {PNG} droplet and the {A}iry process.
		\newblock \emph{J. Statist. Phys.} 108(5):1071--1106, 2002.
		\newblock Dedicated to David Ruelle and Yasha Sinai on the occasion of their
		65th birthdays.
		
		\bibitem{CEPEFP}
		J.~Quastel and S.~Sarkar.
		\newblock Convergence of exclusion processes and the {KPZ} equation to the
		{KPZ} fixed point.
		\newblock {\em J. Amer. Math. Soc.}, 36(1):251--289, 2023.
		
		\bibitem{NBPP}
		H.~Rost.
		\newblock Non-equilibrium behaviour of a many particle process: Density profile
		and local equilibria.
		\newblock {\em Z. Wahrsch. Verw. Gebiete}, 58(1):41--53, 1981.
		
		\bibitem{ERDM}
		T.~Sasamoto and M.~Wadati.
		\newblock Exact results for one-dimensional totally asymmetric diffusion
		models.
		\newblock {\em J. Phys. A}, 31(28):6057--6071, 1998.
		
		\bibitem{SDPBC}
		T.~Sepp\"{a}l\"{a}inen.
		\newblock Scaling for a one-dimensional directed polymer with boundary
		conditions.
		\newblock {\em Ann. Probab.}, 40(1):19--73, 2012.
		
		\bibitem{II}
		J.~Warren.
		\newblock Dyson's {B}rownian motions, intertwining and interlacing.
		\newblock {\em Electron. J. Probab.}, 12:no. 19, 573--590, 2007.
		
		\bibitem{BRLE}
		X.~Wu.
		\newblock Brownian regularity for the {KPZ} line ensemble.
		\newblock Preprint, ar{X}iv:2106.08052.
		
		\bibitem{EL}
		X.~Wu.
		\newblock The {KPZ} equation and the directed landscape.
		\newblock Preprint, ar{X}iv:2301.00547.
		
		\bibitem{TDLEE}
		X.~Wu.
		\newblock Tightness of discrete {G}ibbsian line ensembles with exponential
		interaction {H}amiltonians.
		\newblock {\em Ann. Inst. Henri Poincar\'{e} Probab. Stat.}, 59(4):2106--2150,
		2023.
		
		\bibitem{SMHSVMS}
		Z.~Yang.
		\newblock Stationary measures for higher spin vertex models on a strip.
		\newblock ar{X}iv:2309.04897.
		
	\end{thebibliography}
\end{document}